\documentclass[11pt]{amsart} 
\usepackage{amscd,amssymb,amsxtra}
\usepackage{mathrsfs}  
\DeclareSymbolFontAlphabet{\mathrsfs}{rsfs}
\usepackage[mathscr]{eucal} 
\usepackage{comment}
\usepackage{color}
\usepackage{enumitem} 
\usepackage{marginnote}
\usepackage{quiver}

\usepackage{xspace}

\makeatletter
\let\@secnumfont\bfseries
\def\section{\@startsection{section}{1}%
  \z@{4\linespacing\@plus\linespacing}{\linespacing}%
  {\bfseries\centering}}
\def\introsection{\@startsection{section}{1}%
  \z@{3\linespacing\@plus\linespacing}{\linespacing}%
  {\bfseries\centering}}
\def\subsection{\@startsection{subsection}{2}%
   \z@{1.25\linespacing\@plus.7\linespacing}{.5\linespacing}%
   {\normalfont\bfseries}}
\def\subsectionsinline{\def\subsection{\@startsection{subsection}{2}%
  \z@{1\linespacing\@plus.7\linespacing}{-.5em}%
  {\normalfont\bfseries}}}

\def\thmqed{\qed\vskip1.5em}

\newenvironment{qedequation}{%
   \pushQED{\qed}%
   \incr@eqnum
   \mathdisplay@push
   \st@rredfalse \global\@eqnswtrue
   \mathdisplay{equation}%
}{%
   \endmathdisplay{equation}%
   \mathdisplay@pop
   \ignorespacesafterend
   \popQED\@endpefalse
}

\makeatother

\theoremstyle{definition}
\newtheorem{definition}[equation]{Definition}
\newtheorem{defn}[equation]{Definition}
\newtheorem{example}[equation]{Example}

\newtheorem{construction}[equation]{Construction}
\newtheorem{constr}[equation]{Construction}

\newtheorem{convention}[equation]{Convention}

\newtheorem*{definition*}{Definition}
\newtheorem*{example*}{Example}
\newtheorem*{problem*}{\color{blue}Problem}
\newtheorem*{probsec*}{\color{blue}Problem}
\newtheorem*{exercise*}{Exercise}
\newtheorem*{question*}{\color{blue}Question}
\newtheorem*{project*}{\color{blue}Project}
\newtheorem*{construction*}{Construction}
\newtheorem*{notation*}{Notation}

\theoremstyle{remark}

\newtheorem{remark}[equation]{Remark}

\newtheorem*{note*}{Note}
\newtheorem*{remark*}{Remark}
\newtheorem*{data*}{Data}

\theoremstyle{plain}
\newtheorem{theorem}[equation]{Theorem}
\newtheorem{thm}[equation]{Theorem}
\newtheorem{corollary}[equation]{Corollary}
\newtheorem{cor}[equation]{Corollary}
\newtheorem{lemma}[equation]{Lemma}
\newtheorem{lem}[equation]{Lemma}
\newtheorem{proposition}[equation]{Proposition}
\newtheorem{prop}[equation]{Proposition}

\newtheorem{assumption}[equation]{Assumption}

\newtheorem*{theorem*}{Theorem}
\newtheorem*{corollary*}{Corollary}
\newtheorem*{lemma*}{Lemma}
\newtheorem*{proposition*}{Proposition}
\newtheorem*{conjecture*}{Conjecture}
\newtheorem*{claim*}{Claim}
\newtheorem*{proposal*}{Proposal}
\newtheorem*{conclusion*}{Conclusion}
\newtheorem*{hypothesis*}{Hypothesis}
\newtheorem*{assumption*}{Assumption}

\newenvironment{proof*}[1][\proofname]{
  \begin{proof}[#1]}{  \renewcommand\qedsymbol{\relax}
\end{proof}}

\numberwithin{equation}{section}

\definecolor{refkey}{rgb}{0,.6,.4}

\renewcommand{\:}{\colon}
\renewcommand{\AA}{{\mathbb A}}

\DeclareMathOperator{\Aut}{Aut}
\newcommand{\CC}{{\mathbb C}}
\newcommand{\CP}{{\mathbb C\mathbb P}}

\newcommand{\HH}{{\mathbb H}}
\DeclareMathOperator{\Hom}{Hom}
\DeclareMathOperator{\id}{id}

\newcommand{\PP}{{\mathbb P}}
\DeclareMathOperator{\pt}{pt}

\newcommand{\RP}{{\mathbb R\mathbb P}}
\newcommand{\RR}{{\mathbb R}}
\newcommand{\TT}{\mathbb T}
\DeclareMathOperator{\Spin}{Spin}

\DeclareMathOperator{\tr}{tr}

\newcommand{\ZZ}{{\mathbb Z}}

\newcommand{\chiup}{\raise.5ex\hbox{$\chi$}}
\newcommand{\cir}{S^1}

\newcommand{\inv}{^{-1}}
\DeclareRobustCommand{\mstrut}{^{\vphantom{1*\prime y\vee M}}}

\newcommand{\res}[1]{\negmedspace\bigm|\mstrut_{#1}}
\newcommand{\Res}[1]{\negmedspace\biggm|\mstrut_{#1}}
\newcommand{\temsquare}{\raise3.5pt\hbox{\boxed{ }}}

\newcommand{\zmod}[1]{\ZZ/#1\ZZ}

\newcommand{\zt}{\zmod2}

\newcommand{\hneg}{\mkern-.5\thinmuskip}
 
\DeclareFontFamily{U}{mathx}{}
\DeclareFontShape{U}{mathx}{m}{n}{<-> mathx10}{}
\DeclareSymbolFont{mathx}{U}{mathx}{m}{n}
\DeclareMathAccent{\widehat}{0}{mathx}{"70}
\DeclareMathAccent{\widecheck}{0}{mathx}{"71}
 
\DeclareMathSymbol{\bigtimes}{1}{mathx}{"91}

\usepackage[all,2cell]{xy}\renewcommand{\cir}{\ensuremath{S^1}}
\UseAllTwocells
\usepackage{tikz}
\usetikzlibrary{tikzmark,calc}
\usepackage{tikz-cd}
\usepackage{graphicx}\usepackage{epsf}
\usepackage{trimclip}
\usepackage[colorlinks,backref=page,citecolor=refkey]{hyperref}

\definecolor{refkey}{rgb}{0,.8,.2}\definecolor{labelkey}{rgb}{1,0,0} 

\DeclareMathOperator{\SO}{SO}

\let\O\relax
\DeclareMathOperator{\O}{O}

\DeclareMathOperator{\SU}{SU}
\DeclareMathOperator{\GL}{GL}
\DeclareMathOperator{\SL}{SL}
\DeclareMathOperator{\PSL}{PSL}

\newcommand{\tatp}{\tau \mstrut _{t'}}
\newcommand{\tat}{\tau \mstrut _t}
\DeclareMathOperator{\Ad}{Ad}
\DeclareMathOperator{\Area}{Area}
\DeclareMathOperator{\Arf}{Arf}
\DeclareMathOperator{\Bord}{Bord}
\DeclareMathOperator{\Bun}{Bun}
\DeclareMathOperator{\Conv}{Conv}
\DeclareMathOperator{\Det}{Det}
\DeclareMathOperator{\Euler}{Euler}
\DeclareMathOperator{\Lie}{Lie}
\DeclareMathOperator{\Line}{Line}
\DeclareMathOperator{\Li}{Li}
\DeclareMathOperator{\Man}{Man}
\DeclareMathOperator{\Pin}{Pin}
\DeclareMathOperator{\Vol}{Vol}

\DeclareMathOperator{\hol}{hol}
\DeclareMathOperator{\pro}{proj}
\DeclareMathOperator{\sSet}{sSet}
\DeclareMathOperator{\supp}{supp}
\DeclareMathOperator{\trace}{trace}
\newcommand{\BC}{B\Cx}
\newcommand{\BFC}{B^{}_{\delta }\Cx}
\newcommand{\BNC}{B^{}_{\nabla }\Cx}
\newcommand{\BNG}{B^{}_{\nabla }G}
\newcommand{\BNS}{B^{}_{\nabla }\!\SLC}
\newcommand{\BNU}{B^{}_{\nabla }U}
\newcommand{\BNZ}{B^{}_{\nabla }\ZZ}
\newcommand{\BSLCd}{B\!\SLC^\delta }
\newcommand{\BSLC}{B\!\SLC}
\newcommand{\Bbmut}{B\!\bmut}
\newcommand{\fBbmut}{B\!\fbmut}
\newcommand{\CE}{\widecheck{E}_{\CC}}

\newcommand{\CS}{CS}
\newcommand{\CZo}{\CC/\Zo}
\newcommand{\CZ}{\CC/\ZZ}
\newcommand{\Cx}{\CC^\times}
\newcommand{\ENC}{E^{}_{\nabla }\Cx}
\newcommand{\ENG}{E^{}_{\nabla }G}
\newcommand{\ENS}{E^{}_{\nabla }\!\SLC}

\newcommand{\F}[1]{\rF\mstrut _{\!#1}}
\newcommand{\GLC}{\GL_2\!\CC}
\newcommand{\Gd}{G^\delta }
\newcommand{\Gq}{\mathscr{G}_{(q)}}
\newcommand{\Hext}{\widehat{H_1(\tY)}}
\newcommand{\IZ}{I\ZZ}
\newcommand{\LC}{\Line\mstrut _\CC}
\newcommand{\MSO}{M\!\SO}
\newcommand{\MSpin}{M\!\Spin}
\newcommand{\PSLC}{\PSL_2\!\CC}
\newcommand{\RZ}{\RR/\ZZ}
\newcommand{\SCx}{\rS_{\Cx}}
\newcommand{\SLC}{\SL_2\!\CC}
\newcommand{\Tu}{\Theta^{\textnormal{univ}}}
\newcommand{\VV}{\mathbb{V}}
\newcommand{\Zo}{\ZZ(1)}
\newcommand{\Zt}{\ZZ(2)}
\newcommand{\aX}{\vec{\sX}}
\newcommand{\ain}{\alpha ^{\infty}}
\newcommand{\bC}{\partial Y }
\newcommand{\bM}{\partial M }
\newcommand{\bS}{\partial C }
\newcommand{\bX}{\partial X }
\newcommand{\bY}{\mathscr{M}}
\newcommand{\bg}{\bar{\gamma }}
\newcommand{\bmut}{\bmu 2}
\newcommand{\fbmut}{\fbmuu_{2}}
\newcommand{\bmu}[1]{\bmuu _{#1}}
\newcommand{\brdf}{\Bord _{\langle 2,3  \rangle}(\GL^+_3\!\RR\times G^\delta )}
\newcommand{\brdn}{\Bord_{\langle 2,3  \rangle}(\GL^+_3\!\RR\times G^\nabla )}
\newcommand{\cCC}[1]{\widecheck{C}\mstrut _{\CC}(#1)}
\newcommand{\cCd}[1]{\widecheck{C}^{#1}}
\newcommand{\cCc}[2]{\widecheck{C}(#1)^{#2}}

\newcommand{\cGq}[1]{\widecheck{\mathscr{G}}_{(q)}(#1)}
\newcommand{\cG}[2]{\widecheck{\mathscr{G}}_{(#1)}(#2)}
\newcommand{\cHC}{\widecheck{H}_{\CC}}
\newcommand{\cH}{\widecheck{H}}
\newcommand{\cZd}[1]{\widecheck{Z}^{#1}}
\newcommand{\cZ}[2]{\widecheck{Z}(#1)^{#2}}
\newcommand{\cb}{\widecheck{\beta }}
\newcommand{\cct}{\widecheck{c}_2}
\newcommand{\cc}{\widecheck{c}}
\newcommand{\cdr}{\iota\mstrut _{\partial /\partial r}}

\newcommand{\cdt}{\iota\mstrut _{\partial /\partial t}}

\newcommand{\ce}{\widecheck{\eta}}
\newcommand{\chb}{\widecheck{h}^{\bullet }}
\newcommand{\cl}{\widecheck{\lambda }}
\newcommand{\corp}[2]{^{#1}_{#2}}
\newcommand{\corn}[2]{^{#1}_{-#2}}
\newcommand{\co}{\widecheck{\omega }}
\newcommand{\cq}{\widecheck{q}}
\newcommand{\cs}{\widecheck{\sigma }}
\newcommand{\ct}{\widecheck{\tau}}
\newcommand{\cz}{\widecheck{0}}
\newcommand{\ddt}{\partial /\partial t}
\newcommand{\din}{\delta ^{\infty}}

\newcommand{\eCx}{\epsilon \mstrut _{\Cx}}
\newcommand{\eSLi}{\epsilon ^\infty _{\SLC}}
\newcommand{\eSL}{\epsilon \mstrut _{\SLC}}
\newcommand{\emti}{\epsilon ^{\infty} _{\!\scalebox{0.8}{$\bmut$}}}
\newcommand{\emt}{\epsilon \mstrut _{\!\scalebox{0.8}{$\bmut$}}}
\newcommand{\fg}{\mathfrak{g}}
\newcommand{\fgl}{\mathfrak{gl}}
\newcommand{\form}{\langle -,-  \rangle}

\newcommand{\gE}[1]{\lambda \mstrut _{E_{#1}}}
\newcommand{\hY}{\widehat{\bY}}
\newcommand{\hG}{\widehat{G}}
\newcommand{\hH}{\widehat{H}}
\newcommand{\hT}{\widehat{T}}

\newcommand{\hb}{h^{\bullet }}
\newcommand{\hc}{\hat{c}}
\newcommand{\hl}{\widehat\lambda }
\newcommand{\holL}{\hol\mstrut _{L}}
\newcommand{\ho}{z }

\newcommand{\iQ}{\iota(\Theta _Q)}
\newcommand{\lG}{\lambda \mstrut _{\hG}}
\newcommand{\mstrat}[1]{M _{#1}}
\newcommand{\proj}[1]{\pro_{\mskip0.5\thinmuskip\ell _{#1}}}
\newcommand{\rF}{\mathrsfs{F}}
\newcommand{\rS}{\mathrsfs{S}}

\newcommand{\sB}{\mathscr{B}}
\newcommand{\sE}{\mathscr{E}}

\newcommand{\sL}{\mathscr{L}}

\newcommand{\sN}{\mathscr{N}}
\newcommand{\sQ}{\mathcal{Q}}
\newcommand{\sT}{\mathscr{T}}

\newcommand{\sU}{\mathcal{U}}
\newcommand{\sX}{\mathscr{X}}
\newcommand{\slc}{\mathfrak{s}\mathfrak{l}_2\CC}
\newcommand{\sqmo}{\sqrt{-1}}
\newcommand{\strat}[1]{Y _{#1}}
\newcommand{\tD}{\widetilde{\tetra }}
\newcommand{\tM}{\widetilde{M }}
\newcommand{\tS}{\widetilde{S }}
\newcommand{\tX}{\widetilde{X }}
\newcommand{\tY}{\widetilde{Y }}
\newcommand{\tain}{\ta ^{\infty}}
\newcommand{\ta}{{}^{\textnormal{tt}}\hneg\alpha }

\newcommand{\tl}{\tilde\lambda }
\newcommand{\torus}{\cir\times \cir}

\newcommand{\tpi}{2\pi \sqmo}
\newcommand{\tstrat}[1]{\widetilde Y _{#1}}

\newcommand{\tom}{\tilde{\omega }}
\newcommand{\tz}{{}^{\textnormal{tt}}\hneg\zeta }
\newcommand{\zoM}{[0,1]\times M}
\newcommand{\zo}{[0,1]}
\newcommand{\ztil}{\widetilde{\zeta }}
\renewcommand{\Im}{\text{Im}}
\newcommand{\sV}{s\hspace{-.1ex}\VV}
\newcommand{\GG}{\raisebox{.3ex}{\scalebox{.65}{$|$}}\hspace{-.1em}{\Gamma}}
\DeclareMathOperator{\Cliff}{Cliff}
\newcommand{\spfl}[1]{\Phi _{\hneg#1}}
\newcommand{\sC}{\mathscr{C}}
\newcommand{\bcir}{S^1_{\textnormal{bounding}}}
\newcommand{\nbcir}{S^1_{\textnormal{nonbounding}}} 
\newcommand{\bDt}{\partial \tilD}
\newcommand{\bD}{\partial D}
\newcommand{\tilD}{\widetilde{D}}

\newcommand{\cQ}{{\mathcal Q}}

\newcommand{\C}{\ensuremath{\mathbb C}}
\newcommand{\scrF}{{\rF}}
\newcommand{\scrT}{{\sT}}
\newcommand{\scrS}{{\rS}}

\newcommand{\eps}{\epsilon}
\newcommand{\cL}{\ensuremath{\mathcal L}}
\newcommand{\cN}{\ensuremath{\mathcal N}}
\newcommand{\insfig}[3]{\begin{figure}[ht] \centering \includegraphics[scale=#2]{figures/#1.pdf} \caption{#3}\label{fig:#1}  \end{figure}}
\newcommand{\insfigpng}[3]{\begin{figure}[ht] \centering \includegraphics[scale=#2]{figures/#1.png} \caption{#3}\label{fig:#1}  \end{figure}}
\newcommand{\cA}{\ensuremath{\mathcal A}}
\newcommand{\tDelta}{{\widetilde \Delta}}

\newcommand{\vertices}{{\mathbf {vertices}}}

\newcommand{\edges}{{\mathbf {edges}}}
\newcommand{\faces}{{\mathbf {faces}}}
\newcommand{\tetrahedra}{{\mathbf {tets}}}
\newcommand{\tet}{{\tetra}}

\newcommand{\rest}{{\mathrm{rest}}}

\newcommand{\gea}{{\ge \threeastrat}}
\newcommand{\geb}{{\ge \twobstrat}}
\newcommand{\getwoa}{{\ge \twoastrat}}
\newcommand{\Z}{\ensuremath{\mathbb Z}}
\newcommand{\twoastrat}{{-2\mathrm{a}}}
\newcommand{\twobstrat}{{-2\mathrm{b}}}
\newcommand{\threeastrat}{{-3\mathrm{a}}}
\newcommand{\threebstrat}{{-3\mathrm{b}}}
\newcommand{\IP}[1]{\langle#1\rangle}
\newcommand{\tw}{\mathrm{tw}}
\newcommand{\Vline}{$\VV$-line\xspace}
\newcommand{\Vlines}{$\VV$-lines\xspace}
\newcommand{\I}{{\sqrt{-1}}}
\newcommand{\R}{\ensuremath{\mathbb R}}
\newcommand{\de}{\mathrm{d}}

\newcommand{\rmout}{{\mathrm {out}}}
\newcommand{\rmin}{{\mathrm {in}}}
\newcommand{\e}{{\mathrm e}}
\newcommand{\tH}{{\widetilde H}}
\newcommand{\tils}{{\widetilde s}}
\newcommand{\half}{\ensuremath{\frac{1}{2}}}

\DeclareMathOperator{\im}{Im}
\newcommand{\cC}{\ensuremath{\mathcal C}}
\newcommand{\cO}{\ensuremath{\mathcal O}}

\newcommand\dhxrightarrow[2][]{%
  \mathrel{\ooalign{$\xrightarrow[#1\mkern4mu]{#2\mkern4mu}$\cr%
  \hidewidth$\rightarrow\mkern4mu$}}
}
  
\begin{document}

\abovedisplayskip18pt plus4.5pt minus9pt
\belowdisplayskip \abovedisplayskip
\abovedisplayshortskip0pt plus4.5pt
\belowdisplayshortskip10.5pt plus4.5pt minus6pt
\baselineskip=15 truept
\marginparwidth=55pt

\makeatletter
\renewcommand{\tocsection}[3]{%
  \indentlabel{\@ifempty{#2}{\hskip1.5em}{\ignorespaces#1 #2.\;\;}}#3}
\renewcommand{\tocsubsection}[3]{%
  \indentlabel{\@ifempty{#2}{\hskip 2.5em}{\hskip 2.5em\ignorespaces#1%
    #2.\;\;}}#3} 
\renewcommand{\tocsubsubsection}[3]{%
  \indentlabel{\@ifempty{#2}{\hskip 5.5em}{\hskip 5.5em\ignorespaces#1%
    #2.\;\;}}#3} 
\def\@makefnmark{%
  \leavevmode
  \raise.9ex\hbox{\fontsize\sf@size\z@\normalfont\tiny\@thefnmark}} 
\def\multfoot{\textsuperscript{\tiny\color{red},}}
\def\footref#1{$\textsuperscript{\tiny\ref{#1}}$}
\makeatother

\newcommand{\bmuu}{\mbox{$\raisebox{-.07em}{\rotatebox{9.9}
  {\tiny {\bf /}
  }}\hspace{-0.53em}\mu\hspace{-0.88em}\raisebox{-0.98ex}{\scalebox{2} 
  {$\color{white}\phantom{.}$}}\hspace{-0.416em}\raisebox{+0.88ex}
  {$\color{white}\phantom{.}$}\hspace{0.46em}$}} 
 
\newcommand{\fbmuu}{\mbox{$\raisebox{-.07em}{\rotatebox{9.9}
  {\scalebox{0.8}{\tiny {\bf /}}
  }}\hspace{-0.59em}\mu\hspace{-0.88em}\raisebox{-0.98ex}{\scalebox{2} 
  {$\color{white}\phantom{.}$}}\hspace{-0.416em}\raisebox{+0.88ex}
  {$\color{white}\phantom{.}$}\hspace{0.46em}$}} 

\newcommand{\tetrah}{\tikzmarknode[text opacity=0,path picture={ \draw[line join=bevel] let \p1=($(path picture bounding box.north)-(path picture bounding box.south)$),\p2=($(path picture bounding box.east)-(path picture bounding box.west)$),\n1={0.5*min(\x2,\y1)},\n2={0.42*max(\x2,\y1)-\n1} in ([yshift=-\n2]path picture bounding box.center) coordinate(aux0)  -- ++ (200:\n1) --  ($(aux0)+(80:\n1)$) -- (aux0) -- ++ (320:\n1)  -- ($(aux0)+(80:\n1)$)  ($(aux0)+(320:\n1)$) --  ($(aux0)+(200:\n1)$); }]{aux}{W}}
\newcommand{\tetra}{\!\hneg\trimbox{0pt 2pt}{\raisebox{-.5ex}{\scalebox{1.3}{\rotatebox{10}{\tetrah}}}}\!}
\newcommand{\stett}{\!\!\hneg\trimbox{0pt 2pt}{\raisebox{-.5ex}{\scalebox{1}{\rotatebox{10}{\tetrah}}}}\!}
\newcommand{\sstett}{\!\!\hneg\trimbox{0pt 2pt}{\raisebox{-.5ex}{\scalebox{.7}{\rotatebox{10}{\tetrah}}}}\!}
\newcommand{\stet}{\mstrut _{\stett}}
\newcommand{\sstet}{\mstrut _{\sstett}}
\newcommand{\stetn}[1]{\mstrut _{\stett_{#1}}}

\setcounter{tocdepth}{2}  



 \title[Spectral Networks and Chern-Simons Invariants]{3d Spectral Networks and Classical Chern-Simons theory}
 \author[D. S. Freed]{Daniel S.~Freed}
 \thanks{This material is based upon work supported by the National Science
Foundation under Grant Numbers DMS-1611957, DMS-1711692, DMS-2005286,
DMS-2005312, and by the Simons Fellowship in Mathematics.  This work was
initiated at the Aspen Center for Physics, which is supported by National
Science Foundation grant PHY-1607611.}
 \address{Department of Mathematics \\ University of Texas \\ Austin, TX
78712} 
 \email{dafr@math.utexas.edu}
 \author[A. Neitzke]{Andrew Neitzke}
 \address{Department of Mathematics \\ Yale University \\ 10 Hillhouse Avenue\\
New Haven, CT 06511} 
 \email{andrew.neitzke@yale.edu}
 \dedicatory{To Chern, who taught us all}
 \date{August 13, 2022}
 \begin{abstract} 
 We define the notion of spectral network on manifolds of dimension 
 $\le 3$. For a manifold $X$ equipped with a spectral network, we 
 construct equivalences between Chern-Simons invariants of
 flat $\SLC$-bundles over $X$ and Chern-Simons invariants of 
 flat $\C^\times$-bundles over
 ramified double covers $\widetilde X$. 
 Applications include a new viewpoint on dilogarithmic formulas for
 Chern-Simons invariants of flat $\SLC$-bundles over triangulated 3-manifolds,
 and an explicit description of Chern-Simons lines of flat $\SLC$-bundles
 over triangulated surfaces. Our constructions heavily exploit the locality
 of Chern-Simons invariants, expressed in the language of extended (invertible)
 topological field theory.
 \end{abstract}
\maketitle 

{\small\tableofcontents}

   \section{Introduction}\label{sec:1}

A classical formula of Lobachevsky-Milnor-Thurston \cite[Chapter~7]{T2}
expresses the volume of a tetrahedron, i.e., 3-simplex, in hyperbolic space
in terms of a dilogarithm function.  It follows that the volume of a
triangulated hyperbolic 3-manifold is a sum of real parts of dilogarithms.
Thurston observed that the Chern-Simons invariant of the associated flat
$\PSLC$-connection has real part equal to the volume, and Meyerhoff~\cite{Me}
extended this to hyperbolic 3-manifolds with cusps.  These ideas have been
refined and extended since their introduction in the late 1970's and early
1980's, as we briefly review in Section~\ref{sec:2}.  After much work, in particular
by Neumann~\cite{Neu}, by the early 2000's the Chern-Simons invariant of a
flat $\PSLC$-connection on a closed oriented 3-manifold was expressed as a
sum of \emph{complex} dilogarithms.  In a closely related development over
the past 20~years, Fock and Goncharov~\cite{FG1} studied special
\emph{cluster} coordinate systems on the moduli space of flat bundles on
a compact oriented 2-manifold with punctures.  The moduli space is
symplectic, and the overlap functions---cluster transformations---between
different coordinate systems are generated by essentially the same complex
dilogarithms. These dilogarithms also serve as transition functions
defining a canonical prequantum line bundle over the moduli space~\cite{FG2}. 
In this paper we introduce new perspectives and techniques
into this circle of ideas.  Our work is inspired by two distinct sources:
\emph{spectral networks} and \emph{invertible field theories}.  Both
originated in physics and both have well-developed mathematical
underpinnings.

\emph{Spectral networks} on 2-manifolds were introduced by
Gaiotto-Moore-Neitzke~\cite{GMN1} as part of their study of four-dimensional
supersymmetric gauge theories.  For our purposes the key point is that, given
a spectral network on a surface $Y$, one can define the notion of stratified
abelianization \cite{GMN1,HN}: this is a linkage between a flat
${\mathrm {GL}}_N \CC$-connection on $Y$ and a $\C^\times$-connection on a ramified covering
$\widetilde Y \to Y$.  This notion has been useful in various contexts, e.g.~
in exact WKB analysis and hyperk\"ahler geometry of moduli of Higgs bundles;
it also gives a reinterpretation of the cluster coordinates of
Fock-Goncharov.  In Section~\ref{sec:4} we extend the notion of a spectral
network and stratified abelianization from 2~dimensions to all
dimensions~$\le3$.  In particular, in \S\ref{subsec:4.3} we express the data
of a spectral network on a smooth manifold as a certain type of
stratification together with a double cover over a dense subset and a section
of the double cover over a certain codimension one subset.  We use it to set
up stratified abelianization for the rank one complex Lie groups~$\GLC$,
$\SLC$, and~ $\PSLC$.  In particular, we construct canonical spectral
networks associated to triangulations and ideal triangulations of 2-~and
3-manifolds.\footnote{We allow the intermediate case of \emph{semi-ideal
triangulations} in which both ideal and interior vertices are allowed; see
Definitions~\ref{thm:76} and ~\ref{thm:78}.}
 
\insfig{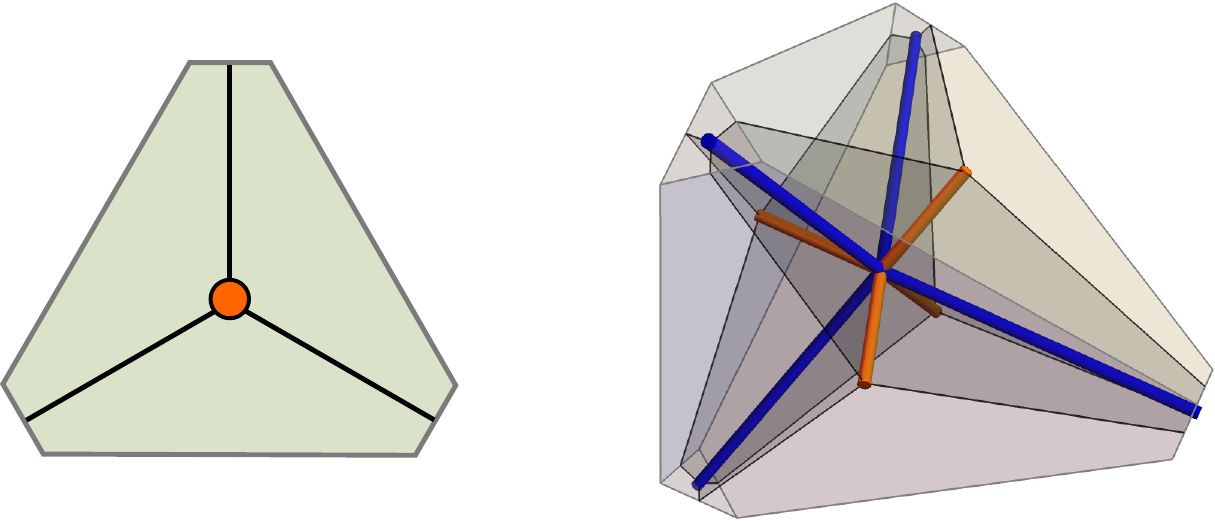}{0.5}{Left: the canonical 2d spectral network in 
a triangle. Right: the canonical 3d spectral network in a tetrahedron. Its restriction
to each face is the canonical 2d network in a triangle.}

The Chern-Simons invariant was introduced in 1971~\cite{CS1,CS2}, and it was
fairly quickly expressed by Cheeger-Simons~\cite{ChS} in terms of their novel
\emph{differential characters}, an amalgam of integral cohomology and
differential forms.  For flat connections, which are our main focus here, the
differential characters are induced from a cohomology class on a classifying
space~\cite{Ch,D,DS}.  With the advent of \emph{quantum} Chern-Simons
invariants~\cite{W}, it became clear that the \emph{classical} Chern-Simons
invariants share the locality properties of the quantum
invariants~\cite{F1,F6,F2,RSW}.  Furthermore, this locality of the classical
invariants is similar to the locality of the integral of a differential form
on a smooth manifold~$M$: if $M$~is expressed as a union $M=\bigcup_{i}M_i$
of manifolds with corners glued along positive codimension submanifolds with
corners, then the integral over~$M$ is the sum of the integrals over
the~$M_i$.  The fullest expression of this locality is in terms of
\emph{invertible field theories}.  They are constructed using the theory of
\emph{generalized differential cocycles}~\cite{HS,BNV,ADH}, and that theory
in turn is a fully local version of the Cheeger-Simons differential
characters.  We give brief introductions to these ideas in
Appendices~\ref{sec:9} and~\ref{sec:10}.
 
These two lines of development lead to the motivating idea behind our
theorems: a stratified abelianization of classical $\SLC$~Chern-Simons theory
for flat connections.  For a manifold $X$ equipped with stratified abelianization
data (defined in \S\ref{sec:4}), this amounts to an equivalence between
the Chern-Simons invariant of an $\SLC$-bundle over $X$ and that of
a $\C^\times$-bundle over a ramified double cover $\widetilde X$.
We develop two main applications:
(1)~a formula for the Chern-Simons line of a flat $\SLC$-connection on a
closed\footnote{\label{bdyu}More generally, we treat flat $\SLC$-connections
on a compact oriented 2-manifold with boundary whose holonomies around
boundary components are unipotent.} oriented 2-manifold~$Y$, derived from the
simpler and more explicit $\Cx$~ Chern-Simons theory applied to a branched
double covering manifold~$\tY$ (Theorem~\ref{thm:cs-line-explicit}); and
(2)~a derivation of the formula for the Chern-Simons invariant of a flat
$\SLC$-connection on a closed\footnote{with extensions as in~\footref{bdyu}}
3-manifold~$M$ as a sum of complex dilogarithms (Theorem
\ref{thm:cs-invariant}).

Here is the rough strategy for~(1), which we develop in Section~\ref{sec:8a}.
Let $Y$~be a closed 2-manifold equipped with a flat principal $\SLC$-bundle
$P\to Y$.  First, choose a triangulation and, over each vertex, a line in the
fiber of the complex 2-plane bundle associated to $P\to Y$; require that this
data satisfy a genericity condition (Assumption~\ref{thm:74}).  The
stratified abelianization derived from the spectral network associated to the
triangulation yields an isomorphism of the Chern-Simons line~ $\F{\SLC}(Y;P)$
with the Chern-Simons line $\SCx(\tY;A)$, where $A\to \tY$ is a flat
$\C^\times$-bundle over a branched double cover $\tY$ of $Y$.  We give a
concrete description of $\SCx(\tY;A)$ in terms of various auxiliary data:
orientations of the edges of the triangulation, a nonzero vector in the line
at each vertex, etc.  Each set of choices trivializes~$\SCx(\tY;A)$, and we
deploy $\Cx$~Chern-Simons theory in its local form to compute explicit
formul\ae\ for the ratio of trivializations under changes of auxiliary data.
Out of this we construct a groupoid whose points are sets of auxiliary data
and whose morphisms are changes of the data.  In this form our description of
the Chern-Simons lines makes contact with dilogarithmic constructions of 
line bundles over cluster varieties in the literature; 
see the discussion in \S\ref{sec:final-cs-line}.
 
For our result~(2), which is the subject of Section~\ref{sec:cs-3-manifolds},
we proceed as in~(1) to choose a triangulation and lines over the
vertices.\footnote{For ideal vertices we choose a flat section of the
associated $\CP^1$-bundle over the corresponding boundary component of~$M$.}
We excise an open ball from the center of each tetrahedron in the
triangulation of the 3-manifold~$M$.  The boundary 2-sphere of each ball is
ramified double covered by a 2-torus in the standard way: there are 4~branch
points.  The $\Cx$~Chern-Simons invariant on the branched double cover
localizes with a contribution from each tetrahedron that we identify as a
complex dilogarithm.  This led us to a new construction of the dilogarithm
function in terms of (classical) $\Cx$~Chern-Simons theory, which we worked
out in~\cite{FN} and which we apply here.
 
Spin structures are used in our stratified abelianization for a simple
reason.  The generating level of $\SLC$~Chern-Simons theory, when restricted
to the maximal torus $\Cx\subset \SLC$, is \emph{half}\footnote{There is a
minus sign at stake here: see Convention~\ref{thm:69} and Convention~\ref{thm:71}
for our choices.} the usual generating level of
$\Cx$~Chern-Simons theory; see equation~\eqref{eq:47}.  The division by~2 is
effected by passing to spin manifolds.  Just as on oriented manifolds
Chern-Simons is a secondary invariant of characteristic classes in integer
cohomology, on spin manifolds there are secondary invariants of
characteristic classes in $KO$-theory.  Here we use a simple 2-stage
Postnikov truncation of~$KO$-theory that we describe in~\S\ref{subsec:5.2}.
The $\SLC$~Chern-Simons theory does not require a spin structure, so
necessarily the results of our computations are independent of the choice of
spin structure on the base, but the intermediate formul\ae\ on the ramified
double cover require us to keep careful track of spin structures there.
 
The stratified abelianization---the production of a flat $\Cx$-connection
from a flat $\SLC$-connection---gives new geometric meaning to some aspects
of the standard constructions.  For example, the \emph{shape parameters} in
Thurston's theory~\cite[\S4.1]{T2} are now holonomies of the flat
$\Cx$-connection around certain loops in the total space of the branched
double cover.  Furthermore, Thurston's \emph{gluing
equations}~\cite[\S4.2]{T2} are a simple relation in the first homology group
of that manifold; see Remark~\ref{thm:67}.  Neumann's ``combinatorial
flattenings''~\cite[\S3]{Neu} correspond to global sections of the principal
$\Cx$-bundle over the branched double cover~$\tM$ (with balls excised).
 
In the course of our work we produced computer programs to implement our
formulas for the Chern-Simons invariants of 3-manifolds.  We have made those
programs and computations available as ancillary files in the arXiv version
of this paper.
 
We conclude this introduction with a brief roadmap to the parts of the paper
not yet discussed.  Section~\ref{sec:3} is a brief recollection of
the Chern-Simons invariant in 3~dimensions, including its status as the
partition function of an invertible field theory.  As a theory of a single
flat connection, this field theory is \emph{topological}; as such it has a
formulation in homotopy theory.  However, it is not topological as a theory
of families of flat connections, and for that reason it requires the setting
indicated in Appendix~\ref{sec:10}.  Section~\ref{sec:5} begins with
cohomological computations relating levels of Chern-Simons theory for
different subgroups of~$\SLC$, both in the oriented and spin cases.  Then we
review the role of differential cochains and prove an important result
(Theorem~\ref{thm:44}) which essentially says that the Chern-Simons theory is
unchanged as connections move in unipotent directions in~$\SLC$.  We also
prove some theorems about the spin $\Cx$~Chern-Simons theory that are
important for our computations.  Section~\ref{sec:5} concludes with a
global---as opposed to stratified---abelianization theorem.
Section~\ref{sec:7a} introduces the auxiliary data we impose on a
triangulated manifold.  Then we prove important technical results which
underpin the abelianization of the Chern-Simons line.  As stated earlier, our
main theorems are in Sections~\ref{sec:8a} and~\ref{sec:cs-3-manifolds}.
We conclude in Section~\ref{sec:11} with suggestions for ambitious
readers who would like to extend our work in new directions.  Finally,
Appendix~\ref{sec:C11} takes up additional $\zt$-gradings in spin
Chern-Simons theory which we suppress in the main text; there we prove a
spin-statistics result which justifies that suppression.
 
March\'e's approach in~\cite{M} is a close cousin to our derivation of the
formula for the Chern-Simons invariant in~ \S\ref{sec:cs-3-manifolds}.
Our stratified abelianization is a \emph{classical} version of a
\emph{quantum} abelianization proposed by
Cecotti-C\'ordova-Vafa~\cite[\S7]{CCV}.

Over the long period in which this work was carried out we benefited from the
comments and insights of many colleagues, including Clay C\'ordova, Tudor
Dimofte, Stavros Garoufalidis, Matthias Goerner, 
Alexander Goncharov, Pavel Safronov, Joerg
Teschner, Christian Zickert.  We warmly thank them all, named and unnamed.

   \section{Hyperbolic volumes and Chern-Simons invariants}\label{sec:2}

As a warmup suppose $Y ^2$~is a complete hyperbolic 2-manifold with finite
area and finitely generated fundamental group.  Then the Gauss-Bonnet theorem
states that $\Area(Y )=-2\pi \Euler(Y )$ is a topological
invariant~\cite{Ro}.  Furthermore, $Y $~is the interior of a compact
surface.  The classification of surfaces shows that the possible areas form a
discrete subset of~$\RR$. 
 
Now suppose $X^3$~is a complete oriented hyperbolic 3-manifold with finite
volume and finitely generated fundamental group.  Then $X$~is the interior of
a compact 3-manifold whose boundary is a union of
tori~\cite[Proposition~5.11.1]{T2}.  Mostow rigidity~\cite{Mo,Pr} asserts
that $\Vol(X)$~is again a topological invariant.  Jorgenssen-Thurston proved
basic properties of this invariant~\cite{T1}.  For example, the set of
hyperbolic volumes is a well-ordered subset of~$\RR$, and there is a finite
set of hyperbolic 3-manifolds of a given volume.  The volume is an important
invariant which orders hyperbolic 3-manifolds by complexity.  The
``simplest'' is the Weeks manifold of volume~0.9427\dots, the minimal volume
closed orientable hyperbolic 3-manifold~\cite{GMM}.  Further analytic
properties of the set of hyperbolic volumes were explored early on
in~\cite{NZ,Y}.
 
There is a classical formula for the volume of an ideal tetrahedron~$\tetra
\subset \HH^3$ in hyperbolic 3-space; it can be used to compute the volume of
an ideally triangulated hyperbolic 3-manifold.  Suppose the vertices
of~$\tetra $ are distinct points $Z_0,Z_1,Z_2,Z_3\in \CP^1=\partial \HH^3$.
Introduce the \emph{Bloch-Wigner dilogarithm function}~\cite[\S3]{Z} 
  \begin{equation}\label{eq:1}
     \begin{aligned} D\:\CP^1\setminus \{0,1,\infty \}&\longrightarrow \RR \\
      z&\longmapsto \Im \,\Li_2(Z) \,+\,\log|z|\arg(1-z),\end{aligned} 
  \end{equation}
where $\Li_2$~is the classical dilogarithm, defined for~$|z|<1$ by the power
series 
  \begin{equation}\label{eq:2}
     \Li_2(z) = \sum\limits_{n=1}^{\infty}\frac{z^n}{n^2} 
  \end{equation}
and analytically continued to~$\CC\setminus [1,\infty )$.  Let
  \begin{equation}\label{eq:3}
     z=\frac{(Z_0-Z_2)(Z_1-Z_3)}{(Z_0-Z_3)(Z_1-Z_2)} 
  \end{equation}
be the cross-ratio of the vertices of~$\tetra $.

  \begin{theorem}[Lobachevsky,
  Milnor-Thurston~\hbox{\cite[Chapter~7]{T2}}]\label{thm:1}
   \quad $\Vol(\tetra )=|D(z)|$.  
  \end{theorem}

In this PhD thesis Meyerhoff~\cite{Me} initiated the detailed study of the
Chern-Simons invariant $\CS(X)\in \RR/\Zo$ of of the Levi-Civita
connection~$\Theta _{LC}$ of a closed oriented hyperbolic 3-manifold.  Here
and throughout we deploy the notation
  \begin{equation}\label{eq:83}
      \Zo=2\pi \sqmo\ZZ,\qquad \ZZ(n)=\Zo^{\otimes n} = (2\pi
      \sqmo)^n\ZZ,\qquad n\in \ZZ^{\ge1}.  
  \end{equation}
  This \emph{real} Chern-Simons invariant is the real part of the
\emph{complex} Chern-Simons invariant of the associated flat
$\PSLC$-connection~$\Theta $.  Recall that the $\SO_3$-bundle $\sB_{\SO}(X)\to
X$ of frames carries not only the Levi-Civita connection~$\Theta _{LC}$ but
also the $\RR^3$-valued ``soldering form''~$\theta $; the complex combination
$\Theta =\Theta _{LC} +\sqrt{-1}\,\theta$ is a \emph{flat} connection on the
associated principal $\PSLC$-bundle:
   $$ \begin{gathered}
     \xymatrix{\mathscr{B}_{SO}(X)\ar[dr]_{SO_3}\ar@{^{(}->}[rr]&&P
     \ar[dl]^{\PSLC}\\&X}
     \end{gathered} $$ 
The exponentiated complex Chern-Simons invariant, which we review
in~\S\ref{sec:3}, satisfies
  \begin{equation}\label{eq:4}
     \F{\PSLC}(X;\Theta )=\exp\left(\Vol(X)+\sqrt{-1}CS(X)\right)\; \in \Cx.
  \end{equation}
Our focus in this paper is the complex Chern-Simons invariant for arbitrary
flat connections, mainly for structure group~$\SLC$.  (In the intrinsic case
of connections on the frame bundle, the passage from~$\PSLC$ to~$\SLC$ is the
introduction of a spin structure.)
 
Just as the real volume is related to a real dilogarithm~\eqref{eq:1}, so too
the complex Chern-Simons invariant is related to a complex dilogarithm, the
\emph{enhanced Rogers dilogarithm}.  Let 
  \begin{equation}\label{eq:5}
     \bY = \{(z_1,z_2)\in \Cx\times \Cx: z_1+z_2=1\} 
  \end{equation}
and 
  \begin{equation}\label{eq:6}
     \hY = \{(u_1,u_2)\in \CC\times \CC : e^{u_1} + e^{u_2}=1\}. 
  \end{equation}
Then $\bY\approx\CP^1\setminus \{0,1,\infty \}$ and $\hY\to\bY$ is a
universal abelian covering map with Galois group isomorphic to~$\ZZ\times
\ZZ$.  The dilogarithm in question~\cite[\S4]{ZG}, \cite[\S II.1.B]{Z},  is the
unique function
  \begin{equation}\label{eq:7}
     L\:\hY\longrightarrow \CC/\Zt, 
  \end{equation}
which satisfies the differential equation
  \begin{equation}\label{eq:8}
     dL=(u_1\,du_2-u_2\,du_1)/2 
  \end{equation}
and $\lim\limits_{}L(u_1,u_2)=0$ as $u_1\to\infty $ and~$u_2\to0$.  (We
encounter variants in~\S\S\ref{sec:7a}--\ref{sec:cs-3-manifolds}.)  The
imaginary part of~$L$ is the Bloch-Wigner function~\eqref{eq:1} plus
$\Im(\overline{u_1}\,u_2)/2$.  See~\cite{FN} for a construction of the
enhanced Rogers dilogarithm using Chern-Simons invariants for
$\Cx$-connections.
 
Let $\BSLCd$ denote the classifying space of flat $\SLC$-bundles.\footnote{We
use standard terminology: `flat' is a \emph{structure}---a flat
connection---on a principal bundle.}  The
universal Chern-Simons class for flat bundles 
  \begin{equation}\label{eq:9}
     \hc_2\in H^3(\BSLCd;\CZo) 
  \end{equation}
was constructed by Cheeger-Simons~\cite{ChS} and is known as the
Cheeger-Chern-Simons class.  It has an expression in terms of the
dilogarithm~\eqref{eq:7}, going back to work of Dupont and collaborators in
the~1980's; see~\cite{D,DS}.  The most precise relationship can be found
in~\cite[\S4]{DZ}, which is based on~\cite{Neu}; see both papers for exact
statements, history, and extensive references.

In the early 2000's, the formula for the Chern-Simons invariant of a flat
connection on a 3-manifold as a sum of dilogarithms was taken up again in
such works as \cite{Neu,DZ,Zi,GTZ,DGG}.  The formula for flat
$\PSLC$-connections on closed 3-manifolds is in~\cite{Neu}; the formula for
flat $\SLC$-connections on closed 3-manifolds is in~\cite{DZ}.  The formula
for boundary-unipotent flat $\SL_N\hneg\CC$-connections appears
in~\cite{GTZ}; for boundary-unipotent flat $\PSLC$-connections it is in the
earlier paper~\cite{Zi}. 

  \begin{remark}[]\label{thm:75}
 These previous works rely on global ordering data/conditions on the vertices
of a triangulation or ideal triangulation of the 3-manifold.  By contrast, in
our work we only use edge orientations with no constraints.  
As a consequence, our formula in Theorem~\ref{thm:cs-invariant} is a bit more
complicated: it involves four variants of the dilogarithm, and also some cube
roots of unity enter from a $\bmu3$-symmetry not present in earlier
approaches.
  \end{remark}

   \section{Chern-Simons as a topological field theory}\label{sec:3}

The integral of a differential form over a smooth manifold~$M$ is
\emph{local}: if $M=\bigcup M_i$ is a finite union of submanifolds, possibly
with boundaries and corners, and if $M_i\cap M_j$ has measure zero for
all~$i\neq j$, then the integral over~$M$ is the sum of the integrals
over~$M_i$.  The exponentiated Chern-Simons invariant of a connection on a
principal bundle $P\to M$ is not the integral of a differential form on~$M$,
yet it still satisfies strong locality properties: it is the \emph{partition
function} of an \emph{invertible field theory}.  We review this aspect of
Chern-Simons invariants.  See~\cite[\S2]{FN} for an exposition of the theory
with gauge group~$\Cx$.
 
Let $G$~be a Lie group with finitely many components, called the \emph{gauge
group}, and let $\mathfrak{g}$~be its Lie algebra.  In this section we make
the simplifying assumption that $G$~is simply connected.  Let $\pi \:P\to M$
be a principal $G$-bundle with connection\footnote{We use `$G$-connection' as
a shorthand for `principal $G$-bundle with connection'.} $\Theta \in \Omega
^1_P(\mathfrak{g})$.  Suppose
  \begin{equation}\label{eq:19}
      \form\:\mathfrak{g}\times \mathfrak{g}\to\CC 
  \end{equation}
is a $G$-invariant symmetric bilinear form on the Lie algebra~$\mathfrak{g}$.
Chern-Simons~\cite{CS2} define a scalar 3-form on the total space~$P$,
  \begin{equation}\label{eq:10}
     \eta (\Theta ) = \langle \Theta \wedge \Omega \rangle-\frac 16 \langle
     \Theta \wedge [\Theta \wedge \Theta ] \rangle \; \in \Omega ^3_P, 
  \end{equation}
where $\Omega =d\Theta +\frac 12[\Theta \wedge \Theta ]\in \Omega
^2_P(\mathfrak{g})$ is the curvature of~$\Theta $.  If $\dim M\le 3$, then
$\eta (\Theta )$~is closed.  Also, in that case the simple connectivity
of~$G$ ensures the existence of sections $s\:M\to P$ of $\pi \:P\to M$.  If
$M=X$~is a closed oriented 3-manifold, then 
  \begin{equation}\label{eq:11}
     \int_{X}s^*\eta (\Theta )\; \in \CC 
  \end{equation}
is unchanged under a homotopy of~$s$, since $\eta (\Theta )$~is closed.  The
space of sections is generally not connected, so to ensure that
\eqref{eq:11}~is independent of~$s$ we make two modifications: (i)~we impose
an integrality hypothesis on~$\form$, and (ii)~we reduce the integral
to~$\CC/\Zo$.  The integrality condition lies in topology if $G$~is compact
or $G$~is complex, which we now assume.  Namely, the vector space of
forms~$\form$ is canonically isomorphic\footnote{The isomorphism maps a
form~$\form$ to the de Rham cohomology class of $\langle \Omega ,\Omega
\rangle$, where $\Omega $~is the curvature of a universal principal
$G$-connection over~$BG$.}  to~$H^4(BG;\CC)$, where $BG$~is the classifying
space of~$G$.  The image of $H^4(BG;\ZZ)\to H^4(BG;\CC)$ is a lattice of
integral forms.\footnote{There is also a distinguished cone of forms whose
restriction to a maximal compact Lie subalgebra is positive definite.  For
$G$~connected, the map $H^4(BG;\ZZ)\to H^4(BG;\CC)$ is injective:
$H^4(BG;\ZZ)$~is torsionfree.}  Then if $\form$~is an integral form,
  \begin{equation}\label{eq:12}
     \tpi\int_{X}s^*\eta (\Theta )\mod{\Zo} 
  \end{equation}
is independent of~$s$.  Define the \emph{exponentiated Chern-Simons
invariant}\footnote{The notation is deficient, as it does not include the
form~$\form$, but the choice should be clear from the context.}
  \begin{equation}\label{eq:13}
     \F G (X;\Theta)  = \exp\left(\tpi\int_{X}s^*\eta (\Theta )
     \right)\;\in \Cx. 
  \end{equation}

  \begin{remark}[]\label{thm:31}
 The exponentiated Chern-Simons invariant is defined without the simple
connectivity assumption on~$G$.  In that case the form $\langle -,-  \rangle$ is
replaced by a class in~$H^4(BG;\ZZ)$, called the \emph{level}.
See~\cite[Appendix]{F4} for the general construction. 
  \end{remark}

  \begin{example}[$G=\SLC$]\label{thm:2}
 The special linear group $G=\SLC$ is a matrix group with Lie algebra
$\mathfrak{g}=\mathfrak{s}\mathfrak{l}_2\CC$ the space of $2\times 2$
traceless complex matrices.  There is an isomorphism $H^4(\BSLC;\ZZ)\cong
\ZZ$ with generator
  \begin{equation}\label{eq:14}
     \langle A,B \rangle = - \frac 1{8\pi ^2}\trace(AB),\qquad A,B\in
     \mathfrak{s}\mathfrak{l}_2\CC, 
  \end{equation}
the complex image of $-c_2\in H^4(\BSLC;\ZZ)$.  On the trivial bundle
over~$X$ an $\SLC$-connection is a traceless matrix of 1-forms $A\in
\Omega ^1_X(\mathfrak{sl}_2\CC)$ and the Chern-Simons invariant~\eqref{eq:12}
is 
  \begin{equation}\label{eq:33}
     \frac{1}{4\pi \sqmo}\int_{X}\trace\left( A\wedge dA + \frac 23\,A\wedge
     A\wedge A \right) \mod{\Zo}. 
  \end{equation}

  \end{example}

We remark that our choice of $-c_2$ is motivated by \eqref{eq:4}; if we chose $c_2$ instead, we would have an extra minus sign in that equation.

Now suppose $X'$~is a compact 3-manifold with boundary, and let $\Theta '$~be
a $G$-connection on~$X'$ for $G$~a simply connected
Lie group.  We define $\F G(X';\Theta ')$ so that if $X=X_1\cup_NX_2$ is a
decomposition of a closed oriented 3-manifold~$X$ along an embedded closed
codimension one oriented submanifold~$N$, then
  \begin{equation}\label{eq:15}
     \F G (X;\Theta) = \F G(X_1;\Theta _1)\cdot \F G(X_2;\Theta _2), 
  \end{equation}
where $\Theta _i=\Theta \res{X_i}$.  If $\partial X'\neq \emptyset $ then
\eqref{eq:12}~is not independent of~$s$; it depends on~$s\res{\partial X'}$.
That dependence satisfies a cocycle relation which leads to the construction
of a complex line~$\F G(\partial X';\partial \Theta ')$ which only
depends on $\partial \Theta '=\Theta '\res{\partial X'}$.  The exponentiated
Chern-Simons invariant~$\F G(X';\Theta ')$ is an element of~$\F G
(\partial X';\partial \Theta ')$, and \eqref{eq:15}~is satisfied if the dot
on the right hand side is interpreted as the pairing of this line with its
dual; see~\cite[\S2]{F1}.

We summarize the situation in the language of field theory.  Let $\brdn$ be
the bordism category with objects closed oriented 2-manifolds~$Y$ equipped
with a $G$-connection~$\Theta _Y$.  (The notation indicates the structure
group of the manifold, and the superscript~$\nabla $ connotes the connection
on the principal $G$-bundle.)  A morphism $(Y_0,\Theta _0)\to (Y_1,\Theta
_1)$ in $\brdn$ is then a compact oriented 3-manifold~$X$ equipped with a
$G$-connection~$\Theta _X$, together with a diffeomorphism $-Y_0\amalg
Y_1\xrightarrow{\;\cong \;}\partial X$, and an isomorphism $\Theta _0\amalg
\Theta _1\xrightarrow{\;\cong \;}\partial \Theta _X$.  (These diffeomorphisms
need to be on collar neighborhoods---or germs of collar neighborhoods---of
the boundary.)  As usual in bordism categories, composition of morphisms is
defined by gluing bordisms, and there is a symmetric monoidal structure given
by disjoint union.  Let $\LC$ denote the groupoid whose objects are
1-dimensional complex vector spaces, and whose morphisms are invertible
linear maps.  It is a Picard groupoid under tensor product of lines.

  \begin{theorem}[]\label{thm:3}
 The exponentiated Chern-Simons invariant is a symmetric monoidal functor 
  \begin{qedequation}\label{eq:16}
     \F G \:\brdn\longrightarrow \LC. \qedhere
  \end{qedequation}
  \end{theorem}

\noindent
 So $\F G $~is an \emph{invertible field theory}, called \emph{classical
Chern-Simons theory}; see~\cite{HS,F2}.
 
Our interest in this paper is the restriction to \emph{flat} $G$-connections
  \begin{equation}\label{eq:17}
     \F G \:\brdf\longrightarrow \LC. 
  \end{equation}
This restricted theory is \emph{topological} in a restricted sense---at least
on single manifolds (see Remark~\ref{thm:38} below).  Namely, the domain
bordism category has no continuously varying parameters.\footnote{More
precisely, for any~$t\in \AA^1$ the restriction map $\Bun_{\Gd}(\AA^1\times
M)\to\Bun_{\Gd}(\{t\}\times M)$ on \emph{flat} connections is an equivalence
of stacks.}  There is a well-developed mathematical theory of topological
field theories.  In this topological case it is technically easier to
implement strong locality in the form of an \emph{extended} field theory.
For \emph{invertible} topological theories, homotopy-theoretic methods can be
brought to bear~\cite{FHT,FH1}: an invertible topological field theory can be
realized as a map of spectra.  The domain is a bordism spectrum and the
codomain a spectrum of ``higher lines''.  In that context, for $G=\SLC$ the
extended version of~ \eqref{eq:17} is realized as the composition
  \begin{equation}\label{eq:18}
     \MSO\wedge B\Gd_+ \xrightarrow{\;\;\id\wedge (-\hc_2)\;\;} \MSO\wedge
     (H\CZo_3)_+\xrightarrow{\;\;\;\int_{}\;\;\;}\Sigma ^3H\CZo. 
  \end{equation}
Here $\MSO$~is the Thom spectrum of oriented manifolds, $H\CZo$~is the
Eilenberg-MacLane spectrum associated to the abelian group~$\CZo$, the
cohomology class~$\hc_2$ is introduced in~\eqref{eq:9}, and
$(H\CZo_3)_+$~denotes the 3-space of the spectrum~$H\CZo_3$.  The first map
is the Cheeger-Simons class~\eqref{eq:9}, and the second is integration;
see~\cite[\S4.10]{HS}.  The induced map on~$\pi _3$ is a bordism invariant of
closed oriented 3-manifolds equipped with a \emph{flat} connection.  The
map~\eqref{eq:18} extends this bordism invariant to an invertible topological
field theory, thereby exhibiting its full locality.

  \begin{remark}[]\label{thm:38}
 Our analysis involves parametrized families of flat connections, that is,
connections on the total space of a fiber bundle $\pi \:M\to S$ which are
flat along the fibers of~$\pi $.  Such connections need not be flat on~$M$,
and for that reason we need more than the homotopy theoretic
map~\eqref{eq:18}, since the latter only incorporates families of flat
connection in which the connection is flat on the total space~$M$.  It is in
this broader sense that the invertible field theory~\eqref{eq:17} is not
topological.  We explain this in Appendix~\ref{sec:10}.
  \end{remark} 

  \begin{remark}[]\label{thm:73}
 In codimension~1---on closed surfaces---we wrote in~\eqref{eq:17} that the
theory~$\F G$ has values complex lines.  Similarly, in codimension~2---on
closed 1-manifolds---we take the values of the theory to be
\emph{$\VV$-lines}, i.e., invertible modules over the tensor category
$\VV$ of complex vector spaces.
  \end{remark}

There is a spin variant of Chern-Simons theory, which we discuss
in~\S\ref{subsec:5.2} in a special case.

   \section{Stratified abelianization and spectral networks}\label{sec:4}

We begin in~\S\ref{subsec:4.1} with an elementary concrete example of
stratified abelianization which motivates all that follows.  Here one
explicitly sees the monodromy around branch points (Lemma~\ref{thm:5}) and
the unipotent automorphism when crossing a wall (Equation~\eqref{eq:30}).  We
abstract this into a general definition in~\S\ref{subsec:4.3}.  We describe
the data of a spectral network as the specification of a particular type of
stratified manifold.  This is all for rank one Lie groups.
In~\S\ref{subsec:4.4} we construct a spectral network and stratified
abelianization for a triangulated surface, and in~\S\ref{subsec:4.2} we do
the same for a triangulated 3-manifold.  An important example is a 2-sphere
triangulated by the boundary of a tetrahedron---this is the boundary of a
3-simplex, which we encounter at the center of each 3-dimensional tetrahedron
in the triangulation of a 3-manifold---and we prove an important relation in
the stratified abelianization in Proposition~\ref{thm:12}.  Our setup here
makes contact with cross ratios (Remark~\ref{thm:51}) and the Thurston gluing
equations (Remark~\ref{thm:67}).

  \subsection{2-dimensional spectral networks: motivation}\label{subsec:4.1}

 To motivate stratified abelianization, begin with an invertible
$2\times 2$~complex matrix~$A\in \GLC$.  For a geometric take, let
$E\to\cir$ be a rank~2 flat complex vector bundle with holonomy~$A$.  Then
$A$~is diagonalizable if and only if there exist
  \begin{align}
      L&\xrightarrow{\;\;\phantom{\cong }\;\;} \cir\amalg \cir \label{eq:22}\\ 
     \pi _*L&\xrightarrow{\;\;\cong \;\;}E,  \label{eq:23}
  \end{align}
where $ \pi \:\cir\amalg \cir\to \cir $~is the product\footnote{Note the
sheets are not ordered.} double cover, \eqref{eq:22}~is a (flat) line bundle,
and \eqref{eq:23}~is an isomorphism.  If so, then the projectivization $\PP
E\to\cir$ has two distinguished horizontal sections; the line bundle
$L\to\cir\amalg \cir$ is isomorphic to the restriction of the tautological
line bundle $\sL\to\PP E$ to the union of the images of those sections.  If
$A$~is not diagonalizable, then the existence of an eigenvector implies that
$\PP E\to\cir$ has a unique flat section.
 
\insfig{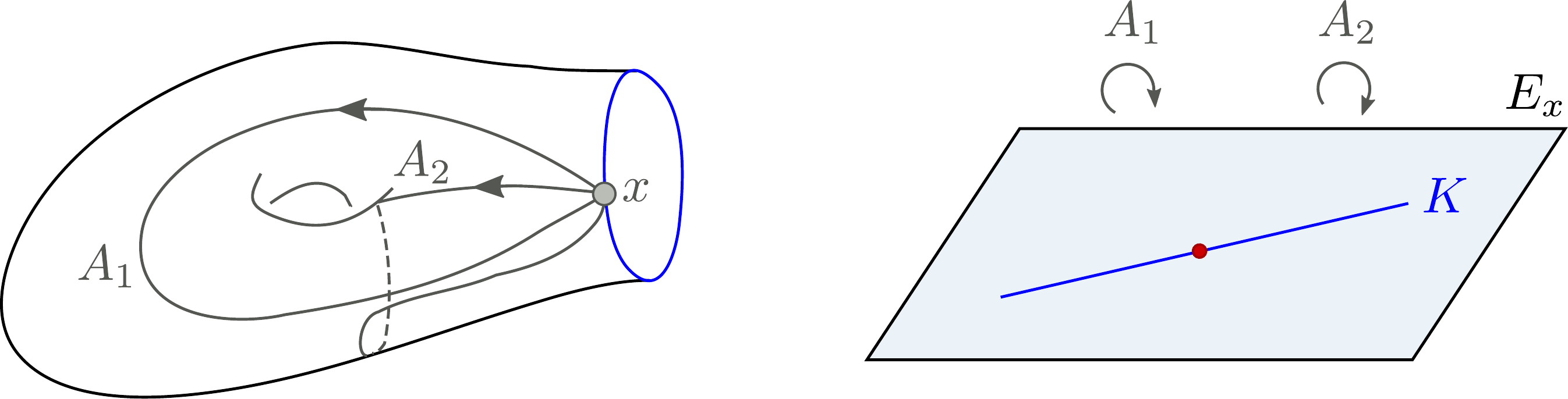}{0.4}{A flat bundle over the one-holed torus.}

Now consider two invertible matrices $A_1,A_2\in \GLC$.  If
$A_1A_2=A_2A_1$, then there is a flat rank~2 complex vector bundle
$E\to\torus$ with holonomies~$A_1,A_2$ about chosen based loops
generating~$\pi _1(\cir\times \cir)$.  Then---assuming each of $A_1,A_2$ is
diagonalizable---there is a global abelianization based on the product double
cover of~$\torus$.  Our story begins when $A_1A_2\neq A_2A_1$.  In this
situation the matrices~$A_1,A_2$ determine a flat rank~2 complex vector
bundle $E\to Y $ over the compact surface $Y =\torus\setminus D^2$, as evoked
by Figure~\ref{fig:hole-torus-holonomies}.  Let $x\in \partial Y $ be a basepoint.  There is no
hope of a global abelianization.  Instead, consider the ideal triangulation
of~$Y $ depicted in Figure~\ref{fig:two-triangles-strata}.  If we collapse the
boundary~$\partial Y$, there are 2~ triangles, glued along 3~edges; each
vertex is the point at ``infinity'' in~$Y/\partial Y$.  We interpret
Figure~\ref{fig:two-triangles-strata} as a stratification
  \begin{equation}\label{eq:24}
     Y = Y _0\amalg Y _{-1}\amalg Y _{-2}.
  \end{equation}
The codimension~2 stratum~$\strat{-2}$ consists of two points, one interior
to each face.  The codimension~1 stratum~$\strat{-1}$ is the union of six
line segments, joining the codimension~2 stratum to the vertices.  The
generic stratum~$\strat0$ is the complement of the lower dimensional strata.

\insfig{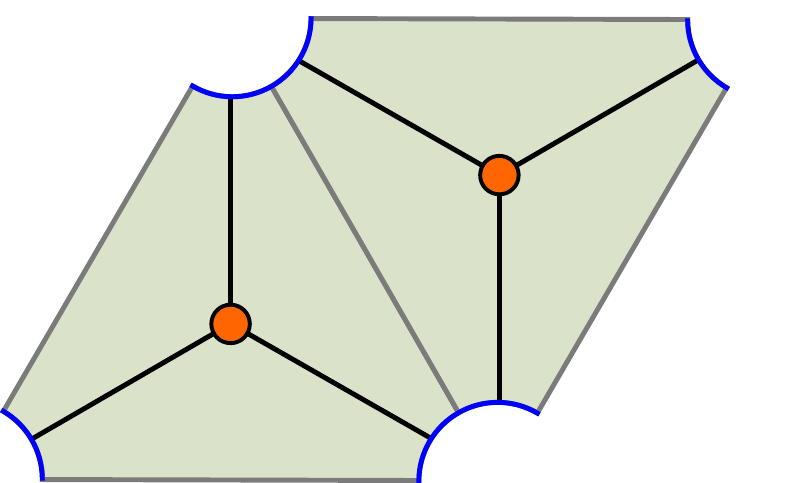}{0.5}{The stratification $Y = Y _0\amalg Y _{-1}\amalg Y
  _{-2}$. The orange points make up $Y_{-2}$; the black segments (walls) make up $Y_1$;
  the rest (including the gray edges of the triangles and the blue boundary
  arcs) is $Y_0$.}

The first step in stratified abelianization is the choice of a parallel
section of the associated projective bundle over~$\partial Y $, equivalently
an eigenline $K\subset E_x$ of the commutator $A\mstrut _1A\mstrut _2A_1\inv
A_2\inv $.  The generic stratum $Y _0=R^{(1)}\amalg R^{(2)}\amalg R^{(3)}$
has three contractible components, and for each $i\in \{1,2,3\}$ the
intersection $R^{(i)}\cap\partial Y $ has two contractible components; see
Figure~\ref{fig:computing-gluing}.  By parallel transport from~$\bS$ we obtain for each
$i\in \{1,2,3\}$ two parallel sections of $\PP E\res{R^{(i)}}\to R^{(i)}$.
In the second drawing of Figure~\ref{fig:computing-gluing} the two sections in each
component~$R$ of~$Y_0$ are labeled by the two vertices in the closure of~$R$.

  \begin{assumption}[genericity]\label{thm:4}
 For each $i\in \{1,2,3\}$, these sections are distinct. 
  \end{assumption}

\noindent
 Then, as in the 1-dimensional case, construct a global abelianization over
the generic stratum:
  \begin{align}
      \phantom{MMMM}\pi \:\tY_0&\longrightarrow Y _0&&\textnormal{double
      cover} \phantom{MMMMMMMMM}\label{eq:25}\\   
      L&\longrightarrow \tY_0&&\textnormal{flat line bundle} \label{eq:26}\\ 
      \pi _*L&\xrightarrow{\;\;\cong \;\;}E\res{Y _0}
      &&\textnormal{isomorphism of flat bundles} \label{eq:27} 
  \end{align}
The map~$\pi $ is the restriction of $\PP E\to Y$ over the image of the two
sections, and the line bundle~\eqref{eq:26} is the restriction of the
tautological line bundle $\sL\to\PP E$ to $\tY_0\subset \PP E$.  The
genericity assumption allows us to construct~\eqref{eq:27} from the embedding
$\sL\to\PP E\times E$.
 
\insfig{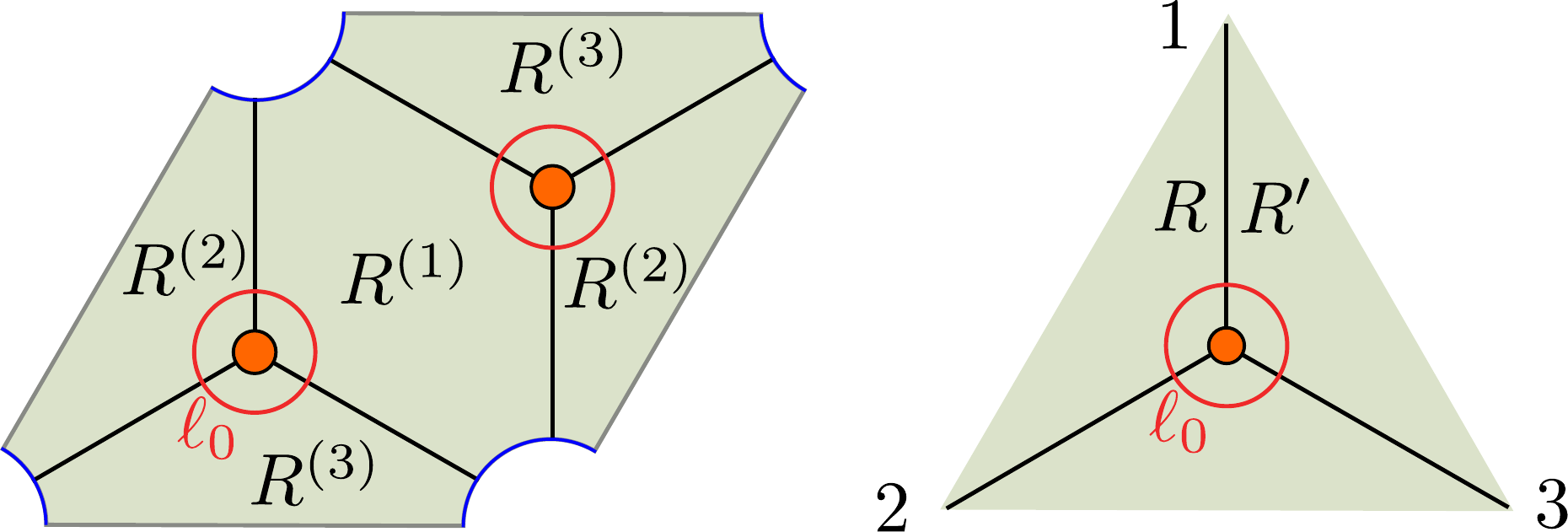}{0.4}{Computing the unipotent gluing and
holonomy.}
 
As a preliminary, suppose $\ell _1,\ell _2,\ell _3$ are three distinct lines
in a 2-dimensional vector space~$F$.  Define
  \begin{equation}\label{eq:29}
     \proj1\:\ell _2\longrightarrow \ell _3 
  \end{equation}
as the composition $\ell _2\lhook\joinrel\xrightarrow{\phantom{WI}}
F\dhxrightarrow{\;\ell _1\;}\ell _3$, where the second map is projection with
kernel~$\ell _1$; the composition is an isomorphism.  

Our task is to extend the abelianization to a structure over the lower
strata.  Fix a component~$I $ of~$\strat{-1}$ and let $R,R'$~be the
components of~$\strat0$ on either side of~$I $.  The intersection point~$I
\cap\partial Y$ picks out contiguous components of~$R\cap\partial Y$ and
$R'\cap\partial Y$.  Glue the corresponding sheets of the double
cover~\eqref{eq:25} along~$I $; there is a distinguished sheet along~$I $
from the contiguous components.  In this manner extend~\eqref{eq:25} to a
double cover
  \begin{equation}\label{eq:28}
     \pi \:\tstrat{\ge-1}\longrightarrow \strat{\ge-1}
  \end{equation}
together with a section~$s$ of~$\pi$ over $\strat{-1}$.  Next,
extend~\eqref{eq:26} to a flat line bundle
  \begin{equation}\label{eq:31}
     L\longrightarrow \tstrat{\ge-1}
  \end{equation}
as follows.  (We refer to Figure~\ref{fig:computing-gluing}.)  In passing from~$R$ to~$R'$,
on the sheet obtained by parallel transport from vertex~1 glue $L\to \tY_0$
via the identity.  Cover the identification of the sheet~2 in~$R$ and the
sheet~3 in~$R'$ with the isomorphism~\eqref{eq:29} of the line bundle
$L\to\tstrat0$ across the segment in~$\tstrat{-1}$.

We compare the isomorphisms~\eqref{eq:27} on each side of~$\strat{-1}$.  By
construction, the unipotent automorphism passing from region~$R$ to
region~$R'$ is
  \begin{equation}\label{eq:30}
     \begin{aligned} \ell _1\oplus \ell _2&\longrightarrow \ell _1\oplus \ell _3 \\
      \xi _1+\xi _2&\longmapsto \xi _1+\proj1(\xi _2)\end{aligned} 
  \end{equation}

  \begin{lemma}[]\label{thm:5}
 Let~$\lambda  $ be the link of $\strat{-2}\subset Y $ and $\lambda  _0\subset
\lambda  $ a component of~$\lambda  $.

 \begin{enumerate}[label=\textnormal{(\roman*)}]

 \item The restriction of the double cover~\eqref{eq:28} to~$\lambda  _0$ is
nontrivial.

 \item The holonomy of~\eqref{eq:31} about~$\pi \inv (\lambda  _0)$ is~$-\id$.

 \end{enumerate}
  \end{lemma}

\insfig{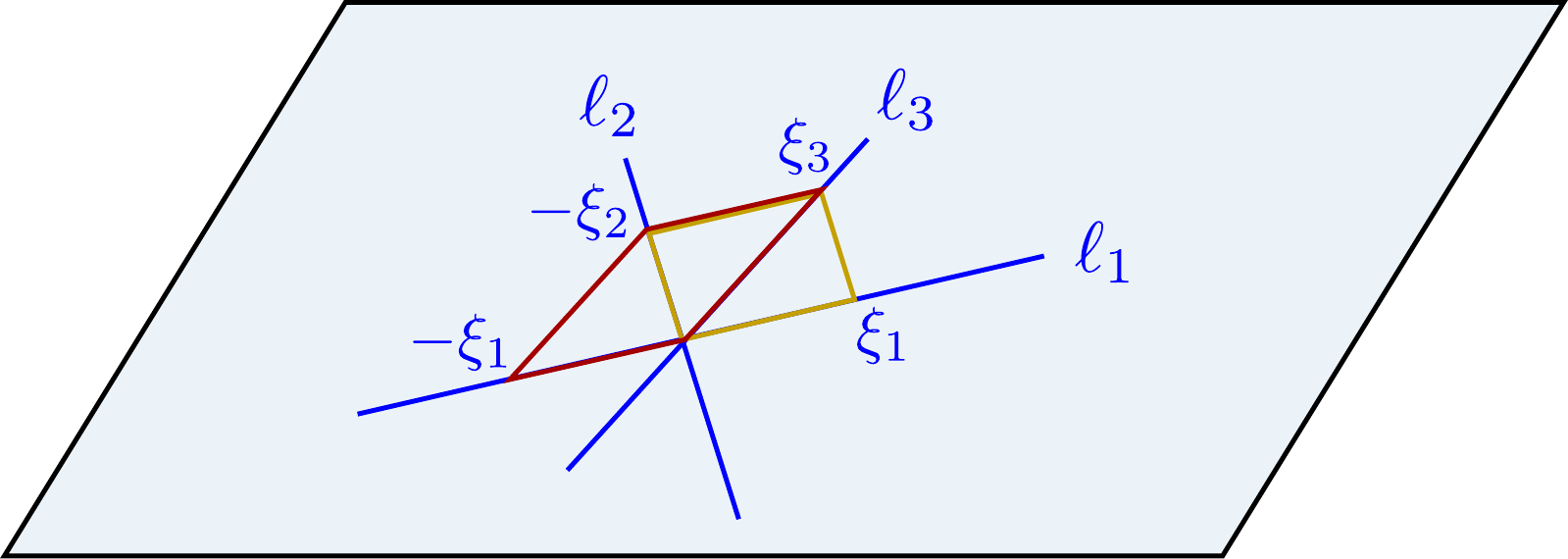}{0.45}{Computing holonomy by composing projections.}

  \begin{proof}
 The proof of~(i) is straightforward, so we omit it.  For~(ii), the holonomy
about~$\pi \inv (\lambda _0)$ is the composition
  \begin{equation}\label{eq:32}
      \ell _1\xrightarrow{\;\;\id\;\;}\ell _1\xrightarrow{\;\proj2\;}\ell _3\xrightarrow{\;\;\id\;\;}
     \ell _3\xrightarrow{\;\proj1\;}\ell _2\xrightarrow{\;\;\id\;\;}\ell _2
     \xrightarrow{\;\proj3\;}\ell _1.
  \end{equation}
Fix $\xi _1\in \ell _1$, and let $\xi _2\in \ell _2$, $\xi _3\in \ell _3$ be
the unique vectors such that $\xi _1=\xi _2+\xi _3$.  Then
under~\eqref{eq:32} 
  \begin{qedequation}\label{eq:127}
     \xi _1\longmapsto \xi _1\longmapsto \xi _3\longmapsto \xi _3\longmapsto
     -\xi _2\longmapsto -\xi _2\longmapsto -\xi _1. 
  \qedhere\end{qedequation}
  \renewcommand{\qedsymbol}{}
  \end{proof}

  \subsection{Stratifications, spectral networks, and
  abelianization}\label{subsec:4.3} 

 The double cover~\eqref{eq:28} together with the section over~$\strat{-1}$
is called a \emph{spectral network} (subordinate to the
stratification~\eqref{eq:24}).  Components of~$\strat{-1}$ are the
\emph{walls} and $\strat{-2}$~is the \emph{branch locus}.  Notice that
Lemma~\ref{thm:5}(i) implies that $\tstrat{\ge-1}\to\strat{\ge-1}$ extends to
a \emph{branched} double cover $\tY\to Y $ with branch locus~$\strat{-2}$.
The stratified abelianization of $E\to Y $ is the data:

\medskip
  \begin{itemize}

 \item the flat line bundle $L\to\tstrat{\ge-1}$

 \item the isomorphism~\eqref{eq:27} on~$\strat0$

 \item the unipotent gluing~\eqref{eq:30} on~$\strat{-1}$

  \end{itemize}
\medskip

\noindent
 In this subsection we give formal definitions of this structure which apply
in some generality.

Two-dimensional spectral networks were introduced by
Gaiotto-Moore-Neitzke~\cite{GMN1} in their study of supersymmetric
4-dimensional gauge theories.  They have motivated many mathematical
constructions and conjectures since, related to hyperk\"ahler geometry,
enumerative invariants, and asymptotic analysis of complex ODE, among others.

   \subsubsection{Stratifications}\label{subsubsec:4.3.1}
 We use the definition~\cite[4.3.2]{L}.  In that approach a type~$\aX$ of
stratified manifold of dimension~$n$ is defined from the \emph{top down}.
Namely, begin with a geometric structure\footnote{That is, a topological
space~$\sX$ equipped with a continuous map $\sX\to BO_n$.  An
$n$-manifold~$M$ with an $\sX$-structure is equipped with a lift $M\to\sX$ of
the classifying map of its tangent bundle.}  for the generic stratum in
codimension~0.  Then for each $1\le k\le n$ specify the geometric structure
and link of a codimension~$k$ stratum; the link is an $\aX$-stratified
$(k-1)$-dimensional manifold.  An $\aX$-stratified manifold of dimension~$\le
n$ is built from the \emph{bottom up}: first the highest codimension strata
are specified, then higher strata with the proper link are glued in
successively.  This heuristic depiction is fleshed out precisely
in~\cite[4.3.2]{L}, and the heuristic specifications of the following
definition can easily be formulated in that precise framework.

  \begin{definition}[]\label{thm:32}
 An \emph{SN-stratification} on a manifold-with-corners 
 of dimension~$\le3$  
has the following specifications.

 \begin{enumerate}[label=\textnormal{(\roman*)}]

 \item \emph{codimension 0}: a codimension~0 smooth manifold;

 \item \emph{codimension 1}: a codimension~1 submanifold---the
link is a 0-sphere;

 \item \emph{codimension 2}: a
\emph{Type}\footnote{Mnemonic: Type~a is ``anodyne'', Type~b is ``branch.''}~\emph{a}
codimension~2 stratum has link a circle with an arbitrary
codimension~1 submanifold consisting of a finite set of points; a
\emph{Type~b} codimension~2 stratum has link a circle with a
codimension~1 submanifold consisting of 3~points;

 \item \emph{codimension 3}: a \emph{Type a} codimension~3 stratum has link a 2-sphere with an SN-stratification consisting of a codimension~1 trivalent graph whose vertices are of Type~a; a \emph{Type b} codimension~3 stratum has link
a 2-sphere with the standard SN-stratification of the boundary of a
tetrahedron (Construction~\ref{thm:35} below).

 \end{enumerate}

We use the term \emph{SN-stratified manifold} for a manifold equipped with an
SN-stratification. 
  \end{definition}

\insfig{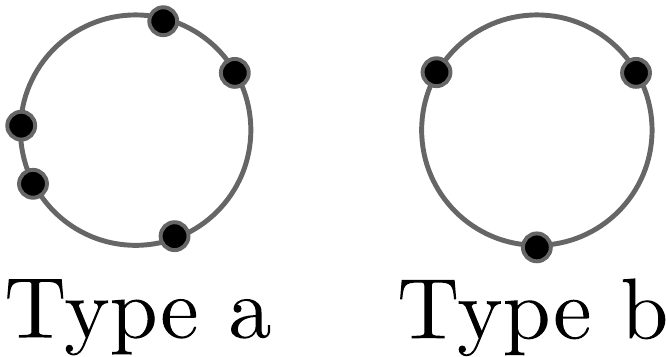}{0.4}{Links of codimension~2 strata. For Type a the link can contain 
an arbitrary number of points of $M_{-1}$, while for Type b it must contain exactly three.}

  \begin{remark}[]\label{thm:59}
 There is a generalization of Definition~\ref{thm:32} to manifolds with
boundary and corners.  The key point is that the SN-strata intersect
boundaries and corners transversely.  
  \end{remark}

 In \S\S\ref{subsec:4.4}, \ref{subsec:4.2} we define canonical
SN-stratifications associated to semi-ideal triangulations of 2-~and 3-manifolds.

  \begin{remark}[]\label{thm:33}
 An SN-stratified manifold~$M$ is decomposed as a disjoint union 
  \begin{equation}\label{eq:82}
     M=\mstrat0\,\amalg\, \mstrat{-1}\,\amalg\, \mstrat{\twoastrat}\,\amalg\,
      \mstrat{\threeastrat}\,\amalg\,\mstrat{\twobstrat} \, \amalg\, \mstrat{\threebstrat} 
  \end{equation}
where $\mstrat{\twoastrat}$ is the union of codimension~2 strata of Type~a,
$\mstrat{\twobstrat}$ is the union of codimension~2 strata of Type~b,
and likewise $\mstrat{\threeastrat}$, $\mstrat{\threebstrat}$ are the unions
of codimension~3 strata.  
This unusual
notation is convenient for subsequent definitions: the Type~a strata of
codimension~2,3 behave as singular parts of the codimension~1 strata.  Hence
define  
  \begin{equation}\label{eq:84}
     \begin{aligned} \mstrat{\gea} &= \mstrat0\amalg \mstrat{-1}\amalg
      \mstrat{\twoastrat} \amalg \mstrat{\threeastrat}, \\ \mstrat{\geb} &= \mstrat0\amalg \mstrat{-1}\amalg
      \mstrat{\twoastrat} \amalg \mstrat{\threeastrat} \amalg \mstrat{\twobstrat}. \end{aligned} 
  \end{equation}
  \end{remark}

   \subsubsection{Rank one Lie groups}\label{subsubsec:4.3.2}
  Spectral networks and abelianization data are conveniently formalized in
terms of a triple of complex Lie groups $G\supset H\supset T$ in which $T$~is
a (complex) maximal torus of~$G$ and $H$~its normalizer.  In this paper we
restrict to the groups $\GLC$, $\SLC$, and very occasionally $\PSLC$.  For
$G=\GLC$ we choose $T\cong \Cx\times \Cx$ the subgroup of diagonal matrices;
then its normalizer is
  \begin{equation}\label{eq:34}
     H=\left\{ \begin{pmatrix} *&0\\0&* \end{pmatrix} \right\} \cup \left\{
     \begin{pmatrix} 0&*\\ *&0 \end{pmatrix} \right\} \subset \GLC, 
  \end{equation}
a 2-component Lie group with identity component~$T$.  Choose the diagonal
matrices to be the maximal torus of~$\SLC$ and its image in~$\PSLC$ to be the
maximal torus in the projective linear group; in each case the normalizer~$H$
of~$T$ has two components.  Let $U\subset \GLC$ be the subgroup
  \begin{equation}\label{eq:35}
     U=\left\{ \begin{pmatrix} 1&z\\0&1 \end{pmatrix}:z\in \CC \right\} 
  \end{equation}
of upper triangular unipotent matrices.  Then $U\subset \SLC$ as well, and
$U$~projects to a unipotent subgroup of~$\PSLC$.  

  \begin{remark}[]\label{thm:8}
 Each of the three groups~$G$ acts on the projective line $\PP(\CC^2)=\CP^1$.
In each case $H$~is the stabilizer subgroup of the 2-point subset $A\subset
\CP^1$ of the axes in~$\CC^2$, and $T$~is the subgroup of elements of~$H$
that act as the identity on~$A$.  The stabilizer of the first axis $\ell
\subset \CC^2$ is a Borel subgroup~$B\subset G$, and there is a
diffeomorphism $G/B\approx \CP^1$.  Then for $G=\GLC$ or $\SLC$, the
Borel~ $B$ acts linearly on~$\ell $ and $\CC^2/\ell $, and $U\subset B$ is
the subgroup of elements that act trivially on both~$\ell $ and $\CC^2/\ell
$.
  \end{remark}

   \subsubsection{Definition of spectral networks and stratified
   abelianization data}\label{subsubsec:4.3.3}
 Assume $G\supset H\supset T$ is one of the three triples defined
in~\S\ref{subsubsec:4.3.2}.  We refer to it as the pair~$(G,T)$, since $H$~is
determined as the normalizer of~$T\subset G$.  Some
notation: If $Q\to M$ is a principal $H$-bundle, then we denote by
$\iota(Q)=Q\times _HG\to M$ its ``inflation'' to a principal $G$-bundle.
Also, if $w\subset \mstrat{-1}$ is a wall (a component),  and $R\to w$ is a
principal $T$-bundle, then there is an associated fiber bundle of groups 
  \begin{equation}\label{eq:180}
     U_w = R\times \mstrut _{T}U\longrightarrow w, 
  \end{equation}
where $T$~acts on~$U$ by conjugation.

The following definition applies to rank one groups, as does
Definition~\ref{thm:32}; there are stratifications, spectral networks, and
stratified abelianization data in higher rank as well \cite{GMN1,GMN2,LP,IM}.

  \begin{definition}[]\label{thm:6}
 Let $M$~be a compact manifold of dimension~$\le 3$ with boundary.  Suppose
$M$~is equipped with an SN-stratification $M\setminus \bM=M_0\amalg
M_{-1}\amalg \mstrat{\twoastrat}\amalg \mstrat{\threeastrat} \amalg \mstrat{\twobstrat}\amalg \mstrat{\threebstrat}$.  

 \begin{enumerate}[label=\textnormal{(\roman*)}]

 \item A \emph{spectral network $\sN=(\pi ,s)$ subordinate to the
stratification of~$M$} is:

  \begin{itemize}

 \item a double cover $\pi \:\tM_{\gea}\to \mstrat{\gea}$ which restricts
nontrivially to the link of each point in~$\mstrat{\twobstrat}$

 \item a section~$s$ of~$\pi $ over~$\mstrat{-1}\amalg  \mstrat{\twoastrat} \amalg \mstrat{\threeastrat}$

  \end{itemize}

 \item \emph{Stratified abelianization data $\cA=(P,Q,\mu ,\theta )$ of
type~$(G,T)$ over~$(M,\sN)$} is the data:

  \begin{itemize}

 \item a principal $G$-bundle $P\to M$ with flat connection

 \item a principal $H$-bundle $Q\to\mstrat{\gea}$ with flat connection

 \item an isomorphism of double covers $\mu \:\tM_{\gea}\to Q/T$
over~$\mstrat{\gea}$

 \item a flat isomorphism $\theta \:\iota(Q)\to P$ over~$M_0$

  \end{itemize}

\noindent
 We require that the discontinuity of~$\theta $ lie in~$U_w\to w$ as we cross
a point of the wall~$w\subset M_{-1}$.

 \end{enumerate}
  \end{definition}

\noindent
 Observe that the section~$s$ reduces the restriction of $Q\to
\mstrat{\gea}$ over $\mstrat{-1}\amalg \mstrat{\twoastrat} \amalg \mstrat{\threeastrat}$ to a principal
$T$-bundle; on a wall $w\subset M_{-1}$ the fiber bundle of groups $U_w\to w$
is defined in~\eqref{eq:180}.  Stratified abelianization 
data over a given $(M,\cN)$ form a category; we leave to the reader the definition
of the morphisms.
Our usage of the term `spectral network' often
includes the underlying SN-stratification.

  \begin{remark}[]\label{thm:7}
 Definition~\ref{thm:6} is adequate for our purposes but 
 does not capture the most general rank one spectral
networks which can occur in nature, e.g. from trajectory structures of
meromorphic or holomorphic quadratic differentials on Riemann surfaces; 
see ~\cite{GMN2,HN,Fe}, for example.
  \end{remark}

  \subsection{2-dimensional spectral networks from triangulations}\label{subsec:4.4}

Let $\AA^2$~denote the standard affine plane.  Denote the convex hull of a
subset~$T\subset \AA^2$ as~$\Conv(T)$.  An affine triangle~$\hat \Delta $ is the
convex hull $\Conv(p_0,p_1,p_2)$ of three non-collinear points~$p_0,p_1,p_2$.
Fix some $\varepsilon \in (0,\frac12)$, say $\varepsilon = \frac{1}{10}$. 
The truncated triangle $\Delta \subset \hat \Delta$ is the convex hull 
of the six points $(1-\varepsilon) p_i + \varepsilon p_j$ for $i \neq j$, 
as shown in Figure~\ref{fig:truncated-triangle}.
\insfig{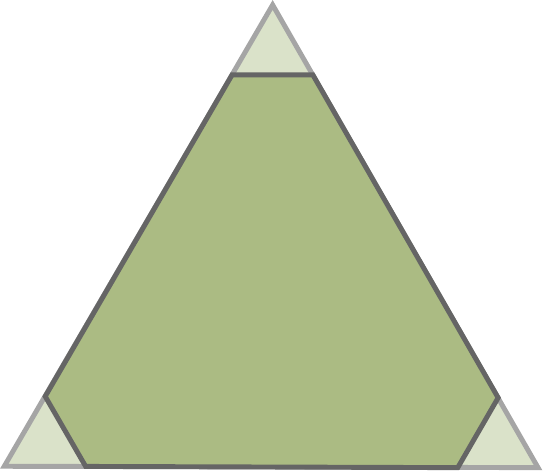}{0.4}{The truncated affine triangle $\Delta$ contained in the affine triangle $\hat\Delta$.}
We will sometimes refer to ``edges'' or ``vertices'' of $\Delta$, meaning 
the corresponding edges or vertices of $\hat\Delta$.

Let $S^\circ$~ be the quotient of a finite union of disjoint truncated affine
triangles~$\{\Delta_i\}_{i\in I}$ 
whose edges are identified in pairs via
affine isomorphisms.
Then $S^\circ$~ can 
be given the structure of a smooth compact 2-manifold with boundary.\footnote{Indeed, 
since edges are identified in pairs, a neighborhood of any point on a glued edge
is a disc; moreover the link of a vertex is easily seen to be a circle.}
The gluing of edges induces an equivalence relation on the $3N$ vertices of the
$N$ triangles $\Delta_i$; 
each equivalence class of vertices
corresponds to a boundary component of $S^\circ$, with the topology of a circle.
We sometimes call such an equivalence class a ``glued vertex'' or simply a ``vertex''.
Partition the glued vertices into two subsets of
\emph{interior} and \emph{ideal} vertices.
Let $S$ be a space 
obtained by gluing a copy of the standard disc to each boundary component 
of $S^\circ$ corresponding
to an interior vertex. $S$ is a smooth compact 2-manifold with boundary;
$\pi_0(\partial S)$ is canonically identified with the set of ideal
vertices.

  \begin{definition}[]\label{thm:76}
  Let $Y$ be a compact 2-manifold with boundary. A \emph{semi-ideal triangulation}
  of $Y$ is a diffeomorphism $S \to Y$, where $S$ is a
  space of the sort just described. The semi-ideal triangulation
  is called \emph{ideal} if all vertices are ideal, and just a \emph{triangulation}
  if all vertices are interior.
  \end{definition}


\insfig{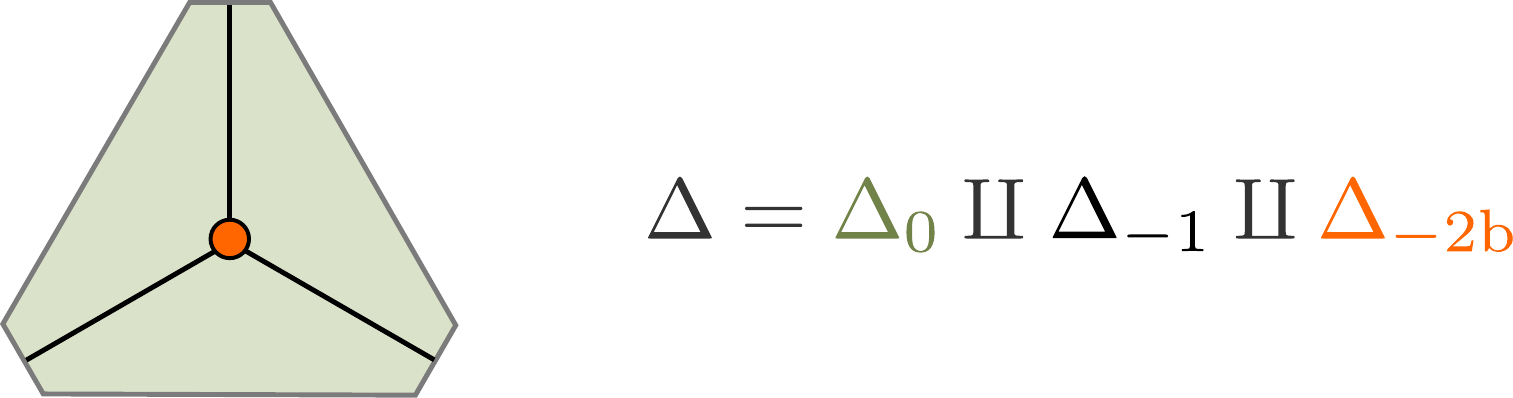}{0.4}{The SN-stratification of a truncated affine triangle.}

\begin{construction}[SN-stratification of a truncated triangle]\label{thm:77}
 A truncated affine triangle $\Delta$ carries a canonical
SN-stratification $\Delta = \Delta _0\amalg \Delta _{-1}\amalg \Delta _{\twobstrat}$ as follows. Let $c=(p_0+p_1+p_2)/3$~be the
barycenter of~$\Delta $.  Set $\Delta _{\twobstrat}=\{c\}$; 
the stratum~$\Delta _{-1}$ is the union of the three
line segments $(\Conv(p_i,c) \cap \Delta) \setminus \Delta \mstrut _{\twobstrat}$, $i=0,1,2$;
and $\Delta _0$~is the complement of~$\Delta _{\le-1}$.  This
SN-stratification is depicted in Figure~\ref{fig:triangle-stratification}.
  \end{construction}

\begin{construction}[SN-stratification of the standard disc] \label{constr:sn-disc}
Let $D$ be the standard closed disc. For any finite subset $W \subset \partial D$
we obtain an SN-stratification $D = D_0 \amalg D_{-1} \amalg D_{\twoastrat}$ as follows.
Let $c$ be the center of $D$. Then $D_{\twoastrat} = \{c\}$; $D_{-1}$ is the union of line
segments connecting $c$ to each point of $W$; and $D_0$ is the complement of $D_{\le -1}$.
See Figure~\ref{fig:stratified-disc}.
\end{construction}

\insfig{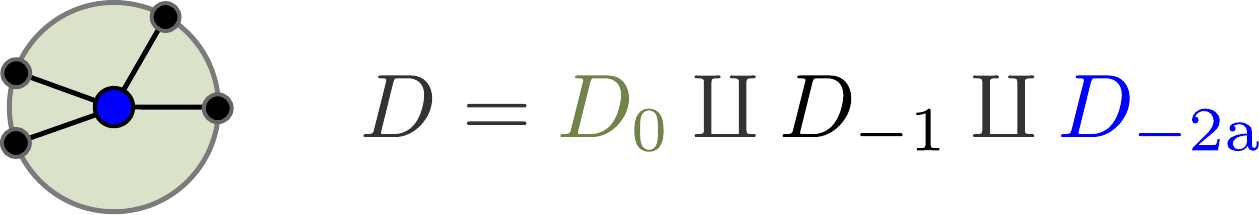}{0.4}{An SN-stratification of the disc, determined by a finite subset of the boundary circle.}
 
  \begin{construction}[SN-stratification of a triangulated
  surface]\label{thm:34}

Let $Y$~be a closed 2-manifold equipped with a semi-ideal 
triangulation~$\sT$.  Then transport of the SN-stratifications on 
the truncated triangles and the discs around interior vertices
defines an SN-stratification of~$Y$. Figure \ref{fig:two-triangles-strata}
is an example where there are no interior vertices.
See Figure~\ref{fig:disc-glued-in} for an example with an interior vertex.

\insfig{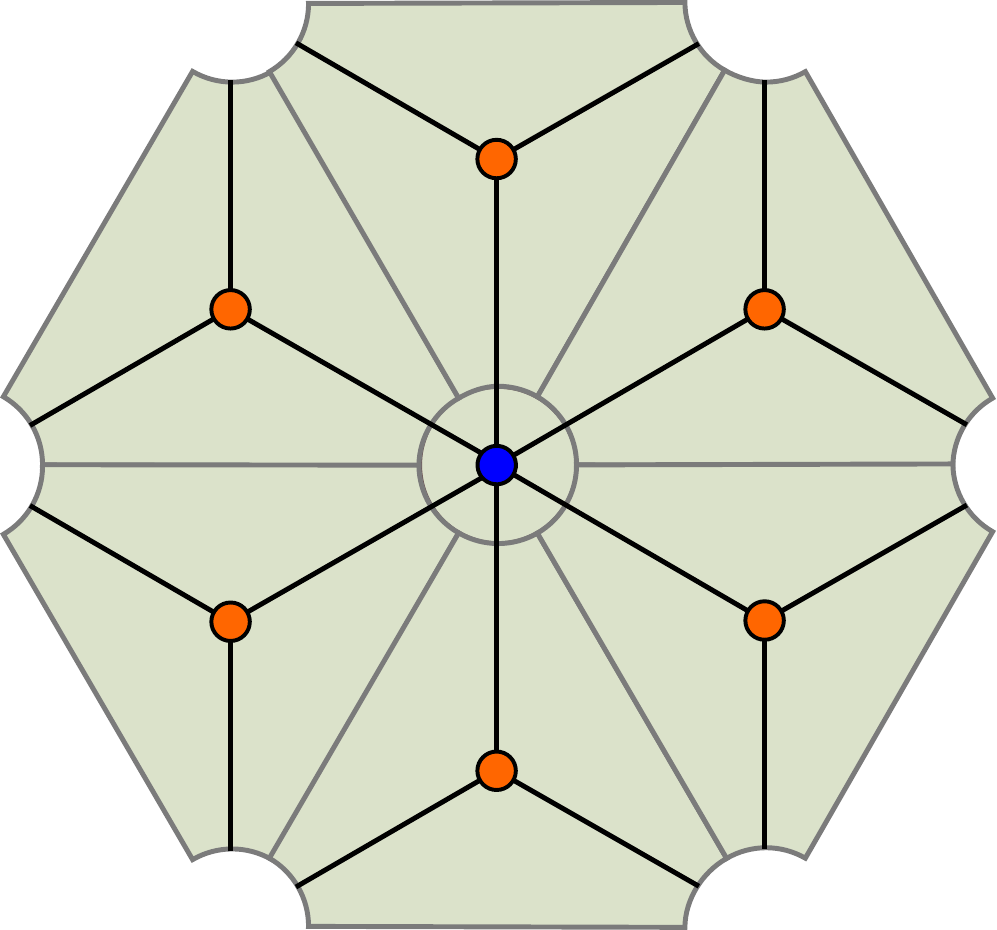}{0.4}{A portion of the SN-stratification of a semi-ideally triangulated 
closed 2-manifold with an interior vertex (center).}

  \end{construction} 

  \begin{construction}[spectral network on a triangulated surface]\label{thm:9}
 Let 
$\Delta =\Delta _0\amalg \Delta _{-1}\amalg \Delta_{\twobstrat}$ 
be the canonical SN-stratification of a truncated affine triangle 
(Figure~\ref{fig:triangle-stratification}).  
Each component of~$\Delta _0$ has boundary containing 
precisely one edge of $\Delta$, with two vertices. Let $\pi
\:\tDelta_0\to\Delta _0$ be the (trivializable) double cover whose fiber
consists of those two vertices.  Each component of~$\Delta _{-1}$ is in the
closure of two components of~$\Delta _0$, with one vertex in common.  
Glue the corresponding sheets of the double cover to
define $\pi \:\tDelta_{\ge-1}\to\Delta _{\ge-1}$ together with a section~$s$
over~$\Delta _{-1}$, i.e., a spectral network on ~$\Delta$.  
This construction glues across edges and extends to discs around interior
vertices, and thus transports to give a spectral network $\cN^{\sT}$ over 
a semi-ideally triangulated surface~$(Y,\sT)$.
  \end{construction}

  \begin{construction}[stratified abelianization data on a semi-ideally
triangulated surface]\label{thm:14}

   Consider first $G=\GLC$ or $G=\SLC$.  Assume that $Y$~has no closed
components and is equipped with a semi-ideal triangulation~$\sT $.  Let $P\to
Y$ be a flat principal $G$-bundle. On each component of~$\partial Y$,
choose a flat section of the associated $\CP^1$-bundle $P/B\res{\bC}\to\bC$,
as in~\S\ref{subsec:4.1}; see Remark~\ref{thm:8} for the definition of the
Borel subgroup~$B\subset G$. Also choose an element of the 
fiber of $P/B$ over each interior vertex. Use parallel transport---as
in~\S\ref{subsec:4.1}---to obtain two flat sections~$s,s'$ of
$P/B\res{Y_0}\to Y_0$.   The following is a generalization of Genericity
Assumption~\ref{thm:4}:  

  \begin{assumption}[genericity]\label{thm:74}
 The sections~$s,s'$ are nowhere equal.  
  \end{assumption}

Identify~$P\to Y$ as a bundle of bases of a rank~2 complex vector bundle
$E\to Y$.  The submanifold of bases contained in the lines defined by the
sections~$s,s'$ determines a reduction of the principal $G$-bundle $P\to Y_0$
to a principal $H$-bundle $Q\to Y_0$.  For~$c\in Y_{-1}$, the limits
of~$s,s'$ from the two sides of $Y_{-1}\subset \strat{\ge-1}$ give three
points $\ell \mstrut _1,\ell '_2,\ell ''_2\in (P/B)_c$ in the projective
line~$\PP E_c$ over~$c$.  One of the sections has the same limit~$\ell _1$ on
both sides; the other has two possibly distinct limits.  Let $B_{\ell
_1}\subset \Aut P_c$ be the subgroup of elements which fix~$\ell _1$.  Then
$(P/B)_c$~is the projectivization~$\PP E_c$ of the 2-dimensional vector
space~$E_c$, the group~$B_{\ell _1}$ acts linearly on~$E_c$, and we
define~$\varphi _c$ to be the unipotent element~\eqref{eq:30}.  Glue
using~$\varphi _c$ at each~$c\in Y_{-1}$ to construct a flat principal
$H$-bundle $Q\to Y_{\ge-1}$.  This gives most of the stratified
abelianization data Definition~\ref{thm:6}(ii).  We leave the rest to the
reader, as do we the slight modifications for $G=\PSLC$.
  \end{construction} 

  \begin{remark}[]\label{thm:10}
 The projection $Q\to\tY_{\getwoa}$, defined via the isomorphism~$\mu $ of
double covers, is a principal $T$-bundle.  If $G=\SLC$ or~$\PSLC$, then
$T\cong \Cx$.  If $G=\GLC$, there is an associated principal $\Cx$-bundle
from the character $\left(\begin{smallmatrix} \lambda _1\\&\lambda
_2  \end{smallmatrix}\right)\mapsto\lambda _1$ of~$T$.  Let $L\to\tY_{\getwoa}$
be the associated flat line bundle.  Lemma~\ref{thm:5} holds in this more
general situation.
  \end{remark}

We conclude with a theorem about stratified abelianizations over a single
triangle~$\Delta $ equipped with the standard spectral network~$\cN$ depicted
in Figure~\ref{fig:triangle-stratification}.  Specialize to $G=\SLC$ and the corresponding
subgroups~$T,H,B,U$.  In this case there is a unique stratified
abelianization, whose automorphism group is $\bmuu_2$, in the following
sense.

\begin{prop} \label{prop:uniqueness-of-triangle-ab} Let $\cA =
(P,Q,\mu,\theta)$ and $\cA' = (P',Q',\mu',\theta')$ be stratified
abelianization data over $(\Delta,\cN)$. Then there is an isomorphism $\cA \to
\cA'$, unique up to composition with the simultaneous action of $-1$ on $P$
and $Q$.
\end{prop}

  \begin{proof} 
 First we construct a map of flat bundles $\varphi_Q: Q \to Q'$.  The
monodromy of $Q$ around $\partial \Delta$ lies in $H \setminus T$, since $Q/T
\simeq \tDelta$ is the nontrivial double cover, and likewise for $Q'$.  But
now recall that all elements of $H \setminus T$ are conjugate in $H$.  It
follows that there exists an isomorphism $\varphi_Q: Q \to Q'$ of flat
$H$-bundles, unique up to composition with an automorphism of $Q\to \Delta
$. 

The automorphism group of $Q\to \Delta $ is the commutant of the monodromy,
which is a cyclic group of order $4$; either generator acts nontrivially on
$Q/T$, and the order $2$ element acts by $-1 \in H$.  Thus, by composing with
an automorphism of $Q\to \Delta $ if necessary, we may arrange that $\mu
\circ \varphi_Q = \mu'$, and the remaining freedom in $\varphi_Q$ is
composition with the action of $-1 \in H$.

Next we construct a map of flat bundles $\varphi_P: P \to P'$.  Along each
wall $w$ we have a section $s_w$ of $Q/T$.  On either side of the wall,
$\theta(s_w)$ then gives a section of $P/T\to w$; the condition on the
discontinuity of $\theta$ ensures that their projections to $P/B$ agree, thus
giving a section $o_w$ of $P/B\to w$.  Because $G / \{ \pm 1 \} = \PSLC$
acts simply transitively on triples of distinct points of $G/B \simeq
\C\PP^1$, there exists $\varphi_P: P \to P'$ which maps $o_w$ to $o'_w$ for
all three walls $w$, and such a $\varphi_P$ is unique up to a sign.

Finally we need to check that on $\Delta_0$ we have (possibly after composing
$\varphi_P$ with the action of $-1 \in G$)
\begin{equation}
\varphi_P = \theta' \circ \varphi_Q \circ \theta^{-1}.
\end{equation}
For this we consider the difference $\xi = \varphi_P^{-1} \circ \theta' \circ
\varphi_Q \circ \theta^{-1}$ which is a covariantly constant section 
of $\Aut(P) \vert_{\Delta_0}$, with two properties:
\begin{itemize}
  \item
In a component of $\Delta_0$ bounded by two walls $w$, $w'$, the difference~
$\xi$ belongs to the subgroup $T_{ww'} \simeq T$ preserving $o_w$ and $o_{w'}$.
Thus $\xi$ acts by a constant scalar $\lambda_w$ on $o_w$, with $\lambda_{w'} = \lambda_w^{-1}$.
\item
The discontinuity of $\xi$ across $w$ belongs to the subgroup $U_w \simeq U$. It follows that
$\lambda_w$ is the same on both sides of $w$.
\end{itemize}
Labeling the three walls as $w_i$ (with $i$ mod 3), the above properties say
$\lambda_{w_{i+1}} = \lambda_{w_i}^{-1}$, which gives $\lambda_{w_{i+3}} = \lambda_{w_i}^{-1}$, so $\lambda_w = \lambda_w^{-1} = \lambda_{w_{i+1}}$ 
and thus $\xi = \pm 1$. This completes the proof.
\end{proof}

  \subsection{3-dimensional spectral networks from
triangulations}\label{subsec:4.2} 

We begin with a 3-dimensional analog of Definition~\ref{thm:76}.  A
\emph{tetrahedron} $\hat\tetra$ in~ $\AA^3$ is the convex hull $\Conv(p_0,p_1,p_2,p_3)$ of
four points in general position.  
The truncated tetrahedron $\tetra \subset \hat\tetra$ is
the convex hull of the 12 points $(1 - \varepsilon) p_i + \varepsilon p_j$ with $j \neq i$. See Figure~\ref{fig:bare-tetrahedron}. 
We will sometimes refer to ``faces'', ``edges''
or ``vertices'' of $\tetra$, meaning those of $\hat\tetra$.

\insfigpng{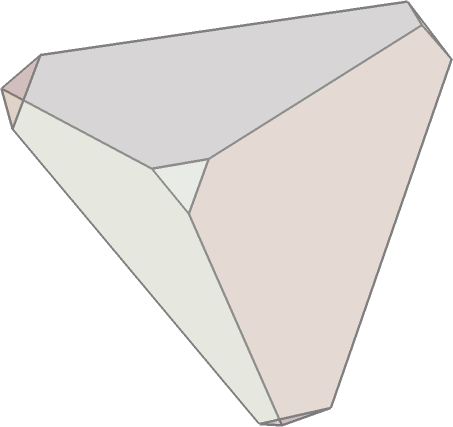}{0.25}{A truncated tetrahedron.}

Let $S^\circ$~ be the quotient of a finite union of disjoint affine truncated
tetrahedra~$\{\tet_i\}_{i\in I}$ 
whose faces are identified in pairs via
affine isomorphisms.
Then $S^\circ$~ can 
be given the structure of a smooth compact 3-manifold with boundary.
The gluing of faces induces an equivalence relation on the $4N$ vertices of the
$N$ tetrahedra $\tetra_i$; 
each equivalence class of vertices
corresponds to a boundary component of $S^\circ$, which is a compact connected surface.
We sometimes call such an equivalence class a ``glued vertex'' or simply a ``vertex''.
Partition the glued vertices into two subsets of
\emph{interior} and \emph{ideal} vertices, subject to the condition
that the boundary component corresponding to an interior vertex must be diffeomorphic to $S^2$.
Let $S$ be a space 
obtained by gluing a copy of the standard 3-disc to each boundary component 
of $S^\circ$ corresponding
to an interior vertex. $S$ is a smooth compact 3-manifold with boundary;
$\pi_0(\partial S)$ is canonically identified with the set of ideal
vertices.

  \begin{definition}[]\label{thm:78}
  Let $Y$ be a compact 3-manifold with boundary. A \emph{semi-ideal triangulation}
  of $Y$ is a diffeomorphism $S \to Y$, where $S$ is a
  space of the sort just described. The semi-ideal triangulation
  is called \emph{ideal} if all vertices are ideal, and just a \emph{triangulation}
  if all vertices are interior.
  \end{definition}

\insfigpng{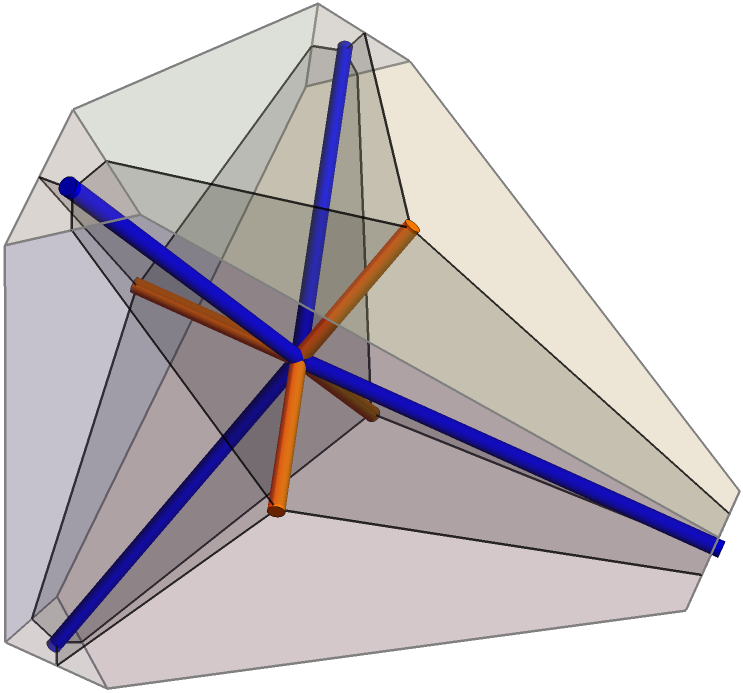}{0.22}{The canonical SN-stratification on a truncated affine tetrahedron.}
\begin{construction}[Spectral network on a tetrahedron]\label{thm:35}
 Let $\hat\tetra =\Conv(p_0,p_1,p_2,p_3)$ be a tetrahedron in~$\AA^3$.  Let
$q_i=(p_{i+1}+p_{i+2}+p_{i+3})/3$ be the barycenter of the face
opposite~$p_i$, $i=0,1,2,3$; set $c=(p_0+p_1+p_2+p_3)/4$ the barycenter
of~$\tetra $.  (We use $p_{i+4}=p_i$, $i=0,1,2,3$.)  Figure~\ref{fig:tetrahedron-network}
depicts a canonical
SN-stratification of $\tetra$,
  \begin{equation}\label{eq:37}
     \begin{aligned} \tetra _{\threebstrat}&=\{c\} \\ 
     \tetra
      _{\twobstrat}&=\bigcup\limits_{i=0}^4\Conv(q_i,c)\;\setminus \;\tetra _{\threebstrat} \\
\tetra_{\threeastrat}&=\emptyset \\
     \tetra_{\twoastrat}&=\bigcup\limits_{i=0}^4\Conv(p_i,c) \cap \tetra \;\setminus \;\tetra
     _{\threebstrat} \\ \tetra
      _{-1}&=\bigcup\limits_{i=0}^4\bigcup\limits_{j=1}^3\Conv(p_{i+j},q_i,c) \cap \tetra
      \;\setminus \; \tetra _{\gea} \\ \tetra _{0}\;\;&=\tetra
      \;\setminus \;\tetra _{\ge-1} \\
      \end{aligned} 
  \end{equation}
The link~$Y_c$ of~$\tetra _{\threebstrat}$ is a 2-sphere triangulated as the boundary
of a tetrahedron.  By Construction~\ref{thm:9} it has a canonical
SN-stratification---the restriction of~\eqref{eq:37} to~$Y_c$---and
subordinate spectral network; see~Figure~\ref{fig:tet-unfolded}.  The same construction
extends the SN-stratification to a 3-dimensional spectral network subordinate
to~\eqref{eq:37}.  Namely, each component~$U$ of~$\tetra _0$ contains one edge
with two vertices, and each component of~$\tetra _{-1}$
in~$\overline U$ corresponds to one of those vertices.  Let $\pi
\:\tD_0\to\tetra _0$ be the (trivializable) double cover whose fiber over~$U$
is the aforementioned set of two vertices, and glue along~$\tetra
_{\twoastrat}\amalg\tetra _{-1}$ by identifying the common vertex on each wall.
This produces a double cover $\pi \:\tD_{\gea}\to\tetra _{\gea}$ with a
section~$s$ over~$\tetra _{\twoastrat}\amalg\tetra _{-1}$, i.e., a spectral network.
There is an extension to a branched double cover $\pi \:\tD
_{\geb}\to\tetra _{\geb}$ with branch locus~$\tetra _{\twobstrat}$.
  \end{construction}

\insfig{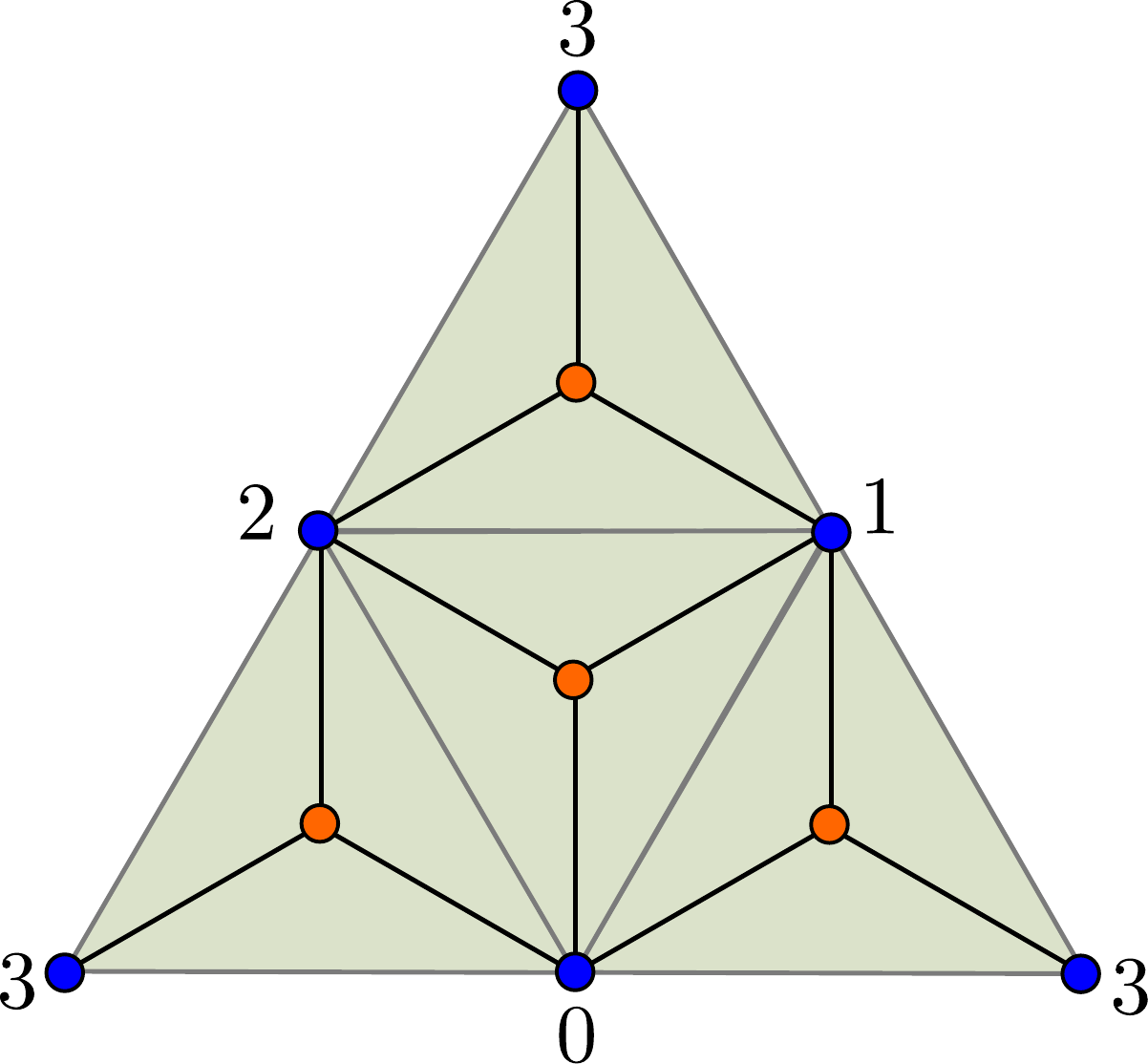}{0.4}{The spectral network on~$Y=Y_c$.}

  \begin{remark}[]\label{thm:11}
 It will be convenient to excise an open ball about~$c$ as well as its
inverse image on the branched double cover.
  \end{remark}

\insfig{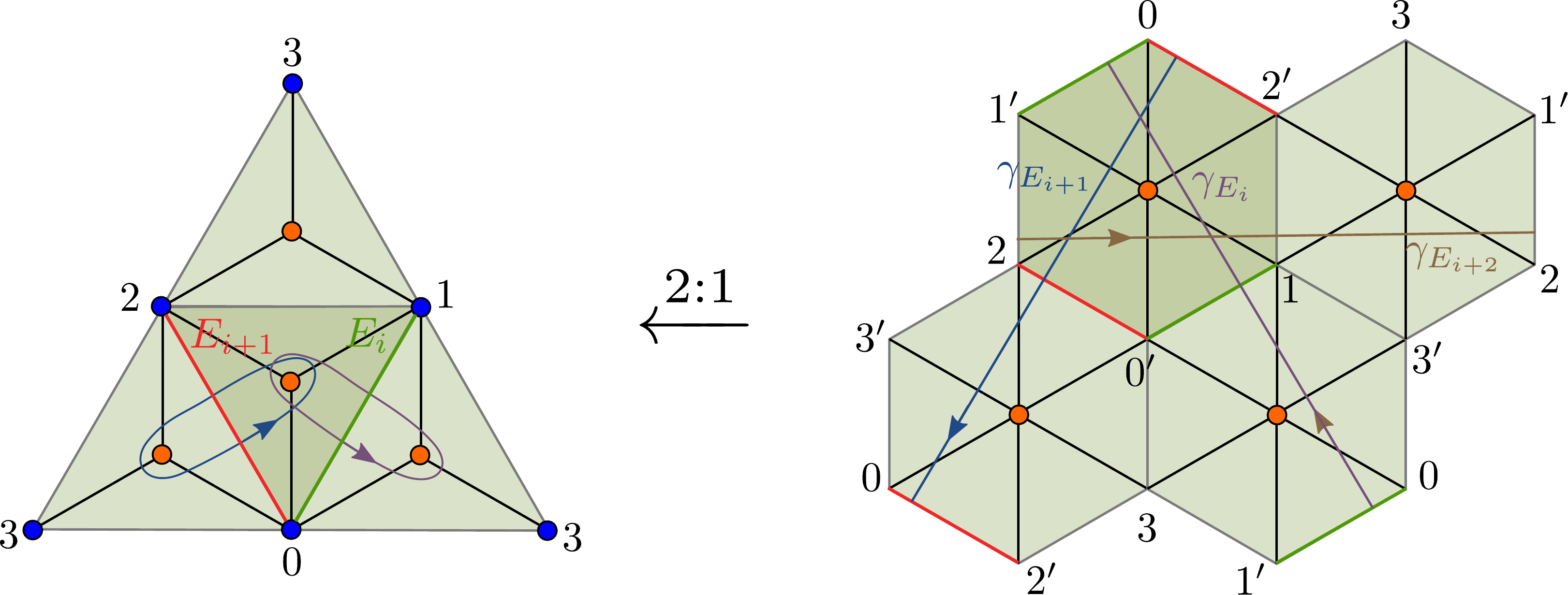}{0.35}{The branched double cover of the tetrahedrally triangulated
  2-sphere $Y_c$ by a 2-torus $\tY_c$, and a distinguished triple of cycles on $\tY_c$.}

 We investigate stratified abelianization on the link~$Y=Y_c$ of the
barycenter of~$\tetra $.  The stratum~$Y_{\twobstrat}$ of~$Y$ consists of 4~points,
and the double cover $\pi \:\tY_{\getwoa}\to Y_{\getwoa}$ extends to a branched
double cover $\pi \:\tY\to Y$ in which $\tY$~is diffeomorphic to a 2-torus.
The double cover~$\pi $ is depicted in Figure~\ref{fig:tet-cover-cycles}.  The boundary of
the tetrahedron has been unfolded, as in Figure~\ref{fig:tet-unfolded}, as has been the
covering 2-torus.  Assume $Y$~is oriented, and use~$\pi $ to induce an
orientation on~$\tY$.  To each edge~$E$ in~$\tetra$ associate an element
$\Gamma _E\in H_1(\tY_{\getwoa})$ as follows.  Let $\bar\lambda _E\subset
Y_{\ge-1}$ be a loop which crosses~$E$ twice transversely and encircles the
branch points in the faces of~$\tetra$ which abut~$E$.  Orient~$\bar\lambda
_E$ as the boundary of the region which contains these two branch points.
The desired lift~$\lambda _E$ is distinguished from the other lift
of~$\bar\lambda _E$ as follows: the lifts to~$\lambda _E$ of the two
intersection points~$\bar\lambda _E\cap E$ lie on the sheet of the double
cover labeled by the closest endpoint of~$E$.  Then $\Gamma _E$~is the
homology class of~$\lambda _E$ in~$H_1(\tY_{\getwoa})$, the homology of the
torus with the 4~branch points excised.  Let $\gamma _E\in H_1(\tY)$ be its
image in the homology of the torus.  We invite the reader to deduce the
following, using Figure~\ref{fig:tet-cover-cycles}.

  \begin{proposition}[]\label{thm:68}
 \ 
\begin{enumerate}[label=\textnormal{(\arabic*)}]

 \item The image of $\Gamma _E\in H_1(\tY_{\getwoa})$ under the deck
transformation is~$-\Gamma _E$.\\[-3pt]

 \item Opposite edges of~$\tetra$, such as~13 and~02 in the figure, induce
the same homology class in~$H_1(\tY_{\getwoa})$.\\[-3pt]

 \item The three pairs of opposite edges lead to three homology classes
$\gamma _0,\gamma _1,\gamma _2\in H_1(\tY)$ which sum to zero.

 \end{enumerate}
  \end{proposition}

\noindent 
 Cyclically order the three pairs of opposite edges so that the intersection
product $\langle \gamma _i,\gamma _{i+1} \rangle=+1$ for~$i\in \zmod3$.
Denote the corresponding loops in~$\tY_{\getwoa}$ as $\gE0,\gE1,\gE2$.

Let $\tY_{\twobstrat}\subset \tY$ be the set of 4~branch points~$\pi \inv
(Y_{\twobstrat})$.  Since $Y$~is simply connected, any flat $G$-bundle $P\to Y$ is
trivializable.  Fix~$\ell _i$ in the fiber of the associated $\CP^1$-bundle
$P/B\to Y$ at each vertex~$p_i\in \tetra$.  Let $\sE\approx \CP^1$ be the
space of horizontal sections of $P/B\to Y$, and let~$\ell _i\in \sE$ be the
extension of the previous~$\ell _i$ to a horizontal section.  Assume that
$\ell _0,\dots ,\ell _3\in \sE$ are distinct; this implies the Genericity
Assumption~\ref{thm:74}.  Use Construction~\ref{thm:14} to produce a
stratified abelianization.  By Remark~\ref{thm:10} there is a flat line
bundle
  \begin{equation}\label{eq:38}
     L\longrightarrow \tY_{\getwoa} 
  \end{equation}
with holonomy~$-1$ around each of the 4~(branch) points in~$\tY_{\twobstrat}$.  The
isomorphism class of the flat bundle~\eqref{eq:38} is determined by its
holonomy, a homomorphism 
  \begin{equation}\label{eq:128}
     \holL\:H_1(\tY_{\getwoa})\longrightarrow \Cx. 
  \end{equation}
Set $\ho_i=\holL(\gE i)$ for~$i\in \zmod3$.

  \begin{proposition}[]\label{thm:12}
 The holonomies of~$L$ satisfy 
  \begin{equation}\label{eq:183}
     \ho_{i+1} = 1 - \frac{1}{\ho_i}. 
  \end{equation}
  \end{proposition}

  \begin{proof}
 We compute as in the proof of Lemma~\ref{thm:5} using Figure~\ref{fig:tet-cover-cycles} as
a guide.  The holonomy of~$L\to\tY$ around~$\gE i$ is the
composition
  \begin{equation}\label{eq:129}
     \ell _0\xrightarrow{\;\;\proj3\;\;}\ell _1
     \xrightarrow{\;\;\proj2\;\;}\ell _0
  \end{equation}
and the holonomy of~$L\to\tY$ around~$\gE{i+1}$ is the
composition  
  \begin{equation}\label{eq:41}
     \ell _2\xrightarrow{\;\;\proj3\;\;}\ell _0
     \xrightarrow{\;\;\proj1\;\;} \ell _2 
  \end{equation}
(Recall the projections in~\eqref{eq:29}.)  Choose $\xi _j\in \ell _j^{\neq
0}$, $j=1$ and then $j=3,0,2$, such that
  \begin{equation}\label{eq:42}
     \begin{aligned} \xi _1&=\xi _3+\xi _0 \\ &=\xi _2 + z \xi
      _0\end{aligned} 
  \end{equation}
for some~$z\in \CC\setminus \{0,1\}$.  Then $\proj3(\xi _0)=\xi _1$ and
$\proj3(\xi _2)= (1-z)\xi _0$, etc.  Hence the image of~$\xi _0$
under~\eqref{eq:129} is~$z\xi _0$, and the image of~$\xi _2$
under~\eqref{eq:41} is $(1-\frac1z)\xi _2$.  Therefore,
  \begin{qedequation}\label{eq:130}
     \ho_{i+1}=\holL(\gE{i+1})=1-\frac{1}{z} = 1-\frac{1}{\holL(\gE i)} = 1 -
     \frac{1}{\ho_i}.  
  \qedhere\end{qedequation}
  \renewcommand{\qedsymbol}{}
  \end{proof}

  \begin{remark}[]\label{thm:51}
 \ 
 \begin{enumerate}[label=\textnormal{(\arabic*)}]

 \item One interpretation of $z=\holL(\gE i)$ is as follows.  Recall that
4~distinct points in a projective line~$\PP F$ are characterized up to
isomorphism by their cross-ratio.  If $\ell _0,\ell _1,\ell _2,\ell _3$ are
the corresponding lines in the 2-dimensional vector space~$F$, then the
cross-ratio is
  \begin{equation}\label{eq:40}
     \frac{(\xi _0\wedge \xi
     _3)(\xi _1\wedge \xi _2)}{(\xi _0\wedge \xi _2)(\xi _1\wedge \xi _3)}
     \in \CC\setminus \{0,1\},\qquad 
     \textnormal{$\xi _i\in \ell _i$ 
     nonzero}, 
  \end{equation}
where the numerator and denominator are nonzero elements in~$(\Det
F)^{\otimes 2}$; the ratio is independent of the choice of~$\xi _i\in \ell
_i^{\neq0}$.  Permuting the lines we obtain numbers $z$,\, $1/z$,\, $1-z$,\,
$1/(1-z)$,\, $z/(z-1)$,\, $(z-1)/z$ for some~$z\in \Cx\setminus \{1\}$.  In
the case at hand, with the chosen vectors~$\xi _1,\xi _2,\xi _3,\xi _4$
in~\eqref{eq:42}, we compute
  \begin{equation}\label{eq:43}
          \frac{(\xi _0\wedge \xi
     _3)(\xi _1\wedge \xi _2)}{(\xi _0\wedge \xi _2)(\xi _1\wedge \xi
          _3)}=z=\holL(\gE i). 
  \end{equation}

 \item As a corollary of Proposition~\ref{thm:12} the product of the
holonomies around the loops $\gE0,\gE1,\gE2$ defined after
Proposition~\ref{thm:68} is 
  \begin{equation}\label{eq:186}
     z\biggl(1-\frac 1z\biggr) \biggl(1 - \frac{1}{1-\frac 1z}\biggr) = -1. 
  \end{equation}
This leads to a sharpening of Proposition~\ref{thm:68}(3).  Let $S\subset
\tY_{\getwoa}$ be a link of the 4~points $\tY_{\twobstrat}\subset \tY_{\getwoa}$, so
$S=\bigsqcup_{k=1}^4 S_k$~is a union of 4~disjoint circles~$S_k$, one
surrounding each branch point.  Form the commutative diagram
  \begin{equation}\label{eq:187}
     \begin{gathered} \xymatrix{& H_1(S)\ar[r]\ar[d]^\chi &
     H_1(\tY_{\getwoa})\ar[r]\ar[d] & H_1(\tY) \ar@{=}[d] \ar[r] & 0\\ 0\ar[r]
     &\bmut\ar[r] & \Hext\ar[r] & H_1(\tY)\ar[r] & 0} \end{gathered} 
  \end{equation}
in which the homomorphism~$\chi $ maps a generator of~$H_1(S_k)$ to~$-1\in
\bmut$.   The bottom row of~\eqref{eq:187} is a central group extension.  The
refinement of Proposition~\ref{thm:68}(3) is that the product of the images
of~ $[\gE i]$ in~$\Hext$ is $-1\in \bmut$.\\[-3pt]

 \item The space $\bY=\CC\setminus \{0,1\}\approx \CP^1\setminus \{0,1,\infty
\}$ is the domain of the real dilogarithm function~\eqref{eq:1}, and the
total space of an abelian cover $\hY\to \bY$ is the domain of the enhanced
Rogers dilogarithm~\eqref{eq:7}.  In our current setup $\bY$~is a space of
flat $\Cx$-connections on a punctured torus.  In~\S\ref{sec:dilog}  we introduce
an extra twist to get rid of the punctures, and so identify~$\bY$ as a space
of flat $\Cx$-connections on a torus.  See~\cite{FN} for a development of the
dilogarithm function with this starting point.
 \end{enumerate}
  \end{remark}

\begin{construction}[SN-stratification on a 3-disc]
Let $D$ be the standard closed 3-disc. Given an SN-stratification of the boundary
$\partial D = S^2$,
of the form $\partial D = (\partial D)_0 \amalg (\partial D)_{-1} \amalg (\partial D)_{\twoastrat}$,
we obtain an SN-stratification of $D$ as follows. Let $c$ be the center of $D$.
Then $D_\threebstrat = \emptyset$, $D_\threeastrat = \{ c \}$,
and each other stratum $D_\alpha$ is the cone over $(\partial D)_\alpha$ with $c$ removed. 
\end{construction}

  \begin{construction}[Stratified abelianization data on a
  3-manifold]\label{thm:36} 
  Let $X$~be a compact 3-manifold with boundary, and suppose~$\sT$ is a
semi-ideal triangulation.  The SN-stratification ~\eqref{eq:37} and
subordinate spectral network on each truncated tetrahedron transport to~$X$,
and extend over the 3-discs around interior vertices.
In particular, there is a branched double cover $\pi \:\tX\to
X_{\geb}$ with branch locus~$X_{\twobstrat}$.
 
Suppose $P\to X$ is a flat principal $G$-bundle.  Assume there
exists\footnote{Existence condition: on each component of~$\bX$ the
holonomies around loops at a basepoint have a common eigenline.} a flat
section of the restriction of the associated $\CP^1$-bundle $P/B\to X$
to~$\bX$, and furthermore that we can and do choose a section such that
Genericity Assumption~\ref{thm:74} hold.  Excise from~$X$ open balls about
the barycenters of the tetrahedra.  Let~$X\subset \tX$ be the total space of
the double cover~$\pi $ with the inverse images of the balls excised.  Then
$\tX$~ is a compact manifold with boundary $\bX\amalg \bX\amalg \tS_1\amalg
\cdots\amalg \tS_N$, where each~$\tS_i$ is a 2-torus.  The preceding gives an
SN-stratification of~$X$ with strata of codimension~0, 1, and~2, and a flat
line bundle $L\to X_{\gea}$.  The holonomy around a circle
linking~$X_{\twobstrat}$ is~$-1$.
  \end{construction} 

\insfig{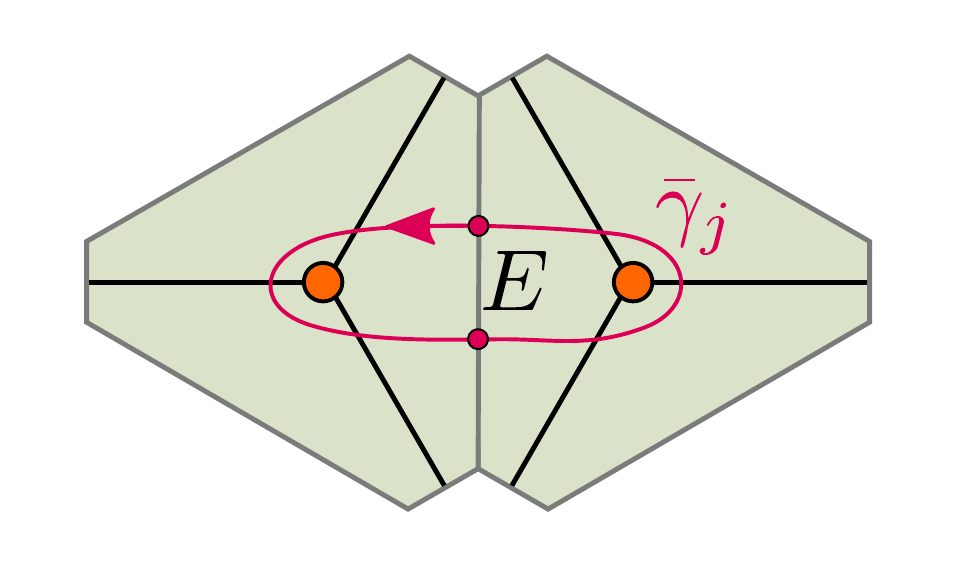}{0.4}{The two faces of the $j^{\textnormal{th}}$~tetrahedron which abut
  the edge~$E$.}

  \begin{remark}[Thurston gluing equations]\label{thm:67}
 Each tetrahedron~$\tetra^{\!(j)}$ in Construction~\ref{thm:36} has a shape
parameter~$\ho^{(j)}\in \CC\setminus \{0,1\}$ which is one of the holonomies
defined before Proposition~\ref{thm:12}.  (There are three possibilities
labeled by the three pairs of opposite edges of~$\tetra^{\!(j)}$.)  Let
$E$~be an edge in the triangulation~$\sT $, and let $S_E\subset \{1,\dots
,N\}$ be the set of~$j$ such that $E$~is an edge of~$\tetra^{\!(j)}$.
For~$j\in S_E$, let $\gamma _j$~be the loop in the torus~$\tS_j$ which is
called~`$\gamma_E$' in the text following Remark~\ref{thm:11}.  Then 
  \begin{equation}\label{eq:184}
     \sum\limits_{j\in S_E} [\gamma _j]=0\quad \textnormal{in $H_1(X)$}. 
  \end{equation}
To prove this relation consider Figure~\ref{fig:faces-abutting-edge}.  Depicted are the two
faces of~$\tetra^{\!(j)}$, $j\in S_E$, which abut~$E$ and the image~$\bg_j$
of the corresponding loop~$\gamma _j$.  Now each of the triangular faces
occurs in exactly one additional tetrahedron~$\tetra^{\!(j')}$, $j'\in
S_E\setminus \{j\}$, and it does so with the opposite orientation.  Hence the
halves of~$\bg_j$ and~$\bg_{j'}$ contained in that face cancel, as do the
halves of their lifts~$\gamma_j$ and~$\gamma _{j'}$.  This leads
to~\eqref{eq:184}.  (The cancellation is in homology; the actual half curves
are not strictly opposite.)  The relation~\eqref{eq:184} in homology
immediately implies the Thurston gluing equation~\cite[\S4.2]{T2}
  \begin{equation}\label{eq:185}
     \prod\limits_{j\in S_E} z^{(j)} = 1, 
  \end{equation}
where we choose the edge~$E$ to define the shape parameter in
each~$\tetra^{\!(j)}$, $j\in S_E$.
  \end{remark}

   \section{Levels and Chern-Simons invariants}\label{sec:5}

We begin in~\S\ref{subsec:5.1} by proving relations among the Chern-Simons
levels of $\GLC$, $\SLC$, and their various subgroups.  This is the
topological basis for abelianization.  These topological computations imply
relations among secondary differential geometric invariants via differential
cohomology.  We provide a brief introduction to differential cohomology in
Appendix~\ref{sec:9}.  In~\S\ref{subsec:5.2} we introduce the spin refinement
of Chern-Simons theory and prove appropriate relations among the ``spin
levels''.  We fully embrace differential cohomology in~\S\ref{subsec:5.5},
where we prove a key result: Theorem~\ref{thm:44}.  It states, heuristically,
that moving in the unipotent direction does not change Chern-Simons
invariants.  We also prove results about $\Cx$~Chern-Simons theory
(Theorem~\ref{thm:55}, Corollary~\ref{thm:57}, Corollary~\ref{thm:72}) that
are important in our later work.  We conclude with a \emph{global} statement,
Theorem~\ref{thm:23}, of abelianization.  Our main focus, \emph{stratified}
abelianization, is the subject of the
subsequent~\S\S\ref{sec:7a}--\ref{sec:cs-3-manifolds}.
 
In this section we change notation slightly.  Set $\hG=\GLC$ and let
$\hH,\hT$ be the subgroups defined in~\S\ref{subsubsec:4.3.2}.  Also, set
$G=\SLC$ and let $H,T$ be the associated subgroups; the unipotent
subgroup~$U$ is a subgroup of~$G$, hence too of~$\hG$.

We remind of a choice made in Example~\ref{thm:2}. 

  \begin{convention}[]\label{thm:69}
 3-dimensional Chern-Simons theory~$\F{\SLC}$ is based on the level
$-c_2\in H^4(\BSLC;\ZZ)$. 
  \end{convention}

\noindent
 In that section the level is encoded in a symmetric bilinear
form~\eqref{eq:19} on the Lie algebra, and \eqref{eq:14}~is the form that
corresponds to~$-c_2$.  In the next section we compute the restriction
of~$-c_2$ to the subgroup~$H\subset G$, and then we will define Chern-Simons
theory on~$H$---or, rather, a spin refinement---in terms of that restricted
level.

  \subsection{Levels and abelianization}\label{subsec:5.1}

   \subsubsection{Levels in~$\GLC$}\label{subsubsec:5.1.1}
 Our goal is to relate Chern-Simons invariants of principal $\hG$-bundles to
Chern-Simons invariants of $\hH$-, $\hT$-, and $U$-bundles, and to do the same
for~$G$-bundles.  These derive from relationships among appropriate degree
four integral cohomology classes on the classifying spaces, which we prove in
this section.  The inclusions $\hT\subset \hH\subset \hG$ and surjective
homomorphism $\hH\to\bmut$ lead to a diagram
  \begin{equation}\label{eq:44}
     \begin{gathered} \xymatrix{B\hT\ar[d]^p\\B\hH\ar[r]^r\ar@<.5ex>[d]^q&B\hG\\
     \Bbmut\ar@<.5ex>[u]^s} \end{gathered} 
  \end{equation}
in which $p$~is a double cover and the vertical maps~$p,q$ form a fibration
sequence.  The section~$s$ of~$q$ is the classifying map of the inclusion $
\bmut\cong \left\{ \left(\begin{smallmatrix}
1&0\\0&1 \end{smallmatrix}\right), \left(\begin{smallmatrix}
0&1\\1&0 \end{smallmatrix}\right)\right\} \hookrightarrow \hH $.  Let $c_i\in
H^{2i}(B\hG;\ZZ)$, $i=1,2$, be the universal Chern classes, and $p\mstrut
_1=c_1^2-2c\mstrut _2$ the universal first Pontrjagin class.  Let $c',c''\in
H^2(B\hT;\ZZ)$ be the first Chern class of the homomorphisms $\hT\to\Cx$
indicated by the matrix $\left(\begin{smallmatrix}
z'&0\\0&z'' \end{smallmatrix}\right)\in \hT$.  Let $a\in H^2(\Bbmut;\ZZ)$ be
the generator; note $2a=0$.

  \begin{proposition}[]\label{thm:15}
 In diagram~\eqref{eq:44} we have the following equality
in~$H^4(B\hH;\ZZ)$\textnormal{:} 
  \begin{equation}\label{eq:45}
     p_*(c')^2 = r^*p_1 + q^*a^2. 
  \end{equation}
  \end{proposition}

  \begin{proof}[Proof of Proposition~\ref{thm:15}]
 From the Leray-Serre spectral sequence of the vertical fibration
in~\eqref{eq:44}, we deduce the split short exact sequence\footnote{It helps
to observe that the action of~$\fbmut$ on~$H^2(B\hT;\ZZ)\cong \ZZ\oplus \ZZ$
exchanges the two summands, so the resulting local system on~$\fBbmut$ is the
pushforward of the trivial local system on its contractible double cover.
Hence the cohomology vanishes in positive degrees.} 
  \begin{equation}\label{eq:48}
     \begin{gathered} \xymatrix{0\ar[r]& H^4(\Bbmut;\ZZ)\ar@<.5ex>[r]^(.55){q^*}&
     H^4(B\hH;\ZZ)\ar@<.5ex>[l]^(.45){s^*} \ar[r]^{p^*}& H^4(B\hT;\ZZ)\ar[r] &0}
     \end{gathered} 
  \end{equation}
Hence a class in~$H^4(B\hH;\ZZ)$ is determined by its pullbacks under~$p^*$
and~$s^*$.   
 
For $\sigma \:B\hT\to B\hT$ the deck transformation, we have $p^*p_*=1+\sigma
^*$.  Hence 
  \begin{equation}\label{eq:49}
     p^*p_*(c')^2 = (c')^2 + (c'')^2 = p^*r^*(c_1^2-2c\mstrut _2) = p^*r^*(p_1), 
  \end{equation}
because $c',c''$ are the Chern roots of the universal $\hG$-bundle.  Since
$s$~induces an isomorphism on~$\pi _1$, the fiber product of~$s$ and~$p$ is
contractible, from which $s^*p_*=0$.  Also, $s^*r^*(c_1^2-2c\mstrut
_2)=(s^*r^*c\mstrut _1)^2$, since $H^4(\Bbmut;\ZZ)$ is torsion of order two.
The composition~$r\circ s$ classifies the sum of the complex sign and trivial
representations of~$\zt$, so its first Chern class is the generator $a\in
H^2(\Bbmut;\ZZ)$.  Combining the preceding with~$s^*q^*=\id$ we deduce~
\eqref{eq:45}. 
  \end{proof}

  \begin{remark}[]\label{thm:17}
 For $\hG=\GLC$ a level $mc_1^2 + nc\mstrut _2$ is parametrized by
integers~$m,n\in \ZZ$.  By a similar argument to the preceding proof,
$p_*c'=r^*c_1 + q^*a$, from which 
  \begin{equation}\label{eq:50}
     (p_*c')^2 = r^*c_1^2 + q^*a^2. 
  \end{equation}
Thus we can realize any level with $n$~even by a linear combination
of~$p_*(c')^2$ and~$(p_*c')^2$, up to~$q^*a^2$. 
  \end{remark}

   \subsubsection{Levels in~$\SLC$}\label{subsubsec:5.1.2}
 By restriction we deduce a formula for the ``special'' subgroups which
appear in the diagram
  \begin{equation}\label{eq:46}
     \begin{gathered}
     \xymatrix{BT\ar[d]^p\\BH\ar[r]^r\ar[d]^q&BG\\ 
     \Bbmut} \end{gathered} 
  \end{equation}

  \begin{lemma}[]\label{thm:29}
 In the diagram~\eqref{eq:46} we have $q^*a^2=0$. 
  \end{lemma}

  \begin{proof}
 Let $x\in H^1(\Bbmut;\zt)$ be the generator.  Then $a=\beta (x)$, where
$\beta \:H^q(-;\zt)\to H^{q+1}(-;\ZZ)$ is the integral Bockstein, and also
$a^2=\beta (x^3)$.  It suffices to prove $q^*x^3=0$.  Passing to maximal
compact subgroups we replace $T\to H\to\bmut$ by $\Spin_2\to
\Pin^-_2\to\bmut$.  In the Leray-Serre spectral sequence for the fibration
sequence $\Bbmut \to B\!\Pin^-_2\to BO_2$, the differential
$d_2\:E_2^{0,1}\to E_2^{2,0}$ sends the generator $y\in H^1(\Bbmut;\zt)$ to
$w_1^2+w_2\in H^2(BO_2;\zt)$.  (See~\cite{KT} for a review of pin groups.)
Then $d_2(w_1y)=w_1^3+w_1w_2$.  Also, since $y$~transgresses so too do its
Steenrod squares, and in particular $d_3(y^2)=d_3(Sq^1y)=Sq^1(d\mstrut
_2y)=Sq^1(w_1^2+w_2)=w_1w_2$.  Hence $w_1^3$~is killed when pulled back
to~$B\!\Pin^-_2$.  Conclude by observing that
  \begin{equation}\label{eq:77}
     \begin{gathered}
     \xymatrix@C-1em@R-.9em{\Pin_2^-\ar[dr]\ar[rr]&&\bmut\\&O_2\ar[ur]_{\det}}
     \end{gathered} 
  \end{equation}
commutes.  So the pullback of~$x$ equals the pullback of~$w_1$. 
  \end{proof}

The classifying map of the inclusion $i\:T\hookrightarrow \hT$ satisfies
$(Bi)^*c' = -(Bi)^*c''=c$ for $c\in H^2\bigl(BT;\ZZ \bigr)$ a generator.
Also, $i^*r^*c_1=0$.  The following is a corollary of
Proposition~\ref{thm:15} and Lemma~\ref{thm:29}.

  \begin{corollary}[]\label{thm:16}
 In diagram~\eqref{eq:46} we have the following equality
in~$H^4\bigl(BH;\ZZ\bigr)$:
  \begin{qedequation}\label{eq:47}
     p_*c^2 = -2r^*c_2. \qedhere
  \end{qedequation}
  \end{corollary}

  \begin{remark}[]\label{thm:70}
 Note the minus sign in~\eqref{eq:47}!  We must be mindful of it when we
define a $\Cx$ Chern-Simons theory which is compatible with our
Convention~\ref{thm:69} for $\SLC$ Chern-Simons theory.
  \end{remark}

   \subsubsection{Abelianization of connections}\label{subsubsec:5.1.3}
 Let us now focus on~$G=\SLC$.  If $X$~is a 3-manifold with a flat
$H$-connection, then \emph{global} abelianization of the associated flat
$\SLC$-connection is encoded in the commutative diagram
  \begin{equation}\label{eq:51}
     \begin{gathered} \xymatrix{\tX\ar[r]\ar[d]_{\pi } & B(\Cx)^\delta
     \ar[d]^{p} \\ 
     X\ar[r]^{} & BH^\delta  \ar[r]^<<<<{r} & B(\SLC)^\delta
     } \end{gathered}  
  \end{equation}
where we write~$\Cx$ for the group of diagonal matrices $T\subset \SLC$.  The
pullback square defines the (unramified) double cover~$\pi $.
\emph{Stratified} abelianization is encoded in the diagram 
  \begin{equation}\label{eq:52}
     \begin{gathered} \xymatrix{\tX_{\gea}\ar[r]\ar[d]_{\pi } &
     B(\Cx)^\delta \ar[d]^{p} \\ 
     X_{\gea}\ar[r]^{} \ar@{^{(}->}[d]& BH^\delta  \ar[r]^<<<<{r} &
     B(\SLC)^\delta \\ X\ar[urr]} \end{gathered} 
  \end{equation}
in which the bottom triangle commutes on~$X_0$.  In both the global and
stratified cases our goal is to compute the Chern-Simons invariant of the
flat $\SLC$-connection on~$X$ in terms of a Chern-Simons invariant of the
flat $\Cx$-connection on~$\tX$.  The $\SLC$ Chern-Simons invariant is the
secondary invariant of $-c_2\in H^4(B\!\SLC;\ZZ)$; the $\Cx$
Chern-Simons invariant is the secondary invariant of $c^2\in H^4(B\Cx;\ZZ)$.
There is a mismatch for abelianization: the factor of~$-2$ in~\eqref{eq:47}.
To rectify we must divide the $\Cx$-level by~2 (and include the minus sign).
This can be done---a secondary invariant for~``$c^2/2$'' exists---but at the
cost of introducing a new cohomology theory and a spin structure on~$\tX$, as
we explain in~\S\ref{subsec:5.2}.

  \begin{remark}[]\label{thm:18}
 Levels have a refinement in \emph{differential cohomology}, and the
Chern-Simons invariants are nicely located in the differential theory;
see~\cite{ChS, HS, F2, FH2}.  We give a pr\'ecis of differential cochains in
Appendix~\ref{sec:9} and use this point of view on Chern-Simons invariants
in~\S\ref{subsec:5.5}; see also \cite[Appendix~A]{FN}.  This framework makes
clear that cohomology identities immediately imply corresponding relations
among secondary invariants.  
  \end{remark}

We conclude our discussion of levels in ordinary cohomology by examining the
restriction to the unipotent subgroup~$U$ in~\eqref{eq:35}.  Recall
(Definition~\ref{thm:6}(ii)) that the failure of the bottom triangle
in~\eqref{eq:52} to commute on all of~$X_{\gea}$ is due to the unipotent
gluing along the walls (components of~$X_{-1}$) of the spectral network.
Since $U\cong \CC$ is contractible, so is~$BU$, and the following is
immediate.

  \begin{proposition}[]\label{thm:19}
 The restriction of any level of~$\GLC$ or~$\SLC$ to~$U$ vanishes. \thmqed
  \end{proposition}

\noindent
 In principle, then, the unipotent gluing does not change the Chern-Simons
invariant and essentially allows us to proceed as if the bottom triangle
in~\eqref{eq:52} commutes on~$X_{\gea}$, though this heuristic requires a
bit of work to make precise; see Theorem~\ref{thm:44}.

  \subsection{Levels for spin Chern-Simons theory}\label{subsec:5.2}

   \subsubsection{$E$-cohomology and spin $\Cx$~ Chern-Simons
   theory}\label{subsubsec:5.2.1}
 To divide $c^2\in H^4(B\Cx;\ZZ)$ by~2, we pass to a cohomology theory simply
denoted~$E$, the nontrivial extension
  \begin{equation}\label{eq:53}
     H\ZZ\xrightarrow{\;\;i\;\;} E\xrightarrow{\;\;j\;\;} \Sigma ^{-2}H\zt 
  \end{equation}
of Eilenberg-MacLane spectra; the $k$-invariant $\Sigma ^{-2}H\zt\to \Sigma
H\ZZ$ is $\beta \circ Sq^2$, the composition of the integral Bockstein and
the Steenrod square.  For any topological space~$S$, the
extension~\eqref{eq:53} leads to a long exact sequence of cohomology groups
  \begin{equation}\label{eq:54}
     \cdots \longrightarrow H^q(S;\ZZ)\xrightarrow{\;\;i\;\;}
     E^q(S)\xrightarrow{\;\;j\;\;} 
     H^{q-2}(S;\zt) \xrightarrow{\;\;\beta\circ  Sq^2\;\;}
     H^{q+1}(S;\ZZ)\longrightarrow \cdots 
  \end{equation}
Multiplication by~2 on~$E^q(S)$ factors through~$i$:
  \begin{equation}\label{eq:78}
     \begin{gathered} \xymatrix{H^q(S;\ZZ)\ar[r]^{i} \ar[d]_{2} &
     E^q(S)\ar[d]^{2}\ar@{-->}[dl]_k \\ H^q(S;\ZZ)\ar[r]^{i} & E^q(S)}
     \end{gathered} 
  \end{equation}
For $S=B\Cx$, a slice of the long exact sequence~\eqref{eq:54}
is the \emph{nontrivial} abelian group extension
  \begin{equation}\label{eq:55}
     0\longrightarrow H^4(B\Cx;\ZZ)\xrightarrow{\;\;i\;\;} E^4(B\Cx)
     \longrightarrow H^2(B\Cx;\zt) \longrightarrow 0, 
  \end{equation}
i.e., $E^4(B\Cx)$~is infinite cyclic and $i(c^2)$~is twice a generator
$\lambda \in E^4(B\Cx)$.  The class~$\lambda $ plays the role of ``$c^2/2$''.
Passing to maximal compact subgroups there is a generalization from $\TT\cong
SO_2$ to~$SO_N$ for any~$N\ge2$.  Namely, there is a characteristic class
$\lambda \in E^4(BSO_N)$ whose image under~$k\oplus j$ is $(p_1,w_2)\in
H^4(BSO_N;\ZZ)\oplus H^2(BSO_N;\zt)$.  Furthermore, $\lambda $~is additive:
for real vector bundles $V',V''\to X$ over a space~$X$ we have
  \begin{equation}\label{eq:79}
     \lambda (V'\oplus V'')=\lambda (V')\oplus \lambda (V''). 
  \end{equation}
The pullback of~$\lambda $ to~$E^4(B\!\Spin_N)$ is the image under~$i$ of a
class $\tl\in H^4(B\!\Spin_N;\ZZ)$ whose double is~$p_1$.  Also, $\tl\equiv
w_4\pmod2$ if~$N\ge4$.  We refer to~\cite[\S1]{F3} for background about this
cohomology theory~$E$ and proofs\footnote{Even if the precise statement does
not appear in~\cite{F3}, the same techniques apply.  The standard fact that
$\tl\neq 0\pmod2$ follows since $H^4(B\!\Spin_N;\ZZ)\cong \ZZ$ and $\tl$~is a
generator, if~$N\ge4$.} of these assertions.

The characteristic class $\lambda \in E^4(B\Cx)$ has a lift $\cl\in
\CE^4(\BNC)$ to the \emph{differential} $E$-cohomology of the classifying
object for principal $\Cx$-connections.  (See \cite[Appendix~A]{FN}.)  Here
$\BNC$~is a simplicial sheaf on smooth manifolds, in the sense of~\cite{FH2},
for example.  There is also a simplicial sheaf~$\BFC$ which classifies flat
$\Cx$-connections, as well as a map $i\:\BFC\to \BNC$.  The pullback $i^*\cl$
is a flat differential class.  Define the spectrum $E\CZo$~as the cofiber of
the composition
  \begin{equation}\label{eq:57}
     E\xrightarrow{\;\;k\;\;}H\ZZ\xrightarrow{\;\;\pi \sqmo\;\;}H\CC. 
  \end{equation}
Its nonzero homotopy groups are $\pi _0\cong \CZo$ and $\pi _{-1}\cong \zt$.
The topological space~$B(\Cx)^\delta $ is a geometric realization of the
simplicial sheaf~$\BFC$.  Then $i^*\cl$ determines a characteristic class 
  \begin{equation}\label{eq:85}
     \hl\in E^3\bigl(B(\Cx)^\delta ;\CZo \bigr) 
  \end{equation}
in the cohomology theory~$E\CZo$.

An \emph{oriented} real vector bundle has a Thom class in integer cohomology,
but a Thom class in $E$-cohomology requires a \emph{spin}
structure~\cite[Proposition~4.4]{F3}.  In particular, $E$-cohomology classes
can be integrated on compact spin manifolds.  This leads immediately to a
\emph{fully extended} unitary 3-dimensional topological field
theory~$\rS_{\Cx}$ on spin manifolds equipped with a \emph{flat}
$\Cx$-connection, analogous to the usual Chern-Simons theory~\eqref{eq:17} on
oriented manifolds.  It has a fully local version defined as a map of spectra
analogous to~\eqref{eq:18}:
  \begin{equation}\label{eq:56}
     \MSpin\wedge B(\Cx)^\delta _+ \xrightarrow{\;\;\id\wedge \hl\;\;}
     \MSpin\wedge 
     (E\CZo_3)_+\xrightarrow{\;\;\;\int_{}\;\;\;}\Sigma ^3E\CZo. 
  \end{equation}
The field theory~\eqref{eq:56} assigns a \emph{$\zt$-graded} line to a closed
spin 2-manifold with flat $\Cx$-connection.  As noted in Remark~\ref{thm:38}
we need the theory for parametrized families of flat connections, so for
nonflat connections.   

  \begin{remark}[]\label{thm:82}
 In fact, the grading of the spin Chern-Simons line of a $\Cx$-connection on
a surface is determined by the parity of the degree of the underlying
principal $\Cx$-bundle.  For a flat connection that degree is zero, hence the
line is even.  Also, to a $\Cx$-connection over a spin 1-manifold, the spin
Chern-Simons theory assigns an invertible module over \emph{super} vector
spaces.  See Appendix~\ref{sec:C11} for more details as well as a
justification for ignoring these $\zt$-gradings in the body of this paper.
  \end{remark}

As a companion to Convention~\ref{thm:69} we
signpost our choice of sign for the level, which is motivated by
Corollary~\ref{thm:22} below. 

  \begin{convention}[]\label{thm:71}
  3-dimensional spin Chern-Simons theory~$\rS_{\Cx}$ is based on the level
$\lambda \in E^4(\BC)$.  
  \end{convention}

This spin Chern-Simons theory is developed in some detail in~\cite{FN}.  For
future use we recall one particular result: \cite[Theorem~3.9(vii)]{FN}.  Let
$Y$~be a closed 2-manifold endowed with a spin structure~$\sigma $, and fix a
principal $\Cx$-bundle $\pi \:Q\to Y$ with connection~$\Theta \in \Omega
^1(Q;\CC)$.  A section~$t$ of~$\pi $ produces
  \begin{equation}\label{eq:188}
     \tau _t\in \SCx(Y ; \Theta ; \sigma) , 
  \end{equation}
a nonzero element in the spin Chern-Simons line computed from the
$\Cx$-connection~$\Theta $ and the spin structure~$\sigma $.  Let $h\:Y\to
\Cx$ be a smooth function.  Then $t'=t\cdot h$~is another section of~$\pi $,
and the ratio of nonzero elements in~$\SCx(Y ; \Theta ; \sigma)$ is
  \begin{equation}\label{eq:189}
     \frac{\tatp}{\tat} = \epsilon \mstrut _{t,h}\exp\left( \frac{1}{4\pi \sqmo}
     \int_{Y} t^*\Theta\wedge \frac{dh}{h} \right)  , 
  \end{equation}
where 
  \begin{equation}\label{eq:s65}
     \epsilon \mstrut _{t,h}(s)=(-1)^{q_\sigma ([h])}.
  \end{equation}
Here $q_\sigma \:H^1(Y;\zt)\to\zt$ is the quadratic refinement of the
intersection pairing given by the spin structure~$\sigma $, and $[h]\in
H^1(Y;\zt)$ is the reduction modulo two of the homotopy class of~$h$.

   \subsubsection{Levels in $E$-cohomology}\label{subsubsec:5.2.2}
 We revisit Proposition~\ref{thm:15} and Corollary~\ref{thm:16} in
$E$-cohomology, so effectively divide~\eqref{eq:45} and~\eqref{eq:47} by~2.

  \begin{lemma}[]\label{thm:21}
\

 \begin{enumerate}[label=\textnormal{(\arabic*)}]

 \item  The map 
  \begin{equation}\label{eq:66}
     H^4(\BSLC;\ZZ)\xrightarrow{\;\;i\;\;}E^4(\BSLC) 
  \end{equation}
is an isomorphism.

 \item   The group extension 
  \begin{equation}\label{eq:67}
     0\longrightarrow H^4(\Bbmut;\ZZ)\xrightarrow{\;\;i\;\;}
     E^4(\Bbmut)\xrightarrow{\;\;j\;\;} H^2(\Bbmut;\zt)\longrightarrow 0 
  \end{equation}
is nontrivial: $E^4(\Bbmut)$~is cyclic of order~4.  

 \item The pullback map $E^4(\Bbmut)\to E^4(B\ZZ)$ is zero.

 \end{enumerate} 
  \end{lemma}

  \begin{proof}
 Statement~(1) follows from $H^2(\BSLC;\zt)=0$.  For~(2), we claim
  \begin{equation}\label{eq:80}
     \alpha :=\lambda (L\oplus L)\in E^4(\Bbmut) 
  \end{equation}
has order~4, where $L\to\Bbmut$ is the real Hopf line bundle.  For this
observe $L^{\oplus 2}\to\Bbmut$ is orientable, $j\lambda (L^{\oplus
2})=w_2(L^{\oplus 2})=x^2$, and $2\alpha =2\lambda (L^{\oplus 2})=\lambda
(L^{\oplus 4})\neq 0$ since $w_2(L^{\oplus 4})=0$ and $w_4(L^{\oplus
4})=x^4\neq 0$, so $\tl(L^{\oplus 4}) \neq 0$.  (We use the Whitney sum
formula~\eqref{eq:79}.)  Finally, (3)~follows immediately from $E^4(B\ZZ)=0$.
  \end{proof}

\noindent
 Observe that $2\alpha =i(a^2)$, where $a\in H^2(\Bbmut;\zt)$ is the
generator. 

  \begin{remark}[]\label{thm:37}
 Let $p\:\ZZ\to \bmut$ be the homomorphism $n\mapsto(-1)^n$.  Observe that
the pullback $(Bp)^*\alpha \in E^4(B\ZZ)$ vanishes, since $B\ZZ\simeq \cir $
and $E^4(\cir)=0$.  The $\bmut$~Chern-Simons theory based on~$\alpha $
defines invariants of compact oriented manifolds equipped with a double
cover.  A lift of a double cover to a principal $\ZZ$-bundle
trivializes\footnote{The Chern-Simons theory is defined using a geometric
representative of~$\alpha $, say a map $\Bbmut\to E_4$, where $E_4$ is the
4-space in the spectrum~$E$, sometimes denoted $\Omega ^{\infty+4}E$.  The
trivialization is based on a choice of null homotopy of the composition
$B\ZZ\to \Bbmut\to E_4$.}  the $\bmut$~Chern-Simons invariant.
  \end{remark}

Let $\lG\in E^4(B\hG)$ be the pullback of the generator $\lambda \in E^4(B\Cx)$
under $\det\:\hG\to\Cx$, where $\hG=\GLC$; then $2\lG=i(c_1^2)$.  Let $\lambda
'\in E^4(B\hT)$ be the unique class such that $k(\lambda ')={c'}^2$.
(See~\eqref{eq:78} for the definition of~$k$.)  Identify $c_2\in
H^4(\BSLC;\ZZ)$ with its image under~$i$ in~$E^4(\BSLC)$.

  \begin{proposition}[]\label{thm:30}
 In diagram~\eqref{eq:44} we have the following equality in~$E^4(B\hH)$: 
  \begin{equation}\label{eq:81}
     p_*\lambda '=r^*(\lG-c\mstrut _2) + q^*\alpha . 
  \end{equation}
  \end{proposition}

  \begin{proof}
 In the diagram 
  \begin{equation}\label{eq:69}
     \begin{gathered} \xymatrix{H^4(\Bbmut;\ZZ)\ar[r]^{i} \ar@<.5ex>[d]^{q^*} &
     E^4(\Bbmut)\ar[r]^(.4){j} \ar@<.5ex>[d]^{q^*} &
     H^2(\Bbmut;\zt)\ar@<.5ex>[d]^{q^*} \\ 
     H^4(B\hH;\ZZ)\ar[r]^{i} \ar[d]^{p^*}\ar@<.5ex>[u]^{s^*} & E^4(B\hH)\ar[r]^(.4){j}
     \ar[d]^{p^*} \ar@<.5ex>[u]^{s^*}
     & H^2(B\hH;\zt)\ar[d]^{p^*}\ar@<.5ex>[u]^{s^*} \\ H^4(B\hT;\ZZ)\ar[r]^{i} &
     E^4(B\hT)\ar[r]^(.4){j} & H^2(B\hT;\zt)} \end{gathered} 
  \end{equation}
the rows are exact, and the first and third columns are exact;
see~\eqref{eq:48}.  It follows that the second column is also exact.  In
other words, a class in~$E^4(B\hH)$ is determined by its pullbacks under~$p^*$
and~$s^*$.  Also, observe that twice~\eqref{eq:81} is~\eqref{eq:45}, which
implies that the two sides of~\eqref{eq:81} differ by an element of order
dividing~2.  Since $E^4(B\hT)$~is torsionfree, as can be deduced
from~\eqref{eq:55}, it follows that the pullback under~$p^*$ of the two sides
of~\eqref{eq:81} agree.  For the pullback under~$s^*$ we argue as in the
proof of Proposition~\ref{thm:15}: the $\lG$-class of the \emph{complex} sign
representation is~$\alpha $; see~\eqref{eq:80}.
  \end{proof}

  \begin{corollary}[]\label{thm:22}
 In diagram~\eqref{eq:46} we have the following equality in $E^4(B(H))$:
  \begin{qedequation}\label{eq:68}
     p_*\lambda =-r^*c_2 + q^*\alpha . \qedhere
  \end{qedequation}
  \end{corollary}

\noindent
 Note from Lemma~\ref{thm:29} that $2q^*\alpha =q^*a^2=0$.
Corollary~\ref{thm:22} follows immediately from Proposition~\ref{thm:30}.
The minus sign in~\eqref{eq:68} is the spin echo of the minus sign
in~\eqref{eq:47}; see Remark~\ref{thm:70}.

  \subsection{Chern-Simons theory and differential cochains}\label{subsec:5.5}

Shortly after the introduction of secondary invariants of connections by
Chern-Simons~\cite{CS2}, Cheeger-Simons~\cite{ChS} recast them in terms of new
objects in differential geometry: differential characters.  Differential
cohomology, which we discuss briefly in Appendix~\ref{sec:9}, introduces
cochains into the theory of differential characters; it is the natural home
in which to express the full locality of Chern-Simons invariants.
In~\cite[Appendix~A]{FN} we prove some properties of \emph{spin} $\Cx$
Chern-Simons theory using \emph{generalized} differential cohomology, and
in~\S\ref{subsubsec:5.5.4} we take this up to prove a lemma we need later.
Otherwise, in this section we restrict to \emph{ordinary} differential
cohomology with complex coefficients and Chern-Simons theory for~$G=\SLC$.
Our main goal is to prove Theorem~\ref{thm:44} about the behavior of
Chern-Simons invariants under unipotent modifications.  We begin with some
preliminaries in~\S\S\ref{subsubsec:5.5.1}--\ref{subsubsec:5.5.2}.  A
\emph{global} abelianization theorem appears in~\S\ref{subsubsec:5.2.4}.

   \subsubsection{The universal $\SLC$-connection}\label{subsubsec:5.5.1}

Let $G$~be a Lie group with finitely many components.  There is a
groupoid-valued sheaf~$\BNG$ on the category of smooth manifolds whose value
on a test manifold~$M$ is the groupoid of $G$-connections; see~\cite{FH2} for
an introduction and details.  The sheaf~$\BNG$ classifies $G$-connections:
there is a universal principal $G$-bundle
  \begin{equation}\label{eq:86}
     \pi \:\ENG\longrightarrow \BNG 
  \end{equation}
with connection~$\Tu$, and if $P\to M$ is a principal $G$-bundle with
connection~$\Theta $ over a smooth manifold~$M$, then there is a unique
$G$-equivariant map $\varphi \:P\to \ENG$ which satisfies~$\varphi
^*\Tu=\Theta $.  Moreover, the universal connection on~\eqref{eq:86} is a
weak equivalence
  \begin{equation}\label{eq:87}
     \Tu\:\ENG\longrightarrow \Omega ^1\otimes \mathfrak{g}, 
  \end{equation}
where $\Omega ^1\otimes \mathfrak{g}$ is the set-valued sheaf which assigns
to a test manifold~$M$ the set~$\Omega ^1_M\mathfrak{g}$ of
$\mathfrak{g}$-valued 1-forms on~$M$.  The total space~$\ENG$
of~\eqref{eq:86} assigns to~$M$ the discrete groupoid of principal
$G$-bundles $Q\to M$ with connection $\Theta \in \Omega ^1(Q;\fg)$ and
section $s\:M\to Q$; the universal connection~\eqref{eq:87} maps the triple
$(Q,\Theta ,s)$ to the $\mathfrak{g}$-valued 1-form~$s^*\Theta $. 
 
The universal Chern-Simons-Weil invariant is a differential cohomology class
on~$\BNG$.  The variant~$\cHC^{\bullet }$ of differential cohomology we need
uses \emph{complex} differential forms.  The construction of~$\cHC^{\bullet
}$ as a homotopy fiber product~\cite{HS,BNV,ADH} leads to the exact sequence 
  \begin{equation}\label{eq:88}
     0\longrightarrow \cHC^4(\BNG)\longrightarrow H^4(BG;\ZZ)\times
     \Omega 
     ^4_{\textnormal{cl}}(\BNG;\CC)\xrightarrow{\;\;-\;\;}H^4(BG;\CC) 
  \end{equation}
in which $\Omega ^4_{\textnormal{cl}}(\BNG;\CC)$ denotes the vector space of
closed complex differential forms.  The main theorem of~\cite{FH2} computes
$\Omega ^4_{\textnormal{cl}}(\BNG;\CC)$ as the vector space of \emph{real}
linear $G$-invariant symmetric bilinear forms $\fg\times \fg\to \CC$.  For
$G=\SLC$ we choose $-c_2\in H^4(\BSLC;\ZZ)$ and the bilinear form 
  \begin{equation}\label{eq:96}
     \langle A,B \rangle = -\frac 1{8\pi ^2}\trace(AB),\qquad A,B\in
     \mathfrak{s}\mathfrak{l}_2\CC, 
  \end{equation}
as in Example~\ref{thm:2}; see also Convention~\ref{thm:69}.  By~\eqref{eq:88}
there is a unique lift $-\cct\in \cHC^4(\BNS)$, the desired universal
Chern-Simons-Weil class.  This gives, for each principal $G$~bundle with
connection~$\Theta $ over a smooth manifold~$M$, a differential
characteristic class~$-\cct(\Theta )\in \cHC^4(M)$.  We need a refinement to a
differential cocycle representative of this class in~$\cZd4_{\CC}(M)$;
see~\S\ref{subsec:9.1}.  This depends on a contractible choice; see
\cite[\S3.1]{F2} or \cite[\S3.3]{HS} for detailed constructions.  We use the
same symbol~$\cct$ for the differential cocycle representative, and make
clear whether it denotes the cocycle or cohomology class.
 
Suppose $P\xrightarrow{\;\;\pi \;\;}M\xrightarrow{\;\;p\;\;}S$ is an iterated
fiber bundle in which $\pi $~is a principal $\SLC$-bundle and the fibers
of~$p$ are manifolds with boundary of dimension~$n\le 3$.  Assume given an
orientation on~$p$, i.e., on the relative tangent bundle $T(M/S)\to M$.  Let
$\Theta \in \Omega ^1(P;\slc)$ be a connection.  We obtain a differential
cocycle~$-\cct(\Theta )\in \cZd4_{\CC}(M)$.  The Chern-Simons invariant of
this family of $\SLC$-connections is
  \begin{equation}\label{eq:89}
     \F{\SLC}(M\to S;\Theta )=\tpi\int_{M/S}(-\cct(\Theta )), 
  \end{equation}
a differential cochain in~$\cCd{4-n}_{\CC}(S)$; see~\S\ref{subsec:9.4} for
the integral.  For $n=3$ and assuming the fibers of $M\xrightarrow{\;p\;}S$
are closed, \eqref{eq:89}~is a function $S\to \CZo$, as in~\eqref{eq:12}.
For~$n=2$ and closed fibers \eqref{eq:89}~is a complex line bundle with
covariant derivative over~$M $, the Chern-Simons line bundle.  For~$n=3$ and
a fiber bundle of manifold with boundary, \eqref{eq:89}~is a section of the
Chern-Simons line bundle computed from the boundaries; see
Theorem~\ref{thm:48}.
 
Consider the pullback~$-\pi ^*\cct\in \cHC^4(\ENS)$ to the total space of the
universal bundle~\eqref{eq:86}.  Let $\Omega ^{\bullet }_{\ZZ}\subset \Omega
\mstrut _{\textnormal{cl}}$ denote the presheaf of closed differential forms
with integral periods.  The exact sequence
  \begin{equation}\label{eq:90}
     0\longrightarrow \frac{\Omega ^3(\ENS;\CC)}{\Omega
     ^3_{\ZZ}(\ENS;\CC)}\longrightarrow\cHC^4(\ENS)\longrightarrow H^4(\ENS;\CC) 
  \end{equation}
from \cite[(3.3)]{HS} reduces to an isomorphism (the middle map), since
$H^4(\ENS;\CC)=0$.  Hence $-\pi ^*\cct$ reduces to a 3-form modulo closed
3-forms with integral periods.  There is a canonical choice of
3-form,\footnote{Use~\eqref{eq:116} to deduce the existence of this 3-form.}
the Chern-Simons form $\eta \in \Omega ^3(\ENS;\CC)$; see~\eqref{eq:10}.  To
a triple $(Q\to M,\Theta ,s)$ which represents a map $M\to \ENS$, the
pullback of~$\eta $ to~$M$ is
  \begin{equation}\label{eq:91}
     -\frac{1}{8\pi ^2}\trace\bigl(\alpha \wedge d\alpha + \frac 23\alpha
     \wedge \alpha \wedge \alpha \bigr)\;\in \Omega ^3(M;\CC), 
  \end{equation}
where $\alpha =s^*\Theta \in \Omega ^1(M;\slc)$.  The Chern-Simons invariant
of an oriented family $P\xrightarrow{\;\pi \;}M\xrightarrow{\;p\;}S$ with
connection and trivialization $s\:M\to P$ can be computed by integrating the
3-form~\eqref{eq:91} over the fibers of~$p$.  For example, if the fibers
of~$p$ are closed of dimension~2, then the resulting 1-form on~$S$ is the
connection form of a trivialized complex line bundle over~$S$.

   \subsubsection{Restriction to the unipotent subgroup}\label{subsubsec:5.5.5}
 Recall the unipotent subgroup $U\subset \SLC$ defined in~\eqref{eq:35}.   

  \begin{definition}[]\label{thm:79}
 Let $M$~be a smooth manifold with boundary.

 \begin{enumerate}[label=\textnormal{(\arabic*)}]

 \item A flat
principal $\SLC$-bundle $P \to M$ is \emph{boundary-unipotent} if its restriction
$\partial P\to \partial M$ to the boundary admits a reduction to a flat
principal $U$-bundle.

 \item Such a $P \to M$ is \emph{boundary-reduced} if a reduction is chosen.

\item Stratified abelianization data $(P,Q,\mu,\theta)$ over $(M,\cN)$ is \textit{boundary-reduced}
if $P$ is boundary-reduced and $\theta(Q) \vert_{\partial M_0}$ lies in the $U$-bundle given by the reduction.

 \end{enumerate}  
  \end{definition}

\noindent
Note that a flat $\SLC$-bundle is boundary-unipotent iff on each boundary component the holonomies around
loops at a basepoint have a common eigenline. Moreover, if $(P,Q,\mu,\theta)$
is boundary-reduced, then $Q \vert_{\partial M}$ is a trivializable flat bundle, since $H \cap U = \{1\}$.

Let~$\BNU$ be the groupoid-valued sheaf of $U$-connections.  Then there is a
map $\BNU\to \BNS$. 

  \begin{lemma}[]\label{thm:80}
 The restriction of the universal second differential Chern class $\cct\in
\cHC^4(\BNS)$ to~$\BNU$ vanishes. 
  \end{lemma}

  \begin{proof}
 Since $U$~is contractible, the restriction of~$c_2\in H^4(B\SLC;\ZZ)$
to~$H^4(BU;\ZZ)$ vanishes; also, the restriction of the bilinear
form~\eqref{eq:96} to the Lie algebra of~$U$ vanishes. 
  \end{proof}

Recall that we choose a differential cocycle representative of~$\cct$; see
the text following~\eqref{eq:96}.  Now choose a trivialization of its
restriction to~$\BNU$.   

  \begin{remark}[]\label{thm:81}
 With these choices, the Chern-Simons invariant of a boundary-reduced flat
$\SLC$-bundle is trivialized on the boundary.  For example, on a compact
2-manifold with boundary, the invariant is a complex line. 
  \end{remark}

   \subsubsection{A lemma in differential cohomology}\label{subsubsec:5.5.2}

Let $M$~be an oriented $n$-manifold with corners, equipped with the extra
structure of a bordism outlined  in~\S\ref{subsec:9.4}; let $S$~be a smooth
manifold, which plays the role of parameter space; and
suppose\footnote{$N=[0,1]\times M$ has the structure of a bordism: set 
  \begin{equation}\label{eq:133}
     \begin{aligned} N_0&=(0,1)\times M_0 \\ N\corn01&=\{0\}\times M_0 \\
      N\corn11&=\{1\}\times M_0 \\ N\corn\delta {j+1}&=M\corn\delta
     j,\qquad  j\ge 1,\quad \delta \in \{0,1\}.\end{aligned} 
  \end{equation}
} $\co\in \cZd
q_{\CC}\bigl(S\times [0,1]\times M \bigr)$ is a differential cocycle of some
degree~$q$.  Let $\omega \in \Omega ^q_{\ZZ}\bigl(S\times [0,1]\times M;\CC
\bigr)$ be the ``curvature'' of~$\co$, i.e., the differential form underlying
the differential cocycle~$\co$.  Let $\ddt$ denote the standard vector field
on~$[0,1]$, lifted to $S\times [0,1]\times M$, and let $\cdt$~denote its
action via contraction on differential forms.   Theorem~\ref{thm:48} implies
the following for $M$~closed.

  \begin{lemma}[]\label{thm:40}
 If $M$~is closed, then the integral
  \begin{equation}\label{eq:99}
     \int_{[0,1]\times M}\co\:\int_{\{0\}\times M}\co\longrightarrow
     \int_{\{1\}\times M}\co 
  \end{equation}
is a nonflat isomorphism of the differential cocycles on~$S$ computed in the
domain and codomain.  Its covariant derivative is
  \begin{equation}\label{eq:100}
     \int_{\zoM}\omega \;\in \Omega ^{q-n-1}(S;\CC). 
  \end{equation}
In particular, if $\cdt\omega =0$, then \eqref{eq:99}~is a flat isomorphism. 
  \qed
  \end{lemma}

If $M$~is a manifold with corners---a bordism of positive depth---then the
integrals of~$\co$ over~$\{0\}\times M$ and $\{1\}\times M$ are higher
morphisms in a groupoid of differential cochains on~$S$:
see~\S\ref{subsec:9.3}.  For example, if $M$~has depth~$\le2$, then the
integrals are 2-morphisms in~$\cG{q-n+2}S$, as depicted in~\eqref{eq:132}. 

  \begin{lemma}[]\label{thm:53}
 \ 
 \begin{enumerate}[label=\textnormal{(\arabic*)}]

 \item  If $M$~has corners of depth~$\le 2$ and $\cdt\omega =0$, then
$\int_{\zoM}\co$ is an isomorphism of 2-morphisms in $\cG{q-n+2}S$: 
  \begin{equation}\label{eq:134}
  \begin{tikzcd}[column sep=16ex]
   \int\mstrut _{\{0\}\times M\corn02}\rar[bend left=25,"\int\mstrut
   _{\{0\}\times M\corn11}", ""{below,yshift=-3pt,name=A} ] \rar[bend
   right=25,"\int\mstrut_{\{0\}\times M\corn01}"', ""{above,yshift=3pt,name=B} ] 
  & \int\mstrut_{\{0\}\times M\corn12}
   \arrow[Rightarrow, from=B, to=A, "\;\int\mstrut_{\{0\}\times M_0}"{right}]
  \end{tikzcd}
   \quad \xrightarrow{\quad\cong \quad}\quad 
  \begin{tikzcd}[column sep=16ex]
   \int\mstrut _{\{1\}\times M\corn02}\rar[bend left=25,"\int\mstrut
   _{\{1\}\times M\corn11}", ""{below,yshift=-3pt,name=A} ] \rar[bend
   right=25,"\int\mstrut_{\{1\}\times M\corn01}"', ""{above,yshift=3pt,name=B} ] 
  & \int\mstrut_{\{1\}\times M\corn12}
   \arrow[Rightarrow, from=B, to=A, "\;\int\mstrut_{\{0\}\times M_0}"{right}]
  \end{tikzcd}
  \end{equation}

 \item Suppose $\ct\in \cCd{q-1}(S\times \zo\times M)$ is a nonflat
trivialization of~$\co$ with covariant derivative $\tau \in \Omega
^{q-1}(S\times \zo\times M)$, and assume $\cdt\tau =0$.  Then the
isomorphism~\eqref{eq:134} preserves the nonflat trivializations and nonflat
isomorphisms in Theorem~\ref{thm:60}.
  \qed
 \end{enumerate}
  \end{lemma}

\noindent
 We omit the integrand~`$\co$' in~\eqref{eq:134} for readability.

  \begin{example}[]\label{thm:41}
 If $n=q-2$, then $\int_{\{i\}\times M}\co$, $i=0,1$, is a complex line
bundle $L_i\to S$ with connection, and \eqref{eq:99}~is an isomorphism
$L_0\to L_1$ of the underlying line bundles; its usual covariant derivative
is the 1-form~\eqref{eq:100}.
  \end{example}

  \begin{example}[]\label{thm:42}
 If $q=2$ and $M=\cir$, then \eqref{eq:100}~reduces to a well-known formula
for the ratio of holonomies of a line bundle with connection around the ends
of a cylinder. 
  \end{example}

  \begin{remark}[]\label{thm:43}
 If $\co$~is equipped with a nonflat trivialization, then so too are its
integrals over~$M$, and then \eqref{eq:99}~ becomes an equation in differential
forms which follows from the usual Stokes' theorem.  The assertion in
Lemma~\ref{thm:53}(2) is a variation for manifolds with corners.
  \end{remark}

   \subsubsection{Moving along unipotents}\label{subsubsec:5.5.3}

Theorem~\ref{thm:44} below is based on the fact that the bilinear
form~\eqref{eq:96} vanishes if $A$~is diagonal and $B$~is upper
triangular.
 
Let $M$~be an oriented manifold with corners of depth~$\le2$ and
dimension~$\le 3$.  Suppose $M=M_0\amalg M_{-1}\amalg M_{\twoastrat}\amalg M_{\threeastrat}$ is equipped
with an SN-stratification which satisfies $M_{\twobstrat}=M_{\threebstrat}=\emptyset $.  In
other words, $M=\mstrat{\gea}$.  Suppose $(\pi ,s)$ is a subordinate
spectral network.  Thus $\pi \:\tM\to M$ is a double cover and
$s\:M_{-1}\amalg \mstrat{\twoastrat} \amalg \mstrat{\threeastrat} \to \tM_{-1}\amalg \tM_{\twoastrat} \amalg \tM_{\threeastrat}$ is a section
of~$\pi $ over~$M_{-1}\amalg \mstrat{\twoastrat} \amalg \mstrat{\threeastrat}$.  Let $T\subset \SLC$ denote the
diagonal subgroup, and $\iota\: H\hookrightarrow \SLC$ its normalizer.
Suppose $\cA=(P,Q,\mu ,\theta )$ is stratified abelianization data of type
$(\SLC,T)$.  Thus

 \begin{enumerate}[label=\textnormal{(\arabic*)}]

 \item $P\to M$ is a principal $\SLC$-bundle with flat
connection~$\Theta _P$,

 \item $Q\to M$ is a principal $H$-bundle with flat
connection~$\Theta _Q$,

 \item $\mu \:\tM\to Q/T$ is an isomorphism of double
covers, and 

 \item $\theta \:\iota (Q)\to P$ is an isomorphism of flat principal
$\SLC$-bundles over~$M_0$.

 \end{enumerate}
Furthermore, let $U\subset \SLC$ be the
subgroup~\eqref{eq:35} of unipotent matrices.  Then we require that the
discontinuity of~$\theta $ along~$M_{-1}$ lie in~$U$, relative to the
reduction of $Q\to M_{-1}$ to a principal $T$-bundle given by the
section~$s$.
 
Our task is to compute the Chern-Simons invariants~\eqref{eq:89} of the flat
$\SLC$-bundles $\iota (Q)\to M$ and $P\to M$.  To state the theorem we posit
a family of this data over a smooth manifold~$S$.  Thus we work over $S\times
M$; the connections $\Theta _P,\Theta _Q$ over~$S\times M$ are only assumed
flat along~$M$.

  \begin{theorem}[]\label{thm:44}
 There is a natural flat isomorphism 
  \begin{equation}\label{eq:101}
     \F{\SLC}\bigl(S\times M\to S,\iota (\Theta
     _Q)\bigr)\xrightarrow{\;\;\cong \;\;} \F{\SLC}\bigl(S\times M\to
     S,\Theta _P \bigr). 
  \end{equation}
  \end{theorem}

\noindent
 Intuitively, moving a connection in unipotent directions does not affect the
$\SLC$ Chern-Simons invariant.  In the proof we construct a flat isomorphism
which depends on a set of choices, and then we check that the isomorphism is
independent of the choices.

\insfig{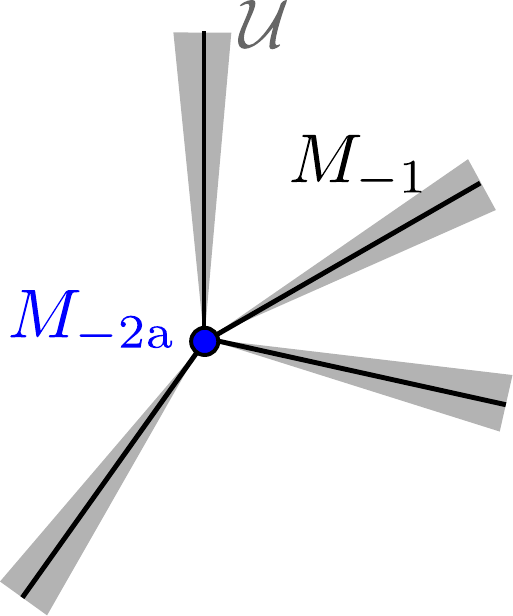}{0.5}{The tubular neighborhood~$\sU$.}

  \begin{proof}
 Fix a smooth function $\phi \:\RR^{\neq 0}\to \RR^{\neq 0}$ which is odd and
satisfies 
  \begin{equation}\label{eq:146}
     \phi (x)=\begin{cases} 0,&x\le-1;\\-\frac 12,&-\frac 12\le x<0\\\frac
     12,&0<x\le\frac 12\\0,&1\le x.\end{cases} 
  \end{equation}
Furthermore, require that $\phi $~ be monotonic nonincreasing on~$\RR^{<0}$
and~$\RR^{>0}$.  Choose a tubular neighborhood of $M_{-1}\subset M$: an open
subset $\sU\subset M_{\ge-1}$ which contains~$M_{-1}$, a surjective
submersion $\rho \:\sU\to M_{-1}$, and an isomorphism of~$\rho $ with the
normal bundle $\nu \to M_{-1}$ to $M_{-1}\subset M_{\ge-1}$.  (See
Figure~\ref{fig:tubular-neighborhood}.)  Fix an inner product on $\nu \to M_{-1}$.  Let $R\to
M_{-1}$ be the reduction of $Q\to M_{-1}$ to a principal $T$-bundle; it is
defined via the section~$s$ and isomorphism~$\mu $.  Locally, for each
orientation of $\nu \to M_{-1}$ the discontinuity in~$\theta $ along~$M_{-1}$
is a section~$u$ of the bundle $S\times R\times _{T}U\to S\times M_{-1}$ of
unipotent groups.  Under reversal of orientation, $u$~maps to~$u\inv $.
Globally, write $u=e^X$ for $X$~a section of the bundle of Lie algebras
  \begin{equation}\label{eq:147}
     S\times R\times \mstrut _{\!T}\mathfrak{u}\to S\times M_{-1} 
  \end{equation}
twisted by the orientation bundle of $\nu \to M_{-1}$.  Extend~$X$ to~$\sU$
using parallel transport along the fibers of $\rho \:\sU\to M_{-1}$.  The
inner product on the normal bundle identifies each fiber of~$\rho $
with~$\RR$ after choosing an orientation of the normal bundle.  Hence the
product~$\phi X$ is a well-defined section of the pullback of~\eqref{eq:147}
over $S\times (\sU\cap M_0)$.  It extends\footnote{Since the codomain
of~$\phi X$ does not so extend, this is not strictly correct.  What we mean
simply is that in formulas below replace~$\phi X$ by~`0' on $S\times
(M_0\setminus\, \sU)$.} by zero to $S\times M_0$. 

We now construct a connection~$\Xi $ on the principal $\SLC$-bundle
  \begin{equation}\label{eq:148}
     \sQ=\zo\times S\times \iota (Q)\to \zo\times S\times M
  \end{equation}
whose restriction to $\{0\}\times S\times M$ is isomorphic to~$\iQ$ and whose
restriction to $\{1\}\times S\times M$ is isomorphic to~$\Theta _P$.  First,
set $\Xi \res{\{0\}\times S\times M}=\iQ$.  Then over $\{1\}\times S\times
M_0$ let $\varphi $ be the gauge transformation of the restriction
of~\eqref{eq:148} which equals $e^{-\phi X}$ on $\{1\}\times S\times (\sU\cap
M_0) \subset \zo\times S\times (\sU\cap M_0)$ and is the identity map on
$\{1\}\times S\times (M_0\setminus\, \sU)$.  Construct an isomorphism
  \begin{equation}\label{eq:102}
     \psi \:\sQ\Res{S\times \{1\}\times M}\xrightarrow{\;\;\cong \;\;} P 
  \end{equation}
which equals~$\theta \circ \varphi $ on $\{1\}\times S\times M_0$; it extends
over $\{1\}\times S\times (M_{-1}\amalg M_{\twoastrat} \amalg M_{\threeastrat})$ 
using the fact that $\theta
$~ jumps by~$u$ on $\{1\}\times S\times M_{-1}$.  Set $\Xi \res{\{1\}\times
S\times M} = \psi ^*(\Theta _P)$.  Finally, define~$\Xi $ on~$\zo_t\times
S\times M$ by affine interpolation between the specified connections~$\Xi
_0,\Xi _1$:
  \begin{equation}\label{eq:150}
     \Xi _t = \iQ + t\alpha ,\qquad \alpha =\Xi _1-\Xi _0. 
  \end{equation}
Then $\alpha $~is a 1-form on~$S\times M$ with values in the adjoint bundle
of Lie algebras isomorphic to~$\slc$; it has support on $S\times \sU$.  We
claim that $\alpha $~takes values in the subbundle~\eqref{eq:147} of
nilpotent subalgebras, extended over~$S\times \sU$ by parallel transport
along the fibers of~$\rho $.  Namely, on $S\times (\sU\cap M_0)$ we have 
  \begin{equation}\label{eq:151}
     \alpha =\left[ \Ad_{e^{-\phi X}}\bigl(\iQ \bigr) - \iQ \right] +
     d_{\iQ}(\phi X). 
  \end{equation}
The second term clearly lies in the nilpotent subalgebra.  For the first,
observe  
  \begin{equation}\label{eq:153}
     \begin{pmatrix} 1&x\\0&1 \end{pmatrix}\begin{pmatrix} y&\phantom{-}0\\0&-y
     \end{pmatrix}\begin{pmatrix} 1&-x\\0&\phantom{-}1 \end{pmatrix} -
     \begin{pmatrix} 
     y&\phantom{-}0\\0&-y \end{pmatrix} = \begin{pmatrix}
     0&-2xy\\0&0 \end{pmatrix}  
  \end{equation}
is nilpotent.
 
Let $\co=\cct(\Xi )$ be the Chern-Simons-Weil differential cocycle.  As
in~\eqref{eq:134}, 
  \begin{equation}\label{eq:104}
     \tpi\int_{[0,1]\times M}\co\:\F{\SLC}\bigl(S\times M\to S,\iota (\Theta _Q)
     \bigr) \longrightarrow \F{\SLC}\bigl(S\times M\to S,\Theta _P \bigr) 
  \end{equation}
is an isomorphism.  We claim that it is a \emph{flat} isomorphism.  By
Lemma~\ref{thm:53} it suffices to show that $\cdt\omega =0$, where
  \begin{equation}\label{eq:154}
     \omega = \langle \Omega (\Xi ),\Omega (\Xi )  \rangle
  \end{equation}
is the Chern-Weil 4-form of~$\Xi $.  The only nonzero contribution
to~$\cdt\omega $ is potentially on $\zo\times \supp \alpha \subset \zo\times
S\times \sU$.  From~\eqref{eq:150} we compute the curvature
  \begin{equation}\label{eq:106}
     \Omega (\Xi ) = \iota \bigl(\Omega (\Theta _Q) \bigr) + dt\wedge \alpha  +
     td_{\iQ}\alpha  + \frac{t^2}{2}[\alpha \wedge \alpha ].
  \end{equation}
The last term vanishes since the Lie algebra~$\mathfrak{u}$ of the unipotent
group is abelian.  The first term in~\eqref{eq:106} takes values in the
diagonal subalgebra $\mathfrak{t}\subset \slc$ and the other terms take
values in the nilpotent subalgebra~$\mathfrak{u}\subset \slc$.  It follows
that  
  \begin{equation}\label{eq:155}
     \cdt\omega =2\langle \alpha \wedge \iota \bigl(\Omega (\Theta _Q) \bigr)
     \rangle + 2t\langle \alpha \wedge d_{\iQ}\alpha \rangle = 0. 
  \end{equation}

It remains to prove that \eqref{eq:104}~is independent of the choices of
$\phi ,\sU,\rho $ and the isomorphism of~$\rho $ with the normal bundle.  Any
two sets of choices can be joined by a path, so we extend the previous setup
by taking the Cartesian product with $[0,1]_r$.  If $\tom$~is the resulting
Chern-Weil 4-form, then  $\cdr\tom=0$ by a similar argument.
  \end{proof}

   \subsubsection{A theorem in $T=\Cx$ Chern-Simons theory}\label{subsubsec:5.5.4}

 Recall from~\S\ref{subsubsec:5.2.1} and \cite[Appendix~A]{FN} the
characteristic class~$\lambda \in E^4(\BC)$ and its differential refinement
$\cl\in \CE^4(\BNC)$, the universal differential class.  The
class~$\lambda $ is the image of the generator under a quadratic function
  \begin{equation}\label{eq:135}
     q\:H^2(\BC;\ZZ)\longrightarrow E^4(\BC) 
  \end{equation}
which for $c,c'\in H^2(\BC;\ZZ)$ satisfy
  \begin{align}
      2q(c)&=c\smile c = c^2 \label{eq:136}\\ 
      jq(c)&=\bar c \label{eq:137}\\ 
      q(c+c')&=q(c)+q(c')+ i(c\smile c'), \label{eq:138}
  \end{align}
where $j\:E^4(\BC)\to H^2(BC;\zt)$ and $i\:H^4(\BC;\ZZ)\to E^4(\BC)$ are the
maps in~\eqref{eq:53}, and we denote $\bar c=c\pmod2$.  The differential
refinement
  \begin{equation}\label{eq:139}
     \cq\:\cHC^2(\BNC)\longrightarrow \CE^4(\BNC) 
  \end{equation}
satisfies analogous properties.  We implicitly use refinements of~$q,\cq$ to
 cochains.  Recall that integration\footnote{In~\eqref{eq:140} we also use
 the integral symbol for the pairing of a mod~2 cohomology class with the
 fundamental class.} of $E$-cocycles over a manifold~$M$ requires a spin
 structure~$\sigma $ on~$M$.  Furthermore, if $\delta \in H^1(M;\zt)$ is the
 class of a double cover over~$M$, and we write the shifted spin structure
 as\footnote{Our notation conflates a double cover and its equivalence class,
 an overload we also deploy in this section for spin structures and
 $\Cx$-connections.} $\sigma \to \sigma +\delta $, then
  \begin{equation}\label{eq:140}
     \int_{M,\sigma +\delta }q(c) = \int_{M,\delta }q(c) + \frac
     12\int_{M}\delta \smile\bar c, 
  \end{equation}
where $\frac 12\:\zt\hookrightarrow \CC/\ZZ$; see \cite[Proposition~4.4]{F3}
and \cite[Theorem~3.9 (ii)]{FN}.  For $\dim M=3$ and $M$~closed,
\eqref{eq:140}~is an equation in~$\CZ$; for manifolds with boundary and
manifolds of lower dimension it is a canonical isomorphism of cochains in
$E$-cohomology theory.  We use the differential refinement
of~\eqref{eq:140}. 
 
Not only do double covers shift spin structures, but they also shift
$\Cx$-bundles via the homomorphism $\zt\hookrightarrow \Cx$.  The
corresponding shift of Chern classes is via the integer Bockstein 
  \begin{equation}\label{eq:141}
     \beta \:H^1(M;\zt)\longrightarrow H^2(M;\ZZ). 
  \end{equation}
The differential refinement shifts $\Cx$-connections by a flat
$\Cx$-connection of order two.

The following is a restatement of part of Lemma~\ref{thm:21}. 

  \begin{lemma}[]\label{thm:54}
 $\CE^4(\Bbmut)\cong \zmod4$ with generator~$\cq\bigl(\cb(\delta ) \bigr)$ for
$\delta \in H^1(\Bbmut;\zt)$ the nonzero class.  Also, $2\cq\bigl(\cb(\delta )
\bigr)$~is the image of~$\delta ^3$ under the map $\phi \:H^3(\Bbmut;\zt)\to
\CE^4(\Bbmut)$.  
  \end{lemma}

  \begin{proof}
 The main theorem in~\cite{FH2} implies $\CE^4(\Bbmut)\to E^4(\Bbmut)$ is an
isomorphism.  Now apply Lemma~\ref{thm:21}(2).  For the last assertion
apply~\eqref{eq:136}.
  \end{proof}

Our main result in this section expresses the change of spin
$\Cx$~Chern-Simons invariants under the simultaneous shift of spin structure
and $\Cx$-connection by a double cover. We express our result as a relation
among 3-dimensional invertible field theories whose background fields are
independent choices of: a spin structure~$\sigma $, a double cover~$\delta $,
and a $\Cx$-connection~$\cc$.  The partition functions which define the
theories are:
  \begin{equation}\label{eq:142}
     \begin{aligned} \alpha _1(\sigma ,\delta ,\cc) &=
      \textnormal{spin $\Cx$~Chern-Simons invariant of~$\cc$ in spin
      structure~$\sigma$} \\ \alpha _2(\sigma ,\delta ,\cc) &=
      \textnormal{spin $\Cx$~Chern-Simons invariant of~$\cc+\cb(\delta )$ in
     spin 
      structure~$\sigma+\delta $} \\ \alpha _3(\sigma ,\delta ,\cc) &=
      \textnormal{integral of~$3 \cq\bigl(\cb(\delta ) \bigr)$ in spin
      structure~$\sigma$} \\ \end{aligned} 
  \end{equation} 
The theory~$\alpha _3$ is topological (of order~4).  The Chern-Simons
invariants in~$\alpha _1,\alpha _2$ are based on~$\lambda $, and so are
computed by integrating~$q$, the quadratic
function~\eqref{eq:135}. 
 
The proofs in the rest of this section draw on the material in
Appendix~\ref{sec:10}. 

  \begin{theorem}[]\label{thm:55}
 There is an isomorphism $\alpha  _1\otimes \alpha _3\cong \alpha _2$ of
invertible field theories. 
  \end{theorem}

  \begin{proof}
 By~\eqref{eq:165} the curvatures of $\alpha _1$~ and $\alpha _2$~are equal.
Therefore $\alpha _1\inv \otimes \alpha \mstrut _2\otimes \alpha _3\inv $ is
a flat invertible field theory, so it is topological in the strong sense.  To
verify that it is trivializable, it suffices to check the partition function
on a closed 3-manifold~$X$.  The quadratic property~\eqref{eq:138} implies
  \begin{equation}\label{eq:143}
     \cq\bigl(\cc+\cb(\delta ) \bigr) = \cq(\cc) + \cq\bigl(\cb(\delta )
     \bigr) + i\bigl(\cc\smile\cb(\delta ) \bigr), 
  \end{equation}
and \eqref{eq:140}~implies that its integral in spin structure~$\sigma
+\delta $ is the integral of 
  \begin{equation}\label{eq:144}
     \cq(\cc) + \cq\bigl(\cb(\delta ) \bigr) + i\bigl(\cc\smile\cb(\delta )
     \bigr) + \phi \biggl(\frac 12(\delta \smile\bar c\;+\;\delta
     \smile\delta ^2) \biggr) 
  \end{equation}
in spin structure~$\sigma $.  Here $\phi \:H^3(X;\zt)\to \CE^4(X)$ is the
inclusion of flat elements of order two.  (The first two terms
of~\eqref{eq:144} lie in~$\CE^4(X)$ and their integral uses the spin
structure; the last two terms lie in $H^3(X;\CZ)$ and no spin structure is
used to integrate.)  The difference of~$\cq(\cc)$ and~\eqref{eq:144} computes
the partition function in the theory $\alpha _1\inv \otimes \alpha \mstrut
_2$:
  \begin{equation}\label{eq:145}
     \begin{aligned} \cq\bigl(\cb(\delta ) \bigr) + \phi \biggl(\frac 12(\bar
      c\smile\delta \;+\;\delta \smile\bar c\;+\;\delta \smile\delta ^2
     )\biggr) &= 
      \cq\bigl(\cb(\delta ) \bigr) + \phi \biggl(\frac 12(\delta ^3)\biggr)
     \\ &= 3\cq\bigl(\cb(\delta ) \bigr) \end{aligned} 
  \end{equation}
where we use Lemma~\ref{thm:54}.  Since the integral of this last quantity is
by definition the partition function in the theory~$\alpha _3$, we see that
the partition function in the theory $\alpha _1\inv \otimes \alpha \mstrut
_2\otimes \alpha _3\inv $ is trivial.
  \end{proof}

  \begin{example}[]\label{thm:58}
 Let $X=\RP^3$ equipped with either spin structure~$\sigma $ and the
nontrivial double cover~$\delta $.  The spin Chern-Simons partition function
of the product $\Cx$-connection~$\cc$ is~1; the partition function of the
$\Cx$-connection $\cc+\delta $ in either spin structure is a primitive
$4^{\textnormal{th}}$~root of unity.
  \end{example}

A principal $\ZZ$-bundle~$\din$ has a mod~2 reduction~$\delta $ which is a
double cover.  The invertible field theories~$\alpha _1,\alpha _2,\alpha _3$
in~\eqref{eq:142} lift to invertible field theories~$\ain_1,\ain_2,\ain_3$
with background fields~$(\sigma ,\din,\cc)$.

  \begin{theorem}[]\label{thm:56}
 $\ain_3$~is isomorphic to the trivial theory. 
  \end{theorem}

  \begin{proof}
 The integer Bockstein of the mod~2 reduction is trivial. 
  \end{proof}

  \begin{corollary}[]\label{thm:57}
 There exist isomorphisms $\zeta \:\ain_1\xrightarrow{\;\cong \;}\ain_2$.
  \end{corollary}

For our work in~\S\ref{sec:7a} we need an isomorphism which satisfies a
particular property that we specify below in~\eqref{eq:178}.  We proceed to
construct it.  To begin, fix a \emph{flat}\footnote{Isomorphisms of
invertible field theories may be flat or nonflat;
compare~\S\ref{subsec:9.2}.} isomorphism
  \begin{equation}\label{eq:169}
     \zeta \:\ain_1\xrightarrow{\;\;\cong \;\;}\ain_2. 
  \end{equation}
Pull back the theories~$\ain_1,\ain_2$ to invertible
theories~$\tain_1,\tain_2$ with background fields $(\sigma ,\din,\cc,t)$ in
which $t$~is a nonflat trivialization of the $\Cx$-connection~$\cc$, i.e.,
$t$~is a section of the underlying principal $\Cx$-bundle.  Then $t$~induces
a \emph{nonflat} trivialization of the theories~$\tain_1,\tain_2$: they are
\emph{topologically trivial}.  In the formalism of Appendix~\ref{sec:10} we
omit the spin structure, and the remaining fields are sections of the
sheaf\footnote{A principal $\ZZ$-bundle has a unique connection, so the
`$\nabla $' in `$\BNZ$' is redundant; the latter is better
denoted~`$B^{}_{\bullet }\ZZ$'} $\BNZ\times \ENC$ on~$\Man$.  It is
convenient to replace~$\BNZ$ with the representable sheaf~$\cir$.  This
amounts to specifying a classifying map for each principal $\ZZ$-bundle.
Then the topologically trivialized theories~$\tain_i$, $i=1,2$, give rise to
differential forms (see~\eqref{eq:168})
  \begin{equation}\label{eq:170}
     \eta _i\:\cir\times \ENC\longrightarrow \Omega ^3_{\CC}. 
  \end{equation}
Let $\omega \in \Omega _{\cir}^1$ be the rotation-invariant closed 1-form
which integrates to $1$, and let $a\in \Omega _{\CC}^1(\ENC)$ be $\frac{\I}{2 \pi}$ 
times the
universal connection 1-form; the latter gives the equivalence $\ENC\to \Omega
_{\CC}^1$.  Then from~\eqref{eq:136} we deduce
  \begin{equation}\label{eq:171}
     \begin{aligned} \eta _1&=\frac 12 a\wedge da \\ \eta _2&=\frac 12 (a +
      \frac 12\omega )\wedge d(a+\frac 12\omega ) = \eta _1 + \frac
      14d(a\wedge \omega ).\end{aligned} 
  \end{equation}
The topologically trivialized theories defined by~$\eta _1,\eta _2$ are
isomorphic as invertible theories (forgetting the topological
trivialization)---as they must be by Corollary~\ref{thm:57}---since the
3-forms differ by an exact 3-form. 
 
Our constructions yield two isomorphisms $\tain_1\to \tain_2$.  First, the
flat isomorphism~$\zeta \:\ain_1\to \ain_2$ lifts to a \emph{flat}
isomorphism 
  \begin{equation}\label{eq:181}
     \tz\:\tain_1\xrightarrow{\;\;\cong \;\;}  \tain_2. 
  \end{equation}
Second, the topological trivializations induce a \emph{nonflat} isomorphism
  \begin{equation}\label{eq:182}
     \lambda \:\tain_1\xrightarrow{\;\;\cong \;\;}  \tain_2. 
  \end{equation}
From~\eqref{eq:171} we compute that the curvature of~$\lambda $ is $\frac
14d(a\wedge \omega )$.  The ratio $\tz/\lambda $ is a 2-dimensional
invertible theory on spin manifolds with a background field in $\cir\times
\ENC$, and its curvature is the 3-form $-\frac 14d(a\wedge \omega )$.  Let
$\beta $~be the 2-dimensional invertible field theory defined by the 2-form
  \begin{equation}\label{eq:172}
    - \frac 14\,a\wedge \omega , 
  \end{equation}
and define the \emph{flat} 2-dimensional theory~$\gamma $ by 
  \begin{equation}\label{eq:173}
     \tz=\beta \gamma \lambda . 
  \end{equation}

  \begin{lemma}[]\label{thm:66}
 The abelian group of topological invertible 2-dimensional theories with
background fields $(\sigma ,\din,\cc,t)$ is isomorphic to the Klein group
$\bmut\times \bmut$.  Furthermore, each theory depends only on the spin
structure~$\sigma $ and the double cover~$\delta $ induced by the principal
$\ZZ$-bundle~$\din$.   
  \end{lemma}

\noindent
 The partition functions of these four theories on a closed
2-manifold~$\Sigma $ are
  \begin{equation}\label{eq:174}
     1,\quad (-1)^{\Arf(\sigma )},\quad (-1)^{\Arf(\sigma +\delta )},\quad
     (-1)^{\Arf(\sigma +\delta )-\Arf(\sigma )}, 
  \end{equation}
where $\Arf$~is the Arf invariant of the spin structure. 

  \begin{proof}
 Since $\ENC\cong \Omega ^1_{\CC}$ is contractible, the group of invertible
theories is isomorphic\footnote{Both unitary and nonunitary theories are
discussed in~\cite{FH1}.  Here we do not assume unitarity, but the background
fields are for 3-manifolds, even for the 2-dimensional theory which is the
ratio of isomorphisms of 3-dimensional theories, hence the domain should at
first glance have $\Sigma ^3MT\!\Spin_3$ in place of~$\MSpin$.  However, the
obstruction theory argument in the proof of \cite[Theorem~7.22]{FH1} allow us
to replace~$\Sigma ^3MT\!\Spin_3$ with~$\MSpin$ in~\eqref{eq:176}.} to the
group of characters of
  \begin{equation}\label{eq:175}
     \pi _2(\MSpin\wedge \cir_+)\cong \pi _2(\MSpin)\,\oplus \,\pi
     _1(\MSpin)\cong \zt\oplus \zt. 
  \end{equation}
One can see that the theories listed in~\eqref{eq:174} exhaust the
possibilities, or can check that
  \begin{equation}\label{eq:176}
     \MSpin\wedge \cir_+\longrightarrow \MSpin\wedge \RP^{\infty}_+ 
  \end{equation}
induces an isomorphism on~$\pi _2$. 
  \end{proof}

It follows that the theory~$\gamma $ in~\eqref{eq:173} depends only
on~$(\sigma ,\delta )$.  Therefore, replace~\eqref{eq:169} by the isomorphism
  \begin{equation}\label{eq:177}
     \ztil =\gamma \inv \zeta \:\ain_1\xrightarrow{\;\;\cong \;\;}\ain_2. 
  \end{equation}
For this choice of isomorphism we have 
  \begin{equation}\label{eq:178}
     {}^{\textnormal{tt}}\hneg\ztil =\beta \lambda , 
  \end{equation}
where recall that $\beta $~is defined by the 2-form~\eqref{eq:172}.  

We summarize with this refinement of Corollary~\ref{thm:57}. 

  \begin{corollary}[]\label{thm:72}
 There exists an isomorphism  
  \begin{equation}\label{eq:190}
     \zeta \:\ain_1\xrightarrow{\;\cong \;}\ain_2 
  \end{equation}
such that the induced isomorphism $\tz\:\tain_1\to \tain_2$ of theories which
include a nonflat trivialization of the $\Cx$-connection satisfies 
  \begin{equation}\label{eq:191}
     \frac{\tz}{\lambda } =\beta , 
  \end{equation}
where $\lambda $~is the isomorphism~\eqref{eq:182} and $\beta $~is defined by
the differential form~\eqref{eq:172}.  
  \end{corollary}

\noindent
 This is the isomorphism we use in~\S\ref{sec:7a}.

   \subsubsection{Global abelianization}\label{subsubsec:5.2.4} 

Diagram~\eqref{eq:51} illustrates global abelianization of an
$\SLC$-connection.  We apply Corollary~\ref{thm:22} to deduce an isomorphism
of Chern-Simons invariants, expressed as an isomorphism among three
invertible 3-dimensional field theories\footnote{As in Theorem~\ref{thm:55}
we restrict to flat connections, so to topological invertible field
theories.} $\eSL$, $\eCx$, and~$\emt$.  Each is defined on the bordism
multicategory of dimension~$\le 3$ manifolds with corners equipped with a
spin structure~$\sigma $ and a flat $H$-connection~$\Theta $.  The first
theory~$\eSL$ uses only the underlying orientation of~$\sigma $, and it
evaluates the Chern-Simons theory~$\F{\SLC}$ at level~$-i(c_2)$ on the flat
$\SLC$-connection $r(\Theta )$.  The second theory~$\eCx$ maps a spin
manifold with flat $H$-connection to the total space of the associated
$\bmut$-bundle with its induced spin structure and flat $\Cx$-connection, and
then evaluates this data using spin Chern-Simons theory~$\rS_{\Cx}$ at
level~$\lambda $.  The third theory~$\emt$ evaluates the spin Chern-Simons
theory~$\rS_{\!\bmut}$ at level~$\alpha $ on the associated
$\bmut$-connection\footnote{Of course, this is simply a double cover, but we
have endeavored to use consistent and transparent notation.}~$q(\Theta )$.

  \begin{theorem}[]\label{thm:23}
 There is an isomorphism $\eSL\cong \eCx \otimes \emt$. 
  \end{theorem}

  \begin{proof}
 As in the proof of Theorem~\ref{thm:55}, it suffices to check equality of
partition functions on a closed oriented 3-manifold~$X$ equipped with a flat
$H$-connection.  For this, apply the secondary invariant version
of~\eqref{eq:68} to the following slight enlargement of
the diagram~\eqref{eq:51}: 
  \begin{equation}\label{eq:158}
     \begin{gathered} \xymatrix{\tX\ar[r]\ar[d]_{\pi } & B(\Cx)^\delta
     \ar[d]^{p} \\ X\ar[r]^{} & B(H)^\delta \ar[r]^<<<<{r} \ar[d]^q&
     B(SL_2\CC)^\delta \\&\Bbmut} \end{gathered} 
  \end{equation}
  \vskip -2.5pc\qedhere
  \renewcommand{\qedsymbol}{}
  \end{proof}

Let $\eSLi,\epsilon ^{\infty} _{\Cx},\emti$~denote the pullbacks
of~$\eSL,\epsilon \mstrut _{\Cx},\emt$ to the bordism multicategory of
dimension~$\le 3$ manifolds with corners equipped with a spin structure, a
flat $H$-connection, and a lift of the associated $\bmut$-bundle to a
principal $\ZZ$-bundle.  Then Lemma~\ref{thm:21}(3) immediately implies 

  \begin{theorem}[]\label{thm:62}
 $\emti$~is isomorphic to the trivial theory. 
  \end{theorem}

  \begin{corollary}[]\label{thm:63}
 A trivialization of~$\emti$ determines an isomorphism 
$\nu: \eSLi\xrightarrow{\;\cong \;}\epsilon ^{\infty}_{\Cx}$.
  \end{corollary}

\noindent
 In Appendix~\ref{sec:C11} we constrain the trivialization, based on
considerations in~\S\ref{sec:difference-line}.

\section{Abelianization of Chern-Simons lines}\label{sec:7a}

Throughout this section we take $G = \SLC$.

So far we have discussed generalities about the Chern-Simons
theories $\scrF_G$, $\scrS_{\C^\times}$, and their relation to 
one another via stratified abelianization.
Now we begin discussing applications.

Suppose $Y$ is a compact 2-manifold, equipped with a 
boundary-reduced flat $G$-bundle $P \to Y$.
In this section and
the next we give a new description of the line
$\scrF_G(Y;P)$. 
The idea is to identify $\scrF_G(Y;P)$ with 
$\scrS_{\C^\times}(\tY; Q^\eps_\tw; \sigma^\eps_\tw) \otimes \cL(Y,\eps)$,
where $Q^\eps_\tw$ is a flat $\C^\times$-bundle
over a branched double cover $\tY \to Y$,
and $\cL(Y,\eps)$ is a universal line which does not depend on $P$.
In the rest of this section we give a sketch of the construction.

The bundle $Q^\eps_\tw$ will be constructed as follows.
First, we fix a semi-ideal triangulation $\scrT$ of $Y$,
and let $\tY \to Y$ be the associated branched double cover,
and $\cN^\scrT$ the associated spectral network (Construction \ref{thm:9}).
Next, we fix a section of $P/B$ over each vertex of $Y$
(as usual, for ideal vertices this means a flat section over the
corresponding boundary component),
obeying the genericity Assumption \ref{thm:74}.
From these data, by Construction \ref{thm:14} we obtain stratified
abelianization data $(P, Q, \mu, \theta)$.

Now, suppose we ignore the branch locus $Y_{\twobstrat}$ 
for a moment, i.e. we work just over $Y_\getwoa$.
Then, according to Theorem~\ref{thm:44}, $\scrF_G(Y_\getwoa;P)$
is naturally isomorphic to $\scrF_H( Y_\getwoa;Q)$; and by Corollary~\ref{thm:63},
if we choose a spin structure $\sigma$ on $Y$, and a lift of $\tY_\getwoa \to Y_\getwoa$ to a $\Z$-bundle $Y^{\eps,\infty} \to Y_\getwoa$,
$\scrF_H(Y_\getwoa;Q)$ is naturally isomorphic to $\scrS_{\C^\times}(\tY_\getwoa; Q; \pi^* \sigma)$.
Composing these two we would get an isomorphism
\begin{equation}
  \scrF_G(Y_\getwoa;P) \to \scrS_{\C^\times}(\tY_\getwoa; Q; \pi^* \sigma) \; .
\end{equation}
This is the kind of statement we are after, but to extend the 
right side over the full $Y$ it needs to be modified.
The complication is that 
$Q \to \tY_\getwoa$ has holonomy $-1$ around these points,
as noted in Lemma \ref{thm:5},
and $\pi^* \sigma$ does not extend over them either.
One could think of this holonomy as a kind of singularity,
and try to define a modified version of the theory 
$\scrS_{\C^\times}$ which works directly with these singular objects.
Here we take an alternative path: we twist both $Q \to \tY_\getwoa$ 
and $\pi^* \sigma$
by a $\bmut$-bundle $\tY^{\eps,4} \to \tY_\getwoa$, 
which cancels the unwanted holonomy.
Fortunately, Corollary~\ref{thm:72} ensures that this 
twisting does not change the Chern-Simons theory away
from the branch locus, i.e., we get an isomorphism 
\begin{equation}
  \scrF_G(Y_\getwoa;P) \to \scrS_{\C^\times}(\tY_\getwoa; Q^\eps_\tw; \sigma^\eps_\tw) \; .
\end{equation}
Now both sides extend over the branch locus.
The isomorphism however does not: there is a mismatch between
the two Chern-Simons theories over the branch locus. We 
measure this mismatch by a $P$-independent 
line we call $\cL(Y,\eps)$. Thus ultimately 
what we get is an isomorphism
\begin{equation} \label{eq:intro-iso}
 \scrF_G(Y;P) \to \scrS_{\C^\times}(\tY; Q^\eps_\tw; \sigma^\eps_\tw) \otimes \cL(Y,\eps) \; .
\end{equation}
In the above we needed to make various choices: a semi-ideal triangulation $\scrT$ of $Y$, a $\Z$-bundle $\tY^\infty \to Y$, and a $\bmut$-bundle $\tY^{\eps,4} \to \tY_\getwoa$.
It turns out that both $\tY^\infty \to Y$ and $\tY^{\eps,4} \to \tY_\getwoa$ can be conveniently built from the data of edge-orientations
on $\scrT$; this is the data we call $\eps$.

In \S\ref{sec:edge-orientations}-\S\ref{sec:difference-line} 
we discuss
the necessary twisting and the properties of the difference line 
$\cL(Y,\eps)$; in particular, we compute the action of
rotations of a triangle on this line. Note also Appendix~\ref{sec:C11} in
which we fix a choice in the construction so that the super line
$\cL(Y,\eps)$ is \emph{even}.  (Recall the discussion in Remark~\ref{thm:82}.)
In \S\ref{sec:abelianization-cs-triangulated-surface} we
develop the stratified abelianization map \eqref{eq:intro-iso},
as Construction \ref{constr:abelianization-surface}.
In the remaining sections \S\ref{sec:edge-reversal-triangle}-\S\ref{sec:edge-reversal} we discuss some aspects of the dependence of
stratified abelianization on the edge-orientations $\eps$, which will be used
in the explicit calculations to follow.

\subsection{Edge-orientations on a triangle} \label{sec:edge-orientations}

Let $(\Delta,\scrT)$ be a triangle.
Choose an orientation $\eps_E$ for each $E \in \edges(\scrT)$, and let
$\eps = (\eps_E)_{E \in \edges(\scrT)}$.
Let 
\begin{equation}
\eps_E(v,v') = \begin{cases} +1 & \text{if the edge $E$ with vertices $v$, $v'$ is oriented from $v$ to $v'$}, \\
-1 & \text{otherwise.}
\end{cases}  
\end{equation}

\begin{constr} \label{constr:Z-bundle-from-eps}
$\eps$ determines a lift of the $\bmuu_2$-bundle 
$\tDelta_{\getwoa} \to \Delta_{\getwoa}$ to a $\Z$-bundle $\tDelta^{\eps,\infty} \to \Delta_{\getwoa}$.
\end{constr}

\begin{proof}
$\Delta_0$ has three connected components $\Delta_0^E$, each containing one edge $E$.
Fix $E$ with vertices $v$, $v'$.
Over $\Delta_0^E$ we define the fiber of $\tDelta^{\eps,\infty}$ as the $\Z$-torsor 
\begin{equation}
(\{v,v'\} \times \Z) / \sim, \qquad (v,n) \sim (v',n+\eps_E(v,v')) .
\end{equation}

Each wall in $\Delta_{-1}$ lies in the boundary of two 
domains $\Delta_0^E$, $\Delta_0^{E'}$, where the edges $E$, $E'$ have 
one vertex $w$ in common.
We glue $\tDelta^{\eps,\infty}$ across such a wall by 
identifying $[(w,n)]$ on one side with $[(w,n)]$ on the other side.

The $\bmuu_2$-bundle $\tDelta^{\eps,\infty} / 2\Z$ is isomorphic to $\tDelta_{\getwoa}$, 
via the map
which takes $[(v,n)] \mapsto v$ when $n \in 2\Z$. Thus $\tDelta^{\eps,\infty}$ is indeed
a lift of $\tDelta_{\getwoa}$ to a $\Z$-bundle as claimed.
\end{proof}

Note that the clockwise monodromy of $\tDelta^{\eps,\infty}$ around $\partial \Delta$
is $n_+ - n_- \in \{3,1,-1,-3\}$ where $n_+$ ($n_-$) is the number of edges oriented clockwise (counterclockwise).
Reducing mod $2$, we recover the fact that the monodromy of the double cover $\tDelta_{\getwoa} 
\to \Delta_{\getwoa}$ around $\partial \Delta$ is the nontrivial element of $\bmuu_2$.

\subsection{Edge-orientations on a triangulated surface}

In the last subsection we considered a single triangle. More generally,
suppose we have a semi-ideally triangulated
surface $(Y,\scrT)$ and edge-orientations $\eps = (\eps_E)_{E \in \edges(\scrT)}$.
All of our constructions glue canonically across edges,
and thus we obtain
a $\Z$-bundle
\begin{equation}
\tY^{\eps,\infty} \to Y_{\getwoa}.
\end{equation}
The action of $2\Z$ on $\tY^{\eps,\infty}$ commutes with the projection
$\tY^{\eps,\infty} \to \tY_{\getwoa}$, 
so $\tY^{\eps,\infty}$ is also a $2\Z$-bundle over $\tY_{\getwoa}$.
Let 
\begin{equation}
\tY^{\eps,4} = \tY^{\eps,\infty} / 4\Z \, .  
\end{equation}
This is a $\bmuu_2$-bundle over $\tY_{\getwoa}$, since $2\Z / 4\Z = \bmuu_2$.
We can describe its holonomies around cycles explicitly, as follows.

\begin{defn} \label{def:quadrilateral-sign}
For any $E \in \edges(\scrT)$, let $\cQ_E$ be the quadrilateral
formed by the two triangles containing $E$. 
The $\eps$-sign of $E$ is $(-1)^n$, where $n$ is the number of edges of $\cQ_E$
which are oriented clockwise by $\eps$. (Replacing ``clockwise'' by ``counterclockwise''
here would give the same definition.)
\end{defn}

\insfig{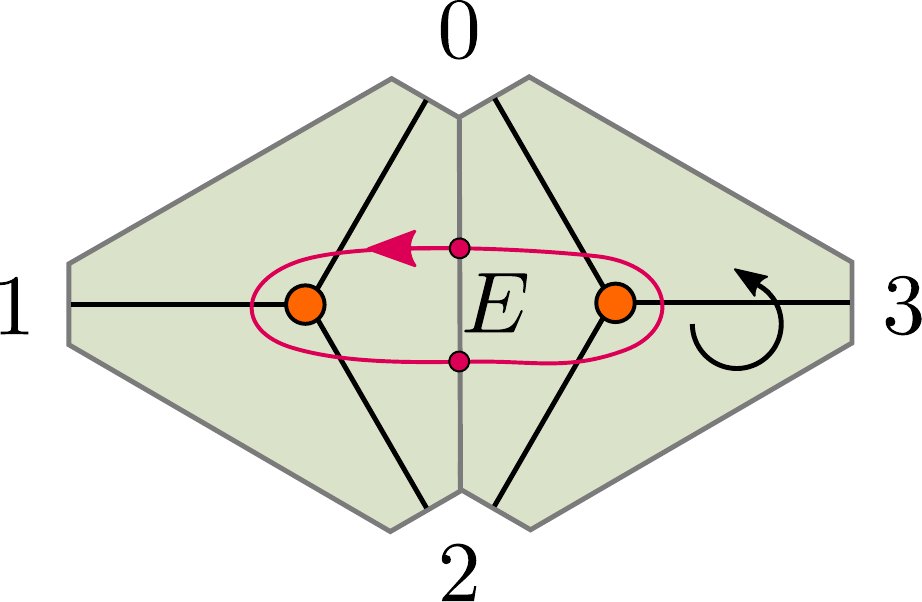}{0.4}{The quadrilateral $\cQ_E$, with vertices labeled and the class
$\gamma_E$ shown.}

We recall the class $\gamma_E \in H_1(\tY)$ defined before
Proposition~\ref{thm:68}. This class depends on an orientation of $Y$, and reversing
the orientation sends $\gamma_E \mapsto -\gamma_E$; the assertions in the rest of
this section hold independent of the choice of orientation.

\begin{prop} \label{prop:loop-hol-formula} 
The holonomy $\hol_{\tY^{\eps,4}} (\gamma_E)$ 
is the $\eps$-sign of $E$
as defined in Definition \ref{def:quadrilateral-sign}.
\end{prop}
\begin{proof}
Over each domain $Y_0^E$, the fiber over the 
sheet labeled by vertex $w$ of $E$ is
$\{[(w,n)]: n \in 2\Z\}$.
Thus $\tY^{\eps,\infty}$ is the trivial $2\Z$-bundle on each 
connected component of $\tY_0$.
The gluing across preimages of walls is as follows. Each wall 
runs into a vertex $v$.
On the sheet labeled by the vertex $v$,
$[(v,n)]$ is glued to $[(v,n)]$.
The other sheet is
labeled by a vertex $v'$ on one side and $v''$ on the other.
There the gluing takes $[(v',n)]$ to $[(v'', n + k)]$,
where
\begin{equation}
 k = \eps(v',v) + \eps(v,v'') \in \{-2,0,2\} \, .
\end{equation}

In traversing $\gamma_E$, 
referring to Figure~\ref{fig:quad-loop-labeled} we see that 
we cross two walls where the gluing is nontrivial (the horizontal walls in the figure). 
Summing their contributions,
$\hol_{\tY^{\eps,\infty}} (\gamma_E)$ is a shift by
$\eps(0,1) + \eps(1,2) + \eps(2,3) + \eps(3,0)$,
which agrees mod $4$ with $2n$, where $n$ was defined in Definition \ref{def:quadrilateral-sign}. 
Reducing mod $4$ gives the desired
statement.
\end{proof}

\begin{constr} \label{constr:Z-bundle-delta-E}
Consider two edge-orientations $\eps$, $\eps'$
which differ by reversing the orientation on a single edge $E$.
The difference $\bmuu_2$-bundle 
\begin{equation} \label{eq:difference-bundle}
\tY^{\eps,4} \otimes_{\bmuu_2} \tY^{\eps',4} \to \tY_{\getwoa}  
\end{equation}
admits a lift to a $\Z$-bundle (canonical up to isomorphism)
\begin{equation}
\varpi^E \to \tY_{\getwoa}
\end{equation}
which extends over $\tY$.

For any class $\mu \in H_1(\tY)$
\begin{equation} \label{eq:varpi-holonomy}
 \hol_{\varpi^E} (\mu) = \IP{\gamma_E,\mu} \, . 
\end{equation}
\end{constr}

\begin{proof}
We first describe the difference bundle \eqref{eq:difference-bundle}, as we did in the proof of Proposition~\ref{prop:loop-hol-formula}.
It is the trivial $\bmuu_2$-bundle on each
connected component of $\tY_0$. The gluing across preimages of
walls is as follows. Consider a wall ending on a vertex $v$.
On the sheet labeled by $v$,
the gluing is given by the identity element in $\bmuu_2$. On the other sheet, the gluing
is given by the nontrivial element in $\bmuu_2$ if $v$ is a vertex of $E$, 
and otherwise by the identity.

Now we can define the lift $\varpi^E$: it 
is the trivial $\Z$-bundle on each connected component of $\tY_0$,
with gluing across preimages of walls as follows.
Consider a wall ending on a vertex $v$.
On the sheet labeled by $v$,
the gluing is given by $0 \in \Z$.
On the other sheet, if $v$ is a vertex of $E$, then the gluing in the direction away from 
$E$ is given by $+1 \in \Z$; if $v$ is not a vertex of $E$, then the gluing
is given by $0 \in \Z$.

The holonomy of $\varpi^E$ 
around a loop in $\tY_{\getwoa}$ encircling a branch point comes to $-1+1 = 0$, 
so $\varpi^E$ extends across the branch points, and thus over $\tY$ as desired.
Finally, the formula $\hol_{\varpi^E} \mu = \IP{\mu,\gamma_E}$ is
obtained directly by evaluating both sides on arcs $\mu$ 
crossing the quadrilateral $\cQ_E$.
\end{proof}

\subsection{Twistings over a triangulated surface}\label{subsec:S6.3}

Now let $(Y, \eps, \sigma, \cA)$ be a triangulated
surface, with edge-orientations, spin structure, 
and stratified abelianization data $\cA = (P, Q, \mu, \theta)$.
Then define the $\C^\times$-bundle
\begin{equation} \label{eq:qepstw-def}
Q^{\eps}_\tw = Q \otimes_{\bmuu_2} \tY^{\eps,4} \to \tY_{\getwoa}. 
\end{equation}
$Q^\eps_\tw$ has trivial monodromy around  
each point of $\tY_{\twobstrat}$, and thus
extends to $\tY$, unlike $Q$.
We use the name $Q^\eps_\tw$ also for the extension.

The spin structure $\pi^* \sigma$ on $\tY_\getwoa$ 
is non-bounding on a circle around a branch point and thus
does not extend from $\tY_{\getwoa}$ to $\tY$,
but its twist
\begin{equation}
\sigma^\eps_{\tw} = \pi^* \sigma \otimes_{\bmuu_2} \tY^{\eps,4}  
\end{equation}
does extend to a spin structure over $\tY$.

It will be useful below to have some concrete information about this spin
structure. 
We recall a convenient bit of notation first. 
Given a spin structure $\sigma$ on a surface, and a simple 
closed curve $\lambda$, 
we define
\begin{equation}
  \sigma(\lambda) = \begin{cases} 1 & \text{ if } \sigma\vert_\lambda \text{ extends to a spin structure on the disc}, \\
                                  -1 & \text{ otherwise.} \end{cases}
\end{equation}
We also recall the class $\gamma_E \in H_1(\tY)$ defined before
Proposition~\ref{thm:68}.

\begin{prop} \label{prop:ss-formula}
$\sigma^\eps_\tw(\gamma_E)$ is the $\eps$-sign of $E$ as defined
in Definition \ref{def:quadrilateral-sign}.
\end{prop}

\begin{proof} Since $\pi_*\gamma_E$ can be represented by the boundary of a disc
in $\tY$, $\pi^*\sigma(\gamma_E) = +1$. Thus $\sigma^\eps_\tw(\gamma_E)$
is the monodromy of the $\bmuu_2$-bundle $\tM^{\eps,4} \to \tY$
around $\gamma_E$, which we computed in Proposition~\ref{prop:loop-hol-formula}.
\end{proof}

\subsection{Stratified abelianization}\label{subsec:S6.4}

Fix a manifold $X$ of dimension $\le 3$ with corners,
with $X_{\twobstrat} = \emptyset$.
Suppose $X$ is equipped with a spin structure $\sigma$, a
spectral network $\cN$, and
stratified abelianization data $\cA = (P,Q,\mu,\theta)$ 
over $(X,\cN)$.
Also suppose given a $\Z$-bundle $\tX^\infty$ whose mod
2 reduction is the double cover $\tX$.
Then let $\tX^4$ denote the mod 4 reduction of $\tX^\infty$.
$\tX^4 \to \tX$ is a $\bmuu_2$-bundle.

\begin{thm} \label{thm:strat-abelianization}
There is a canonical isomorphism
\begin{equation}\label{eq:S1}
  \chi(X;\cA;\sigma;\tX^\infty): \scrF_G(X;P) \to \scrS_{\C^\times}(\tX;Q \otimes_{\bmuu_2} \tX^4 ; \pi^* \sigma \otimes_{\bmuu_2} \tX^4) \, .
\end{equation}
\end{thm}
 
\noindent 
 We remark that the isomorphism~\eqref{eq:S3} is based on a choice of
trivialization in Corollary~\ref{thm:63}.  We constrain that choice in
Appendix~\ref{sec:C11}. 

\begin{proof} We combine ingredients as follows.
First, Theorem \ref{thm:44} (triviality of Chern-Simons in unipotent directions) gives
\begin{equation}\label{eq:S2}
 \rho(X;\cA): \scrF_G(X;P) \to \scrF_H(X;Q) \, .
\end{equation}
Second, Corollary \ref{thm:63} (identity between Chern-Simons for $H$-bundles over $X$ and $\C^\times$-bundles over $\tX$) gives
\begin{equation}\label{eq:S3}
 \nu(X;Q;\tX^\infty): \scrF_H(X;Q) \to \scrS_{\C^\times}(\tX;Q;\pi^* \sigma) \, .
\end{equation}
Finally, Corollary \ref{thm:72} (invariance of spin $\C^\times$ Chern-Simons under $\bmuu_2$-twists) gives
\begin{equation}\label{eq:S4}
  \zeta(\tX;Q;\pi^* \sigma;\tX^\infty): \scrS_{\C^\times}(\tX;Q;\pi^* \sigma) \to \scrS_{\C^\times}(\tX;Q \otimes_{\bmuu_2} \tX^4;\pi^* \sigma \otimes_{\bmuu_2} \tX^4) \, .
\end{equation}
The composition of these three is the desired isomorphism.
\end{proof}

The concrete nature of the isomorphism $\chi(X)$ 
in Theorem~\ref{thm:strat-abelianization}
depends on the nature of $X$. In general, $\chi(X)$ 
is an isomorphism between objects in appropriate diagram categories.
For instance, if $X$ is a closed 2-manifold, $\chi(X)$ is an isomorphism of
lines; if $X$ is a 3-manifold with boundary, $\chi(X)$ is an isomorphism of lines together
with an isomorphism of objects in those lines; if $X$ is a closed 3-manifold, $\chi(X)$ 
is just an equation.

In what follows we will need to know that $\chi$ has good gluing properties.
These properties are most succinctly summarized as follows: they are
just as if $\chi$ came from an isomorphism of 3-dimensional
topological field theories, defined on a bordism category of 
oriented manifolds $X$ equipped with a spectral network, 
stratified abelianization data, and a lift of the double
cover $\widetilde X$ to a $\Z$-bundle.
We will apply this below to various individual manifolds $X$
carrying this data.
Of the three ingredients above, 
two of them were formulated as isomorphisms of
topological field theories; the third,  Theorem \ref{thm:44}, was not
formulated in this language, but it was constructed in a fully local and canonical
way. This is sufficient to imply the desired gluing properties.

\subsection{The difference line for a triangle} \label{sec:difference-line}

Let $(\Delta, \eps, \sigma, \cA)$ be an oriented triangle with edge-orientations, spin structure, 
and stratified abelianization data $\cA = (P, Q, \mu, \theta)$

\insfig{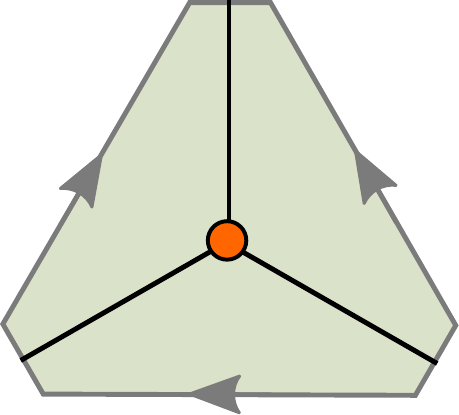}{0.5}{A triangle $\Delta$ with 
edge-orientations $\eps$ and the canonical SN-stratification.}

Using $G$ Chern-Simons theory on $\Delta$ gives an object in
the \Vline associated to the boundary $\partial \Delta$,
\begin{equation}
 \scrF_G(\Delta; P) \in \scrF_G(\partial \Delta; \partial P) \, .
\end{equation}
Using $\C^\times$ Chern-Simons theory on $\tDelta$ likewise gives an object
in the \Vline associated to $\partial \tDelta$,
\begin{equation}
 \scrS_{\C^\times}(\tDelta; Q^{\eps}_{\tw}; \sigma^\eps_{\tw}) \in \scrS_{\C^\times}(\partial \tDelta; \partial Q^{\eps}_{\tw}; \sigma^\eps_{\tw}) \, .  
\end{equation}
Because $\Delta_{\twobstrat} \neq \emptyset$, we cannot use 
Theorem~\ref{thm:strat-abelianization} to identify these two objects. However,
$\Delta_{\twobstrat}$ does not intersect $\partial \Delta$,
so Theorem~\ref{thm:strat-abelianization} gives an 
equivalence of \Vlines,
\begin{equation}
\chi(\partial \Delta; \partial \cA; \sigma; \tDelta^{\eps,\infty}): \scrF_G(\partial \Delta; \partial P) \to \scrS_{\C^\times}(\partial \tDelta; \partial Q^\eps_\tw; \sigma^\eps_\tw) \, .
\end{equation}
Now we can compare the two objects: we define a line
\begin{equation} \label{def:L}
\cL(\Delta,\eps,\sigma,\cA) = \chi(\partial \Delta; \partial \cA; \sigma; \tDelta^{\eps,\infty})(\scrF_G(\Delta;P))
\otimes (\scrS_{\C^\times}(\tDelta ; Q^{\eps}_{\tw} ; \sigma^\eps_{\tw}))^*.  
\end{equation}
\noindent 
 We explain in Appendix~\ref{sec:C11} how to make a choice of
isomorphism~$\chi $ in~\eqref{eq:S1} so that the super line~$\cL(\Delta
,\epsilon ,\sigma ,\cA)$ is \emph{even}; without further argument it could be
odd.

The line $\cL(\Delta,\eps,\sigma,\cA)$ depends only on $(\Delta,\eps)$,
in the following sense.

\begin{prop} \label{prop:triangle-iso}
Suppose $(\Delta, \eps, \sigma, \cA)$ and $(\Delta', \eps', \sigma', \cA')$ are 
triangles with edge-orientations, spin structure, 
and stratified abelianization data.
An orientation-preserving affine-linear isomorphism $f: \Delta \to \Delta'$ which carries $\eps$ to $\eps'$ induces a canonical map
\begin{equation}
f_*: \cL(\Delta,\eps,\sigma,\cA) \to \cL(\Delta',\eps',\sigma',\cA') \, .
\end{equation}
An orientation-reversing affine-linear isomorphism $f: \Delta \to \Delta'$ which carries $\eps$ to $\eps'$ induces a canonical map
\begin{equation}
f_*: \cL(\Delta,\eps,\sigma,\cA) \to \cL(\Delta',\eps',\sigma',\cA')^* \, .  
\end{equation}
\end{prop}

\begin{proof}
By uniqueness of spin structures on $\Delta$, 
we can lift $f$ to an isomorphism $\sigma \to f^* \sigma'$.
By Proposition~\ref{prop:uniqueness-of-triangle-ab} we can lift
$f$ to an isomorphism $\cA \to f^* \cA'$.
Finally, since $f^* \eps' = \eps$ we can lift $f$ to
an isomorphism $\tDelta^{\eps,\infty} \to \tDelta^{f^* \eps',\infty}$.
All of our constructions are canonical and depend only
on these data, so we obtain a map $f_*$ as desired. 
It only remains to check that this map is independent of the choices we 
made in lifting $f$. To see this, we need to show 
that the nontrivial automorphism of the spin structure---the \emph{spin flip}---and the automorphism of $\cA$ induced by the action of
$-1 \in G$ and $-1 \in \C^\times$ both act trivially
on $\cL(\Delta,\eps,\sigma,\cA)$.

The argument that the spin flip acts trivially is the subject of
Appendix~\ref{sec:C11}. 

To show that the automorphism $-1$ of $\cA$ acts trivially we argue
as follows. Our definition of $\cL(\Delta,\eps,\sigma,\cA)$ can equally
well be made using $G = \GL_2 \C$ rather than $\SL_2 \C$. 
In this case the whole center
$Z(\GL_2 \C) \simeq \C^\times$ acts on $\cA$ and thus on $\cL(\Delta,\eps,\sigma,\cA)$.
Now, to compute the action of the element $\lambda \in \C^\times$ 
on $\cL(\Delta,\eps,\sigma,\cA)$ we consider the mapping torus
$\Delta \times S^1$, with abelianization data $\cA_\lambda = (P_\lambda, Q_\lambda, \mu_\lambda, \theta_\lambda)$ 
obtained by gluing $\cA$ to itself
with the action of $\lambda$. 
The action of $\lambda \in \C^\times$ is given by the ratio
\begin{equation} \label{eq:lambda-ratio}
 \frac{\chi(\partial\Delta \times S^1;\partial \hat\cA;\sigma \times S^1;\tDelta^{\eps,\infty} \times S^1) (\scrF_G(\Delta \times S^1;P_\lambda))}{\scrS_{\C^\times}(\tDelta \times S^1;Q^\eps_{\lambda,\tw};\sigma^\eps_\tw \times S^1)} \, .
\end{equation}
Now we consider the dependence on $\lambda$.
Since $P_\lambda$ is a flat $G$-connection for each $\lambda$, and the
curvature of $\scrF_G$ vanishes when evaluated on a $1$-parameter 
family of flat connections, it follows that
$\scrF_G(\Delta \times S^1;P_\lambda)$ gives a covariantly constant section of the 
bundle over $\C^\times$ with fiber $\scrF_G(\partial \Delta \times S^1;\partial P_\lambda)$.
Likewise $\scrS_{\C^\times}(\tDelta \times S^1;Q^\eps_{\lambda,\tw};\sigma^\eps_\tw \times S^1)$
is covariantly constant.
Finally, the isomorphism $\chi(\partial\Delta \times S^1;\partial \hat\cA;\sigma \times S^1;\tDelta^{\eps,\infty} \times S^1)$ is flat, i.e. it is an isomorphism of line bundles
with connection. Thus the ratio \eqref{eq:lambda-ratio} is locally constant as
a function of $\lambda$. But at $\lambda = 1$ it gives the action of 
$1$, which is trivial; thus it must be $1$ for all $\lambda$, as desired.
\end{proof}

\begin{cor}
The line $\cL(\Delta,\eps,\sigma,\cA)$ depends only
on $(\Delta,\eps)$ up to canonical isomorphism.
\end{cor}

\begin{proof}
Given
 $(\Delta,\eps)$ and two different data $(\sigma,\cA)$ and $(\sigma',\cA')$
 we apply Proposition~\ref{prop:triangle-iso}
 taking $f: \Delta \to \Delta$ to be the identity map.
 This gives the desired isomorphism $f_*: \cL(\Delta,\sigma,\cA,\eps) \to \cL(\Delta,\sigma,\cA',\eps')$.
\end{proof} 

With this corollary in mind we just write the line as $\cL(\Delta,\eps)$.
The most important feature of this line for concrete computations
is that it transforms nontrivially under
the $\bmuu_3$ rotational symmetry of $\Delta$, 
as measured by the following proposition. 

\begin{prop} \label{prop:L-rotation} Suppose  
$\eps$ induces a consistent orientation of $\partial \Delta$,
and $f$ is a positively oriented rotation by $\frac{2 \pi}{3}$ 
with respect to the orientation of $\Delta$.
Then $f_*$ acts on $\cL(\Delta,\eps)$ as multiplication
by $\exp(2 \pi \I / 3)$. 
\end{prop}

\begin{proof} Fix stratified abelianization data $\cA = (P,Q,\mu,\theta)$
and a spin structure $\sigma$ over $\Delta$. Lift the action of $f$
to $\cA$ and $\sigma$, 
in such a way that $f^3 = 1$.
(This is possible, since each of $\cA$ and $\sigma$ is unique
up to isomorphism and has 
only a single nontrivial automorphism $\rho$; 
an arbitrary lift of $f$ will have
either $f^3 = 1$ or $f^3 = \rho$, and in the latter
case we replace the lift by $f \circ \rho$.)
Also lift $f$ to the $\bmuu_2$-bundle $\tDelta^{\eps,4} \to \tDelta$, again
in such a way that $f^3 = 1$.
Combining this lift
with the actions of $f$ on $Q$ and $\sigma$
gives actions of $f$ on $Q^{\eps}_\tw$ and $\sigma^\eps_\tw$.

We consider the \Vlines associated to the
boundary,
$\scrF_G(\partial \Delta;\partial P)$
and $\scrS_{\C^\times}(\partial \tDelta;\partial Q_\tw^{\eps};\sigma^\eps_\tw)$. $f$ gives actions of $\bmuu_3$ on both \Vlines,
the equivalence $\chi(\partial \Delta; \partial \cA; \sigma; \tDelta^{\eps,\infty})$
is $\bmu3$-equivariant, and $\scrF_G(\Delta;P)$ and $\scrS_{\C^\times}(\tDelta;Q_\tw^{\eps};\sigma^\eps_\tw)$ are 
$\bmuu_3$-invariant objects.
The line 
\begin{equation}
\cL(\Delta,\eps) = \Hom(\scrS_{\C^\times}(\tDelta;Q_\tw^{\eps};\sigma^\eps_\tw), \chi(\partial \Delta; \partial \cA; \sigma; \tDelta^{\eps,\infty})(\scrF_G(\Delta;P)))
\end{equation}
is thus acted on by $\bmuu_3$,
and we want to compute this action.

We will use an explicit picture of the
$\bmuu_3$-equivariant equivalence $\chi(\partial \Delta; \partial \cA; \sigma; \tDelta^{\eps,\infty})$, obtained
by chopping $\partial \Delta$ into three segments, ending at 
the midpoints of the three edges.
First, we fix an $f$-invariant trivialization of the
restriction of $\cA$ to the midpoints. (There is a $T$-torsor 
worth of freedom in this choice, which we will fix below.)
Given this trivialization, 
we may factorize each of our \Vlines 
as a tensor product
of three \Vlines associated to the three segments, and factorize the 
equivalence $\chi(\partial \Delta; \partial \cA; \sigma; \tDelta^{\eps,\infty})$
into a tensor product of three equivalences: for each segment $S$, we have
\begin{equation}
\chi(S; \partial \cA \vert_S; \sigma; \tDelta^{\eps,\infty} \vert_S): 
\scrF_G(S;P \vert_S) \to
\scrS_{\C^\times}(\pi^{-1}(S);Q_\tw^{\eps} \vert_{\pi^{-1}(S)}; \sigma^\eps_\tw) \, .
\end{equation}

Now comes the crucial technical step:
for the purposes
of our computation 
we may replace the equivalences $\chi(S; \partial \cA \vert_S; \sigma; \tDelta^{\eps,\infty} \vert_S)$ by any
other equivalences $\xi_S$ 
between the same \Vlines, compatible
with the $f$-action. Indeed, any two equivalences
differ by tensorization with a line $L_S$, and
the $f$-equivariance identifies the three lines $L_S$
with a single line $L$, so the effect of changing from
$\chi(S; \partial \cA \vert_S; \sigma; \tDelta^{\eps,\infty} \vert_S)$ to $\xi_S$ would be to replace
$\cL(\Delta,\eps)$ by $\cL(\Delta,\eps) \otimes L^3$;
the $\bmuu_3$-action on $L^3$ induced
by cyclic permutation of the factors is 
trivial, so the $\bmuu_3$-action we want to compute
is insensitive to this replacement.

We construct a convenient $\xi_S$ as follows.
We extend the trivializations of $P$ and $Q^\eps_\tw$ from the
midpoints to sections $s_P$ and $s_Q$ of $P\vert_{\partial \Delta}$ and $Q^{\eps}_\tw \vert_{\partial \tDelta}$ respectively,
in an $f$-invariant way. On each segment $S$
this gives trivializations of our two \Vlines, 
and we choose $\xi_S$
to intertwine these trivializations.
Tensoring the $\xi_S$ we get
\begin{equation}
  \xi: \scrF_G(\partial \Delta; \partial P) \to \scrS_{\C^\times}(\partial \tDelta; \partial Q^\eps_\tw; \sigma^\eps_\tw) \, .
\end{equation}

We may choose the trivialization of $\cA$ 
at the midpoints in such a way that the parallel
transport of $Q^{\eps}_\tw$ along any of the $6$ preimages
of segments is given by $1 \in \C^\times$.
(Indeed, for an arbitrary $f$-invariant trivialization,
the parallel transport of $Q$ along each segment is given
by some fixed element $h \in H \setminus T$; changing the
trivialization by $t \in T$ at each vertex conjugates this 
transport by $t$, and by so doing we can set the off-diagonal
entries to $\pm 1$ as needed.)
From now on we fix such a choice. Having done so,
we can choose $s_Q$ to be covariantly constant.

However, we cannot choose $s_P$ to be covariantly constant. Indeed,
the parallel transport of $P$ along an edge 
is given, relative to the trivializations at the
midpoints, by an element $g = hb \in G$, 
where $h \in H \setminus T$ and
$b \in B$; in particular, $g \neq 1$.
The $f$-invariance implies that $g$ is 
independent of the choice of edge. 
Moreover, $g^3 = 1$, since the 
holonomy of $P$ around $\partial \Delta$ is trivial.
We choose $s_P$ as follows. Let $t$ be a covariantly constant section
of $\partial P \to \partial \Delta$
(necessarily not $f$-invariant). Then choose $s_P = \phi t$,
where $\phi: \partial \Delta \to G$ obeys
\begin{equation}
\phi(f(y)) = g \, \phi(y) \, .
\end{equation}

Now, to compute the action of $f$ on $\cL(\Delta,\eps)$ we consider the mapping torus of $f$,
\begin{equation}
 \Delta_f = (\Delta \times \R) \, / \, [(y,x) \sim (f(y),x+1)] \, ,
\end{equation}
and the mapping torus $\tDelta_f$ of the lift of $f$ to $\tDelta$.
The $f$-equivariant
flat bundles $P \to \Delta$ and $Q^\eps_\tw \to \tDelta$ induce
flat bundles
$P_f \to \Delta_f$ and
$Q^\eps_{\tw,f} \to \tDelta_f$
respectively.
Moreover, the $f$-equivariant equivalence $\xi$ of
\Vlines
induces an isomorphism of lines,
$\xi_f: \scrF_G(\partial \Delta_f;\partial P_f) \to \scrS_{\C^\times}(\partial \tDelta_f;\partial Q_{\tw,f}^{\eps};\sigma^\eps_\tw)$.
The $\bmuu_3$-action on $\cL(\Delta,\eps)$ is multiplication by
\begin{equation}
\frac{\xi_f(\scrF_G(\Delta_f;P_f))} {\scrS_{\C^\times}(\tDelta_f;Q^\eps_{\tw,f};\sigma^\eps_\tw)} \, .
\end{equation}

To compute this, first note that, being $f$-invariant, $s_P$ and $s_Q$ induce
sections $s_{P_f}$ and $s_{Q_f}$ of $\partial P_f \to \partial \Delta_f$
and $\partial Q^\eps_{\tw,f} \to \partial \widetilde \Delta_f$ respectively.
The resulting trivializations of the boundary lines have,
essentially by definition of $\xi$,
\begin{equation}
\xi_f (\tau_{s_{P_f}}) = \tau_{s_{Q_f}} \, .
\end{equation}
Our task now is to compute the numerator and denominator
relative to these trivializations.

To compute $\scrS_{\C^\times}(\tDelta_f ; Q^\eps_{\tw,f} ; \sigma^\eps_\tw)$ we note that
$s_{Q_f}$ is covariantly constant on $\partial \widetilde \Delta_f$, 
and it can be extended to a covariantly constant section
over the full $\tDelta_f$; thus the $\C^\times$ Chern-Simons form vanishes,
and
$\scrS_{\C^\times}(\tDelta_f ; Q^\eps_{\tw,f} ;\sigma^\eps_\tw) = \tau_{s_{Q_f}}$.

To compute $\scrF_G(\Delta_f ; P_f)$ is more interesting.
We choose an arbitrary extension $s$ of $s_{P_f}$ to the solid
torus $\Delta_f$. Let $A$ denote the connection form in $P_f$
relative to the section $s$; then using \eqref{eq:33}
\begin{equation} \label{eq:computing-cube-root-explicit}
  \scrF_G(\Delta_f;P_f) = \tau_{s_{P_f}} \exp \left[ \frac{1}{4\pi \sqmo} \int_{\Delta_f} \trace\left( A\wedge dA + \frac 23\,A\wedge
     A\wedge A \right) \right] \, .
\end{equation}
To compute explicitly,
we pull back to a triple cover $p: {\Delta \times S^1} \to \Delta_f$,
\begin{equation}
 {\Delta \times S^1} = (\Delta \times \R) \, / \, [(y,x) \sim (y,x+3)] \, , \quad p ( [(y,x)] ) = [(y,x)] \, .
\end{equation}
The covariantly
constant section $t$ of $\partial P \to \partial \Delta$ 
induces a covariantly constant section of 
$\partial (p^* P) \to \partial {\Delta \times S^1}$,\footnote{Since $t$ is not $f$-invariant, it would not induce a section of $P_f \to \partial \Delta_f$; this is the reason why we had to pull back to the triple cover.}
which we can further extend to a covariantly constant section $\hat t$
of $p^* P \to {\Delta \times S^1}$.
Then $p^* s = \phi \hat t$ for some $\phi: {\Delta \times S^1} \to G$,
and the invariance of $p^* s$ under the $\bmuu_3$ deck group
gives 
\begin{equation} \label{eq:phi-equivariant}
  \phi(f(y),x+1) = g \, \phi(y,x) \, .
\end{equation}
The connection form in $p^* P$ relative to $p^* s$ is 
$A = \phi^{-1} \de \phi$, and this form is pulled back from
$\Delta_f$;
thus the integral over $\Delta_f$ in \eqref{eq:computing-cube-root-explicit} can be rewritten as one-third of a more explicit 
integral over $\Delta \times S^1$,
\begin{equation} \label{eq:wzw-integral}
  \scrF_G(\Delta_f;P_f) = \tau_{s_{P_f}} \exp \left[ \frac13 \frac{1}{12 \pi \I} \int_{{\Delta \times S^1}} \trace(\phi^{-1} \de \phi)^3 \right] .
\end{equation}
Moreover, 
$s \vert_{\partial {\Delta \times S^1}}$ is covariantly constant along
translation in the $x$-direction holding $y \in \partial \Delta$ fixed, 
while $t$ is covariantly constant in every direction.
Thus $\phi \vert_{\partial {\Delta \times S^1}}$ is constant in the $x$-direction. 
Attaching a solid torus $S^1 \times D^2$ to ${\Delta \times S^1}$ along this direction we obtain
a closed 3-manifold $M \simeq S^3$.
The map $\phi$ naturally extends to the added $S^1 \times D^2$,
by choosing it to be constant along the $D^2$ factor.
Then $(\phi^{-1} \de \phi)^3 = 0$ there, and thus we can replace
the domain of integration in \eqref{eq:wzw-integral} by $M$.
This integral gives $\exp(\frac{2 \pi \I}{3} k)$, where $k$ is the degree of the map $\phi: M \to G$ (by which we mean the degree of the retraction
of $\phi$ from $G$ to $\SU(2)$, when we equip $\SU(2)$ with the
orientation for which $- \trace (h^{-1} \de h)^3$ is a positive 3-form.)
To compute this degree, we will use only the fact
that $\phi$ commutes with certain
$\bmuu_3$-actions on $M$ and $\SL(2,\C)$, as follows.

First, we can identify $M$ with 
$\{|\alpha|^2 + |\beta|^2 = 1\} = S^3 \subset \C^2$ as follows.
On the torus $\partial \Delta \times S^1$, we fix coordinates
$\alpha = \frac{1}{\sqrt{2}} \e^{\I \theta}$, where $\theta$ parameterizes
$\partial \Delta$ (positively with respect to the boundary orientation), and
$\beta = \frac{1}{\sqrt{2}} \e^{2 \pi \I x / 3}$.
These coordinates naturally extend to the two solid tori 
$\Delta \times S^1$ and $S^1 \times D^2$, identifying them
respectively as the loci $|\alpha| < \frac12$
and $|\beta| < \frac12$ in $S^3$.
The orientation
of $M$ matches the standard orientation of $S^3$.
The $\bmuu_3$-action $(y,x) \mapsto (f(y),x+1)$ on
$\Delta \times S^1$ becomes in these coordinates $(\alpha, \beta) \mapsto (\e^{2 \pi \I/3} \alpha, \e^{2 \pi \I/3} \beta)$ (and thus extends to a fixed-point-free
action on the whole $S^3$).

Next, parameterize $\SU(2)$ by $h = \begin{pmatrix} \alpha & \beta \\ -\bar\beta & -\bar\alpha \end{pmatrix}$. This gives $\SU(2) \simeq S^3 \subset \C^2$.
Computing $- \trace (h^{-1} \de h)^3$ in this parameterization 
we see that it is positive 
for the standard orientation on $S^3$. 
By composing $\phi$ with an inner automorphism of $\SL(2,\C)$ we may assume 
$g = \mathrm{diag}(\e^{2 \pi \I/3}, \e^{-2 \pi \I/3})$.
The $\bmuu_3$-action $h \mapsto g h$ then acts by $(\alpha,\beta) \mapsto (\e^{2 \pi \I/3} \alpha, \e^{2 \pi \I/3} \beta)$.

We have shown how to identify both $M$ and $\SU(2)$ with $S^3$,
in such a way that the $\bmuu_3$-actions and orientations agree with the standard
ones for $S^3$.
Using \eqref{eq:phi-equivariant},
$\phi$ intertwines the
$\bmuu_3$-action on $M$ with the $\bmuu_3$-action
on $\SU(2)$.
Then using Lemma \ref{lem:s3-degree} below completes the proof.
\end{proof}

\begin{lemma} \label{lem:s3-degree} Any continuous 
map $\phi: S^3 \to S^3$ commuting with
the standard $\bmuu_3$-action has degree equal to $1$ mod $3$.
\end{lemma}

\begin{proof}
Such a $\phi$ descends to a map
$\bar\phi: S^3 / \bmuu_3 \to S^3 / \bmuu_3$ which lifts to
the $\bmuu_3$-bundle $S^3 \to S^3 / \bmuu_3$;
this bundle has a nonzero characteristic class lying in 
$H^3(S^3 / \bmuu_3; \bmuu_3) \simeq \bmuu_3$ (because the
inclusion $S^3 / \bmuu_3 \to B\bmuu_3 = S^\infty / \bmuu_3$ 
induces an isomorphism on $H^3(\cdot; \bmuu_3)$), 
and thus $\bar\phi^*$ must act trivially on $H^3(S^3 / \bmuu_3; \bmuu_3)$,
i.e. the degree of $\bar\phi$ is 1 mod 3.
Since the degree of $\phi$ agrees with that of $\bar\phi$, this finishes
the proof.
\end{proof}

To finish this section we remark on a diagrammatic perspective on $\cL(\Delta,\eps)$ which will be useful in some
of the arguments to follow. Here we suppress most of the background fields to reduce clutter.
We regard $\Delta$ as a morphism in the bordism category
\begin{equation}
  \emptyset \xrightarrow{\Delta} \partial \Delta \, .
\end{equation}
Applying $\scrF_G$ and $\scrS_{\C^\times}$ to this diagram, and including the map $\chi(\partial \Delta, \eps)$, we get
\begin{equation} \label{eq:L-diagram}
\begin{tikzcd}
  && {\scrF_G(\partial \Delta)} \\
  {{\mathrm {Line}}} \\
  && {\scrS_{{\mathbb C}^\times}(\partial \Delta,\eps)}
  \arrow["{\scrF_G(\Delta)}", from=2-1, to=1-3]
  \arrow["{\scrS_{{\mathbb C}^\times}(\Delta,\eps)}"', from=2-1, to=3-3]
  \arrow["{\chi(\partial \Delta,\eps)}"', from=1-3, to=3-3]
\end{tikzcd}
\end{equation}
Here, and in various diagrammatic arguments to follow, 
we freely identify morphisms $\Line \to \cC$ with objects of $\cC$.
Then the composition 
\begin{equation} \label{eq:L-as-composition}
 \cL(\Delta,\eps) = \scrS_{\C^\times}(\Delta,\eps)^{-1} \circ \chi(\partial \Delta, \eps) \circ \scrF_G(\Delta) \, .
\end{equation}

\subsection{Abelianization of Chern-Simons over triangulated surfaces} \label{sec:abelianization-cs-triangulated-surface}

Now suppose given an oriented surface $Y$ with a semi-ideal triangulation $\scrT$, edge-orientations $\eps$,
and a spin structure $\sigma$. 
Also fix boundary-reduced stratified abelianization data $\cA = (P, Q, \mu, \theta)$ over $(Y,\cN^\scrT)$.
As we have discussed in \S\ref{subsubsec:5.5.5}, because $P$ is boundary-reduced, the \Vline $\scrF_G(\partial Y; \partial P)$ is canonically
trivial, and thus
$\scrF_G(Y;P)$ is a line.
The \Vline $\scrS_{\C^\times}(\partial \tY; \partial Q^\eps_\tw;\sigma^\eps_\tw)$
is also canonically trivial,
since $Q^\eps_\tw$ has trivial holonomy around the boundary components; thus
$\scrS_{\C^\times}(\tY;Q^\eps_\tw;\sigma^\eps_\tw)$ is also a line.

Define the difference line
\begin{equation} \label{eq:def-leps}
  \cL(Y,\eps) = \bigotimes_{\Delta \in \faces(\scrT)} \cL(\Delta,\eps \vert_{\Delta}) \, .
\end{equation}


\begin{constr} \label{constr:abelianization-surface} There is a canonical 
isomorphism of lines
\begin{equation} \label{eq:abelianization-for-lines}
  \chi^\eps_Y: \scrF_G(Y;P) \to \scrS_{\C^\times}(\tY;Q^\eps_\tw;\sigma^\eps_\tw) \otimes \cL(Y,\eps) \, .
\end{equation}
\end{constr}

\begin{proof}
First for simplicity suppose that there are no ideal vertices, so $Y$ is a closed triangulated surface.

For each $\Delta \in \faces(\scrT)$ we consider the dilation
$\mu: \Delta \to \Delta$ which rescales distance from the barycenter 
by $\frac12$. Then we have a decomposition
 $Y = Y_\rmout \cup Y_\rmin$, where $Y_\rmin$ is
the union of the rescaled triangles $\mu(\Delta)$.
\insfig{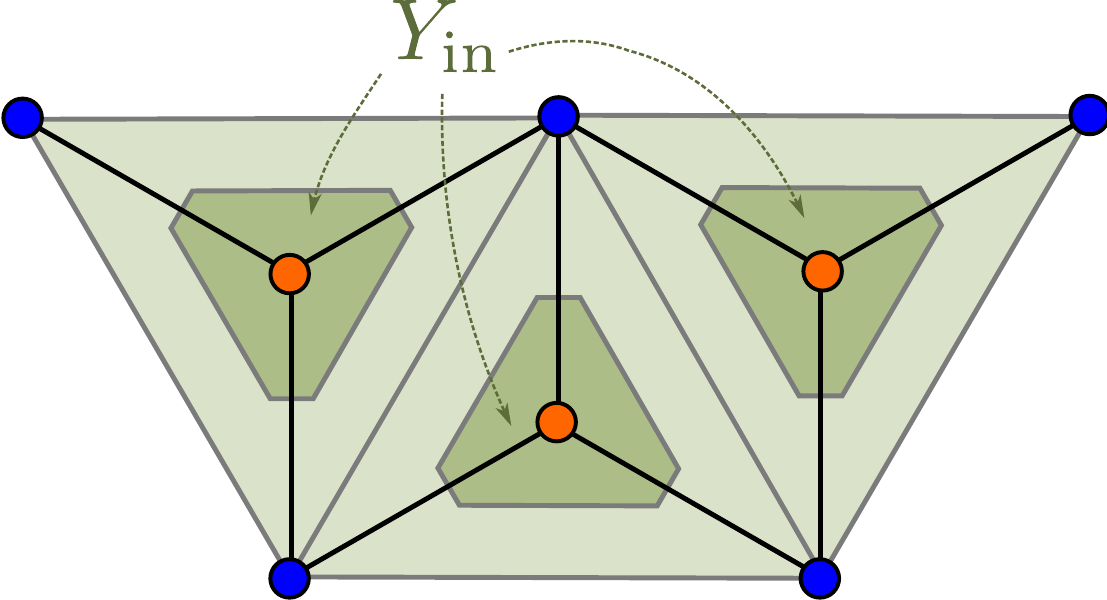}{0.4}{A portion of the triangulated surface $Y$
with its SN-stratification, and the decomposition $Y = Y_\rmout \cup Y_\rmin$.}

Let $R = \partial Y_{\rmin}$.
To condense the notation we will just write the manifolds, suppressing all the extra background fields,
including the dependence on $\eps$ (since $\eps$ is held fixed throughout this proof).
Then we have a diagram in the bordism category,
\begin{equation}
  \emptyset \xrightarrow{Y_\rmin} R \xrightarrow{Y_\rmout} \emptyset
\end{equation}
Applying $\scrF_G$ and $\scrS_{\C^\times}$ to this diagram, and inserting the maps $\chi(R)$ and $\chi(Y_\rmout)$
provided by Theorem~\ref{thm:strat-abelianization}, we get the diagram below:
\begin{equation}
\begin{tikzcd}
  && {{\scrF}_G(R)} \\
  {{\mathrm {Line}}} &&&& {{\mathrm {Line}}} \\
  && {{\scrS}_{{\mathbb C}^\times}(R)}
  \arrow["{{\scrF}_G(Y_{\mathrm{in}})}", from=2-1, to=1-3]
  \arrow["{{\scrS}_{{\mathbb C}^\times}(Y_{{\mathrm{in}}})}"', from=2-1, to=3-3]
  \arrow["{\chi(R)}"', from=1-3, to=3-3]
  \arrow["{{\scrS}_{{\mathbb C}^\times}(Y_{\mathrm{out}})}"', from=3-3, to=2-5]
  \arrow[""{name=0, anchor=center, inner sep=0}, "{{\scrF}_G(Y_{{\mathrm{out}}})}", from=1-3, to=2-5]
  \arrow["{\chi(Y_{{\mathrm{out}}})}"{description}, shorten <=8pt, shorten >=8pt, Rightarrow, from=0, to=3-3]
\end{tikzcd}
\end{equation}
Now whiskering $\chi(Y_\rmout)$ by $\scrF_G(Y_\rmin)$ we get a 2-morphism
\begin{equation}
\begin{tikzcd}
  {{\mathrm {Line}}} &&&& {{\mathrm {Line}}}
  \arrow[""{name=0, anchor=center, inner sep=0}, "{\scrS_{\mathbb C^\times}(Y_{\mathrm{out}}) \circ \chi(R) \circ \scrF_G(Y_{\mathrm{in}})}"', shift left=1, curve={height=18pt}, from=1-1, to=1-5]
  \arrow[""{name=1, anchor=center, inner sep=0}, "{\scrF_G(Y)}", curve={height=-18pt}, from=1-1, to=1-5]
  \arrow[shorten <=5pt, shorten >=5pt, Rightarrow, from=1, to=0]
\end{tikzcd}
\end{equation}
Finally, defining 
\begin{equation} \label{eq:def-L}
\cL = \scrS_{\C^\times}(Y_\rmin)^{-1} \circ \chi(R) \circ \scrF_G(Y_\rmin),  
\end{equation}
we have $\chi(R) \circ \scrF_G(Y_\rmin) = \scrS_{\C^\times}(Y_\rmin) \circ \cL(Y,\eps)$,
so we can rewrite the diagram as
\begin{equation} \label{eq:easy-diagram-line}
\begin{tikzcd}
  {{\mathrm {Line}}} &&&& {{\mathrm {Line}}}
  \arrow[""{name=0, anchor=center, inner sep=0}, "{\scrS_{\C^\times}(Y) \circ \cL}"', shift left=1, curve={height=18pt}, from=1-1, to=1-5]
  \arrow[""{name=1, anchor=center, inner sep=0}, "{\scrF_G(Y)}", curve={height=-18pt}, from=1-1, to=1-5]
  \arrow[shorten <=5pt, shorten >=5pt, Rightarrow, from=1, to=0]
\end{tikzcd}
\end{equation}
Composition of maps $\Line \to \Line$ is tensor product, so this
is a map of lines
\begin{equation}
  \scrF_G(Y) \to \scrS_{\C^\times}(Y) \otimes \cL
\end{equation}
as desired; what remains is to identify $\cL$ with the $\cL(Y,\eps)$ we defined in \eqref{eq:def-leps}.
This follows directly from 
\eqref{eq:L-as-composition} and
the decomposition of $Y_\rmin$ into
the disjoint union of triangles.

So far we discussed only the case where there are no ideal vertices.
In the general case one has to draw slightly
more complicated diagrams:
\begin{equation}
  \emptyset \xrightarrow{Y_\rmin} R \xrightarrow{Y_\rmout} \partial Y
\end{equation}
\begin{equation} \label{eq:harder-diagram-line}
\begin{tikzcd}
  && {{\scrF}_G(R)} && {{\scrF_G(\partial Y)}} \\
  {{\mathrm {Line}}} \\
  && {{\scrS}_{{\mathbb C}^\times}(R)} && {\scrS_{\mathbb C^\times}(\partial Y)}
  \arrow["{{\scrF}_G(Y_{\mathrm{in}})}", from=2-1, to=1-3]
  \arrow["{{\scrS}_{{\mathbb C}^\times}(Y_{{\mathrm{in}}})}"', from=2-1, to=3-3]
  \arrow["{\chi(R)}"', from=1-3, to=3-3]
  \arrow["{{\scrF}_G(Y_{{\mathrm{out}}})}", from=1-3, to=1-5]
  \arrow["{{\scrS}_{{\mathbb C}^\times}(Y_{\mathrm{out}})}"', from=3-3, to=3-5]
  \arrow["{\chi(Y_{{\mathrm{out}}})}"{description}, shorten <=14pt, shorten >=14pt, Rightarrow, from=1-5, to=3-3]
  \arrow["{\chi(\partial Y)}"', from=1-5, to=3-5]
\end{tikzcd}
\end{equation}
However, because $\cA$ is assumed boundary-reduced, we have trivializations
of $\scrF_G(\partial Y)$ and $\scrS_{\C^\times}(\partial Y)$, which are intertwined
by $\chi(\partial Y)$; using these trivializations the diagram \eqref{eq:harder-diagram-line} reduces to \eqref{eq:easy-diagram-line}, and then we can proceed just as above.
This completes the proof.
\end{proof}


We may also consider a family of boundary-reduced stratified abelianization data
$\cA = (P_s,Q_s,\mu_s,\theta_s)$ over a fixed $(Y,\cN)$, 
varying with a parameter $s \in S$.
All of our constructions can be applied to such a family:
then we obtain two invariants $\scrF_G(Y;P_s)$ and $\scrS_{\C^\times}(\tY;Q^\eps_{\tw,s};\sigma^\eps_\tw)$
both of which vary over $S$, and an isomorphism relating them.
In particular, if $Y$ is a semi-ideally triangulated surface as above,
then both $\scrF_G(Y;P_s)$ and $\scrS_{\C^\times}(\tY;Q^\eps_{\tw,s};\sigma^\eps_\tw)$ 
are line bundles over $S$ with connection,
and $\chi^\eps_Y: \scrF_G(Y;P_s) \to \scrS_{\C^\times}(\tY;Q^\eps_{\tw,s};\sigma^\eps_\tw) \otimes \cL(Y,\eps)$
is an isomorphism of line bundles with connection.

\subsection{Reversing an edge orientation on a triangle} \label{sec:edge-reversal-triangle}

Now suppose $(\Delta,\eps,\sigma,\cA)$ is an oriented 
triangle with edge-orientations,
spin structure, and boundary-reduced stratified abelianization data.
We first consider the restriction of all the data to a 
single oriented edge $E$.
Let $\eps$ be the given orientation of $E$ and $\eps'$ the opposite orientation.
Then in the bordism category we have the diagram
\begin{equation}
  \emptyset \xrightarrow{E} \partial E
\end{equation}
Once again we suppress some background fields in the notation: we just keep the manifolds
and (where necessary) the edge-orientations. We will
also use freely the fact that the $\Z$-bundles $\tDelta^{\eps,\infty}$
and $\tDelta^{\infty,\eps'}$ are
canonically trivial over $\partial E$.

Then we have a diagram in the 2-\Vline $\scrS_{\C^\times}(\partial E)$:
\[\begin{tikzcd}
  && {\scrS_{\mathbb C^\times}(E,\epsilon)} \\
  {\chi(\partial E) \circ \scrF_G(E)} \\
  && {\scrS_{\mathbb C^\times}(E,\epsilon')}
  \arrow["{\chi(E,\epsilon)}", from=2-1, to=1-3]
  \arrow["{\zeta(E, \varpi^E, \epsilon)}", from=1-3, to=3-3]
  \arrow["{\chi(E,\epsilon')}"', from=2-1, to=3-3]
\end{tikzcd}\]
We define
\begin{equation} \label{eq:def-xi}
  \Xi(E) = \chi(E,\eps')^{-1} \circ \zeta(E, \varpi^E, \eps) \circ \chi(E,\eps)
\end{equation}
which is an automorphism of the object $\chi(\partial E) \circ \scrF_G(E) \in \scrS_{\C^\times}(\partial E)$, or equivalently a line.

Now suppose $(\Delta,\eps,\sigma,\cA)$ is an oriented 
triangle with edge-orientations,
spin structure, and stratified abelianization data. Also suppose $E$ is 
an edge such that $\eps_E$ agrees with the boundary orientation,
and $\eps'$ is obtained
from $\eps$ by reversing the orientation on $E$.
Let $\tilde{E} = -(\overline{\partial \Delta \setminus E})$, so now we have
\begin{equation}
\begin{tikzcd}
  \emptyset && {\partial E}
  \arrow[""{name=0, anchor=center, inner sep=0}, "E"', curve={height=18pt}, from=1-1, to=1-3]
  \arrow[""{name=1, anchor=center, inner sep=0}, "{\tilde{E}}", curve={height=-18pt}, from=1-1, to=1-3]
  \arrow["\Delta"', shorten <=5pt, shorten >=5pt, Rightarrow, from=1, to=0]
\end{tikzcd}
\end{equation}
Then $\zeta(\cdot,\varpi^E,\eps)$ provides an isomorphism of 2-morphisms
\begin{equation} \label{eq:zeta-diagram-iso}
\begin{tikzcd}
  {{\mathbb V}{\mathrm{Line}}} && {\scrS_{\mathbb C^\times}(\partial E)} & {} & {} & {{\mathbb V}{\mathrm{Line}}} && {\scrS_{\mathbb C^\times}(\partial E)}
  \arrow[""{name=0, anchor=center, inner sep=0}, "{\scrS_{\mathbb C^\times}(E,\epsilon)}"', curve={height=18pt}, from=1-1, to=1-3]
  \arrow[""{name=1, anchor=center, inner sep=0}, "{\scrS_{\mathbb C^\times}(\tilde{E})}", curve={height=-18pt}, from=1-1, to=1-3]
  \arrow[from=1-4, to=1-5]
  \arrow[""{name=2, anchor=center, inner sep=0}, "{\scrS_{\mathbb C^\times}(E,\epsilon')}"', curve={height=18pt}, from=1-6, to=1-8]
  \arrow[""{name=3, anchor=center, inner sep=0}, "{\scrS_{\mathbb C^\times}(\tilde{E})}", curve={height=-18pt}, from=1-6, to=1-8]
  \arrow["{\scrS_{\mathbb C^\times}(\Delta,\epsilon)}"', shift left=5, shorten <=5pt, shorten >=5pt, Rightarrow, from=1, to=0]
  \arrow["{\scrS_{\mathbb C^\times}(\Delta,\epsilon')}"', shift left=5, shorten <=5pt, shorten >=5pt, Rightarrow, from=3, to=2]
\end{tikzcd}
\end{equation}
while $\chi(\cdot,\eps)$ gives a similar isomorphism of diagrams, but without the inner arrow,
\begin{equation}
\begin{tikzcd}
  {{\mathbb V}{\mathrm{Line}}} && {\scrF_{G}(\partial E)} & {} & {} & {{\mathbb V}{\mathrm{Line}}} && {\scrS_{\mathbb C^\times}(\partial E)}
  \arrow["{\scrF_{G}(E)}"', curve={height=18pt}, from=1-1, to=1-3]
  \arrow["{\scrF_G(\tilde{E})}", curve={height=-18pt}, from=1-1, to=1-3]
  \arrow[from=1-4, to=1-5]
  \arrow["{\scrS_{\mathbb C^\times}(E,\epsilon)}"', curve={height=18pt}, from=1-6, to=1-8]
  \arrow["{\scrS_{\mathbb C^\times}(\tilde{E})}", curve={height=-18pt}, from=1-6, to=1-8]
\end{tikzcd}
\end{equation}
and $\chi(\cdot,\eps')$ gives the same isomorphism of diagrams with $\eps$ replaced by $\eps'$.
Combining these, we can whisker the isomorphism \eqref{eq:zeta-diagram-iso} into
an isomorphism between the lines we previously defined in \eqref{eq:L-as-composition}, \eqref{eq:def-xi},
\begin{equation}
 \rho^{\eps',\eps}(\Delta): \cL(\Delta,\eps') \to \cL(\Delta,\eps) \otimes \Xi(E)^{-1} \, .
\end{equation}
This map describes the effect of reversing the orientation on one edge of $\Delta$.

Now we want to consider reversing orientation on two edges of $\Delta$.
We introduce a bit of notation that will be convenient below:
\begin{definition} If $E_1$ and $E_2$ are edges of an oriented triangle $\Delta$,
\begin{equation}
  \IP{E_1,E_2} = \begin{cases} +1 & \text{if $E_1$ is ahead of $E_2$ in the boundary orientation}, \\
                           -1 & \text{if $E_1$ is behind $E_2$ in the boundary orientation}, \\
                           0 & \text{if $E_1 = E_2$}. \end{cases}
\end{equation}
\end{definition}
Suppose $\eps_{E_1}$, $\eps_{E_2}$ both agree with
the boundary orientation,
let $\eps'_i$ be obtained from
$\eps$ by reversing $\eps_{E_i}$, and let $\eps''_{12}$ be
obtained from $\eps$ by reversing both $\eps_{E_1}$ and $\eps_{E_2}$.
\insfig{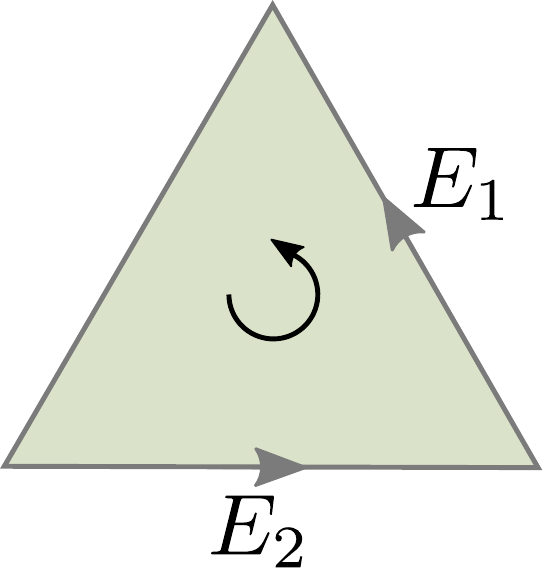}{0.4}{An oriented triangle with two marked edges, with $\IP{E_1,E_2} = 1$, and
edge orientations agreeing with the boundary orientation. The third edge
orientation is arbitrary.}


We want to compare the two 2-morphisms appearing in the diagram:
\[\begin{tikzcd}
  {\scrS_{\mathbb C^\times}(\partial \Delta, \epsilon)} &&&& {\scrS_{\mathbb C^\times}(\partial \Delta, \epsilon''_{12})} \\
  \\
  && {{\mathrm{Line}}}
  \arrow["{\zeta(E_1, \varpi_{E_1}) \circ_h \zeta(E_2, \varpi_{E_2}) \circ_h {\mathrm{id}}}", from=1-1, to=1-5]
  \arrow[""{name=0, anchor=center, inner sep=0}, "{\scrS_{\mathbb C^\times}(\Delta, \epsilon)}", from=3-3, to=1-1]
  \arrow[""{name=1, anchor=center, inner sep=0}, "{\scrS_{\mathbb C^\times}(\Delta, \epsilon''_{12})}"', from=3-3, to=1-5]
  \arrow["{\zeta(\Delta,\varpi_{E_2}) \circ \zeta(\Delta,\varpi_{E_1})}"', shift left=3, shorten <=16pt, shorten >=16pt, Rightarrow, from=0, to=1]
  \arrow["{\zeta(\Delta,\varpi_{E_1}) \circ \zeta(\Delta,\varpi_{E_2})}", shift left=5, shorten <=16pt, shorten >=16pt, Rightarrow, from=0, to=1]
\end{tikzcd}\]
\begin{lem} \label{lem:zeta-commutation}
In this situation, 
\begin{equation}
\frac{\zeta(\tDelta;Q^{\eps'_1}_\tw;\sigma^{\eps'_1}_\tw;\varpi_{E_2}) \circ \zeta(\tDelta;Q^\eps_\tw;\sigma^\eps_\tw;\varpi_{E_1})}{\zeta(\tDelta;Q^{\eps'_2}_\tw; \sigma^{\eps'_2}_\tw; \varpi_{E_1}) \circ \zeta(\tDelta;Q^\eps_\tw;\sigma^\eps_\tw;\varpi_{E_2})} = \exp \left( - \frac{2 \pi \I}{4} \IP{E_1,E_2} \right) \, .
\end{equation}
\end{lem}

\begin{proof}
We use \eqref{eq:191}. For this purpose we choose 
a section $s$ of $Q^\eps_\tw$ and maps $p_1, p_2: \tDelta \to S^1$ representing
$\varpi_{E_1}$, $\varpi_{E_2}$, with $p_i$ constant on $\partial \tDelta$
except for $\widetilde E_i$. Then we have another section
$s''_{12} = e^{\pi \I (p_1 + p_2)} s$ of $Q^{\eps''_{12}}_\tw$,
and in terms of these sections
\begin{equation}
  \zeta(\tDelta;Q^{\eps'_i}_\tw;\sigma^{\eps'_i}_\tw;\varpi_{E_j}) \circ \zeta(\tDelta;Q^\eps_\tw;\sigma^\eps_\tw;\varpi_{E_i}) \tau_{s} = \exp \left( -\frac14 \int_{\tDelta} \left(s^* \alpha + \pi \I \, d p_i \right) \wedge d p_j  +  s^* \alpha \wedge d p_i \right) \tau_{s''_{12}}
\end{equation}
so that the desired ratio is
\begin{equation}
  \exp \left( - \frac{2 \pi \I}{4} \int_{\tDelta} d p_1 \wedge d p_2 \right) = 
  \exp \left( - \frac{2 \pi \I}{4} \IP{E_1,E_2} \right) \, .
\end{equation}
\end{proof}

\begin{prop} \label{prop:edge-reversal-commutation} 
In this situation,
\begin{equation}
  \rho^{\eps'_2,\eps''_{12}}(\Delta) \circ \rho^{\eps,\eps'_2}(\Delta) = \exp\left(- \frac{2 \pi \I}{4} \IP{E_1,E_2} \right) \rho^{\eps'_1,\eps''_{12}}(\Delta) \circ \rho^{\eps,\eps'_1}(\Delta).
\end{equation}

\end{prop}

\begin{proof}
Since $\rho$ is constructed by whiskering $\zeta$ we get
\begin{equation}
  \frac{\rho^{\eps'_2,\eps''_{12}}(\Delta) \circ \rho^{\eps,\eps'_2}(\Delta)}{\rho^{\eps'_1,\eps''_{12}}(\Delta) \circ \rho^{\eps,\eps'_1}(\Delta)} = \frac{\zeta(\tDelta;Q^{\eps'_2}_\tw; \sigma^{\eps'_2}_\tw; \varpi_{E_1}) \circ \zeta(\tDelta;Q^\eps_\tw;\sigma^\eps_\tw;\varpi_{E_2})}{\zeta(\tDelta;Q^{\eps'_1}_\tw;\sigma^{\eps'_1}_\tw;\varpi_{E_2}) \circ \zeta(\tDelta;Q^\eps_\tw;\sigma^\eps_\tw;\varpi_{E_1})} .
\end{equation}
Then use Lemma~\ref{lem:zeta-commutation}.
\end{proof}

Similarly we can consider reversing the orientation on an edge $E$ 
where the initial orientation $\eps_E$ is
opposite to the boundary orientation. Then the same constructions as above give 
an isomorphism
\begin{equation}
 \rho^{\eps,\eps'}(\Delta): \cL(\Delta,\eps') \to \cL(\Delta,\eps) \otimes \Xi(E) \, .
\end{equation}

\subsection{Reversing an edge orientation on a triangulated surface} \label{sec:edge-reversal}

Now suppose $(Y,\scrT)$ is a semi-ideally triangulated surface,
with boundary-reduced stratified abelianization data $(P,Q,\mu,\theta)$.
Suppose
$\eps$, $\eps'$ are edge-orientations on $\scrT$, differing 
by reversing the orientation on one edge $E$.
We would like to compare the maps $\chi_Y^{\eps}$ and $\chi_Y^{\eps'}$
from Construction~\ref{constr:abelianization-surface}.
Their ratio is an isomorphism
\begin{equation}
  \chi^{\eps}_Y \circ (\chi^{\eps'}_Y)^{-1}: \scrS_{\C^\times}(\tY;Q^{\eps'}_\tw;\sigma^{\eps'}_\tw) \otimes \cL(Y,\eps') \to \scrS_{\C^\times}(\tY;Q^{\eps}_\tw;\sigma^{\eps}_\tw) \otimes \cL(Y,\eps) \, .
\end{equation}
We already have an isomorphism
\begin{equation}
\zeta(\tY;Q^{\eps'}_\tw;\sigma^{\eps'}_\tw;{\varpi^{E}}): \scrS_{\C^\times}(\tY;Q^{\eps'}_\tw;\sigma^{\eps'}_\tw) \to \scrS_{\C^\times}(\tY;Q^{\eps}_\tw;\sigma^{\eps}_\tw)  
\end{equation}
given by Corollary~\ref{thm:72}.
Thus we can write
\begin{equation} \label{eq:lambda-def}
\chi_Y^{\eps} \circ (\chi_Y^{\eps'})^{-1} = \zeta(Y;Q^{\eps'}_\tw;{\varpi^{E}}) \otimes (\lambda^{\eps, \eps'}(Y))^{-1}
\end{equation}
for some
\begin{equation}
\lambda^{\eps,\eps'}(Y): \cL(Y,\eps) \to \cL(Y,\eps').
\end{equation}
This map is determined in terms of the edge-reversal maps for the two triangles $\Delta_1$, $\Delta_2$ abutting $E$,
as follows.
\begin{prop} \label{prop:lambda-factorization} $\lambda^{\eps,\eps'}(Y) = \rho^{\eps,\eps'}({\Delta_1}) \otimes \rho^{\eps,\eps'}({\Delta_2})$.
\end{prop}

\begin{proof}
We decompose $Y$ as indicated in the figure:
\insfig{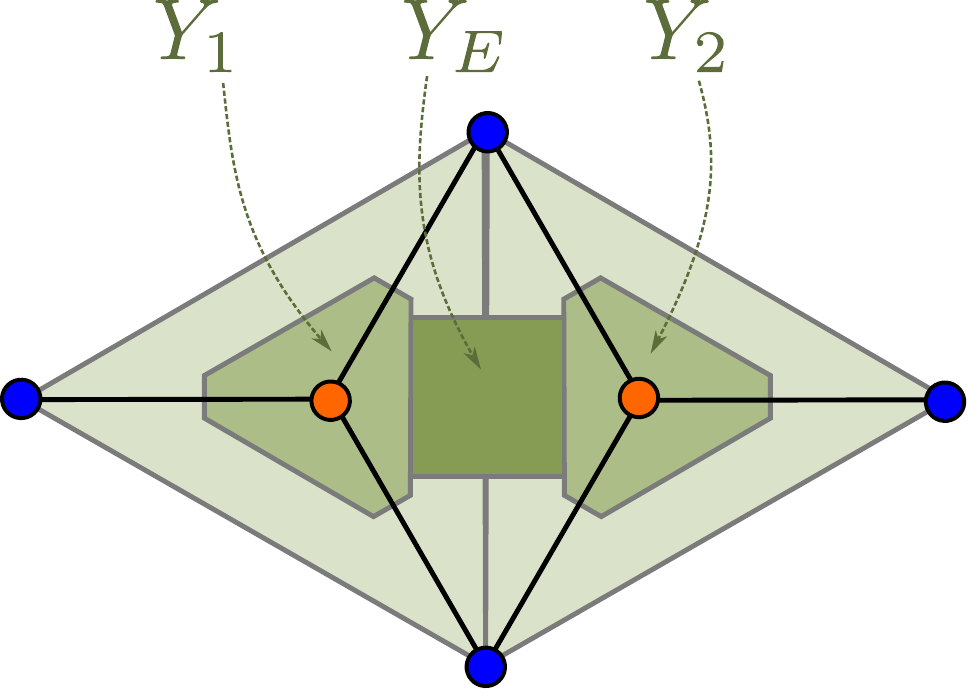}{0.4}{The decomposition of a triangulated surface $Y$ associated to an edge $E$.}
\begin{equation}
  Y = Y_\rmout \cup Y_1 \cup Y_2 \cup Y_E
\end{equation}
Now we want to describe $\zeta(Y;Q^{\eps'}_\tw;{\varpi^{E}})^{-1} \circ \chi_Y^{\eps} \circ (\chi_Y^{\eps'})^{-1}$ relative to this decomposition.
We will content ourselves with a heuristic
description, leaving the full diagrammatics to the reader. 
On $Y_\rmout$ all the background fields associated to $\eps$ and $\eps'$ are canonically isomorphic, and the
background field $\varpi_E$ is canonically trivial, so likewise $\zeta(Y;Q^{\eps'}_\tw;{\varpi^{E}})^{-1} \circ \chi^{\eps} \circ (\chi^{\eps'})^{-1}$ is trivial. 
On $Y_i$, $\zeta(Y;Q^{\eps'}_\tw;{\varpi^{E}})^{-1} \circ \chi^{\eps} \circ (\chi^{\eps'})^{-1}$ is described by the map $(\rho^{\eps,\eps'}(\Delta_i))^{-1}: \cL(\Delta_i, \eps') \otimes \Xi(E_i)^{\pm 1} \to \cL(\Delta_i, \eps)$ where $E_i$ is the edge of $\Delta_i$ corresponding to $E$.
Finally, on $Y_E$, $\zeta(Y;Q^{\eps'}_\tw;{\varpi^{E}})^{-1} \circ \chi^{\eps} \circ (\chi^{\eps'})^{-1}$ gives the isomorphism $\Xi(E_1) \to \Xi(E_2)$ induced by identifying these two edges. Combining these gives the desired result.
\end{proof}

\begin{prop} \label{prop:edge-reversal-double} The composition
\begin{equation}
\lambda^{\eps',\eps}(Y) \circ \lambda^{\eps,\eps'}(Y): \cL(Y,\eps) \to \cL(Y,\eps) 
\end{equation}
acts by multiplication by the $\eps$-sign of the edge $E$, as defined in
Definition \ref{def:quadrilateral-sign}.
\end{prop}

\begin{proof}
First note that from \eqref{eq:lambda-def} it follows 
directly that $\lambda^{\eps',\eps}(Y) \circ \lambda^{\eps,\eps'}(Y)$ acts 
by $\zeta(\tY;Q^{\eps'}_\tw;\sigma^{\eps'}_\tw;{\varpi_E}) \circ \zeta(\tY;Q^\eps_\tw;\sigma^{\eps}_\tw;{\varpi_E})$.
Using \eqref{eq:191} and \eqref{eq:189} we see that this is 
multiplication by ${\sigma^\eps_\tw}([\varpi_E]) = {{\sigma^\eps_\tw}(\gamma_E)}$; by Proposition~\ref{prop:ss-formula} this is the $\eps$-sign of $E$.
\end{proof}

\section{Gluing description of the Chern-Simons line}\label{sec:8a}

We continue with the setup of the last section.
There we used abelianization to produce an isomorphism 
\begin{equation}
  \psi_D: \scrF_G(Y;P) \to \scrS_{\C^\times}(\tY;Q^\eps_\tw;\sigma^\eps_\tw) \otimes \cL(Y,\eps)
\end{equation}
depending on various choices encapsulated in $D$. 
In this section we process this further
to turn it into an explicit description
of $\scrF_G(Y;P)$. The final result is
Theorem~\ref{thm:cs-line-explicit}, which asserts
the existence of nonzero elements
$\tau_{\hat D} \in \scrF_G(Y;P)^*$ depending on
choices $\hat D$ (including an ideal triangulation,
edge-orientations, and certain choices of logarithms),
and gives the cocycle expressing how $\tau_{\hat D}$ changes
when the data $\hat D$ are changed.

\subsection{Abelianization of Chern-Simons over interpolating 3-manifolds} \label{sec:abelianization-cs-interpolating-threemanifolds}

Suppose $M$ is a spin 3-manifold with boundary, equipped with a spectral network $\cN$
such that $M_{-3} = \emptyset$, the restriction of $\cN$
to $Y = \partial M$ is the spectral network $\cN^\scrT$ for a triangulation $\scrT$ of $Y$,
and $M_\twobstrat$ has no closed components.

Each component $B$ of $M_\twobstrat$ is a closed interval, whose two ends lie
on $Y_\twobstrat$, in two triangles $\Delta_0, \Delta_1 \in \faces(\scrT)$,
with edge-orientations $\eps_{\Delta_0}$, $\eps_{\Delta_1}$.
Moreover, a tubular neighborhood of $B$ can be identified with 
$\Delta \times [0,1]$, for a triangle $\Delta$ with its
standard spectral network.
This in particular gives an orientation-reversing identification $f: \Delta_0 \to \Delta_1$.

Now suppose given edge-orientations $\eps$ over $Y$, and 
a $\Z$-bundle $\tM^{\infty} \to \tM_\gea$, which restricts to 
$\tY^{\eps,\infty} \to \tY_\gea$ over the boundary,
and also restricts to $\tDelta^{\eps_B,\infty} \times [0,1]$
over the tubular neighborhood of each component $B \subset M_\twobstrat$,
for some edge-orientations $\eps_B$ on $\Delta$.
In particular $\eps_B$ restricts to match $\eps$ at the two ends, and
so $f_* (\eps \vert_{\Delta_0}) = \eps \vert_{\Delta_1}$.
Thus using Proposition~\ref{prop:triangle-iso} we get 
an isomorphism $\cL(\Delta,\eps_\Delta) \to \cL(\Delta',\eps_{\Delta'})^*$, i.e. an element $\beta \in \cL(\Delta,\eps_\Delta) \otimes \cL(\Delta',\eps_{\Delta'})$.
Tensoring over all components $B \subset M_\twobstrat$ gives a canonical 
element $\beta_M \in \cL(Y,\eps)$.
This element should be thought of as just implementing the matching-up of triangles
provided by $M_\twobstrat$.

Finally, suppose we have stratified abelianization data $\cA = (P, Q, \mu, \theta)$
over $(M, \cN)$.

\begin{prop} \label{prop:cs-interpolating} In this situation,
\begin{equation} \label{eq:cs-interp}
\chi^\eps_Y(\scrF_G(M;P)) = \scrS_{\C^\times}(\tM;Q^\eps_\tw;\sigma^\eps_\tw) \otimes \beta_M \in \scrS_{\C^\times}(\tY;\partial Q^\eps_\tw;\sigma^\eps_\tw) \otimes \cL(Y,\eps).
\end{equation}
\end{prop}
\begin{proof} Decompose $M = M_\rmout \cup M_\rmin$, where
$M_\rmin$ is a small tubular neighborhood of $M_{\twobstrat}$.
$M_\rmin$ is a 3-manifold with 1-dimensional corners $R$;
$R$ is a union of circles, which divides $\partial M_\rmin$
into a union of cylinders $B$ and a union of discs 
$Y_{\rmin} \subset Y$.
Likewise $\partial M_\rmout$ is divided by $R$ into 
$-B$ and $Y_{\rmout} \subset Y$.

Now we need to apply $\scrF_G$ and $\scrS_{\C^\times}$ to the various
parts of this decomposition.
To lighten the notation a bit we write the \Vlines (dimension 1)
\begin{equation}
  \cC_G = \scrF_G(R;P\vert_R), \quad \cC_{\C^\times} = \scrS_{\C^\times}(\widetilde R;Q^\eps_\tw \vert_{\widetilde R};\sigma^\eps_\tw) \, ,
\end{equation}
with objects (dimension 2)
\begin{equation}
  \cO_{G,\rmin} = \scrF_G(Y_\rmin;P \vert_{Y_\rmin}), \quad \cO_{\C^\times,\rmin} = \scrS_{\C^\times}(\tY_\rmin;Q^\eps_\tw \vert_{\tY_\rmin};\sigma^\eps_\tw)\,.
\end{equation}
Thus $\cO_{G,\rmin} \in \cC_G$, $\cO_{\C^\times,\rmin} \in \cC_{\C^\times}$,
and likewise we have $\cO_{G,\rmout} \in \cC_G^*$, $\cO_{\C^\times,\rmout} \in \cC_{\C^\times}^*$,
and $\cO_{G,B} \in \cC_G$, $\cO_{\C^\times,B} \in \cC_{\C^\times}$. Finally, we have the elements (dimension 3)
\begin{align}
  \psi_{G,\rmin} &= \scrF_G(M_\rmin;P\vert_{M_\rmin}) \in \cO_{G,\rmin} \otimes \cO_{G,B}^*\,, \\
  \psi_{G,\rmout} &= \scrF_G(M_\rmout;P\vert_{M_\rmout}) \in \cO_{G,\rmout} \otimes \cO_{G,B}\,, \\
  \psi_{\C^\times,\rmin} &= \scrS_{\C^\times}(M_\rmin;Q^\eps_\tw\vert_{M_\rmin};\sigma^\eps_\tw) \in \cO_{\C^\times,\rmin} \otimes \cO_{\C^\times,B}^*\,, \\
  \psi_{\C^\times,\rmout} &= \scrS_{\C^\times}(M_\rmout;Q^\eps_\tw\vert_{M_\rmout};\sigma^\eps_\tw) \in \cO_{\C^\times,\rmout} \otimes \cO_{\C^\times,B}\,.
\end{align}

We can use Theorem~\ref{thm:strat-abelianization}
to construct various canonical abelianization maps:
an equivalence of \Vlines (associated to dimension 1),
\begin{equation}
  \chi_{\partial B}^\eps: \cC_G \to \cC_{\C^\times},
\end{equation}
homs (associated to dimension 2),
\begin{align}
  \chi^\eps_{Y_\rmout}&\!:  \chi_{\partial B}^{\eps*}(\cO_{G,\rmout}) \to \cO_{\C^\times,\rmout} \, , \\ 
  \chi^\eps_{Y_\rmin}&\!:  \chi_{\partial B}^\eps(\cO_{G,\rmin}) \to \cO_{\C^\times,\rmin} \otimes \cL(Y,\eps) \, , \\
  \chi^\eps_{B}&\!: \chi_{\partial B}^\eps(\cO_{G,B}) \to \cO_{\C^\times,B}
\end{align}
and an equation (associated to dimension 3),
\begin{equation} \label{eq:3d-f-1}
  ((\chi^\eps_{Y_\rmout} \otimes \chi^{\eps *}_{B}) \circ \Omega_{Y_\rmout,\partial B,B}) (\psi_{G,\rmout}) = \psi_{\C^\times,\rmout}
\end{equation}
where
\begin{equation}
 \Omega_{Y_\rmout,\partial B,B}: \cO_{G,\rmout} \otimes \cO_{G,B} \to \chi_{\partial B}^{\eps*}(\cO_{G,\rmout}) \otimes \chi_{\partial B}^{\eps}(\cO_{G,B})
\end{equation}
is the canonical map.
Now, we define an element $\delta_M \in \cL(Y,\eps)$ by
\begin{equation} \label{eq:3d-f-2}
  ((\chi^\eps_{Y_\rmin} \otimes \chi^{\eps}_{B}) \circ \Omega_{Y_\rmin,\partial B,B}) (\psi_{G,\rmin}) = \psi_{\C^\times,\rmin} \otimes \delta_M \, .
\end{equation}
This is a measurement of the difference between the
$G$ and $\C^\times$ theories on the cylinders $M_\rmin$, 
analogous to our definition of $\cL(Y,\eps)$, but one
dimension up, so it gives an element rather than a line.
Tensoring \eqref{eq:3d-f-1} and \eqref{eq:3d-f-2}
gives
\begin{equation}
  \chi^\eps_Y(\psi_G) = \psi_{\C^\times} \otimes \delta_M
\end{equation}
so what remains to prove \eqref{eq:cs-interp} is to show
that $\delta_M = \beta_M$.
Each component of $M_\rmin$ is 
a cylinder, with ends on two discs
carrying the standard spectral network for a triangle;
call these discs $\Delta$, $\Delta'$.
By a diffeomorphism we can identify 
this cylinder with the mapping cylinder $I_f$
of a map $f: \Delta \to \Delta'$.
Moreover, $f$ lifts to the spectral networks, spin structures 
and stratified abelianization data.
Thus we can transport the computation of $\delta_M$ to
the union of mapping cylinders; on each mapping cylinder $I_f$ 
this
gives the action of $f_*: \cL(\Delta,\eps) \to \cL(\Delta',\eps')$, 
and then tensoring over the cylinders gives $\beta_M$ as desired.
\end{proof}

In Proposition~\ref{prop:cs-interpolating} we only considered
the case of a 3-manifold $M$ with boundary a closed triangulated 
surface $Y$. In applications we sometimes want
to use $Y$ with boundary and a semi-ideal or ideal triangulation.
In this case $M$ will have to have extra boundary components extending
$\partial Y$, and corners around the
ideal vertices. As long as we always
work with boundary-reduced stratified abelianization data
over $Y$, these extra boundary
components and corners do not introduce additional complications
in the formal structure: the statement and its proof are unchanged
except for a bit more notation, which we omit here.

\subsection{The dilogarithm in abelianization on one tetrahedron} \label{sec:dilog}

Suppose $Y = S^2$, identified as the boundary of a tetrahedron $\tet$,
with the induced triangulation $\scrT$. Let $\sigma$ denote a spin structure on $Y$,
and $\cA = (P, Q, \mu, \theta)$ stratified abelianization data 
over $(Y,\cN^\scrT)$.
Then $P$ extends to $\hat P \to \tet$ (uniquely up to isomorphism) 
and applying $\scrF_G$ gives an element 
\begin{equation}
  \scrF_G(\tet ; \hat P) \in \scrF_G(Y ; P) \, .
\end{equation}
Fix edge-orientations $\eps$ on $\tet$.
The goal of this section is to describe the image $\chi_Y^\eps(\scrF_G(\tet ; \hat P))$ in terms of a relative
of the dilogarithm function.

Choose a pair of opposite edges in $\scrT$; as discussed above Proposition~\ref{thm:68},
this determines classes $\gamma_1, \gamma_2, \gamma_3 \in H_1(\tY)$.
Also choose a section $t$ of the $\C^\times$-bundle $Q^\eps_\tw$ and
let
\begin{equation} \label{eq:def-u-abstract}
  u_i = \oint_{\gamma_i} t^* \alpha
\end{equation}
where $\alpha$ denotes the connection form in the $\C^\times$-bundle $Q \to
\tY$.  Thus $u_i$ is a logarithm of $\hol_{Q^\eps_\tw}(\gamma_i)$. Changing the choice of section $t$
shifts each $u_i$ by an integer multiple of $2 \pi \I$.  Also let
\begin{equation} \eta_i = \sigma_\tw^\eps(\gamma_i) \in \{\pm 1\} \, .
\end{equation}
Equivalently, $\eta_i$ is the $\eps$-sign of the edge $E_i$.
Thus $(\eta_1, \eta_2)$ measure the isomorphism class of the spin structure $\sigma_\tw^\eps$ on the torus $\tY$,
and are determined by (but have less information than)
the six edge-orientations $\eps$.

\begin{prop}
We have the relation
\begin{equation} \label{eq:uv-constraint}
  \eta_1 \e^{-u_1} + \eta_2 \e^{u_2} = 1 \, .
\end{equation}
\end{prop}

\begin{proof}
Theorem~\ref{thm:12} shows that the holonomies $z_i = \hol_{Q} (\gamma_i)$ obey $z_1^{-1} + z_2 = 1$, and \eqref{eq:qepstw-def}
gives
$\e^{u_i} = \hol_{Q^\eps_\tw} (\gamma_i) = \eta_i z_i$;
combining these gives \eqref{eq:uv-constraint}.
\end{proof}

Now
let $\eta = (\eta_1,\eta_2)$,
\begin{equation}
 S^\eta = \{ (u_1,u_2): \eta_1 \e^{-u_1} + \eta_2 \e^{u_2} = 1 \} \subset \C^2,  
\end{equation}
and let
\begin{equation}
  \ell^{\eta}: S^\eta \to \C / 4 \pi^2 \Z
\end{equation}
be any function obeying
\begin{equation} \label{eq:L-de}
  \de \ell^\eta = \frac12 (u_2 \, \de u_1 - u_1 \, \de u_2).
\end{equation}
Each $\ell^\eta$ is a variant of the dilogarithm function;
see \eqref{eq:lpp}-\eqref{eq:lmm} for concrete examples.

\begin{prop} \label{prop:hyper-dilog}
\begin{equation}
  \frac{\chi^{\eps}_Y(\scrF_G(\tet ; \hat P))}{\tau_t} = c^\eps \exp \left[ \frac{1}{2 \pi \I} \ell^\eta(u_1,u_2) \right], \label{eq:hyper-dilog}
\end{equation}
for some $c^\eps \in \cL(Y,\eps)$ independent of $u_1$, $u_2$.
\end{prop}

\begin{proof}  We first note that
$(P,Q,\mu,\theta)$ fits into a 
family of stratified abelianization data $(P_{u_1,u_2},Q_{u_1,u_2},\mu,\theta)$ over $(Y, \cN^\scrT)$,
parameterized by $(u_1,u_2) \in S^\eta$, 
with the property 
$\e^{u_i} = \eta_i \hol_{Q_{u_1,u_2}} (\gamma_i)$. To construct this
family we can use Construction \ref{thm:14}, choosing the trivial flat bundle $P_{u_1,u_2} \to S^2$, and taking the 4 elements $s_i \in \C\PP^1$ at the 4 vertices to have cross-ratio $\eta_1 \e^{u_1}$; then the calculation in Remark \ref{thm:51} shows that the holonomies of $Q_{u_1,u_2}$ will be as desired. To show that our given $(P,Q,\mu,\theta)$ is indeed isomorphic to a member of this family, we use the fact that
stratified abelianization data is determined up to isomorphism
by the holonomies of $Q$ (for this calculation see e.g. \cite{HN}).

Because $P_{u_1,u_2} \to S^2$ is trivial and its extension $\hat P_{u_1,u_2} \to \tet$ is unique,
the elements $\scrF_G(\tet;\hat P_{u_1,u_2})$ sweep out a covariantly constant section
of the line bundle with fibers $\scrF_G(Y;P_{u_1,u_2})$ over $S^\eta$.
Using the compatibility of $\chi^\eps_Y$ with the connection in
this bundle, it follows that 
$\chi_Y^{\eps}(\scrF_G(\tet;\hat P_{u_1,u_2}))$ is a covariantly constant
section of the line bundle with fibers
$\scrS_{\C^\times}(\tY ; Q_{\tw,u_1,u_2}^\eps ; \sigma^\eps_\tw) \otimes \cL(Y,\eps)$ over $S^\eta$.
Moreover the section $t$ of the given $Q^\eps_\tw$ deforms to 
a section $t_{u_1,u_2}$ of $Q^\eps_{\tw,u_1,u_2}$, unique up to homotopy.
The computations in Section 4 of \cite{FN} then show that 
the function
\begin{equation}
f(u_1,u_2) = \frac{\chi^\eps_Y(\scrF_G(\tet;\hat P_{u_1,u_2}))}{\tau_{t_{u_1,u_2}}} 
\end{equation}
on $S^\eta$
obeys\footnote{In \cite{FN} we discussed only the case $\eta=(+1,+1)$, but this
computation is independent of $\eta$. Also, beware
that $u_1$ in this paper is $-u_1$ in \cite{FN}, which leads
to a sign flip in comparing.}
\begin{equation}
 \de \log f = \frac{1}{4 \pi \I}(u_2 \de u_1 - u_1 \de u_2) \, .
\end{equation}
This proves \eqref{eq:hyper-dilog}.
\end{proof}

\subsection{Flipping an edge}\label{sec:8.2}

Suppose given a boundary-reduced flat $G$-bundle $P \to Y$.
Suppose $\scrT_0$, $\scrT_1$ are two semi-ideal triangulations of $Y$, which differ
by flipping an edge $E \in \edges(\scrT_0)$ to an edge $E' \in \edges(\scrT_1)$.
Also suppose given a choice of sections of $G/P$ over
the interior vertices of $\scrT$, such that
the genericity Assumption \ref{thm:74} holds
for both $\scrT_0$ and $\scrT_1$.
Then by Construction
\ref{thm:14} we obtain boundary-reduced stratified abelianization data
$(P,Q_0,\dots)$ and $(P,Q_1,\dots)$ over  
$(Y,\cN^{\scrT_0})$ and $(Y,\cN^{\scrT_1})$ respectively.

Let $\eps_0$, $\eps_1$
be systems of edge-orientations on $\scrT_0$ and $\scrT_1$, which agree
on all common edges, i.e. all edges except for $E$ and $E'$.
We would like to compare $\chi_Y^{\scrT_0,\eps_0}$ and $\chi_Y^{\scrT_1,\eps_1}$.
Their ratio is an isomorphism
\begin{equation}
 \chi^{\scrT_1,\eps_1}_Y \circ (\chi^{\scrT_0,\eps_0})^{-1}_Y: \scrS_{\C^\times}(\tY_0 ; Q^{\eps_0}_{0,\tw}; \sigma^{\eps_0}_\tw) \otimes \cL(Y,\eps_0) \to \scrS_{\C^\times}(\tY_1 ; Q^{\eps_1}_{1,\tw} ; \sigma^{\eps_1}_\tw) \otimes \cL(Y,\eps_1).
\end{equation}

To describe this map we consider the 3-manifold 
\begin{equation}
M = Y \times [0,1] \, .
\end{equation}
Let $p: M \to Y$ denote the projection, and let $\hat P = p^* P$.
On $M$ we construct a 
SN-stratification, 
spectral network and stratified abelianization, as follows.

Let $Y_\rest$ be the union of all the triangles in $Y$ except
for the two abutting $E$. Then we have a decomposition
$Y = Y_{\rest} \cup \cQ_E$. 
The stratified abelianizations $(P,Q_0,\dots) \vert_{Y_\rest}$ and $(P,Q_1,\dots) \vert_{Y_\rest}$
are canonically isomorphic, 
and thus extend to boundary-reduced stratified abelianization data
over $(Y_{\rest} \times [0,1], \cN_\rest \times [0,1])$.
Over $Y_\rest \times [0,1]$ we also have a $\Z$-bundle
$p^* \tY^{\eps,\infty} \vert_{Y_\rest}$.

Next we consider $\cQ_E \times [0,1]$. 
Collapsing the vertical intervals in 
$\partial \cQ_E \times [0,1]$ gives a surjective
map $\partial \phi: \partial (\cQ_E \times [0,1]) \to \partial \tet$,
where $\tet$ is a tetrahedron, and this map extends to
\begin{equation}
\phi: \cQ_E \times [0,1] \to \tet
\end{equation}
which is a homeomorphism on the interior.
On $\tet$ we have a standard SN-stratification and spectral network 
as described in Construction \ref{thm:35}; pulling this back by $\phi$ gives a 
SN-stratification and spectral network on $\cQ_E \times [0,1]$, which glue to the ones we already have on $Y_\rest \times [0,1]$.
Using Construction \ref{thm:36} we obtain
stratified abelianization data over $\tet_{\geb}$, which likewise
pulls back to $\cQ_E \times [0,1]$ and glues to the stratified
abelianization data we already have on $Y_\rest \times [0,1]$.

Thus altogether we have:
\begin{itemize} 
\item A SN-stratification and 
spectral network $\cN$ over $M$, with isomorphisms
$\cN \vert_{Y \times \{0\}} \simeq \cN^{\scrT_0}$, $\cN \vert_{Y \times \{1\}} \simeq \cN^{\scrT_1}$,
\item 
 Stratified abelianization data 
$(\hat P, \hat Q, \hat\mu, \hat\theta)$ over $(M,\cN)$, 
which restricts to $(P,Q_0,\mu_0,\theta_0)$, $(P,Q_1,\mu_1,\theta_1)$ 
over $Y_0$, $Y_1$ respectively.
\item 
A $\Z$-bundle $\tM^{\infty} \to M_{\gea}$, 
which restricts to $\tY^{\eps_0,\infty}$, $\tY^{\eps_1,\infty}$ over $Y_0$, $Y_1$.
\end{itemize}
We need to extend our $\bmuu_2$-twisted objects from $Y_0$ and $Y_1$
to $M$: define
$\tM^{4} = \tM^{\infty} / 4\Z$,
and twist $\hat Q \to M_\gea$ to 
$\hat Q_\tw = \hat Q \otimes_{\bmuu_2} \tM^4$ and $\pi^* p^* \sigma$ to $\sigma_\tw = \pi^* p^* \sigma \otimes_{\bmuu_2} \tM^4$.

The stratum 
$M_{\threebstrat}$ consists of a single point $p$, the barycenter of the tetrahedron
$\tet$. Let $B_p$ be a small ball around $p$. 
By radial projection centered at $p$, 
the sphere $H = \partial B_p$ 
acquires a 
triangulation $\scrT_H$, whose four faces correspond naturally to the
two faces of $\cQ_E$ and two of $\cQ_{E'}$.
$\scrT_H$ comes with edge-orientations $\eps_H$,
induced by $\eps_0$ and $\eps_1$ (using the fact that $\eps_0$
and $\eps_1$ agree on the common edges).
Moreover, the restriction $\cN \vert_H$ is the 
spectral network $\cN^{\scrT_H}$, and
$\tH^{\eps_H,\infty}$ is the restriction of
$\tM^\infty$ to $H$.


Now we can describe the effect of a flip on the abelianization
maps $\chi^{\scrT,\eps}_Y$:

\begin{prop} \label{prop:flip-inexplicit}
The map
\begin{equation}
  \chi^{\scrT_1,\eps_1}_Y \circ (\chi_Y^{\scrT_0,\eps_0})^{-1}: \scrS_{\C^\times}(\tY ; Q^{\eps_0}_{0,\tw} ; \sigma^{\eps_0}_\tw) \otimes \cL(Y,\eps_0) \to \scrS_{\C^\times}(\tY; Q^{\eps_1}_{1,\tw}; \sigma^{\eps_1}_\tw) \otimes \cL(Y,\eps_1)
\end{equation}
is the product
\begin{equation}
\scrS_{\C^\times}(M \setminus B_p ; \hat{Q}_\tw \vert_{M \setminus B_p} ; \sigma_\tw) \otimes \chi_H^{\eps_H}(\scrF_G(B_p ; \hat P \vert_{B_p})) \otimes \beta_{M \setminus B_p} \, .
\end{equation}
\end{prop}

\begin{proof}
We apply Proposition~\ref{prop:cs-interpolating} to
the 3-manifold $M \setminus B_p$,
noting that what was called $Y$ there is here the disconnected
boundary $\partial(M \setminus B_p) = Y_0 \cup -Y_1 \cup -H$, which carries a
semi-ideal triangulation $\scrT_0 \cup \scrT_1 \cup \scrT_H$
and edge-orientations $\eps_0 \cup \eps_1 \cup \eps_H$.

\insfig{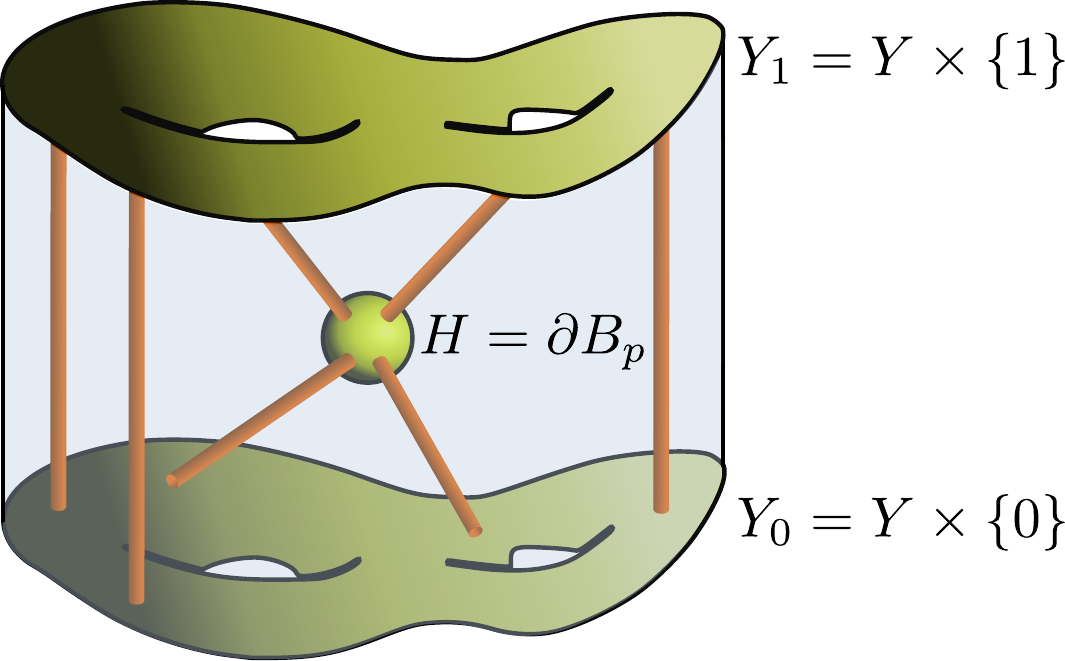}{0.65}{The 3-manifold $M = Y \times [0,1]$ 
which we use to study a flip of the semi-ideal 
triangulation on $Y$. The
branch locus $M_{\twobstrat}$ is shown in orange. In this example
each triangulation has five triangles, and thus there are five points
of $M_{\twobstrat}$ on each of $Y_0$ and $Y_1$.}

This gives the formula
\begin{equation}
\chi_{Y_0 \cup -Y_1 \cup -H}^{\scrT_0 \cup \scrT_1 \cup \scrT_H,\eps_0 \cup \eps_1 \cup \eps_H}(\scrF_G(M \setminus B_p ; \hat P \vert_{M \setminus B_p})) = \scrS_{\C^\times}(M \setminus B_p ; \hat{Q}_\tw \vert_{M \setminus B_p} ; \sigma_\tw) \otimes \beta_{M \setminus B_p} \, .
\end{equation}
Tensoring with $\chi_H^{\eps_H}(\scrF_G(B_p ; \hat P \vert_{B_p}))$ then gives
\begin{equation}
\chi_{Y_0 \cup -Y_1}^{\scrT_0 \cup \scrT_1,\eps_0 \cup \eps_1}(\scrF_G(M ; \hat P)) = \scrS_{\C^\times}(M \setminus B_p ; \hat{Q}_\tw \vert_{M \setminus B_p} ; \sigma_\tw) \otimes \beta_{M \setminus B_p} \otimes \chi_H^{\eps_H}(\scrF_G(B_p ; P \vert_{B_p})) \, .
\end{equation}
But since $M = Y \times [0,1]$ and $\hat P = p^* P$, 
$\scrF_G(M ; \hat P)$ is the identity map
on $\scrF_G(Y ; P)$; the desired result follows by rearranging
factors.
\end{proof}

Using the identification between the $4$ 
triangles in $\cQ_E$, $\cQ_{E'}$
and the $4$ triangles in $H$, we get from
Proposition~\ref{prop:triangle-iso} an element
$\beta: \cL(\cQ_E, \eps_0) \otimes \cL(H,\eps_H) \to \cL(\cQ_{E'}, \eps_1)$.
Contracting that with the normalization 
constant $c^{\eps_H} \in \cL(H,\eps_H)$ defined 
in Proposition~\ref{prop:hyper-dilog}, we obtain an isomorphism
\begin{equation} \label{eq:def-kappa}
\kappa_E^{\eps_0,\eps_1} = \beta(c^{\eps_H}): \cL(\cQ_E,\eps_0) \to \cL(\cQ_{E'},\eps_1)\,.
\end{equation}

\subsection{Gluing the Chern-Simons line}

Let $Y$ be a compact oriented surface with boundary.
Suppose $P \to Y$ is a flat boundary-reduced $G$-bundle.
In this section we use abelianization to give a description of the line
$\scrF_G(Y;P)$.

\begin{constr} \label{constr:triv-from-data} We consider tuples $D = (\scrT,\eps,s,x)$, where:
\begin{itemize}
\item $\scrT$ is a semi-ideal triangulation of $Y$,
\item $\eps = (\eps_E)_{E \in \edges(\scrT)}$ is a system of edge-orientations for $\scrT$,
\item $s = (s_v)_{v \in \vertices(\scrT)}$ 
is a flat section of $P/U = P \times_G (\C^2 \setminus \{0\})$ over each $v$, such that
if $v$ is an ideal vertex then the projection of $s_v$ to $P/B$ agrees with the boundary reduction,
\item $x = (x_E)_{E \in \edges(\scrT)}$ is a collection of complex numbers, 
where 
\begin{equation}
\exp(x_E) = \eps_E(v,v') s_{v'} \wedge s_{v}  
\end{equation}
if $E$ is an edge with vertices $v, v'$. Here $s_v$ means the continuation of $s$ from
$v$ by parallel transport along the edge $E$, similarly $s_{v'}$, 
and the wedge product is evaluated at any point along $E$; here and below, we use the fact that $P$ is 
an $\SLC$-bundle and thus there is a canonical volume form in $P/U$.
  \end{itemize}
For each such tuple there is a canonical isomorphism
\begin{equation} \label{eq:psid}
\psi_D: \scrF_G(Y;P) \to \cL(Y,\eps).  
\end{equation}
\end{constr}

\begin{proof}

Fix $P \to Y$ and $D$ as above.
Construction \ref{thm:14} extends $P$ to stratified 
abelianization data $(P, Q, \mu, \theta)$
over $(Y,\cN^\scrT)$. Then Construction \ref{constr:abelianization-surface} gives an isomorphism
\begin{equation} \label{eq:iso-1}
  \chi^{\scrT,\eps}_Y: \scrF_G(Y;P) \to \scrS_{\C^\times}(\tY;Q^\eps_\tw;\sigma^\eps_\tw) \otimes \cL(Y,\eps).
\end{equation}

Using the data $(s,x)$ we determine a section $\tils_x$ of the line bundle $Q^\eps_\tw \to \tY$ (up to 
homotopy), as follows.
Recall that each $v \in \vertices(\scrT)$ 
has two preimages $(v,v)$ and $(v,\overline v)$
in $\tY$. Each $s_v$ determines an element $\tils(v,v)$ 
of the fiber of $Q^\eps_\tw$.
We also determine an element $\tils(v,\overline v)$ 
by the condition $\theta(\tils(v,v)) \wedge \theta(\tils(v,\overline v)) = 1$.
Now suppose $E \in \edges(\scrT)$ with vertices $v$, $v'$.
The parallel transport of $Q^\eps_\tw$ along preimages of 
$E$ takes 
\begin{equation}
\tils(v,v) \mapsto \exp(x_E) \tils(v',\overline v'), \quad \tils(v,\overline v) \mapsto \exp(-x_E) \tils(v',v') \, .
\end{equation}
The choice of a logarithm $x_E$ 
thus determines (up to homotopy)
an extension $\tils_x$ of $\tils$ over
the preimages of $E$.
Finally, by construction $\tils_x$ has zero winding
around the preimage of the boundary of 
each triangle, because the summands $x_E$ and $-x_E$
cancel over each of the three edges; 
thus $\tils_x$ can be extended 
to a section $\tils_x$ of $Q^\eps_\tw \to \tY$,
uniquely up to homotopy.
The section $\tils_x$ determines a trivialization
$\tau_{\tils_x} \in \scrS_{\C^\times}(\tY;Q^\eps_\tw;\sigma^\eps_\tw)$. Then we define
\begin{equation} \label{eq:def-psid}
  \psi_D = \frac{\chi_Y^{\scrT,\eps}}{\tau_{\tils_x}} \, .
\end{equation}

\end{proof}

We remark that the existence of the data $x_E$ above requires
$s_v \wedge s_{v'} \neq 0$, i.e. that
the genericity Assumption \ref{thm:74} is satisfied.

\begin{prop} \label{prop:line-transitions}
We have relations among the maps associated to data
$D = (\scrT, \eps, s, x)$ and $D' = (\scrT', \eps', s', x')$
as follows:

\begin{enumerate}
\item Suppose $D$ and $D'$ differ only by a change involving 
a single vertex $v$:
$s'_v = s_v \exp(t)$
and $x'_E = x_E + t$ for all edges $E$ incident on $v$. 
Then
\begin{equation}
  \psi_{D'} = \psi_D.
\end{equation}

\item Suppose $D$ and $D'$ differ only by a change involving
a single edge $E$: $\eps'_E = - \eps_E$ and
$x'_E = x_E + \pi \I$.
\insfig{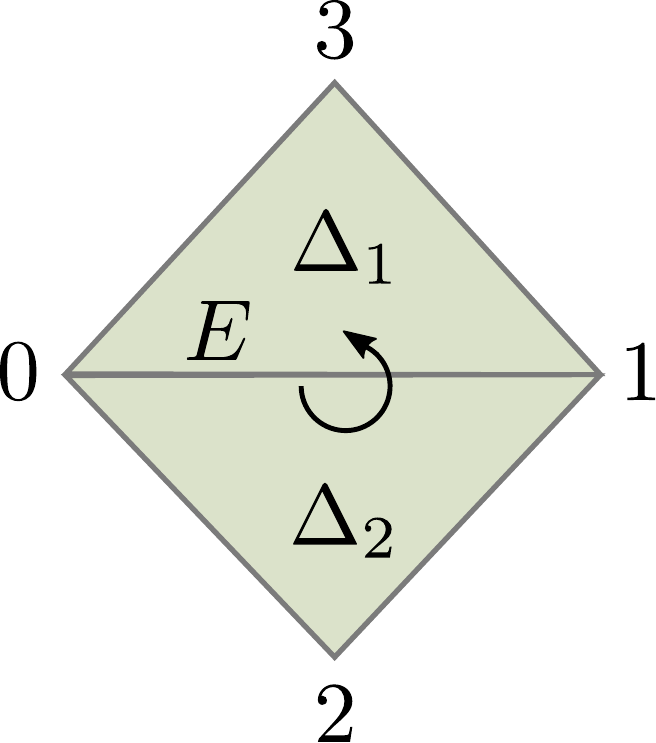}{0.4}{The quadrilateral $\cQ_E$.}
With labeling as in Figure~\ref{fig:quad-labeled} let
\begin{equation} \label{eq:u-def-single}
  u = x_{E_{30}} - x_{E_{02}} + x_{E_{21}} - x_{E_{13}} \, .
\end{equation}
Then
\begin{equation} \label{eq:psid-desired}
  \psi_{D'} = \exp \left[\frac14 u \right] (\rho^{\eps,\eps'}({\Delta_1}) \otimes \rho^{\eps,\eps'}({\Delta_2})) \, \psi_D.
\end{equation}

\item Suppose $D$ and $D'$ differ only in that
$\scrT'$ is obtained from $\scrT$ by flipping edge $E$
to obtain a new edge $E'$, with some orientation 
$\eps'_{E'}$ and a choice of logarithm $x_{E'}$.
\insfig{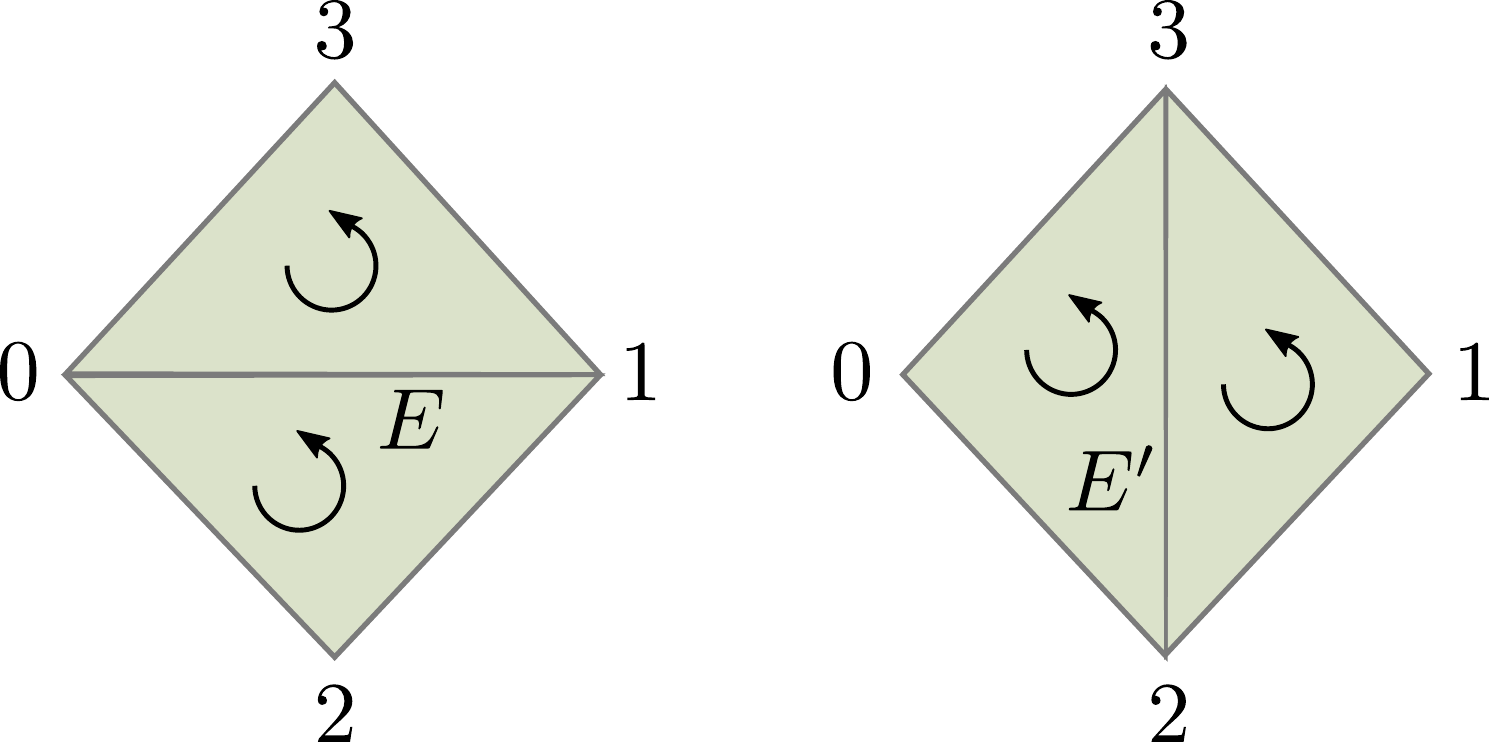}{0.4}{The quadrilaterals $\cQ_E$ in $\scrT$ and $\cQ_{E'}$ in $\scrT'$, related by a flip.}
Labeling the vertices as in Figure~\ref{fig:two-quads}, let
\begin{gather} \label{eq:u-def}
  u_1 = x_{E_{30}} - x_{E_{02}} + x_{E_{21}} - x_{E_{13}} \, , \quad u_2 = x_{E_{32}} - x_{E_{21}} + x_{E_{10}} - x_{E_{03}} \, , \\
  \eta_1 = (-1)^{\frac12(\eps(3,0)+\eps(0,2)+\eps(2,1)+\eps(1,3))} \, , \quad \eta_2 = (-1)^{\frac12(\eps(3,2)+\eps(2,1)+\eps(1,0)+\eps(0,3))} \, . \label{eq:eta-def}
\end{gather}
Then
\begin{equation} \label{eq:psi-flip-relation}
  \psi_{D'} = \exp \left[ \frac{1}{2 \pi \I} \ell^{\eta}(u_1,u_2) \right] \kappa_E^{\eps,\eps'} \circ \psi_D \, .
\end{equation}

\end{enumerate}
\end{prop}

\begin{proof}
We treat the three relations in turn:
\begin{enumerate}
\item In this case the sections $\tils_x$ and $\tils'_{x'}$ are homotopic (the homotopy obtained by continuously varying from $0$ to $t$) and thus they induce the same trivialization of $\scrS_{\C^\times}(Q^\eps_\tw ; \sigma_\tw)$.

\item By~\eqref{eq:def-psid} and \eqref{eq:lambda-def} we have
$\psi_{D'} \cdot \tau_{\tils'_{x'}} = (\lambda^{\eps,\eps'}(Y) \otimes \zeta(\tY;Q^\eps_\tw;\varpi^E)) \circ \psi_D \cdot \tau_{\tils_{x}} $.
We use Proposition~\ref{prop:lambda-factorization} to expand the
$\lambda$ factor. For the $\zeta$ factor,
 we use \eqref{eq:191},
 which gives 
\begin{equation}
\zeta(\tY;Q^\eps_\tw;\varpi^E)(\tau_{\tils_x}) = \exp \left[ -\frac14 \int_{\tY} \tils_x^* \alpha \wedge \varpi^E \right] \tau_{\tils'_{x'}}
\end{equation}
where $\alpha$ denotes the connection form on $Q^\eps_\tw$.
By \eqref{eq:varpi-holonomy} the multiplicative factor here is
\begin{equation}
  \exp \left[ - \frac14 \oint_{\gamma_E} \tils_x^* \alpha \right]
\end{equation}
and comparing \eqref{eq:u-def-single} to the definition of $\gamma_E$, 
we have $\oint_{\gamma_E} \tils_x^* \alpha = - u$, so the factor becomes
\begin{equation}
  \exp \left[ \frac14 u \right]
\end{equation}
as desired.

\item 
We use the setup and notation of \S\ref{sec:8.2}, involving
the 3-manifold $M = Y \times [0,1]$.
Proposition~\ref{prop:flip-inexplicit} gives
$\chi^{\scrT',\eps'}_Y \circ (\chi_Y^{\scrT,\eps})^{-1}$
as a tensor product of three ingredients:
\begin{enumerate}
\item First we have $\scrS_{\C^\times}(M \setminus B_p ; \hat Q_\tw ; \sigma_\tw)$.
The data of $D$ and $D'$ in particular determine sections $\tils_x$ and $\tils'_{x'}$ of $\hat Q^\eps_\tw$ 
on the boundaries $Y_0$ and $Y_1$, as discussed in Construction \ref{constr:triv-from-data}. There is a section of $\hat Q^\eps_\tw \to M \setminus B_p$ which extends $\tils_x$ and $\tils'_{x'}$; let $\tils_H$ be
its restriction to $H$.
The $\C^\times$ Chern-Simons form vanishes for a flat bundle; 
it follows that
\begin{equation}
\scrS_{\C^\times}(M \setminus B_p ; \hat Q_\tw ; \sigma_\tw) = \tau^*_{\tils_x} \otimes \tau_{\tils'_{x'}} \otimes \tau^*_{\tils_H} \, .
\end{equation}

\item Next we have the factor 
$\chi^{\eps_H}(\scrF_G(B_p ; \hat P \vert_{B_p}))$.
Note that $u_1$ and $u_2$ defined in \eqref{eq:u-def}
agree with the $u_1$ and $u_2$ defined in
in \S\ref{sec:dilog}, if we take the section $t = \tils_H$.
Likewise $\eta_1$ and $\eta_2$ defined in \eqref{eq:eta-def}
agree with the $\eta_1$ and $\eta_2$ used there.
Then by \eqref{eq:hyper-dilog},
$$\chi^{\eps_H}(\scrF_G(B_p ; \hat P \vert_{B_p})) = c^{\eps_H} \exp \left[ \frac{1}{2 \pi \I} \ell^{\eta}(u_1,u_2) \right] \tau_{\tils_H} \, .$$

\item Finally we have the isomorphism $\beta_{M \setminus B_p}$.
\end{enumerate}
Tensoring these ingredients together, and using the definition
\eqref{eq:def-psid} of $\psi_D$ 
and the definition \eqref{eq:def-kappa}
of $\kappa_E^{\eps,\eps'}$,
we get the desired \eqref{eq:psi-flip-relation}.
\end{enumerate}
\end{proof}

\subsection{Explicit formulas}

In the last section we gave a description of the line $\scrF_G(Y;P)$ which is
canonical but somewhat
inexplicit: the transition maps described by Proposition~\ref{prop:line-transitions} involve
the maps $\rho^{\eps,\eps'}(\Delta)$
and $\kappa_E^{\eps,\eps'}$, for which we have not yet written
down concrete formulas. Roughly speaking, we have fully described
the $P$-dependence in $\scrF_G(Y;P)$, 
but left some constant phases undetermined.
In this section we rectify this omission, getting
concrete formulas in terms of actual numbers,
at the cost of making some arbitrary choices.

\subsubsection{Trivializing the difference lines}

Let $\Delta_s$ denote the standard triangle with its standard orientation, and vertices labeled 012 in cyclic order given
by the orientation.
\insfig{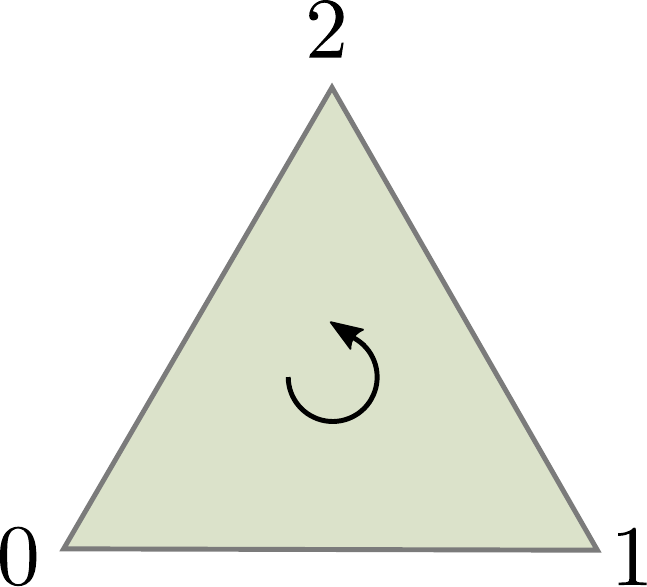}{0.4}{The standard triangle.}
We choose a nonzero element
\begin{equation} \label{eq:t-eps}
  \tau_\eps \in \cL(\Delta_s,\eps)
\end{equation}
for each of the eight possible $\eps$.
We may choose $\tau_\eps$ so that 
if $f: \Delta_s \to \Delta_s$ is a rotation which acts
on the vertices by $012 \to 120$,\footnote{The factor appearing in \eqref{eq:tau-rot-equiv} is dictated by
Proposition~\ref{prop:L-rotation}.}
\begin{equation} \label{eq:tau-rot-equiv}
f_*(\tau_\eps) = \exp\left(\frac{2 \pi \I}{3} \right) {\tau_{f_* \eps}},
\end{equation}
and also 
if $r: \Delta_s \to \Delta_s$ is the reflection
which acts on the vertices by $012 \to 102$,
\begin{equation} \label{eq:reflection-condition}
r_*(\tau_\eps) = \tau_{r_* \eps}^{-1} \, .
\end{equation}
From now on we assume that we make such a choice.

Now suppose $(Y,\scrT,\eps)$ is a general triangulated manifold with 
edge-orientations. We want to trivialize the line $\cL(Y,\eps)$; for this we need one more datum:
\begin{itemize}
\item A marked edge $e_\Delta$ on each $\Delta \in \faces(\scrT)$.
\end{itemize}
The edge $e_\Delta$ determines 
an orientation-preserving identification
$\phi_{e_\Delta}: \Delta_s \to \Delta$,
by the condition that $\phi_{e_\Delta}$ takes the edge $(0,1)$ to $e_\Delta$.
Then, given the extended data $\hat{D} = (\scrT, \eps, s, x, e)$ we obtain
a trivialization
\begin{equation} \label{eq:tauhatD}
  \tau_{\hat{D}} = \frac{\psi_D}{\bigotimes_{\Delta \in \faces(\scrT)} (\phi_{e_\Delta})_* \tau_\eps} \in \scrF_G(Y;P)^* \, .
\end{equation}
What remains is to describe the relations between the elements $\tau_{\hat{D}} \in \scrF_G(Y;P)^*$
associated to different $\hat{D}$.

\subsubsection{Computing $\rho^{\eps,\eps'}(\Delta_s)$}

Recall the edge-reversal map $\rho$ from \S\ref{sec:edge-reversal-triangle}.
Fix arbitrarily a nonzero element $\xi$ in the line $\Xi(E)$ considered there,
and then define constants $b(\eps,\eps') \in \C^\times$ by
\begin{equation} \label{eq:c-const-def}
\rho^{\eps,\eps'}(\Delta_s) (\tau_{\eps}) = b(\eps,\eps') \, \tau_{\eps'} \otimes \xi^{\pm 1} \, .
\end{equation}
Our aim now is to compute these constants.
We write each $\eps$ as a 3-tuple 
$(\eps(0,1), \eps(1,2), \eps(2,0))$.

\begin{prop} \label{prop:lambda-explicit}
We have 
\begin{equation} \label{eq:c-rotation-invariance}
  b(f_* \eps, f_* \eps') = b(\eps, \eps') \, ,
\end{equation}
and for some $p, q \in \C^\times$ we have
\begin{center}
\begin{tabular}{|c|c|c|} \hline
 $(a_1, a_2)$& $b((a_1,a_2,+1),(a_1,a_2,-1))$ & $b((a_1,a_2,-1),(a_1,a_2,+1))$ \\ \hline
$(-1,-1)$ & $p$ & $-p^{-1}$ \\
$(-1,+1)$ & $\omega^{-1} q$ & $\omega q^{-1}$ \\
$(+1,-1)$ & $\omega q$ & $\omega^{-1} q^{-1}$ \\
$(+1,+1)$ & $-p$ & $p^{-1}$ \\ \hline
\end{tabular}
\end{center}
where $\omega = \exp (\frac{2 \pi \I}{8})$.
\end{prop}
\begin{proof}
Applying $f_*$ to both sides of \eqref{eq:c-const-def} and using \eqref{eq:tau-rot-equiv} gives \eqref{eq:c-rotation-invariance}.

Proposition~\ref{prop:edge-reversal-commutation}
gives 
\begin{multline}
b((+1,+1,+1),(+1,+1,-1)) b((+1,+1,-1),(+1,-1,-1)) = \\ \exp \left(- \frac{2 \pi \I}{4}\right) b((+1,+1,+1),(+1,-1,+1)) b((+1,-1,+1),(+1,-1,-1)) \, .
\end{multline}
Applying \eqref{eq:c-rotation-invariance} 
and canceling a common factor 
$b((+1,+1,+1),(+1,+1,-1))$
gives
\begin{equation} \label{eq:c-constraint-1}
b((-1,+1,+1),(-1,+1,-1)) = \exp \left(- \frac{2 \pi \I}{4}\right) b((+1,-1,+1),(+1,-1,-1)) \, .
\end{equation}

Next, using Proposition~\ref{prop:edge-reversal-double},
we have
\begin{multline} \label{eq:c-constraint-2}
  b((a_1,a_2,+1),(a_1,a_2,-1)) \times b((a_3,a_4,-1),(a_3,a_4,+1)) \times \\ b((a_1,a_2,-1),(a_1,a_2,+1)) \times b((a_3,a_4,+1),(a_3,a_4,-1)) = (-1)^{\frac12(a_1+a_2+a_3+a_4)}
\end{multline}

Finally, using the reflection condition \eqref{eq:reflection-condition} gives
\begin{equation}
  b(\eps,\eps') \, b(r_* \eps, r_* \eps') = 1 \, ,
\end{equation}
which implies in particular
\begin{equation} \label{eq:c-constraint-3}
  b((+1,-1,+1),(+1,-1,-1)) \, b((+1,-1,-1),(+1,-1,+1)) = 1
\end{equation}
and
\begin{equation} \label{eq:c-constraint-4}
  b((+1,+1,+1),(+1,+1,-1)) \, b((-1,-1,-1),(-1,-1,+1)) = 1 \, .
\end{equation}

The constraints \eqref{eq:c-constraint-1}, \eqref{eq:c-constraint-2}, \eqref{eq:c-constraint-3}, \eqref{eq:c-constraint-4} determine the $b(\eps,\eps')$ up to two undetermined constants
as indicated in the table.
\end{proof}

Proposition~\ref{prop:lambda-explicit} determines all of the $b(\eps,\eps')$
in terms of the undetermined constants $p$, $q$.
Fixing $p$, $q$ is equivalent to fixing the remaining
freedom in the trivializations $\tau_\eps$, up to an overall 
scale which remains unfixed.

\subsubsection{Fixing the $\ell^\eta$} 
In \S\ref{sec:dilog} we determined the functions
$L^\eta$ up to an overall constant. We now make a definite
choice as follows.
Let $\Li_2$ denote the principal branch of the dilogarithm function, which has a branch 
cut along $(1,\infty)$.
Then let
\begin{align} 
\ell^{(+1,+1)}(u_1,u_2) &= \Li_2(e^{-u_1}) - \frac{1}{2} u_1 u_2 - 2 \pi \I \left\lfloor -\frac{\im u_2}{2\pi} + \half \right\rfloor  u_1  ,  \label{eq:lpp} \\
\ell^{(-1,+1)}(u_1,u_2) &= \Li_2(-e^{-u_1}) - \frac{1}{2} u_1 u_2 - 2 \pi \I \left\lfloor -\frac{\im u_2}{2\pi} + \half \right\rfloor  (u_1 + \pi \I) , \label{eq:lmp} \\
\ell^{(+1,-1)}(u_1,u_2) &= \Li_2(e^{-u_1}) - \frac{1}{2} u_1 u_2 - \pi \I u_1 - 2 \pi \I \left\lfloor -\frac{\im u_2}{2\pi} \right\rfloor u_1 , \label{eq:lpm} \\
\ell^{(-1,-1)}(u_1,u_2) &= \Li_2(-e^{-u_1}) - \half u_1 u_2 - \pi \I u_1  - 2 \pi \I \left\lfloor -\frac{\im u_2}{2\pi} \right\rfloor (u_1 + \pi \I) . \label{eq:lmm}
\end{align}
It is routine to check that these formulas indeed define holomorphic functions $\ell^\eta: S^\eta \to \C / 4 \pi^2 \Z$:
the discontinuity of $\Li_2$ across its branch cut gets compensated by the discontinuity
of the floor function, up to an integer multiple of $4 \pi^2$.
Using the differential equation obeyed by $\Li_2$, 
it is also
straightforward to check that they obey \eqref{eq:L-de}
as needed.

\subsubsection{Computing the $\kappa_E^{\eps,\eps'}$}

Now that we have fixed our choices of $\ell^\eta$ 
and also fixed trivializations of the $\cL(\Delta,\eps)$
we are in position to express the normalization constants
from \eqref{eq:def-kappa},
\begin{equation}
  \kappa_E^{\eps,\eps'}: \cL(\cQ_E,\eps) \to \cL(\cQ_{E'},\eps') \, ,
\end{equation}
as concrete numbers: we write
\begin{equation}
  \kappa_E^{\eps,\eps'} ((\phi_E)_* \tau_{\eps_1} \otimes (\phi_E)_* \tau_{\eps_2} ) = k_E^{\eps,\eps'} ((\phi_{E'})_* \tau_{\eps'_1} \otimes (\phi_{E'})_* \tau_{\eps'_2} ) 
\end{equation}
for some $k_E^{\eps,\eps'} \in \C^\times$.

To compute the $k^{\eps,\eps'}_E$, it will be useful to recall their origin
in Chern-Simons theory on an oriented sphere $H$ with tetrahedral
triangulation $\scrT$, obtained by gluing $-\cQ_E$ to $\cQ_{E'}$
along the common boundary.
We label the edges of $\cQ_E$ and $\cQ_{E'}$ as shown
in Figure~\ref{fig:two-quads}, thus identifying
$H$ with a standard model. We
use $\eps_H$ to represent the full collection of six edge-orientations on $\scrT$ induced by $(\eps,\eps')$,
and use the condensed notation
$k^{\eps_H}$ for $k_E^{\eps,\eps'}$.

The next lemma concerns just $H$, not the original surface $Y$.

\begin{lem} \label{lem:k-relations} 
Suppose $\eps$ and $\eps'$ are
two systems of edge-orientations on $\scrT$, 
which differ by reversing the orientation
on some $\tilde{E} \in \edges(\scrT)$. 
Fix stratified abelianization data $(P, Q, \mu, \theta)$ on $(H,\cN^\scrT)$,
and extend $(\scrT,\eps)$ to data $D = (\scrT,\eps,s,x)$ as
in Construction~\ref{constr:triv-from-data}.
Also define $D' = (\scrT,\eps',s,x')$ where
$x'_E = x_E + \I \pi$, 
and $x'_E = x_E$ for all other edges.
Let
$(u_1,u_2)$ and $(\eta_1,\eta_2)$ be given
by \eqref{eq:u-def} and \eqref{eq:eta-def},
likewise 
$(u'_1,u'_2)$ and $(\eta'_1,\eta'_2)$.
Let $\Delta_1$, $\Delta_2$ denote the two triangles
abutting $\tilde{E}$. Then
\begin{equation}
  \frac{k^{\eps'}}{k^{\eps}} = \exp \left[ \frac{1}{2 \pi \I} \left( \ell^\eta(u_1,u_2) - \ell^{\eta'}(u'_1,u'_2) \right) + \frac{1}{4\pi\I} (u_1 u'_2 - u_2 u'_1) \right] b(\eps_{\Delta_1},\eps'_{\Delta_1}) b(\eps_{\Delta_2},\eps'_{\Delta_2})
\end{equation}
with $b$ determined in Proposition~\ref{prop:lambda-explicit}.

\end{lem}

\begin{proof} We consider the element $o = \scrF_G(\tet;\hat P)$,
which has
\begin{equation}
  \psi_D(o) = c^\eps \exp \left[ \frac{1}{2 \pi \I} \ell^\eta(u_1,u_2) \right]\, \in \cL(H,\eps) \, .
\end{equation}
Part (2) of Proposition~\ref{prop:line-transitions}
gives the relation between $\psi_D(o)$ and $\psi_{D'}(o)$.
The quantity $u$ appearing there is not necessarily 
the $u_1$ here, because $\tilde{E}$ is not necessarily
$E_{01}$; one can show however (e.g. by checking the
six cases for $\tilde E$) that the correct relation is
\begin{equation}
\I \pi u = u_1 u'_2 - u_2 u'_1 \, .  
\end{equation}
Thus we get
\begin{equation}
  \exp \left[ \frac{1}{2 \pi \I} \ell^{\eta'}(u'_1,u'_2) \right] c^{\eps'} = \exp \left[ \frac{1}{2 \pi \I} \ell^\eta(u_1,u_2) + \frac{1}{4\pi\I} (u_1 u'_2 - u_2 u'_1) \right] (\rho^{\eps,\eps'}({\Delta_1}) \otimes \rho^{\eps,\eps'}({\Delta_2}))(c^\eps) \, .
\end{equation}
Now for each $\Delta \in \faces(\scrT)$
let the marked edge $e_\Delta$ be either $E_{01}$ 
or $E_{23}$,
whichever lies on $\Delta$.
Dividing both sides by
$\bigotimes_{\Delta} (\phi_{e_\Delta})_* (\tau_{\eps_\Delta})$
gives the desired result.
\end{proof}

Lemma \ref{lem:k-relations} determines the constants
$k^{\eps}$ up to a single overall multiplicative constant.
The remaining constant can be fixed as follows.
We consider a triangulated surface $(Y,\scrT)$ 
with stratified abelianization
data $(P, Q, \mu, \theta)$.
Suppose $\scrT$ contains three triangles which make up a
pentagon. Then we consider two possible sequences of flips, 
involving various triangulations $\scrT_i$ as 
indicated in Figure~\ref{fig:pentagon-flips}, and 
choose data $\hat{D}_i$ (logarithms, edge-orientations and marked edges) extending the triangulations 
$\scrT_i$.

\insfig{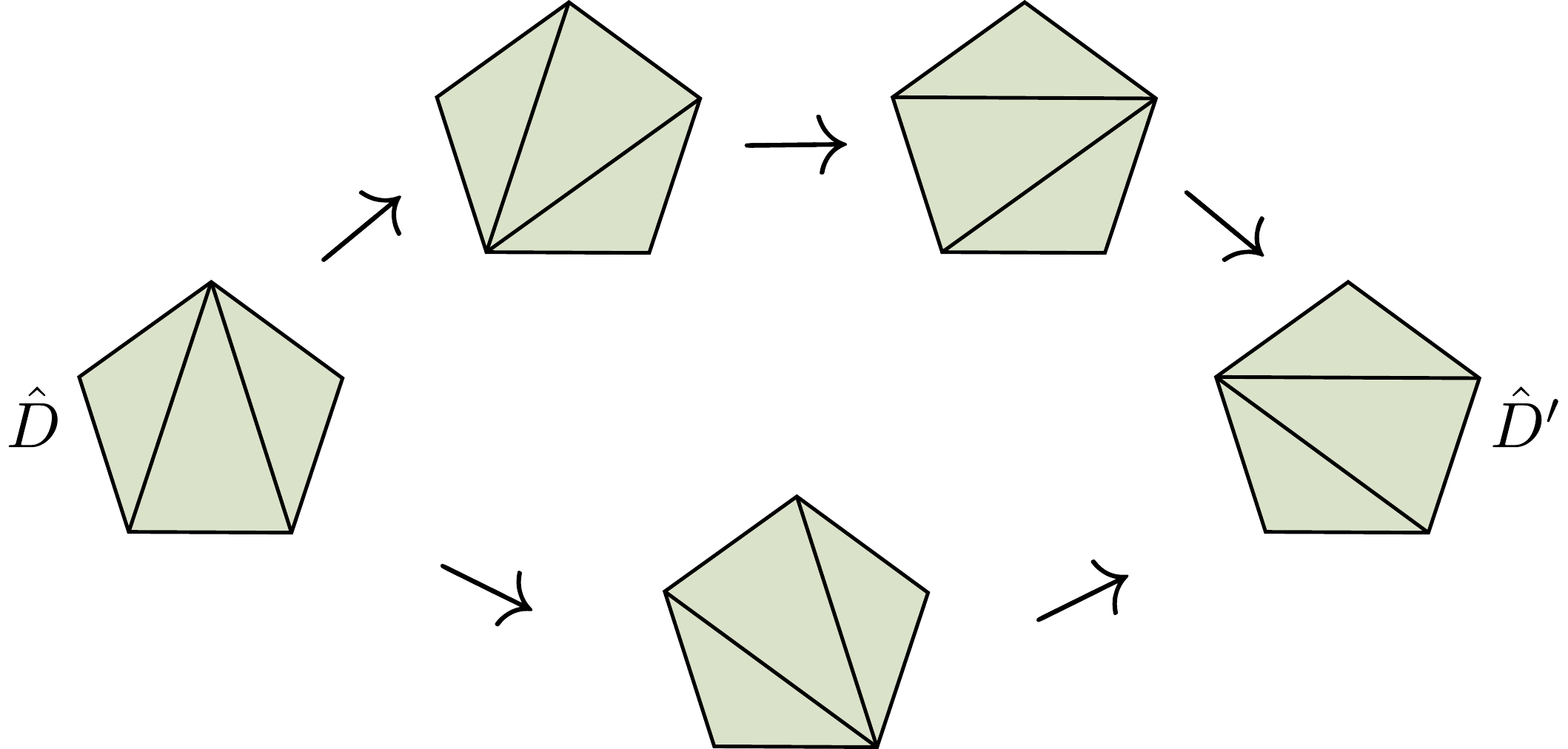}{0.4}{Two sequences of flips in a triangulated pentagon.}

Since both sequences begin at $\hat D$
and end at $\hat D'$, we can follow either sequence to compute the ratio
$\tau_{\hat D'} / \tau_{\hat D}$; but one sequence 
contains $3$ flips and the other contains $2$, so
requiring that they are equal is enough to 
determine the overall multiplicative constant in $k^\eps$.
Carrying this computation out we obtain the following
(the detailed computation can be found in the
file {\tt dilog-compute.nb} included with the arXiv
version of this paper.)

\begin{lem} \label{lem:k-constants}
If we set $p = q = 1$, then
$k^{\eps} = \exp\left(\frac{2 \pi \I}{24} n(\eps) \right)$,
with $n(\eps)$ given below. We specify $\eps$ by the tuple 
$(\eps(0,2), \eps(2,1), \eps(1,3), \eps(3,0), \eps(1,0), \eps(2,3))$.
{\tiny
\vspace{-4pt}
\begin{equation*}
\begin{array}{llll}
 (+1,+1,+1,+1,+1,+1)\to 7 & (+1,-1,+1,+1,+1,+1)\to -11 &
   (-1,+1,+1,+1,+1,+1)\to 7 & (-1,-1,+1,+1,+1,+1)\to 10 \\
 (+1,+1,+1,+1,+1,-1)\to -5 & (+1,-1,+1,+1,+1,-1)\to -8 &
   (-1,+1,+1,+1,+1,-1)\to 10 & (-1,-1,+1,+1,+1,-1)\to -8 \\
 (+1,+1,+1,+1,-1,+1)\to -5 & (+1,-1,+1,+1,-1,+1)\to -2 &
   (-1,+1,+1,+1,-1,+1)\to 10 & (-1,-1,+1,+1,-1,+1)\to 10 \\
 (+1,+1,+1,+1,-1,-1)\to 7 & (+1,-1,+1,+1,-1,-1)\to -11 &
   (-1,+1,+1,+1,-1,-1)\to 1 & (-1,-1,+1,+1,-1,-1)\to -8 \\
 (+1,+1,+1,-1,+1,+1)\to -11 & (+1,-1,+1,-1,+1,+1)\to 1 &
   (-1,+1,+1,-1,+1,+1)\to 10 & (-1,-1,+1,-1,+1,+1)\to 7 \\
 (+1,+1,+1,-1,+1,-1)\to -2 & (+1,-1,+1,-1,+1,-1)\to 1 &
   (-1,+1,+1,-1,+1,-1)\to 10 & (-1,-1,+1,-1,+1,-1)\to 10 \\
 (+1,+1,+1,-1,-1,+1)\to -8 & (+1,-1,+1,-1,-1,+1)\to 1 &
   (-1,+1,+1,-1,-1,+1)\to -8 & (-1,-1,+1,-1,-1,+1)\to 10 \\
 (+1,+1,+1,-1,-1,-1)\to -11 & (+1,-1,+1,-1,-1,-1)\to 1 &
   (-1,+1,+1,-1,-1,-1)\to -8 & (-1,-1,+1,-1,-1,-1)\to 1 \\
 (+1,+1,-1,+1,+1,+1)\to 1 & (+1,-1,-1,+1,+1,+1)\to -8 &
   (-1,+1,-1,+1,+1,+1)\to 1 & (-1,-1,-1,+1,+1,+1)\to -11 \\
 (+1,+1,-1,+1,+1,-1)\to 10 & (+1,-1,-1,+1,+1,-1)\to -8 &
   (-1,+1,-1,+1,+1,-1)\to 1 & (-1,-1,-1,+1,+1,-1)\to -8 \\
 (+1,+1,-1,+1,-1,+1)\to 10 & (+1,-1,-1,+1,-1,+1)\to 10 &
   (-1,+1,-1,+1,-1,+1)\to 1 & (-1,-1,-1,+1,-1,+1)\to -2 \\
 (+1,+1,-1,+1,-1,-1)\to 7 & (+1,-1,-1,+1,-1,-1)\to 10 &
   (-1,+1,-1,+1,-1,-1)\to 1 & (-1,-1,-1,+1,-1,-1)\to -11 \\
 (+1,+1,-1,-1,+1,+1)\to -8 & (+1,-1,-1,-1,+1,+1)\to 1 &
   (-1,+1,-1,-1,+1,+1)\to -11 & (-1,-1,-1,-1,+1,+1)\to 7 \\
 (+1,+1,-1,-1,+1,-1)\to 10 & (+1,-1,-1,-1,+1,-1)\to 10 &
   (-1,+1,-1,-1,+1,-1)\to -2 & (-1,-1,-1,-1,+1,-1)\to -5 \\
 (+1,+1,-1,-1,-1,+1)\to -8 & (+1,-1,-1,-1,-1,+1)\to 10 &
   (-1,+1,-1,-1,-1,+1)\to -8 & (-1,-1,-1,-1,-1,+1)\to -5 \\
 (+1,+1,-1,-1,-1,-1)\to 10 & (+1,-1,-1,-1,-1,-1)\to 7 &
   (-1,+1,-1,-1,-1,-1)\to -11 & (-1,-1,-1,-1,-1,-1)\to 7 \\
\end{array}\end{equation*}
}
\end{lem}

\subsubsection{The final result} \label{sec:final-cs-line}

We summarize our description of the Chern-Simons line, in its most concrete form:
\begin{thm} \label{thm:cs-line-explicit}
Fix a surface $Y$ and a boundary-reduced flat $G$-bundle $P \to Y$.
We consider tuples $\hat D = (\scrT,\eps,s,x,e)$, where:
\begin{itemize}
\item $\scrT$ is a semi-ideal triangulation of $Y$,
\item $\eps = (\eps_E)_{E \in \edges(\scrT)}$ is a system of edge-orientations for $\scrT$,
\item $s = (s_v)_{v \in \vertices(\scrT)}$ 
is a flat section of $P/U = P \times_G (\C^2 \setminus \{0\})$ over each $v$, which when projected to $P/B$ agrees with the boundary reduction,
\item $x = (x_E)_{E \in \edges(\scrT)}$ is a collection of complex numbers, 
where $\exp(x_E) = \eps_E(v,v') s_{v'} \wedge s_{v}$ 
if $E$ is an edge with vertices $v, v'$,
\item $e = (e_\Delta)_{\Delta \in \faces(\scrT)}$ is a system of marked edges, 
where $e_\Delta$ is an edge of $\Delta$.
  \end{itemize}
For each such tuple there is a canonical nonzero element
\begin{equation} \label{eq:taud}
\tau_{\hat D} \in \scrF_G(Y;P)^* \, . 
\end{equation}
They obey relations as follows.
\begin{enumerate}
\item Suppose $\hat D$ and $\hat D'$ differ only by a change involving 
a single vertex $v$:
$s'_v = s_v \exp(t)$
and $x'_E = x_E + t$ for all edges $E$ incident on $v$. 
Then
\begin{equation}
  \tau_{\hat D'} = \tau_{\hat D}.
\end{equation}

\item Suppose $\hat D$ and $\hat D'$ differ only by changing the marking
$e_\Delta$ to $e'_\Delta$ for a single triangle $\Delta$.
Then
\begin{equation} \label{eq:cube-root-triv}
  \tau_{\hat D'} = \exp \left[ \frac{2 \pi \sqrt{-1}}{3} \IP{e_\Delta, e'_\Delta} \right]  \tau_{\hat D} \, .
\end{equation}

\item Suppose $\hat D$ and $\hat D'$ differ only by a change involving
a single edge $E$: $\eps'_E = - \eps_E$ and
$x'_E = x_E + \pi \I$.
Labeling the edges of $\cQ_E$ as in Figure~\ref{fig:quad-labeled}, 
define $u$ by \eqref{eq:u-def-single}.
Then
\begin{equation} \label{eq:tau-b-transition}
  \tau_{\hat D'} = \exp \left[ \frac14 u \right] b({\eps_{\Delta_1},\eps'_{\Delta_1}}) b({\eps_{\Delta_2},\eps'_{\Delta_2}}) \, \tau_{\hat D} \, ,
\end{equation}
where $b(\eps,\eps')$ is given by Proposition~\ref{prop:lambda-explicit}
with $p = q = 1$.

\item Suppose $\hat D$ and $\hat D'$ differ only in that
$\scrT'$ is obtained from $\scrT$ by flipping edge $E$
to obtain a new edge $E'$, with some orientation 
$\eps'_{E'}$ and a choice of logarithm $x_{E'}$.
Suppose that the triangles $\Delta$ abutting $E$ in $\scrT$ both have
$e_\Delta = E$, and the triangles abutting $E'$ in $\scrT'$
both have $e_{\Delta'} = E'$.
Labeling the edges of $\cQ_E$ and $\cQ'_E$ as in Figure~\ref{fig:two-quads}, define $u_1$, $u_2$, $\eta_1$, $\eta_2$ by \eqref{eq:u-def}, \eqref{eq:eta-def}.
Then
\begin{equation} \label{eq:tau-flip-relation}
  \tau_{\hat D'} = \exp \left[ \frac{1}{2 \pi \I} \ell^{\eta}(u_1,u_2) \right] k^{\eps,\eps'} \tau_{\hat D} \, ,
\end{equation}
where $\ell^\eta$ is given in \eqref{eq:lpp}-\eqref{eq:lmm} and 
$k^{\eps,\eps'}$ is given by Lemma \ref{lem:k-constants}.
\end{enumerate}

\end{thm}

\begin{proof} The object $\tau_{\hat D}$
has been defined in \eqref{eq:tauhatD}.
The relations it obeys are obtained from 
Proposition~\ref{prop:line-transitions} by
contracting with $\bigotimes_{\Delta \in \faces(\scrT)} (\phi_{e_\Delta})_* \tau_\eps$.
\end{proof}

Theorem~\ref{thm:cs-line-explicit} closely resembles known patching constructions of a 
prequantum line bundle $L_{cl}$ over a symplectic leaf of the $\SLC$-character variety of a punctured surface. 
In particular, in \cite{FG2} such a line bundle is constructed in the more general
setting of an $X$-cluster variety, using the dilogarithm as a gluing map.
Related constructions appear in the subsequent works \cite{N,APP,BK,CLT}.\footnote{In all of these works the basic dilogarithmic formula for transition functions appears,
but the detailed treatment of the constant factors is somewhat different in each case.}
In all these cases, 
the characteristic property of $L_{cl}$ is that it carries a natural
connection given concretely in terms
of the cluster coordinates on the character variety. From our point of
view, this connection is the one provided by the TFT $\scrF_G$ applied
to families of flat $G$-bundles $P \to Y$, the cluster coordinates
are (up to sign) the holonomies of the corresponding flat $\C^\times$-bundles $Q^\eps_\tw \to \tY$,
and the fact that the connection has a simple expression in cluster coordinates
is obtained by using $\chi_Y^\eps$ to pass from $\scrF_G$ to $\scrS_{\C^\times}$.

\section{Computing CS invariants for flat \texorpdfstring{$\SLC$}{SL(2,C)}-bundles over 3-manifolds} \label{sec:cs-3-manifolds}

Let $M$ be an oriented $3$-manifold with boundary.
Suppose $P \to M$ is a boundary-reduced flat $G$-bundle.
In this section we explain how to use our results to obtain a dilogarithmic
formula for the Chern-Simons invariant $\scrF_G(M;P)$,
Theorem~\ref{thm:cs-invariant} below, assuming that $P$ satisfies a
genericity condition.
This formula 
closely resembles the previously known formulas we recalled in \S\ref{sec:2},
but its precise structure is slightly different, involving several
variants of the dilogarithm function and some extra cube roots of unity,
and not requiring any orderability constraints on the triangulation of $M$.

\subsection{Abelianization of the CS invariant} \label{sec:cs-invariant-abelianization}

We fix data:
\begin{itemize}
\item $\scrT$ a semi-ideal triangulation of $M$, such that 
the genericity Assumption~\ref{thm:74} is satisfied.
\item $\eps = (\eps_E)_{E \in \edges(\scrT)}$ a system of edge-orientations for $\scrT$.
\end{itemize}

According to Construction \ref{thm:36}, 
$\scrT$ determines a spectral network $\cN^\scrT$ on $M$, and $P \to M$
extends to boundary-reduced stratified abelianization data
$\cA = (P, Q, \mu, \theta)$ over $(M,\cN^\scrT)$.

The edge-orientations $\eps$
determine a lift of the double cover $\tM_{\gea} \to M_{\gea}$
to a $\Z$-bundle
over each face of $\scrT$ by Construction \ref{constr:Z-bundle-from-eps}. We can then extend
by radial projection to get a $\Z$-bundle
$\tM^{\eps,\infty} \to M_\gea$. As before, we use
this to twist $Q$ to $Q^\eps_\tw$,
and $\pi^* \sigma$ to $\sigma^\eps_\tw$,
both of which extend over $M_\geb$.

For each tetrahedron $\tet \in \tetrahedra(\scrT)$ we introduce
a small ball $B\stet$ around the barycenter, and define
$H\stet = \partial B\stet$. As we discussed in
\S\ref{sec:8.2},
$H\stet$ is naturally triangulated and has edge-orientations
induced from those of $\tet$.
Then we have
\begin{equation}
  \chi_{H\stet}^{\,\,\eps\stet}\left(\scrF_G(B\stet;P\vert_{B\sstet})\right) \in \scrS_{\C^\times}(H\stet;Q^{\,\,\eps\stet}_\tw \vert_{H\stet};\sigma^{\eps\stet}_\tw) \otimes \cL(H\stet, \eps\stet).
\end{equation}
Moreover, each $\Delta \in \faces(\scrT)$ is pierced by a component
of the branch locus $M_{\twobstrat}$ which connects a triangle
$\Delta_1$ in $H\stetn1$ to another triangle $\Delta_2$ in $H\stetn2$,
and thus gives an element 
\begin{equation}
  \beta_\Delta \in \cL(\Delta_1,\eps) \otimes \cL(\Delta_2,\eps) \, .
\end{equation}
Tensoring over all faces $\Delta$ gives an element
\begin{equation}
  \beta: \bigotimes_{\tet \in \tetrahedra(\scrT)} \cL(\tet,\eps\stet) \to \C \, .
\end{equation}
\begin{prop} \label{prop:cs-invariant-abstract} The Chern-Simons invariant $\scrF_G(M;P)$ is 
\begin{equation}
  \scrF_G(M;P) = \beta \left( \bigotimes_{\tet \in \tetrahedra(\scrT)}  \chi_{H\stet}^{\,\,\eps\stet}\left(\scrF_G(B\stet;P\vert_{B\sstet})\right) \right) \, \otimes \scrS_{\C^\times}\left({M \setminus \cup\stet B\stet};Q^\eps_\tw \vert_{M \setminus \cup\stet B\stet};\sigma^\eps_\tw\right) \, .
\end{equation}
\end{prop}

\begin{proof} Apply Proposition~\ref{prop:cs-interpolating} 
to the 3-manifold $M \setminus \cup\stet B\stet$. This
gives
\begin{equation}
  \chi^\eps_{\cup\sstet -H\sstet} (\scrF_G(M;P)) = \beta_{M \setminus \cup\stet B\stet} \otimes \scrS_{\C^\times}\left({M \setminus \cup\stet B\stet};Q^\eps_\tw \vert_{M \setminus \cup\stet B\stet};\sigma^\eps_\tw\right).
\end{equation}
Then gluing in a factor $\chi_{H\stet}^{\,\,\eps\stet}\left(\scrF_G(B\stet;P\vert_{B\sstet})\right)$ for each $\tet$ gives the desired formula.
\end{proof}

\subsection{Explicit formulas}

To make the formula in Proposition~\ref{prop:cs-invariant-abstract}
more concrete, we fix more data:
\begin{itemize}
\item $s = (s_v)_{v \in \vertices(\scrT)}$ 
a flat section of $P/U = P \times_G (\C^2 \setminus \{0\})$ over each $v$, lying in
the line given by the boundary reduction if $v$ is an ideal vertex,
\item $x = (x_E)_{E \in \edges(\scrT)}$ a collection of complex numbers, 
where $\exp(x_E) = \eps_E(v,v') \, s_{v'} \wedge s_{v}$ 
if $E$ is an edge with vertices $v, v'$.
\item $e = (e\stet)_{\stett \in \tetrahedra(\scrT)}$ a pair of opposite edges 
on each tetrahedron.
\end{itemize}

Note that $e\stet$ induces a marked edge $e(\Delta,\tet)$ for each face
$\Delta$ of $\tet$.
Each face 
$\Delta$ has two abutting tetrahedra $\tet_1$, $\tet_2$, and the marked edges
coming from these two tetrahedra need not agree. 
Suppose we equip $\Delta$
with the boundary orientation induced by $\tet_1$;
then we can consider the pairing
$\IP{e(\Delta,\tet_1),e(\Delta,\tet_2)}$ which measures the mismatch.
Note that this quantity is independent of which tetrahedron we called $\tet_1$,
since that choice enters both in the orientation of $\Delta$ and in the ordering of
the antisymmetric pairing.

Fix a $\tet \in \tetrahedra(\scrT)$, and let $E_1$ denote one of the 
marked edges $e\stet$; then let $E_2$, $E_3$ be defined by following
$E_1$ around a face, with $\IP{E_{i},E_{i+1}} = 1$. Then define
\begin{equation}
  u_i\,\sstett = \sum_{E' \in \edges(\scrT)} \IP{E_i,E'} x_{E'}
\end{equation}
and let $(\eta\stet)_i$ 
be the $\eps$-sign of the edge $E_i$.
Then we have the following concrete formula for the Chern-Simons
invariant:
\begin{thm} \label{thm:cs-invariant} 
The Chern-Simons invariant $\scrF_G(M;P)$ is
\begin{multline*} 
  \scrF_G(M;P) = \left( \prod_{\tet \in \tetrahedra(\scrT)} k^{\eps\sstet} \exp \left[\frac{1}{2 \pi \I} \ell^{\,\eta\sstet}(u_1\,\sstett,u_2\,\sstett) \right] \right) \times \\ \left( \prod_{\Delta \in \faces(\scrT)} \exp \left[ - \frac{2 \pi \I}{3} \IP{e(\Delta,\tet_1) , e(\Delta,\tet_2) } \right] \right) \, ,
\end{multline*}
where $\ell^\eta$ is given in \eqref{eq:lpp}-\eqref{eq:lmm} and 
$k^{\eps}$ is given by Lemma \ref{lem:k-constants}.
\end{thm}

\begin{proof}
As in Construction \ref{constr:triv-from-data}, 
the choice of logarithms $x_E$ determines a section of $Q^\eps_\tw$
over each face; using radial projection in each tetrahedron,
this extends to a section $t$ of $Q^\eps_\tw \to M_\geb$.
Then, using the fact that the $\C^\times$
Chern-Simons form vanishes for a flat bundle, 
we have
\begin{equation} 
\scrS_{\C^\times}\left({M \setminus \cup\stet B\stet};Q^\eps_\tw \vert_{M \setminus \cup\stet B\stet};\sigma^\eps_\tw\right) = \bigotimes_{\tet \in \tetrahedra(\scrT)} \tau^*_{t \vert_{H\sstet}} \, .
\end{equation}
We plug this into the formula of Proposition
\ref{prop:cs-invariant-abstract}, which gives
\begin{equation} \label{eq:cs-invariant-interm}
 \scrF_G(M;P) = \beta \left( \bigotimes_{\tet \in \tetrahedra(\scrT)}  \frac{\chi_{H\sstet}^{\,\,\eps\sstet}\left(\scrF_G(B\stet;P\vert_{B\sstet})\right)}{\tau_{t \vert_{H\sstet}}} \right) \, .
\end{equation}
On each $\tet$, we have
\begin{equation}
  \frac{\chi_{H\sstet}^{\,\,\eps\sstet}\left(\scrF_G(B\stet;P\vert_{B\sstet})\right)}{\tau_{t \vert_{H\sstet}}} = k^{\eps\sstet} \exp \left[\frac{1}{2 \pi \I} \ell^{\,\eta\sstet}(u_1\,\sstett,u_2\,\sstett) \right] \bigotimes_{\Delta \in \faces(\tet)} (\phi_{e(\Delta,\tet)})_*(\tau_{\eps_\Delta})
\end{equation}
Applying $\beta$ to this tensor product pairs up the factors corresponding
to the same face on different tetrahedra.
For a face where the marked edges $e(\Delta,\tet_1)$ and $e(\Delta,\tet_2)$
agree, this pairing just gives $1$,
using \eqref{eq:reflection-condition}.
More generally, using \eqref{eq:tau-rot-equiv}, we see that the 
pairing is
$\exp \left[ \frac{2 \pi \I}{3} \IP{e(\Delta,\tet_1) , e(\Delta,\tet_2) } \right]$. 
Plugging this into \eqref{eq:cs-invariant-interm} gives the desired
result.
\end{proof}

\subsection{An example}

Let $M$ be the
ideally triangulated manifold obtained by gluing 
two copies
$\tet_1$, $\tet_2$ of the standard oriented tetrahedron
across the four triangles $\Delta_1, \dots, \Delta_4$, with each gluing specified
by giving the mapping between three vertices of $\tet_1$
and three vertices of $\tet_2$: \footnote{This manifold is
known as the ``figure-eight sister,'' or {\tt m003} in the SnapPea census. 
To avoid confusion we note that it is not the figure-eight knot complement; the latter is also obtained by gluing two tetrahedra, but it 
does not admit a boundary-unipotent $\SL(2,\C)$-connection --- although it does have boundary-unipotent $\PSL(2,\C)$-connections, e.g. the one induced by the hyperbolic structure. 
See \cite{C} for a general
discussion of the obstruction to lifting a boundary-unipotent $\PSL(2,\C)$-connection to 
a boundary-unipotent $\SL(2,\C)$-connection.
}
\begin{align*}
 \Delta_1:\ & (\tet_1\ 012) \ \leftrightarrow  (\tet_2\ 130)  \\
 \Delta_2:\ & (\tet_1\ 013) \ \leftrightarrow  (\tet_2\ 012)  \\
 \Delta_3:\ & (\tet_1\ 023) \ \leftrightarrow  (\tet_2\ 231)  \\
 \Delta_4:\ & (\tet_1\ 123) \ \leftrightarrow  (\tet_2\ 023)  
\end{align*}
The induced orientations on each face are such that the vertex orderings
shown above are positively oriented on $\tet_1$, negatively oriented on $\tet_2$.
This gluing induces identifications on vertices and edges;
after the gluing there is $1$ vertex and $2$ edges.
Choose edge-orientations so that $\eps_{E_1}$
points from $0 \to 1$ in $\tet_1$ and $\eps_{E_2}$
points from $0 \to 3$ in $\tet_1$.
See Figure~\ref{fig:gluing-figure-eight-sister}.

\insfig{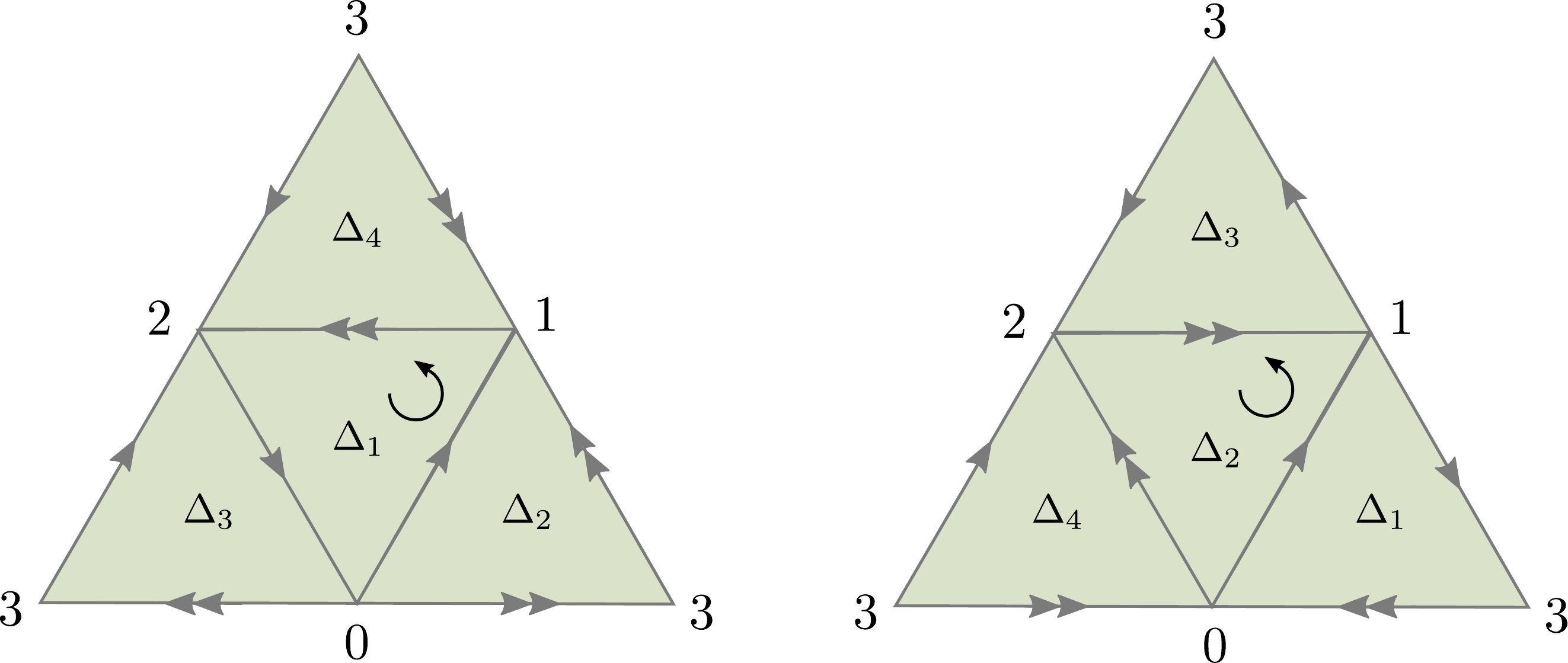}{0.4}{The gluing pattern defining the ideally triangulated 3-manifold $M$. The boundaries of the two tetrahedra are shown, with their
boundary orientations. Edges with single arrows form the equivalence class $E_1$ after gluing; edges with double arrows form the class $E_2$. The direction of the
arrows gives the edge-orientations $\eps_{E_1}$ and $\eps_{E_2}$.}

Suppose given a flat $G$-bundle $P \to M$, with a
section $s$ chosen at the vertex.
Write $X_E = s_v \wedge s_{v'}$ if $E$ is an edge from $v$ to $v'$
and $\eps_E(v,v') = 1$.
The $X_E$ corresponding to the edges of a single tetrahedron are constrained by the ``Ptolemy relation'':\footnote{These relations are exploited systematically in \cite{GGZ1,GGZ2} to parameterize flat bundles over triangulated 3-manifolds.}
if we have four elements $s_0, \dots, s_3$ in a 2-dimensional vector space
then 
\begin{equation}
  (s_0 \wedge s_1)(s_2 \wedge s_3) + (s_0 \wedge s_2)(s_3 \wedge s_1) + (s_0 \wedge s_3)(s_1 \wedge s_2) = 0 \, .
\end{equation}
Applying this to $\tet_1$ we get the relation
\begin{equation} \label{eq:ptolemy-ex}
  - X_1^2 - X_1 X_2 + X_2^2 = 0 \, .
\end{equation}
As it happens, in $\tet_2$ we get the same relation.\footnote{Had we used the figure-eight knot complement instead, we would have gotten a sign flipped here,
with the result that the only common solution of the two equations 
would be $X_1 = X_2 = 0$.}
Assuming that $X_2 \neq 0$,
we may choose the scale of $s$ so that $X_2 = 1$,
and then using \eqref{eq:ptolemy-ex}
\begin{equation} \label{eq:X-sol-1}
  X_1 = \frac{1}{2} (-1 \pm \sqrt{5}) \, , \quad  X_2 = 1 .
\end{equation}
For either choice of the sign in \eqref{eq:X-sol-1},
we can build a flat boundary-reduced $G$-bundle $P \to M$, 
by taking the trivial bundle in each tetrahedron and then
gluing
across each triangle, requiring that the gluing
match up the sections $s_v$ at the three
vertices of the triangle.
In this way we obtain two inequivalent flat boundary-reduced 
$G$-bundles $P \to M$. Moreover, since all $X_E \neq 0$ these $G$-bundles
obey the genericity Assumption \ref{thm:74}. 

In either case, we would like to compute $\scrF_G(M;P)$
using Theorem~\ref{thm:cs-invariant}. For this we have to make a choice
of a pair of marked edges on each tetrahedron. We choose
the edges labeled $01$,$23$.
We also have to choose logarithms $x_i$ of $X_i$.
Then for either $\tet_1$ or $\tet_2$ we have
\begin{equation}
  u_1 = -x_1+x_2, \quad u_2 = 2x_1 - 2x_2, \quad \eta = (+1,+1)
\end{equation}

To be completely explicit let us choose $P \to M$ corresponding to the $-$ sign in
\eqref{eq:X-sol-1} above. Then,
using \eqref{eq:lpp} and \eqref{eq:X-sol-1}, it turns out that\footnote{The fact that we get a closed form expression here comes from the fact that, for $z = -\frac{1+\sqrt{5}}{2}$, 
we have the closed form expression 
$\Li_2(z) = -\frac{\pi^2}{10} - \log(-z)^2$. 
Also, in this example $\ell^\eta(u_1,u_2)$ is
actually independent $ (\text{mod}\,4\pi^2 \Z)$ 
of the choice of logarithms $x_1$, $x_2$.
Both of these are exceptional properties, which do not hold in most examples; more typically, the final result for $\scrF_G(M;P)$ involves a sum of $\Li_2(z_i)$ for algebraic numbers $z_i$,
and changing the choice of logarithms
changes the contributions from individual tetrahedra while leaving
the final result unchanged.}
\begin{equation} \label{eq:one-tet-cs-1}
  \ell^{(+1,+1)}(u_1,u_2) = \frac{9 \pi^2}{20}\ (\text{mod}\,4\pi^2 \Z), \qquad
  \exp \left( \frac{1}{2 \pi \sqrt{-1}} \ell^{(+1,+1)}(u_1,u_2) \right) = \exp \left( - \frac{9 \pi \sqrt{-1}}{20} \right) \, .
\end{equation}
Next we look up from Lemma \ref{lem:k-constants} that for
either tetrahedron
we have
\begin{equation} \label{eq:one-tet-k-1}
  k^{\eps\sstet} = \exp\left(\frac{7 \pi \sqrt{-1}}{12}\right) \, .
\end{equation}
Finally, we have to include the cube roots of unity
from face gluings. For instance, on face $\Delta_1$
we see from the face gluings above that the edge $e(\Delta_1, \tet_1)$ 
(numbered $01$ on $\tet_1$)
and the edge $e(\Delta_1, \tet_2)$ (numbered $01$ on $\tet_2$)
have $\IP{e(\Delta_1, \tet_1),e(\Delta_1, \tet_2)} = 1$.
Computing similarly for the other three faces, we get the overall factor
\begin{equation} \label{eq:cube-roots-1}
  \exp\left(\frac{2 \pi \sqrt{-1}}{3} (1+0+1+0) \right) = \exp \left(- \frac{2 \pi \sqrt{-1}}{3} \right) \, .
\end{equation}
Combining the factors \eqref{eq:one-tet-cs-1} (for each tetrahedron), \eqref{eq:one-tet-k-1} (for each tetrahedron), and \eqref{eq:cube-roots-1} we get
\begin{equation}
 \scrF_G(M;P) = \exp \left( \pi \sqrt{-1} \left( -\frac{9}{10} + \frac{7}{6} - \frac{2}{3} \right) \right) = \exp \left( - \frac{2 \pi \sqrt{-1}}{5} \right) \, .
\end{equation}

If we take $P \to M$ corresponding to the $+$ sign in
\eqref{eq:X-sol-1}, then
\eqref{eq:one-tet-cs-1} is replaced by
\begin{equation}
  \ell^{(+1,+1)}(u_1,u_2) = \frac{\pi^2}{20}\ (\text{mod}\,4\pi^2 \Z), \qquad
  \exp \left( \frac{1}{2 \pi \sqrt{-1}} \ell^{(+1,+1)}(u_1,u_2) \right) = \exp \left( - \frac{\pi \sqrt{-1}}{20} \right)   \, ,
\end{equation}
which when combined with the other factors gives
\begin{equation}
 \scrF_G(M;P) = \exp \left( \pi \sqrt{-1} \left( -\frac{1}{10} + \frac{7}{6} - \frac{2}{3} \right) \right) = \exp \left( \frac{2 \pi \sqrt{-1}}{5} \right) \, .
\end{equation}

\subsection{Another example}

Now let $M$ be the manifold called {\tt m071} in the SnapPea census.
$M$ admits an ideal triangulation $\scrT$ with 5 tetrahedra, 10 triangles, 5 edges
and 1 vertex.
Applying Theorem~\ref{thm:cs-invariant}
in this example we obtain the Chern-Simons invariants 
$\scrF_G(M;P)$ for $7$ inequivalent boundary-reduced flat
$G$-bundles $P \to M$:
numerically they are approximately
\begin{equation*}  
\begin{array}{ccc}
0.697849 + 0.716244 \I & 
-0.99614 + 0.0877733 \I &
-0.26787 + 0.963455 \I \\
-0.948968 - 0.315372 \I &  
0.982867 + 0.184319 \I &
1.51835 + 0.0629475 \I \\
0.65748 + 0.0272577 \I 
\end{array}
\end{equation*}
Details of this computation 
(along with a few others) are given
in the Mathematica notebook {\tt dilog-compute.nb},
available as an ancillary file included with the
arXiv preprint version of this paper.
In making these computations we made use of the software SnapPy \cite{CDGW} and Regina
\cite{BBP}.

   \section{Future directions}\label{sec:11}

\ \,We conclude with a brief description of possible avenues for further
exploration.   
 
\bigskip

 \begin{enumerate}[label=\textnormal{\arabic*. }, leftmargin=*]

 \item Although some basic setup in~\S\ref{sec:4} and~\S\ref{sec:5} applies
to all rank one complex Lie groups, our main results are proved here only
for flat $\SLC$-connections.  One should be able to extend them to the
groups~$\GLC$ and $\PSLC$.

 \item More ambitious is an extension to higher rank complex Lie groups.
Spectral networks in 2~dimensions for higher rank groups are discussed
in~\cite{GMN1,GMN2,LP,IM}.  Some clues for spectral networks in 3~dimensions
can be found in~\cite{GTZ,DGG}.

 \item The flat $\SLC$-connections to which we apply stratified
abelianization are assumed to be boundary-unipotent, this for both 2-~and
3-dimensional compact manifolds with boundary.  The expectation is that the
boundary-unipotent assumption could be dropped at the cost of having
additional lines or $\VV$-lines associated to boundary components.

 \item The 3-dimensional spectral networks in this paper are either induced
from a triangulation or for the Cartesian product of a triangulated
2-manifold and a closed interval (\S\ref{sec:8.2}).  According to~\cite{GMN3}
one obtains a 3-dimensional spectral network from a 1-parameter family of
holomorphic quadratic differentials on a Riemann surface.  This recovers the
spectral network of~\S\ref{sec:8.2}, but one can also get another transition
of triangulations called a ``juggle''.  It would be interesting to study this
kind of 3-dimensional spectral network.

 \item One can strive to develop spectral networks (rank one or higher rank)
as a topological structure on a smooth manifold, much like an orientation or
spin structure.  If so, then one could define bordism categories and bordism
spectra as domains of field theories on manifolds equipped with a spectral
network and stratified abelianization.  This should lead to more powerful
theorems.  Perhaps, then, our main results could be stated and proved as
parts of an isomorphism of field theories.

 \item One of the most important aspects of 2-dimensional spectral networks
is the relationship to the WKB analysis of ordinary differential
equations; see e.g.~\cite{KaTa,GMN3} and references therein. One can inquire whether 3-dimensional spectral
networks have a similar relationship to differential equations.
 
 \item We give a construction of the Chern-Simons line of a flat bundle
on a 2-manifold via stratified abelianization.  In a parametrized family of
flat bundles, this gives a construction of the Chern-Simons line bundle
with its covariant derivative.  This is the same construction that appears
in~\cite{FG2,APP,N,CLT}, and therefore the line bundle constructed in those papers
must be the Chern-Simons line bundle.  One can imagine that this
identification of the line bundle will lead to new insights and results.

 \item 2d spectral networks on a Riemann surface can be constructed in terms of 
 BPS particles in a corresponding supersymmetric quantum field theory
 \cite{GMN1}. It would be very interesting to know whether one can understand
 3d spectral networks in a similar way; candidates for the corresponding quantum
 field theory are known in the physics literature, beginning with
 \cite{DGGu}.

 \item 2d spectral networks can be used to construct the quantum trace map (or $q$-nonabelianization map, or quantum UV-IR map) of \cite{BW},
 as a homomorphism from the $\fgl(2)$ skein algebra of a surface $Y$ to the $\fgl(1)$ skein
 algebra of a double cover $\tY$ \cite{G,NY}.
 It would be interesting to use the 3d spectral networks defined in this paper
 to construct a 3-manifold version of the quantum trace, mapping the $\fgl(2)$ skein module of a 
 3-manifold $M$ to an $\fgl(1)$ skein module of a branched double cover $\tM$. According
 to the recent proposal of \cite{AGLR}, such a 3-manifold quantum trace is one of the 
 necessary ingredients in the formulation of a ``length conjecture'' for link invariants.

 \end{enumerate}

\appendix

   \section{Ordinary differential cochains}\label{sec:9}

Differential cohomology descends on the one hand from the differential
characters of Cheeger-Simons~\cite{ChS} and on the other from Deligne
cohomology~\cite{De} in algebraic geometry.  The ideas for an explicit model
are suggested in~\cite[\S6.3]{DF}.  Hopkins and Singer~\cite[\S3]{HS} develop
the model that we recall in this appendix.  The theory of \emph{generalized
differential cohomology} (and differential cocycles) has been further
developed in many works, for example~\cite{SS,BNV,ADH}.  In this brief
appendix we also include a few complements needed for Chern-Simons theory.

  \subsection{Cochain model}\label{subsec:9.1}

Let $M$~be a smooth manifold, $A$~an abelian group, and $C^p(M;A)$ the group
of singular $p$-cochains with values in the abelian group~$A$.  For
each~$q\in \ZZ^{\ge0}$ define a cochain complex~$\cCc q\bullet$ by
  \begin{equation}\label{eq:108}
     \cCc qp(M) = \begin{cases} C^p(M;\ZZ)\times C^{p-1}(M;\RR)\times \Omega
     ^p(M),&p\ge q;\\C^p(M;\ZZ)\times C^{p-1}(M;\RR),&p<q,\end{cases} 
  \end{equation}
with differential acting on $(c,h,\omega )\in \cCc qp(M)$, $p\ge q$, or on
$(c,h)\in \cCc qp(M)$, $p<q$, given by
  \begin{equation}\label{eq:109}
     \begin{aligned} d(c,h,\omega ) &=\;\;\;(\delta c\,,\,\omega -c-\delta
     h\,,\, 
      d\omega ),\qquad p\ge q; \\ d(c,h) &= \begin{cases} (\delta c,-c-\delta
      h,0),&\qquad \quad \;\;p=q-1;\\(\delta c,-c-\delta h),&\qquad \quad
     \;\;p<q-1.\end{cases}  
 \end{aligned}
  \end{equation}
The cohomology of the cochain complex $\cCc q\bullet$ is 
  \begin{equation}\label{eq:110}
     \cH(q)^p(M)\cong \begin{cases}
     H^p(M;\ZZ),&p>q;\\H^p(M;\RZ),&p<q,\end{cases} 
  \end{equation}
and the diagonal group $\cH(q)^q(M)$ is the \emph{differential cohomology}.
Since we have great use for the diagonal groups, we introduce the notations
  \begin{equation}\label{eq:111} 
  \begin{aligned}
      \cCd q(M) &:= \cCc qq(M) \\
      \cZd q(M) &:= \cZ qq(M) \\
      \cH^q(M) &:= \cH(q)^q(M) ,
  \end{aligned}
  \end{equation}
where $\cZ qp(M)\subset \cCc qp(M)$ is the subgroup of cocycles.  The
group~$\cH^q(M)$ is isomorphic to the group of Cheeger-Simons differential
characters on~$M$.  It can be given~\cite{BSS}, \cite[\S6]{DGRW} the
structure of an infinite dimensional abelian Lie group with Lie algebra and
homotopy groups
  \begin{equation}\label{eq:112}
     \begin{aligned} \Lie \cH^q(M)&\cong \frac{\Omega ^{q-1}(M)}{d\Omega
      ^{q-2}(M)}, \\[6pt] \pi _k\cH^q(M) &\cong \begin{cases}
     H^q(M;\ZZ),&k=0;\\H^{q-1}(M;\ZZ),&k=1;\\ 0,&k\ge2\end{cases}
       \\ \end{aligned} 
  \end{equation}

  \begin{remark}[]\label{thm:45}
 \

 \begin{enumerate}

 \item The case~$q=2$ is instructive: $\cH^2(M)$ is isomorphic to the group
of isomorphism classes of principal $\RZ$-bundles with connection over~$M$.

 \item In the main text we use complex differential forms and complex
singular cochains in~\eqref{eq:108}, for which we use the notation `$\cCC
q^p(M)$'.

 \item The construction is functorial for smooth maps, so can be phrased in
terms of simplicial sheaves on manifolds~\cite{FH2,BNV,ADH}.

 \item There are alternative models which replace singular cochains with
integer coefficients by other models of cochains or by maps to
Eilenberg-MacLane spaces.  The latter approach is used to define differential
versions of generalized cohomology theories, as in~\cite[\S4]{HS}.

 \end{enumerate}
  \end{remark}

  \subsection{Curvature, characteristic class, and nonflat trivializations}\label{subsec:9.2}

The projection
  \begin{equation}\label{eq:113}
     \cCd q(M)\longrightarrow \Omega ^q(M) 
  \end{equation}
is called the \emph{curvature} or \emph{covariant derivative} map, depending
on the context.  The restriction of~\eqref{eq:113} to differential {cocycles}
$\cZd q(M)\subset \cCd q(M)$ factors through $\cH^q(M)$ and has image the
subgroup of closed differential forms with integral periods.  There is also a
\emph{characteristic class} homomorphism
  \begin{equation}\label{eq:114}
     \cH^q(M)\longrightarrow H^q(M;\ZZ). 
  \end{equation}
 
The short exact sequence 
  \begin{equation}\label{eq:115}
     0\longrightarrow \cCc q{\bullet }(M)\xrightarrow{\;\;i\;\;}
     \cCc{q-1}{\bullet }(M)\longrightarrow \Omega
     ^{q-1}(M)[1-q]\longrightarrow 0 
  \end{equation}
induces a long exact sequence in cohomology: 
  \begin{equation}\label{eq:116}
     \dots\longrightarrow \cH^{q-1}(M)\longrightarrow \Omega
     ^{q-1}(M)\longrightarrow \cH^q(M)\xrightarrow{\;\;i\;\;}
     H^q(M;\ZZ)\longrightarrow 0 
  \end{equation}
Let $\co\in \cZd q(M)$ be a differential cocycle.   

  \begin{definition}[]\label{thm:46}
 A cochain $\ct\in \cCd{q-1}(M)$ is a \emph{nonflat\footnote{`Nonflat' could
be replaced by the more accurate `not necessarily flat'.} trivialization}
of~$\co$ if $d\ct=i(\co)$ holds in $\cCc{q-1}q(M)$.
  \end{definition}

\noindent
 By~\eqref{eq:116}, a nonflat trivialization~$\ct$ produces a differential
form $\tau \in \Omega ^{q-1}(M)$, the \emph{covariant derivative} of~$\ct$.
(It is the image of~$\ct$ under~\eqref{eq:113}.)  For~$q=2$ we can
represent~$\co$ as a complex line bundle $\pi \:L\to M$ with connection.
Then a nonflat trivialization is a section of~$\pi $, and its covariant
derivative is as usual.

  \begin{remark}[]\label{thm:49}
 See \cite[Definition~5.12]{F3} for a model of nonflat trivializations which
works in generalized differential cohomology theories.
  \end{remark}

  \subsection{Higher Picard groupoids of differential cocycles}\label{subsec:9.3}

As a warmup, recall that a cochain complex 
  \begin{equation}\label{eq:117}
     0\longrightarrow A^0\xrightarrow{\;\;d\;\;}
     A^1\xrightarrow{\;\;d\;\;}\cdots 
  \end{equation}
gives rise to a sequence of higher Picard groupoids.  For $q\in \ZZ^{\ge0}$,
form the truncated complex 
  \begin{equation}\label{eq:118}
     0\longrightarrow
     A^0\xrightarrow{\;\;d\;\;}A^1\xrightarrow{\;\;d\;\;}\cdots
     \xrightarrow{\;\;d\;\;} A^q_{\textnormal{closed}}\longrightarrow 0 
  \end{equation}
in which $A^q_{\textnormal{closed}}\subset A^q$ is the subgroup of closed
elements.  Define the Picard $q$-groupoid~$\Gq$ as follows.  The set of
objects is~$A^q_{\textnormal{closed}}$.  The
$A^{q-1}_{\textnormal{closed}}$-torsor of 1-morphisms between $a_0,a_1\in
A^q_{\textnormal{closed}}$ is
  \begin{equation}\label{eq:119}
     \Hom\mstrut _1(a_0,a_1) = \bigl\{ a'\in A^{q-1}: a_0 + da' = a_1 \bigr\}. 
  \end{equation}
The $A^{q-2}_{\textnormal{closed}}$-torsor of 2-morphisms between
$a_0',a_1'\in \Hom\mstrut _1(a_0,a_1)$ is
  \begin{equation}\label{eq:120}
     \Hom\mstrut _2(a'_0,a'_1) = \bigl\{ a''\in A^{q-2}: a_0' + da'' = a'_1
     \bigr\},
  \end{equation}
and so on.  The homotopy groups of~$\Gq$ are
  \begin{equation}\label{eq:121}
     \pi _i\Gq = \begin{cases}
     H^{q-i}(A),&i=0,\dots ,q;\\[3pt]
     0,&i>q.\end{cases} 
  \end{equation}
 
A variation of this construction produces higher Picard groupoids out of
differential cohomology on a smooth manifold~$M$.  Define the Picard
$q$-groupoid $\cGq M$ to have as its set of objects $\cZd q(M)\subset \cCd
q(M)$, the set of differential cocycles of degree~$q$.  Then for
$\co_0,\co_1\in \cZd q(M)$, define
  \begin{equation}\label{eq:122}
     \Hom\mstrut _1(\co_0,\co_1) = \bigl\{ \ct\in \cCd{q-1}(M): i(\co_0) +
     d\ct = i(\co_1) \bigr\},
  \end{equation}
where as in Definition~\ref{thm:46} the equation $i(\co_0) + d\ct = i(\co_1)$
lies in $\cCc{q-1}q(M)$.  In other words, $\ct\in \Hom\mstrut _1(\co_0,\co_1)$
is a nonflat trivialization of $\co_1-\co_0$.  Now for $\ct_0,\ct_1\in
\Hom\mstrut _1(\co_0,\co_1)$, define
  \begin{equation}\label{eq:123}
     \Hom\mstrut _2(\ct_0,\ct_1) = \bigl\{ \cl\in \cCd{q-2}(M): i(\ct_0) +
     d\cl = i(\ct_1) \bigr\}. 
  \end{equation}
The homotopy groups of $\cGq M$ are the differential cohomology groups:
  \begin{equation}\label{eq:124}
     \pi _i\cGq M= \begin{cases}
     \cH^{q-i}(M),&i=0,\dots ,q;\\[3pt]0,&i>q.\end{cases} 
  \end{equation}

  \begin{remark}[]\label{thm:47}
 In low degrees the $\cG qM$ have alternative, more geometric, presentations.
For example, $\cG 2M$~is equivalent to the following Picard 2-groupoid: an
object is a principal $\RZ$-bundle $P\to M$ with connection~$\Theta $; a
morphism $(P_0,\Theta _0)\to (P_1,\Theta _1)$ is a section~$s$ of $P\mstrut
_1\otimes P_0\inv \to M$; and a 2-morphism $s_0\to s_1$ is a function
$f\:M\to \RR$ such that $s_0 + \overline f=s_1$, where $\overline f\:M\to
\RZ$ is the mod~$\ZZ$ reduction of~$f$.
  \end{remark}

  \subsection{Integration and Stokes' theorem}\label{subsec:9.4}

We repeat and expand upon material from~\cite[\S3.4]{HS}, slightly
specialized.  Let $p\:M\to S$ be a proper fiber bundle whose base~ $S$ is a
smooth manifold and whose total space~ $M$ is a smooth manifold with
boundary; the fibers of~$p$ are smooth manifolds with boundary.  Suppose an
$\cH$-orientation is given; see \cite[\S2.4]{HS}.  Then if the fibers have
dimension~$k$, integration is a homomorphism
  \begin{equation}\label{eq:125}
     \int_{M/S}\:\cCc qp(M)\longrightarrow \cCc{q-k}{p-k}(S). 
  \end{equation}
We also have the usual integration of differential forms, and the diagram 
  \begin{equation}\label{eq:126}
     \begin{gathered} \xymatrix{\cCc qp(M)\ar[r]^{} \ar[d]_{\int_{M/S}} &
     \Omega ^p(M)\ar[d]^{\int_{M/S}} \\ \cCc{q-k}{p-k}(S)\ar[r]^{} & \Omega
     ^{p-k}(S) } \end{gathered} 
  \end{equation}
commutes.  Also the integration map~\eqref{eq:125} commutes with base
change. 
 
We state a version of Stokes' theorem in this context.  It concerns the
case~$p=q$ in~\eqref{eq:125}, restricted to differential
cocycles.\footnote{We do not vouch for the signs.}

  \begin{theorem}[]\label{thm:48}
 Let $\co\in \cZd q(M)$ be a differential cocycle with curvature $\omega \in
\Omega ^q(M)$.  Then the integral $\int_{M/Z}\co\in \cCd{q-k}(S)$ is a nonflat
trivialization of $\int_{\partial M/S}\co\in \cZd{q-k+1}(S)$ with
covariant derivative $\int_{M/S}\omega \in \Omega ^{q-k}(S)$.  
  \end{theorem}

\noindent
 See~\cite[\S3.4]{HS} for a more general theorem and proof.

We also need a generalization of Theorem~\ref{thm:48} to manifolds with
corners.  Here we state the next simplest version---corners of codimension at
most~2---and leave to the reader the more general version.  For convenience
we use manifolds with corners equipped with the extra structure needed for
objects and morphisms in bordism multicategories, as in \cite[\S A.2]{FT1}.
The data of such a manifold~$M$ of dimension~$k$ and depth~$\le d$ includes
manifolds~$M\corn\delta j$ with corners of depth $\le d-j$, $\delta \in
\{0,1\}$, $j\in \{1,\dots ,d\}$, and embeddings $[0,1]^{j-1}\times
M\corn\delta j\to \partial M$.  The data is designed to fit the formalism of
multicategories.  For example, if $d=2$ then a 2-categorical interpretation
has objects $M\corn02, M\corn12$; 1-morphisms $M\corn01,M\corn11\:M\corn02\to
M\corn12$; and $M$~itself is a 2-morphism $M\corn01\to M\corn11$. 
 
Let $p\:M\to S$ be a proper fiber bundle whose base~$S$ is a smooth manifold
and whose total space~$M$ is a manifold with corners of depth~$\le2$.  Assume
that $M$~carries the extra structure of \cite[\S A.2]{FT1}, fibered over~$S$.
In particular, there are fiber bundles $p\corn\delta j\:M\corn\delta j\to S$,
$\delta \in \{0,1\}$, $j\in \{1,2\}$.  Let $\co\in \cZd q(M)$ be a
differential cocycle with curvature $\omega \in \Omega ^q(M)$.  Define 
  \begin{equation}\label{eq:131}
     \begin{aligned} \ce\corp\delta j&=\int_{M\corn\delta j/S}\co\; &&\in
      \cCd{q-k+j}(S), \\ \eta\corp\delta j&=\int_{M\corn\delta j/S}\omega\;
      &&\in \Omega^{q-k+j}(S). \\ \end{aligned} 
  \end{equation}
These formulas pertain for $j\in \{1,2\}$, $\delta \in \{0,1\}$; for~$j=0$
omit~$\delta $.  

  \begin{theorem}[]\label{thm:52}
 \ 
 \begin{enumerate}[label=\textnormal{(\arabic*)}]

 \item $\ce\corp\delta 1$~is a nonflat trivialization of $\ce\corp12-\ce\corp02$
with covariant derivative~$\eta \corp\delta 1$.

 \item $\ce\mstrut _0$~is a nonflat trivialization of
$\ce\corp11-\ce\corp01$ with covariant derivative~\,$\eta \mstrut _0$.

 \end{enumerate}
  \end{theorem}
\noindent 
 The diagram 
  \begin{equation}\label{eq:132}
  \begin{tikzcd}[column sep=large]
  \ce\corp02 \arrow[r,bend left=35,"\ce\corp11"{name=B},""{below,name=D}]
\arrow[r,bend right=35,"\ce\corp01"{below,name=F},""{above,name=C}] 
 & \ce\corp12   
 \arrow[Rightarrow, from=C, to=D, "\ce_0"{right}]
 \end{tikzcd}
  \end{equation}
captures some of Theorem~\ref{thm:52}.  In the terms of~\S\ref{subsec:9.3},
the integral of~$\co$ over~$M_0/S$ is a 2-morphism in~$\cG{q-k+2}S$.  We
leave generalizations to greater depths to the reader.

We also use a generalization of Theorem~\ref{thm:52} for the integral of a
1-morphism in~$\cGq M$.  For simplicity, consider a 1-morphism of the form
$\widecheck{0}\xrightarrow{\;\ct\;}\co$.  In other words, $\ct\in
\cCd{q-1}(M)$ is a nonflat trivialization of $\co\in \cZd q(M)$.  Let $\tau
\in \Omega ^{q-1}(M)$ be the covariant derivative of~$\ct$.  As
in~\eqref{eq:131}, set 
  \begin{equation}\label{eq:156}
     \begin{aligned} \cs\corp\delta j&=\int_{M\corn\delta j/S}\ct\; &&\in
      \cCd{q-k+j-1}(S), \\ \sigma\corp\delta j&=\int_{M\corn\delta j/S}\tau\;
      &&\in \Omega^{q-k+j-1}(S). \\ \end{aligned} 
  \end{equation}

  \begin{theorem}[]\label{thm:60}
  \ 
 \begin{enumerate}[label=\textnormal{(\arabic*)}]

 \item $\cs\corp\delta 2$~is a nonflat trivialization of $\ce\corp\delta 2$
with covariant derivative~$\sigma \corp\delta 2$.

 \item $\cs\corp\delta 1\:\,\cs\corp12-\cs\corp02\,\to\, \ce\corp\delta 1$~is a
nonflat isomorphism with covariant derivative~$\sigma \corp\delta 1$.

 \item $\cs\mstrut _0 \:\,\cs\corp11-\cs\corp01\,\to\, \ce\mstrut _ 0$ is a
nonflat isomorphism with covariant derivative $\sigma \mstrut _0$.

 \end{enumerate} 
  \end{theorem}

\noindent
 The data of $\ce\corp\delta 2,\ce\corp\delta 1,\ce\mstrut
_{0},\cs\corp\delta 2,\cs\corp\delta 1,\cs\mstrut _0$ assemble to a
3-morphism in $\cG{q-k+2}S$.  We depict all but~$\cs\mstrut _0$ in~
\eqref{eq:157}.  
  \begin{equation}\label{eq:157}
 \begin{tikzcd}[column sep=large]
 \cz\arrow[r,"\cs\corp02"] \arrow[rrr, bend left=60, "\cz", ""{below, name=A}]
 \arrow[rrr, bend right=60, "\cz"{below}, ""{name=E}]
 & \ce\corp02 \arrow[r,bend left=35,"\ce\corp11"{name=B},""{below,name=D}]
\arrow[r,bend right=35,"\ce\corp01"{below,name=F},""{above,name=C}] 
 & \ce\corp12  
 & \cz\arrow[l,"\cs\corp12"{above}]  
 \arrow[Rightarrow, from=A, to=B, "\cs\corp11"{right}]
 \arrow[Rightarrow, from=C, to=D, "\ce_0"{right}]
 \arrow[Rightarrow, from=E, to=F, "\cs\corp01"{right}]
 \end{tikzcd}
  \end{equation}
Again, we leave to the reader generalizations of Theorem~\ref{thm:60} to
greater depths.  

  \begin{remark}[]\label{thm:61}
 Suppose~$q=k$, so that \eqref{eq:157}~is a diagram in~$\cG2S$.  Then as in
Remark~\ref{thm:47} we interpret~$\ce\corp\delta 2$ as a principal
$\RZ$-bundle $\pi ^\delta \:P^\delta \to S$ with connection, $\ce\corp\delta
1$~as a nonflat isomorphism $\pi ^0\to \pi ^1$, and $\ce_0$~as a function
$S\to \RR$ whose reduction mod~$\ZZ$ maps the isomorphism~$\ce\corp01$ to the
isomorphism~$\ce\corp11$.  Now $\cs\corp\delta 2$~is a nonflat section
of~$\pi ^\delta $.  The isomorphism $\ce\corp\delta 1\:\pi ^0\to \pi^1$
maps~$\cs\corp02$ to a section of~$\pi ^1$, and $\cs\corp\delta 1$~is a
function $S\to \RR$ whose mod~$\ZZ$ reduction maps the section~$\cs\corp12$
to the section~$\ce\corp\delta 1(\cs\corp02)$.  Finally, $\cs_0=0$
trivially---there are no nonzero 3-morphisms in~$\cG2S$---which means that
the difference of the functions $\cs\corp11 - \cs\corp01$ equals the
function~$\ce_0$, as is evident from the foregoing. 
  \end{remark}

   \section{Invertible field theories}\label{sec:10}
 
We briefly outline some general facts about invertible field theories,
including those which are not flat, hence not topological.  For simplicity we
confine our exposition to a setting which applies to the Chern-Simons we
encounter in~\S\ref{subsubsec:5.5.4}.  Invertible field theories which are
topological in a \emph{restricted} sense (which applies to evaluation on a
single manifold, as for example is true for Chern-Simons theory~\eqref{eq:18}
on flat connections) or are topological in a \emph{strong} sense (which
applies to evaluation in families) are modeled as maps of spectra in
topology; see~\cite{FHT,FH1,F5} and the references therein.  Here we use
\emph{generalized} differential cohomology, for which we need a model of
``cocycles'' which generalize the singular cocycles used in
Appendix~\ref{sec:9}.  For background and detailed development of topics in
this appendix, see \cite{FH2,HS,BNV,ADH} and the references therein.
 
Let $G$~be a Lie group.  Let $\Man$~denote the category of smooth manifolds
and smooth maps between them, and let $\sSet$~be the category of simplicial
sets.  Then
  \begin{equation}\label{eq:159}
     \BNG\:\Man^{\textnormal{op}}\longrightarrow \sSet 
  \end{equation}
is the simplicial sheaf which assigns to a test manifold~$S$ the nerve of the
groupoid of $G$-connections on~$S$; see \cite[Example~5.11]{FH2}.
Similarly,  
  \begin{equation}\label{eq:166}
     \ENG\:\Man^{\textnormal{op}}\longrightarrow \sSet 
  \end{equation}
is the simplicial sheaf which assigns to a test manifold~$S$ the nerve of the
groupoid of $G$-connections on~$S$ and a trivialization of the underlying
$G$-bundle.  Then $\ENG$~is equivalent to the set-valued sheaf $\Omega
^1\otimes \mathfrak{g}$ which assigns to a test manifold~$S$ the vector
space~$\Omega ^1_S(\mathfrak{g})$, where $\mathfrak{g}$~is the Lie algebra
of~$G$; see \cite[Example~5.14]{FH2}.  A smooth manifold~$M$ defines a
representable set-valued sheaf on~$\Man$; its value on a test manifold~$S$ is
the set of smooth maps $S\to M$.

Let $\hb$~be a generalized cohomology theory.  Define the $\ZZ$-graded real
vector space 
  \begin{equation}\label{eq:179}
     V_h^{\bullet }=h^{\bullet }(\pt)\otimes \RR, 
  \end{equation}
and suppose given an isomorphism of cohomology theories $h^{\bullet }\otimes
\RR\xrightarrow{\;\cong \;}HV_h^{\bullet }$, where the codomain is the
Eilenberg-MacLane theory with $HV_h^{\bullet }(\pt)=V_h^{\bullet }$.  There
is then a differential cohomology theory~$\chb$ (of ``Hopkins-Singer type'')
which refines the topological theory~$\hb$.  Furthermore, these theories---as
well as the de Rham complex---can be evaluated on simplicial sheaves, in
particular on~$\BNG$.  The Anderson dual~$\IZ^{\bullet }$ to the sphere is a
universal choice for the codomain of an invertible field theory.
Nonetheless, for Chern-Simons theory in this context it is more convenient to
use a truncation~$E^{\bullet }$, the cohomology theory introduced
in~\S\ref{subsubsec:5.2.1}, as the codomain.

  \begin{remark}[]\label{thm:64}
 Chern-Simons theory has not only a $G$-connection as a background field, but
also a spin structure.  In the formalism sketched here we do not treat them
symmetrically: we use the spin structure to integrate differential
$E$-cohomology classes (and cochains).  By contrast, in \emph{topological}
invertible theories we usually do treat them symmetrically and use the
universal codomain~$\IZ^{\bullet }$.
  \end{remark}

An $n$-dimensional invertible field theory on $G$-connections is modeled by a
map 
  \begin{align}
     \alpha \:\BNG&\longrightarrow \widecheck{h}^{n+1}. \label{eq:160}\\
\intertext{The theory~$\alpha $ has an underlying \emph{cohomology class}}
     \BNG&\longrightarrow h^{n+1} \label{eq:161}\\
\intertext{and \emph{curvature}}
     \BNG&\longrightarrow (\Omega \otimes V_h)^{n+1} \label{eq:162}
  \end{align}
of total degree~$n+1$.  (For $E$-cohomology\footnote{In the main text we use
a \emph{complexified} version of differential $E$-cohomology.} the
$\ZZ$-graded vector space~$V_E^{\bullet }=\RR$ is supported in degree~0,
hence the codomain of the curvature~\eqref{eq:162} is~$\Omega ^{n+1}$.)  The
theory~$\alpha $ is \emph{flat} if its curvature~\eqref{eq:162} vanishes.  In
that case $\alpha $~factors through a topological theory
  \begin{equation}\label{eq:163}
     \hat\alpha \:\BNG\longrightarrow h^n_{\RZ}, 
  \end{equation}
where $\hb_{\RZ}$~is the cofiber of $\hb\to \hb\otimes \RR$.  The theory is
\emph{topologically trivial} if its underlying cohomology
class~\eqref{eq:161} vanishes.  A trivialization of the underlying cohomology
class lifts~$\alpha $ to a theory defined by the differential form\footnote{Our convention
is that the partition function for a theory defined by a form $\eta$ is obtained 
by integrating $2 \pi \I \eta$.}
  \begin{equation}\label{eq:168}
     \eta \:\BNG\longrightarrow (\Omega \otimes V_h)^n. 
  \end{equation}
(Compare with Definition~\ref{thm:46} and the following paragraph.)
Conversely, a differential form~\eqref{eq:168} determines an invertible
theory~\eqref{eq:160}, and the isomorphism class of the latter does not
change if the form is shifted by an exact form (or, more generally, by a
closed form whose periods are integral in a sense defined by the cohomology
theory~$\hb$).

  \begin{example}[]\label{thm:65}
 The 3-dimensional spin $\Cx$ Chern-Simons theory in~\S\ref{subsubsec:5.5.4}
is a map
  \begin{equation}\label{eq:164}
     \BNC\longrightarrow \CE^4
  \end{equation}
Recall~\cite{FH2} that the complexified de Rham complex of~$\BNC$ is a
complex polynomial algebra generated by a 2-form~$\omega $ which is
$\sqrt{-1}/2\pi $ times the curvature of the universal $\Cx$-connection.
Then the curvature of the theory~\eqref{eq:164} is the 4-form 
  \begin{equation}\label{eq:165}
     \frac{1}{2}\,\omega \wedge \omega . 
  \end{equation}
The curvature, and indeed the entire theory~\eqref{eq:164}, should be
evaluated on families 
  \begin{equation}\label{eq:167}
     P\xrightarrow{\;\;p\;\;}X\xrightarrow{\;\;\pi \;\;}S 
  \end{equation}
in which $\pi $~is a proper fiber bundle equipped with a relative spin
structure and $p$~is a principal $\Cx$-bundle equipped with a connection.
Then \eqref{eq:165}~pulls back to an element of~$\Omega ^4_{\CC}(X)$.  
\end{example}

The discussions in~\S\ref{subsec:9.2} and~\S\ref{subsec:9.3} have analogs for
invertible field theories.  Thus there is a notion of a flat (or nonflat)
isomorphism of invertible theories.  Furthermore, the equivalence classes of
flat isomorphisms of two $n$-dimensional invertible theories form a torsor
over the abelian group of flat $(n-1)$-dimensional invertible theories.  The
latter are \emph{topological} invertible theories in a strong sense, hence
may be treated via methods of stable homotopy theory.

   \section{$\zt$ gradings}\label{sec:C11}

The invertible spin $\Cx$ Chern-Simons theory is obtained by integration in
the cohomology theory~$E$ with two nonzero homotopy groups;
see~\S\ref{subsubsec:5.2.1}.  (As explained in~\S\ref{subsec:5.5}, one
integrates complex \emph{differential} ``cochains'' in the differential
theory~$\CE$ or, for flat connections, ``cochains'' in the secondary
theory~$E_{\CZ}$.)  The homotopy group~$\zt$, which appears along with the
standard~$\ZZ$ or~$\CZ$, introduces an additional $\zt$-grading in the values
of the field theory.  In this appendix we collect some remarks and results
about this grading.  In particular, we complete the proof of
Proposition~\ref{prop:triangle-iso}.  The bottom line is: with a suitable
universal choice we can and do ignore the $\zt$-gradings in the main text.

  \subsection{$\sV$-lines}\label{subsec:C11.1}

Let $\sV$~be the linear category of complex $\zt$-graded\footnote{As usual,
we use `super' as a synonym for `$\zt$-graded'.} vector spaces; for
$V_1,V_2\in \sV$ the hom space $\sV(V_1,V_2)$ consists of \emph{even} linear
maps $V_1\to V_2$.  Impose the symmetric monoidal structure of tensor product
with the Koszul sign rule.  A \emph{super category} is a complex linear
module category over the tensor category~$\sV$.  Super categories form a
2-category.  The multiplicative units in this 2-category are called
\emph{$\sV$-lines}, and there are two such up to isomorphism: the Bose
line~$\sV$ and the Fermi line~$\GG$; the latter is the category of modules
over the complex Clifford algebra~$\Cliff_1$.  An automorphism of an
$\sV$-line is a functor defined by tensoring with a super line~$L$; the complex
line~$L$ can be even or odd.  We refer to~\cite{FT2} for more details
about super categories. 

  \begin{remark}[]\label{thm:C73}
 The integral of the level of spin Chern-Simons theory over a spin 2-manifold
with flat $\Cx$-connection lies in~$E^1_{\CZ}$.  The nonzero homotopy groups
of~$E^1_{\CZ}$ are $\pi _0E^1_{\CZ}\cong \zt$ and $\pi _1E^1_{\CZ}\cong \CZ$;
there is a nonzero $k$-invariant which connects them.  An equivalent linear
Picard groupoid is that of complex super lines, and so the value of~$\SCx$ on
a closed 2-manifold is a complex super line.  Similarly, on a spin 1-manifold
the relevant space is~$E^2_{\CZ}$, whose nonzero homotopy groups are $\pi
_1E^2_{\CZ}\cong \zt$ and $\pi _2E^2_{\CZ}\cong \CZ$.  An equivalent linear
Picard 2-groupoid is that of \emph{Bose} $\sV$-lines.  In particular, we do
not encounter Fermi $\sV$-lines.  (We would have met Fermi $\sV$-lines had
the level lain in the cohomology theory with an additional homotopy
group~$\zt$.)
  \end{remark}

  \subsection{Spin flip}\label{subsec:C11.2}

A Riemannian spin $m$-manifold is a Riemannian manifold~$M$ equipped with a
principal $\Spin_m$-bundle $\sB_{\Spin}(M)\to M$ which lifts the orthonormal
frame bundle $\sB_{\O}(M)\to M$ under the homomorphism $\Spin_m\to \O_m$.  A
diffeomorphism $f\:M'\to M$ of spin $m$-manifolds is an isometry of the
underlying Riemannian manifolds together with a lift to the
$\Spin_m$-bundles.  The \emph{spin flip}~$\spfl M$ is the
automorphism~$\id_M$ with lift given as multiplication by the central element
of~$\Spin_m$.  
 
Let $F$~be an $n$-dimensional field theory of spin manifolds---say
nonextended---so defined on a bordism category of $(n-1)$-~and $n$-manifolds.
If $Y$~is a closed spin $(n-1)$-manifold, then $F(Y)$~is a super vector
space.  We say that $F$~\emph{satisfies spin-statistics} if the spin
flip~$\spfl Y$ acts on the super vector space $F(Y)=F(Y)^0\oplus F(Y)^1$ as
the grading automorphism $\id_{F(Y)^0}\oplus -\id_{F(Y)^1}$.  Any~$F$ which
is topological, invertible, and reflection positive satisfies
spin-statistics~\cite[\S11]{FH1}.  \emph{Complex} Chern-Simons theories are
not reflection positive, and in any case what we need is spin-statistics for
the difference line of a triangle, which is defined
in~\S\ref{sec:difference-line} via a combination of the $\Cx$~spin
Chern-Simons theory~$\SCx$, the $\SLC$ Chern-Simons theory~$\F{\SLC}$, and
the abelianization isomorphism~\eqref{eq:S1}.  Although we cannot simply
quote~\cite[\S11]{FH1} for what we need, the basic setup pertains, and so we
review it briefly.
 
Let $\sC$ be a symmetric monoidal category, and suppose $x\in \sC$ is a
dualizable object.  Then 
  \begin{equation}\label{eq:C192}
     1\xrightarrow{\;\;\textnormal{coevaluation}\;\;}x\otimes
     x^*\xrightarrow{\;\;\textnormal{symmetry}\;\;} x^*\otimes
     x\xrightarrow{\;\;\textnormal{evaluation}\;\;} 1 
  \end{equation}
is by definition the dimension~$\dim x$.  For example, if $\sC=\sV$ and
$x=V^0\oplus V^1$ is a finite dimensional super vector space, then $\dim
x=\dim V^0-\dim V^1$.  If $f\:x\to x$ is a morphism, then the composition 
  \begin{equation}\label{eq:C193}
     1\longrightarrow x\otimes x^*\longrightarrow x^*\otimes
     x\xrightarrow{\;\;1\otimes f\;\;} x^*\otimes x\longrightarrow 1 
  \end{equation}
is by definition the trace~$\tr f$.  For $\sC=\sV$ and $f$~an (even)
endomorphism of~$V^0\oplus V^1$, this categorical trace~$\tr f=\tr f\res{V^0}
- \tr f\res{V^1}$ is usually called the \emph{supertrace}.  For
$\sC=\Bord_{\langle n-1,n \rangle}(\Spin)$ the spin bordism category and
$x=Y$ a closed spin $(n-1)$-manifold with spin flip~$\spfl Y$, we find
  \begin{equation}\label{eq:C194}
     \begin{aligned} \dim Y&= \nbcir\times Y \\ \tr\spfl Y &=
     \bcir\;\;\;\;\,\times Y\end{aligned} 
  \end{equation}
where `(non)bounding' identifies the spin structure on~$\cir$.  If
$F\:\Bord_{\langle n-1,n  \rangle}(\Spin)\to \sV$ is a topological field
theory, then since $F$~is a symmetric monoidal functor it maps traces to
traces:  
  \begin{equation}\label{eq:C195}
     \tr F(\spfl Y) = F(\bcir\times Y). 
  \end{equation}
If $F$~is invertible, then spin-statistics holds for~$Y$ iff this
equals~$+1$; spin statistics holds for~$F$ iff this equals~$+1$ for all~$Y$.

  \subsection{Spin-statistics for the difference line~$\cL$}\label{subsec:C11.3}

Next, recall the construction of the complex line~$\cL=\cL(D,\epsilon
,\sigma ,\cA)$, defined in~\eqref{def:L} in an equivalent description to what
we give here.  Let $D=D^2$ be the standard spin 2-disk; the spin structure is
denoted~$\sigma $.  The edge orientations~$\epsilon $ which appear on the
triangle~$\Delta $ in~\eqref{def:L} give rise to a universal cover
of~$\partial D=\cir$; see~\ref{sec:edge-orientations}.  The 2-disk~$D$ is
equipped with its standard spectral network (Figure~\ref{fig:triangle}), and
$\cA=(P,Q,\mu ,\theta )$ is stratified abelianization data
(Definition~\ref{thm:6}).  

The line~$\cL$ derives from three ingredients.  First $G=\SLC$ Chern-Simons
theory~$\F G$ on~$D$ produces an $\sV$-line\footnote{The $\SLC$ Chern-Simons
theory does not use the spin structure and it factors through ungraded linear
objects, so $\F G(\partial D)$ is canonically a $\VV$-line.  Here we base
extend to~$\sV$ for consistency with~$\SCx$.}~$\F G(\bD)$ and an isomorphism
of $\sV$-lines 
  \begin{equation}\label{eq:C196}
     \sV\xrightarrow[\cong ]{\;\;\F G(D)\;\;}\F G(\bD). 
  \end{equation}
Let $c=D_{-2b}$ be the center point of~$D$ and let $\tilD'\to D\setminus
\{c\}$ the double cover constructed from the spectral network on~$D$.  The
spin structure~$\sigma $ lifts to a spin structure on the deleted
2-disk~$\tilD'$; $\sigma $~ does \emph{not} extend over the deleted point.
The principal $\ZZ$-bundle derived from~$\epsilon $ is used to twist this
spin structure and the flat $\Cx$-bundle
over~$\tilD'$---see~\S\ref{subsec:S6.3}---to obtain a spin
structure~$\tilde\sigma $ and a flat $\Cx$-bundle over the filled-in
2-disk~$\tilD$.  The second ingredient in~$\cL$ is then the $\Cx$-spin
Chern-Simons invariant
  \begin{equation}\label{eq:C197}
     \sV\xrightarrow[\cong ]{\;\;\SCx(\tilD)\;\;}\SCx(\partial \tilD). 
  \end{equation}
The third ingredient is the isomorphism of $\sV$-lines 
  \begin{equation}\label{eq:C198}
     \F G(\bD)\xrightarrow[\cong]{\;\;\chi (\bD)\;\;}\SCx(\partial \tilD) 
  \end{equation}
constructed in~\S\ref{subsec:S6.4}.  For now we simply note that $\chi
(\bD)$~is the composition of three isomorphisms; we dig into the details
in~\S\ref{subsec:C11.4} below.  Finally, the line~$\cL$ is the composition 
  \begin{equation}\label{eq:C199}
     \cL = \SCx(\tilD)\inv \circ \chi (\bD)\circ \F G(D). 
  \end{equation}
That is, the right hand side of~\eqref{eq:C199} is an $\sV$-linear
automorphism of~$\sV$, hence it is tensoring with a line~$\cL$.  Note that
$\cL$~is $\zt$-graded; it may be even or odd. 
 
Let $\spfl D$ denote the spin flip of the spin 2-disk~$(D,\sigma )$.  It
induces the spin flip~$\spfl\tilD$ of~$(\tilD,\tilde\sigma )$, and these spin
flips of 2-disks restrict to the spin flip of the boundary circles.  In this
way $\spfl D$~induces an involution on each of~\eqref{eq:C196},
\eqref{eq:C197}, and~\eqref{eq:C198}, hence too an involution~$(\spfl
D)_*\:\cL\to \cL$ on the line~$\cL$ which is their composition.

  \begin{proposition}[]\label{thm:C74}
 $(\spfl D)_* = \id_{\cL}$ if $\cL$~is even, and  $(\spfl D)_* = -\id_{\cL}$
if $\cL$~is odd. 
  \end{proposition} 

  \begin{proof}
 Define an invertible 1-dimensional field theory~$f$ of spin manifolds as
follows.  If $M$~is a compact 0-dimensional spin manifold or a 1-dimensional
spin bordism, then form the spin manifold~$M\times D$ and equip it with the
structure pulled back from~$(\epsilon ,\cA)$ on~$D$.  Then $f(M)$~is computed
as the composition in~\eqref{eq:C199} applied to $M\times -$\,:
  \begin{equation}\label{eq:C200}
     f(M)= \SCx(M\times \tilD)\inv \circ \chi (M\times \bD)\circ \F G(M\times
     D).  
  \end{equation}
Note that $f(\pt)=\cL$.  By~\eqref{eq:C195} the supertrace of~$(\spfl D)_*$ is 
  \begin{equation}\label{eq:C201}
     \tr(\spfl D)_*=\tr f(\spfl{\pt})=f(\bcir). 
  \end{equation}
Hence the proposition follows if we prove that $f$~satisfies spin-statistics. 
 
Compute $f(\bcir)$ as the composition~\eqref{eq:C200}.  Write $\bcir=\partial
D^2$.  The complex lines $\F G(\bcir\times \bD)$ and $\SCx(\bcir\times \bDt)$
have nonzero (basis) elements $\F G(D^2\times \bD)$ and $\SCx(D^2\times
\bDt)$.  Furthermore, since $\chi $~is a map of theories, the linear
isomorphism 
  \begin{equation}\label{eq:C202}
     \chi (\bcir\times \bD)\:\F G(\bcir\times \bD)\longrightarrow
     \SCx(\bcir\times \bDt) 
  \end{equation}
maps $\F G(D^2\times \bD)$ to $\SCx(D^2\times \bDt)$.  It remains to compute
the ratio of each of the vectors $\F G(\bcir\times D)$ and $\SCx(\bcir\times
\tilD)$ and these basis elements.  The first ratio is $\F G$ evaluated on
  \begin{equation}\label{eq:C203}
     (\bcir\times D)\cup -(D^2\times \bD), 
  \end{equation}
and the second ratio is $\SCx$~evaluated on 
  \begin{equation}\label{eq:C204}
     (\bcir\times \tilD)\cup -(D^2\times \bDt), 
  \end{equation}
where the minus sign denotes the opposite spin structure (and reverse
orientation).  But each of these is a 3-sphere with the trivial connection,
and so each partition function is~$+1$.  It follows that $f(\bcir)=+1$, as
claimed.  
  \end{proof}

  \begin{remark}[]\label{thm:C75}
 In the proof we encounter the 3-sphere~$S^3$ in~\eqref{eq:C203} with the
following spectral network.  The branch locus is an embedded $\cir\subset
S^3$.  The walls are three disjoint embedded open 2-disks~$B^2$, each of
which has boundary $S^1\subset S^3$. 
  \end{remark}

  \subsection{The freedom to eliminate odd lines}\label{subsec:C11.4}

The isomorphism~$\chi $ is the composition of three isomorphisms: see
Theorem~\ref{thm:strat-abelianization}.  One of them, \eqref{eq:S2}, is given
by a canonical construction in Theorem~\ref{thm:44}.  We fixed another,
\eqref{eq:S4}, in Corollary~\ref{thm:72}.  The remaining isomorphism~$\nu $
in~\eqref{eq:S3} derives from Corollary~\ref{thm:63}, which we did not fix
completely.  Namely, since $\nu $~is an isomorphism of invertible
3-dimensional topological theories, we can tensor it with an invertible
2-dimensional theory to obtain a new isomorphism.  The background fields
of~$\nu $---listed at the beginning of~\S\ref{subsubsec:5.2.4}, are: a spin
structure~$\sigma $, a flat $H$-connection, and a principal $\ZZ$-bundle
which lifts the associated principal $\bmut$-bundle~$\delta $.  Observe that
the 2-dimensional invertible theory whose partition function is\footnote{This
is the Arf invariant (0 or~1) of the shifted spin structure $\sigma +\delta
$.} $(-1)^{\Arf(\sigma +\delta )}$ takes value the \emph{odd} line
on~$\bcir$.  Hence, possibly after tensoring~$\nu $ with this shifted Arf
theory, we can arrange that $\cL$~be an \emph{even} complex line.  It then
follows from Proposition~\ref{thm:74} that the spin flip acts as~$\id_{\cL}$,
which is the claim in Proposition~\ref{prop:triangle-iso}.

\newpage
 \bigskip\bigskip
 \bibliographystyle{hyperamsalpha} 
 \bibliography{Specnet}

\end{document}